\documentclass[14pt]{extarticle}
\usepackage[utf8]{inputenc}	
\usepackage[T2A]{fontenc}
\usepackage[english]{babel}
\usepackage{amsmath,amsthm,amsfonts,amssymb,amscd}
\usepackage{mathrsfs}
\usepackage{tikz}

\makeatletter
\@addtoreset{equation}{section}
\makeatother

\DeclareMathOperator{\Ree}{Re}
\DeclareMathOperator{\Imm}{Im}
\DeclareMathOperator{\Ker}{Ker}
\DeclareMathOperator{\Clos}{Clos}
\DeclareMathOperator{\Span}{Span}
\DeclareMathOperator{\Tr}{Tr}
\DeclareMathOperator{\dist}{dist}
\newcommand{\Wo}{\mathop{\rm W^1_2}\limits}
\newcommand{\Wp}{\mathop{\rm W^1_3}\limits}
\newcommand{\Wq}{\mathop{\rm W_1}\limits}
\newcommand{\Wr}{\mathop{\rm W_{-1}}\limits}
\newcommand{\slim}{\mathop{{\rm s\mbox{-}lim}}\limits}

\newcommand{\ph}{\varphi}
\newcommand{\ep}{\varepsilon}

\newcommand{\yt}{\widetilde y_k}
\newcommand{\A}{\mathscr{A}}
\newcommand{\HH}{\mathscr{H}}
\newcommand{\la}{\lambda}

\theoremstyle{plain}
\newtheorem{theorem}{Theorem}[section]

\newtheorem{lemma}{Lemma}[section]
\newtheorem{proposition}{Proposition}[section]
\newtheorem{corollary}{Corollary}[section]
\newtheorem{problem}{Exercise}[section]

\theoremstyle{definition}
\newtheorem{definition}{Definition}[section]

\newtheorem{note}{Note}[section]

\theoremstyle{remark}
\newtheorem*{rem}{Remark}

\newenvironment{pf}
{\par\noindent{\it Proof.}}
{\hfill$\scriptstyle\blacksquare$~\\}

\renewcommand{\leq}{\leqslant}
\renewcommand{\geq}{\geqslant}

\renewcommand{\phi}{\varphi}
\renewcommand{\epsilon}{\varepsilon}

\renewcommand{\Re}{\Ree}
\renewcommand{\Im}{\Imm}

\newcommand{\bigzero}{\mbox{\normalfont\Large\bfseries 0}}
\newcommand{\rvline}{\hspace*{-\arraycolsep}\vline\hspace*{-\arraycolsep}}

\begin{document}
\begin{titlepage}
\makeatletter
\def\@makefnmark{}
\makeatother
{\Large \begin{center} Operator Pencils \\ and Half-range Problem in Operator Theory
 \end{center}}
\bigskip

\centerline{A.~A.~Shkalikov\footnotemark[0]}
\footnotetext[0]{{\normalsize This work is supported by Russian Science Foundation, grant  No 17-11-01215.}}

\begin{center}
Lomonosov Moscow State University,

Department of Mechanics and Mathematics

email: shkalikov@mi.ras.ru
\end{center}

\bigskip
{\centerline{\bf Abstract}}

\medskip
This article can be considered as the first version of a book which  the author  plans to write about half-range  problems  in operator theory. It consists of two parts. The first part is  based  on lectures which the author delivered at University of Calgary and Lomonosov Moscow State University.  The main attention in this part is paid  to the selection of  waves which are involved in the formulation of the Mandelstamm radiation principle (the eigen-pairs,  corresponding to the real eigenvalues)  and to the factorization problems  of self-adjoint and dissipative,   quadratic and polynomial operator pencils. There is  a dramatic difference  between finite dimensional and infinite dimensional  cases.  It is shown that in the finite dimensional case the factorization problems can  be solved completely.
In the second part  we consider abstract models  for concrete problems  of mechanics.  We  demonstrate the methods how concrete problems can be represented in an abstract form. The main results concern the factorization of elliptic operator pencils  satisfying the resolvent growth condition in a double sector  containing the real axis and the investigation of the semi-group properties of a divisor. Using Pontrjagin space methods we obtain a criterium for the stability in the celebrated Sobolev problem about a rotating top with a cavity filled with a viscous liquid.

\medskip


\bigskip
\medskip
\end{titlepage}

\pagestyle{myheadings}

\tableofcontents

%\part{}

\newpage
\centerline{{\huge Part 1}}
%Лекция 01 (05.9.2017)
\section{Operator pencils and Cauchy problem. The finite dimensional case}
\medskip

Among key-stones which ly in the base of the topic represented in this article I would like  to mention three ones. The first one is  due L.~S.~Pontrjagin [P], who proved (1944) the existence of  maximal semi-definite invariant subspaces
 for self-adjoint (and later for dissipative) operators in Pontrjagin space.
 The second one is due M.~V.~Keldysh [K1,  K2], who began to investigate(1951) spectral properties of operator pencils and suggested an analytic  approach to prove the completeness properties of the root functions of wide class of non-self-adjoint operators and operator pencils. The third one is due to M.~G.~Krein and H.~Langer [KL] who proved (1964) the first factorization theorem for self-adjoint quadratic operator pencils in the infinite dimensional case using a generalization of Pontrjagin theorem. Since that time many mathematicians  were involved in the investigation of related problems.
 Here we shall present some our points of view  to this subject.
 
Let us come to the subject: we deal with operator polynomials
\begin{equation}
A(\la) = A_0 + \la A_1 + \cdots + \la^n A_n
\end{equation}
where $A_j, \quad j = 0,1,\dots, n$ are operators in Hilbert space $H$. Further, we will consider the different situations: H is finite dimensional; H is infinite dimensional and $A_j$ are bounded operators; $A_j$ are unbounded operators (in this case $H$ has to be infinite dimensional).

I would like to mention two origins for the study of operator pencils.

\textit{First} the algebraic origin, by which I mean the theory of matrices. This is understandable, because the eigenvalue problem for operator pencil $A(\la)$ is a generalization of the eigenvalue problem for monic linear operator polynomials $A(\la) = A - \la I$, i.e. the classical eigenvalue problem for matrix $A$.

\textit{Second} the Fourier method for solving equations with operator coefficients of the form
\begin{equation}
A (-i\frac{d}{dt}) u(t) = A_0 u - i A_1 \frac{du}{dt} + \cdots + (-i)^n A_n \frac{d^n u}{dt^n} = 0,
\end{equation}
where $u(t)$ is the function (determined for example for $t \geq 0$) with values in Hilbert space $H$.

Certainly, the problem of solvability of such equations and the problem of stability of their solutions are closely connected with some problems on factorization of operator pencil, with problems of its eigenvalue distribution, and with problems on basis properties, completeness and minimality of its eigenfunctions (the latter concept can be considered in different senses: for example, n-multiple completeness and half range completeness).

Now we clarify the connection between operator pencil (1) and differential equation (2). First we can consider the operator pencil $A(\la)$ as the characteristic polynomial of the differential equation to arrive at $A(\la)y = 0$ as the characteristic equation. We say $y^0 \neq \emptyset$ \textit{is an eigenvector corresponding to eigenvalue $\la_0$ of the operator pencil $A(\la)$ if $A(\la_0)y^0 = 0$}. We say also, that $y^1, \dots,y^p$ \textit{is the sequence of vectors associated with $y^0$ if}

$$A(\la_0)y^s + \frac{1}{1!} A' (\la_0) y^{s-1} + \cdots + \frac{1}{s!} A^{(s)} (\la_0)y^0 = 0$$

\textit{for all $s = 1, \dots, p$}. The sequence $y^0, y^1, \dots, y^p$ we call the chain of eigen and associated vectors (EAV) corresponding to eigenvalue $\la_0$ and we say $p+1$ is the length of that chain.

A simple verification shows that if $y^0, y^1, \dots, y^p$ is a chain of EAV corresponding to eigenvalue $\la_0$ then for all $s = 0,1, \dots, p$ the functions
$$
u^s(t) = e^{i\la_0 t}[y^s + \frac{it}{1!}y^{s-1} + \cdots + \frac{(it)^s}{s!}y^0]
$$
are solutions of differential equation (2). These functions we call \textit{elementary solutions} of the equation (2).

Now assume that $\dim \Ker A(\la_0) = \ell < \infty$. Let $y_1^0, \dots, y_\ell^0$ be a basis in subspace $\Ker A (\la_0)$ and
\begin{equation}
y_k^0, y_k^1, \dots , y_k^{p_k}, \quad j = 1, \ldots, \ell
\end{equation}
are the chains of EAV corresponding to $\la_0$ of the maximum possible length $p_k + 1$. The number $m = (p_1+1) + \cdots + (p_\ell + 1)$ certainly depends on the choice of the basis $\{y_k^0\}$. If $m < \infty$ for any choice of $\{y_k^0\}$ then there exists a basis $\{y_k^0\}$ such that the corresponding number $m$ has the maximum possible value, say $m_0$. That number $m_0$ is called the \textit{algebraic multiplicity} f the eigenvalue $\la_0$ and the chains (3) consisting of $m_0$ elements numbered in such a way that $p_1 \geq p_2 \geq \cdots \geq p_\ell$, are called a \textit{canonical system of eigen and associated elements}. Certainly canonical system is not unique. The number $\ell = \dim \Ker A(\la_0)$ is called \textit{the geometric multiplicity} of the eigenvalue $\la_0$.

Note, if $A(\la) = A - \la I$ then the canonical system of EAV (3) coincides with  Jordan chains corresponding to eigenvalue $\la_0$ (see [Lancaster and Tismenetsky, \S 6.4]). Note also that if the operator polynomial $A(\la)$ is not linear then one can not assert the linear independence of the elements of the system (3). Some of elements $y_k^s$ for $1 \leq s \leq p_k$ may even be equal to zero.

Let us try to find the solution of the equation (2) satisfying the initial conditions
\begin{equation}
(-i)^j u^{(j)} (0) = \varphi_j, \quad j = 0,1, \ldots, n-1.
\end{equation}
The problem (2), (4) is called the Cauchy problem. Following the Fourier method we try to find its solution in the form
\begin{equation}
u(t) = \sum_{s,k} c_k^s e^{i\la_k t} (y_k^s + \frac{it}{1!} y_k^{s-1} + \cdots + \frac{(it)^s}{s!}y_k^0) = \sum_{s,k} c_k^s u_{k,s} (t)
\end{equation}
where the $c_k^s$ are unknown coefficients, $y_k^s, \quad s = 0, 1, \ldots ,p_k$, are the elements of canonical systems (3) corresponding to all eigenvalues $\la_k$ of the pencil $A(\la)$. here we avoid the introduction of the third index if addition to $s$ and $k$ and assume that canonical system (3) corresponds to eigenvalue $\la_k$ (instead of $\la_0$) and $\la_k$ is repeated in the sum (5) as many times as its geometric multiplicity.

Using (5) we can rewrite the initial conditions (4) in the form
\begin{equation}
\begin{bmatrix}
\ph_0 \\
\ph_1 \\
\cdots\\
\ph_{n-1}
\end{bmatrix}
= \sum_{s,k} c_k^s
\begin{bmatrix}
y_k^{s,0} \\
y_k^{s,1} \\
\cdots\\
y_k^{s,n-1}
\end{bmatrix}
= \sum_{s,k} c_k^s \widetilde{y}_k^s,
\end{equation}
where for $r = 0,1, \ldots, n-1$
$$y_k^{s,r} = (-i)^r u_{k,s}^{(r)}(0) = (-i)^r \frac{d^r}{dt^r} [e^{i\la_k t} (y_k^s + \cdots + \frac{(it)^s}{s!} y_k^0)] \big|_{t=0} = $$
$$= \frac{d^r}{dt^r} [e^{\la_k t}(y_k^s + \frac{t}{1!} y_k^{s-1} + \cdots + \frac{t^s}{s!} y_k^0)] \big|_{t=0} $$
\textit{The elements} $\yt^s \in H^n = H\times H \times \cdots \times H$ \textit{are called the Keldysh derived chains constructed from canonical system (3)}.
If the eigenvalue $\la_k$ is \textit{semi-simple} (this is the case when there are no associated vectors) then the Keldysh derived chains have the representation
$$
\yt = ( y_k, \la_k y_k, \ldots , \la_k^{n-1} y_k).
$$

Let us assume that $\dim H < \infty$ and let $\det A(\la) \not\equiv 0$. In this case the pencil $A(\la)$ has a finite number of eigenvalues and to establish the Fourier method for the Cauchy problem (2), (4) we have to show that the system of Keldysh derived chains $\{\yt^s\}$ is a basis in $H^n$ (then the Cauchy problem will be solvable for any set of initial vectors $\ph_0, \ldots, \ph_{n-1}$).

\begin{theorem}
Let $\dim H < \infty$ and $\det A(\la) \not\equiv 0$. Then a system of Keldysh derived chains $\yt^s$ is basis in $H^n$ if and only if $\Ker A_n = 0$.
\end{theorem}

\begin{pf}
With pencil $A(\la)$ we associate the following linear pencil in space $H^n$
\begin{equation}
\A (\la) = \A_0 - \la \A_1
\end{equation}
where
$$\A_0 =
\begin{bmatrix}
A_0 & A_1 & \ldots & A_{n-1}\\
0 & I & \ldots & 0\\
\vdots & \vdots & \ddots & \vdots \\
0 & 0 & \ldots & I
\end{bmatrix},
\quad \A_1 =
\begin{bmatrix}
0 & 0 & \ldots & 0 & -A_n\\
I & 0 & \ldots & 0 & 0\\
0 & I & \ldots & 0 & 0\\
\vdots & \vdots & \ddots & \vdots & \vdots\\
0 & 0 & \ldots & I & 0
\end{bmatrix}
$$
A simple verification (see for example [Keldysh 1], [Markus 1]) shows that the EAV (3) are the chains of eigen and associated vectors corresponding to eigenvalue $\la_k$ of the pencil $A(\la)$ if and only if the Keldysh derived chains
\begin{equation}
\yt^0, \yt^1, \ldots , \yt^{p_k}
\end{equation}
are EAV of the linear pencil (7) or linear operator $\A_1^{-1} \A_0$ acting in the space $H^n$ (note that $\A_1$ is invertible if $\Ker A_n = 0$). But the system of EAV of any linear operator in finite dimensional space is a basis. Hence the system of Keldysh derived chains is a basis in $H^n$.

To show that $\{\yt^s\}$ is not a basis if $\Ker A_n \neq 0$ we can assume without loss of generality that $A_0$ is invertible. (Otherwise we have to shift $\la \rightarrow \la + \la_0$, where $\la_0$ is a point such that $A(\la_0)$ is invertible. We can find such a point because $\det A(\la) \not\equiv 0$.) Then the system of EAV for pencil (7) coincides with a system of EAV for $I - \la \A_0^{-1}\A_1$ and in turn coincides with EAV for operator $\A_0^{-1}\A_1$ with the exception of a canonical system corresponding to the eigenvalue $\mu = 0$. The operator $\A_0^{-1} \A_1$ is singular, hence the algebraic multiplicity of the eigenvalue $\mu = 0$ is equal to $k > 0$. Then the system Keldysh derived chains has defect $k$. Theorem 1 is proved.
\end{pf}

\newpage
%Лекция 02 (12.9.2017)
\section{Theorem on holomorphic operator function}

In attempting to generalize theorem 1 to infinite-dimensional spaces one comes up against some deep problems. Under the assumption that the spectrum of the pencil $A(\la)$ is descrete we will show the minimality of Keldysh derived chains in the space $H^n$. Under some reasonable additional assumptions we will sketch the proof of its completeness. But the basis property, as a rule, does not hold. Even for the simple pencil $A(\la) = I - \la^2C$, where $C$ is a self-adjoint positive compactoperator in $H$ the Keldysh derived chains do not form a basis in $H^2$. Nevertheless, for some pencils it is possible to find the space $\HH$ which is embedded in $H^n$ and such that the system $\{\yt^s\}$ consisting of Keldysh derived chains has the basis property in $\HH$. Some  results of this kind may be found in the recent paper [Shkalikov 1, \S 2.3]. They are based on eigenexpansion theorems for p-subordinate linear operators due to V. Kaznelson, A. Markus and V. Matsaev (see [Markus 1], for
  example).

For convenience we give the definitions ot the concepts which we have mentioned.

\begin{definition}
The system $\{y_k\}_1^\infty$ is complete in Hilbert space $H$ if from the equalities
$$(y_k, x) = 0, \quad k = 1, 2,\ldots ,$$
it follows that $x=0$.
\end{definition}

\begin{problem}
(see [A. Kolmogorov and S.Fomin 1]). The system $\{y_k\}_1^\infty$ in separable Hilbert space $H$ is complete if and only if it is dense in $H$, i.e. for any $x \in H$ and for any $\ep > 0$ there exists a linear combination
$Y_N = c_1y_1 + \cdots + c_N y_N$ such that $\|Y_N - x \| < \ep$.
\end{problem}

\begin{definition}
The system $\{y_k\}_1^\infty$ is minimal in Hilbert space $H$ if there exists a system $\{z_j\}_1^\infty \in H$, such that
$$(y_k, z_j) = \delta_{kj}, \quad j, k = 1,2, \ldots ,$$
where $\delta_{kj}$ is the Kronecker symbol.
\end{definition}

\begin{problem}
The system $\{y_k\}_1^\infty$ is minimal in Hilbert space $H$ if and only if for all $j = 1, 2, \ldots ,$
$$y_j \not\in \Clos \{y_1, \ldots, y_{j-1}, y_{j+2}, \ldots\}$$
(by $\Clos \{x_1, x_2, \ldots, \}$ we denote the closure of the linear span of the set $\{x_k\}_1^\infty$).
\end{problem}

\begin{definition}
The system $\{y_k\}_1^\infty$ is a basis in Hilbert space $H$ if $\|y_k\| \asymp 1$ (i.e. $c_1 \leq \|y_k\| \leq c_2$ with some positive constants $c_1, c_2$ independent of $k$) and any element $y \in H$ can be uniquely represented by a series
\begin{equation}
y = \sum_{k=1}^\infty c_k y_k
\end{equation}
with some coefficients $\{c_k\}$ and this series strongly converges in $H$. If this series converges unconditionally for any $y\in H$ then the basis $y_k$ is called an unconditional basis or a Riesz basis.
\end{definition}

\begin{note}
It is not a simple exercise to give an example of a basis which is not a Riesz basis. It was K. Babenko who proved in 1948 that the system $\{y_k(x)\}_1^\infty = \{x^\alpha \sin kx\}_1^\infty$ is a basis in $L_2 [0, \pi]$  but not a Riesz basis, provided that  $-1/2 <\alpha 1/2$.
\end{note}

The following theorem allows us to give another definition of a Riesz basis.

\begin{theorem}
The system $\{y_k\}_1^\infty$ in Hilbert space $H$ forms a Riesz basis if and only if there exists an orthogonal basis $$\{e_k\}_1^\infty$$ in $H$ and a bounded invertible operator $A$ such that
$$A e_k = y_k, \quad k = 1, 2, \ldots .$$
\end{theorem}

The proof of this theorem can be found in [Gohberg, Krein 1, Ch.6]. More precise historical comments are given in [Nikolskii 1].

If the system $\{y_k\}_1^\infty$ is a basis in $H$ then the coefficients $c_k = c_k(y)$ in the representation (1) are linear functionals in $H$. Since these functionals are defined for all $y \in H$, we have by virtue of the Banach-Steinhaus theorem that they  are bounded. Therefore, using the Riesz theorem we can find elements $\{z_k\} \in H$ such that $c_k(y) = (y, z_k)$. Since the representation (1) for element $y = y_j$ is unique we have $(y_j, z_k) = \delta_{jk}$. The theorem $\{z_k\}$ is said to be \textit{adjoint to} $\{y_k\}$. It is known (see [Gohberg, Krein, Ch.6], for example) that if $\{y_k\}$ is a basis  (or a Riesz basis) then $\{z_k\}$ is too.

Hence any basis $\{y_k\}$ in $H$ is a minimal and, obviously, complete system. The converse assertion is certainly not true. For example, the system $\{y_k\}_1^\infty$, where
$$y_1 = \{1,0,0,\ldots\}, \quad y_2 = \frac{1}{\sqrt 2} \{1,1,0,0, \ldots\} , \ldots, y_k = \frac{1}{\sqrt k} \{1,1,\ldots, 1,0,0,\ldots \}$$
is minimal and complete in $H = \ell_2$, bet it is not a basis in $\ell_2$. (Hint: if $\{e_k\}_1^\infty$ is the standard basis in $\ell_2$ then the operator $A$ defined by equalities $Ae_k = y_k$ is not invertible.)

If the system is complete and minimal but is not a basis then it can have some intermediate property.

\begin{definition}
The system $\{y_k\}_1^\infty$ is a basis with parenthesis in Hilbert space $H$ if there exists a sequence of integers $\{m_s\}_1^\infty$ ($m_1 = 0$) such that any element $y \in H$ can be uniquely represented by a series
$$y = \sum_{s=1}^\infty \left(\sum_{k = m_s+1}^{m_{s+1}} c_k Y_k\right) = \sum_{s=1}^\infty Y_s$$
and the series $\Sigma Y_s$ strongly converges in $H$.
\end{definition}

\begin{definition}
Let $\Lambda = \{\la_k\}_1^\infty$ be a sequence of complex numbers such that for some $\alpha > 0 \quad \Ree \la_k^\alpha \geq 0$ for all $k$ sufficiently large (we take the main branch of $\la^\alpha$, i.e. $\la^\alpha > 0$ if $\la > 0$). The minimal system $\{y_k\}_1^\infty \in H$ is a basis for the Abel method of summability of order a $\alpha$ with respect to sequence $\Lambda$ if there exists a sequence of integers $\{m_s\}_1^\infty$ ($m_1 = 0$) such that for any $x\in H$ the series
$$x(t) = \sum_{s=1}^\infty \left(\sum_{k = m_s+1}^{m_{s+1}} e^{-\la_k^{\alpha_t}} (x, y_k^*) y_k\right) = \sum_{s=1}^\infty X_s(t)$$
(the system $\{y_k^*\}_1^\infty$ is adjoint to $\{y_k\}_1^\infty$) strongly converges for any $t>0$ and $\|x(t) - x\| \rightarrow 0$ if $t \rightarrow +0$.
\end{definition}

It was V.Lidskii [1] who introduced this method for summability of Fourier series $\sum (x, y_k^*)y_k$ with respect for systems $\{y_k\}_1^\infty$ of EAV of compact operators $A$. In this situation $\Lambda = \{\la_k\}$ is the sequence of eigenvalues of $A^{-1}$ (the definition of $x(t)$ slightly changes if $\la_k$ are not semi-simple eigenvalues of $A^{-1}$).

According to a theorem of Hilbert the system $\{y_k\}_1^\infty$ of eigenvectors of self-adjoint compact operator $A$ acting in Hilbert space $H$ is an orthogonal basis in $H$. The system $\{y_k\}_1^\infty$ of EAV of non-self-adjoint compact operator $A$ corresponding to eigenvalues $\la_k \neq 0$ forms a minimal system, because the system $\{y_k^*\}_1^\infty$ of EAV of the operator $A^*$ is adjoint to $\{y_k\}_1^\infty$. Certainly, for a nonself-adjoint operator, the system  of its EAV is not always complete, For example the compact operator
$$Ay(x) = \int_0^x y(t) dt$$
in the space $H = L_2[0,1]$ has no eigenvectors. But if the property of completeness is proved, then one can try to prove the basis property, or the basis property for the Abel method of summability. The investigations in this field were very intensive and a number of deep and refined results were obtained. The reader can make acquaintance with some of them in books [Cohberg, Krein 1], [Markus 1], [Agranovich 1]. We will  touch this topic again in a subsequent lecture.

The proof of the completeness theorem in Section~1 depended on the finite dimensional context. To give the new approach we have to start from an important theorem on holomorphic operator functions. First we have to recall some definitions.

We say $A(\la)$ is an analytic vector function of complex variable $\la$ with values in Hilbert space $H$ and defined in a domain $\Omega \subset \mathbb{C}$, if at each point $\la \in \Omega$ the ration
$$\frac{A(\la+h) - A(\la)}{h}$$
converges in the norm of $H$ to a limit $A'(\la)$ if $h\rightarrow 0$.

Futher by $(\cdot, u) v$ we denote one dimensional operator $V$ such that $V_y = (y,u)v$. Obviously, $V^* = (\cdot, v) u$. By $\sigma_\infty$ we denote the class of compact operators in $H$.

The following result is due to Keldysh $[1,2]$ (its first part was independently proved by I. Gohberg).

\begin{theorem}
Let $A(\la) = A_0 + S(\la)$, where $S(\la)$ is an holomorphic operator function in a domain $\Omega$ and $S(\la) \in \sigma_\infty$ for each $\la \in \Omega$. Also, let there exist a point $\la_0 \in \Omega$ such that the operator $A(\la_0)$ is invertible. Then $A^{-1}(\la)$ is a meromorphic operator function in $\Omega$, i.e. it can be represented in the form $A^{-1} (\la) = D(\la)/ \Delta(\la)$, where $D(\la)$ is a holomorphic operator function and $\Delta (\la)$ is a holomorphic scalar function in $\Omega$. The principal part of the function $A^{-1}(\la)$ at the pole $\la = c$ has the representation

\begin{multline}
\sum_{k=1}^\ell \left[\frac{(\cdot, z_k^0) y_k^0}{(\la - c)^{p_k+1}} + \frac{(\cdot, z_k^1)y_k^0 + (\cdot, z_k^0) y_k^1}{(\la - c)^{p_k}} + \cdots \right.\\ \left. +
\frac{(\cdot, z_k^{p_k})y_k^0 + (\cdot, z_{p_k-1})y_k^1 + \cdots + (\cdot, z_k^0)y_k^{p_k}}{\la - c}\right]
\end{multline}
where
\begin{equation}
y_k^0, y_k^1, \ldots, y_k^{p_k}, \quad k = 1, \ldots, \ell
\end{equation}
is an arbitrary canonical system of $A(\la)$ corresponding to eigenvalue $c$ and
\begin{equation}
z_k^0, z_k^1, \ldots, z_k^{p_k}, \quad k = 1, \ldots, \ell
\end{equation}
is the canonical system of operator function $A^*(\la) = [A(\overline{\la})]^*$ corresponding to eigenvalue $\overline{c}$. The adjoint canonical system (4) is  uniquely determined by the given canonical system (3).
\end{theorem}

\begin{pf}
First, note that this theorem generalizes the well known Fredholm theorem for linear operator function $A(\la) = I - \la A, \quad A\ \in \sigma_\infty$. Then note that without loss of generality we can assume that $A_0 = I$. Otherwise we can shift $\la \rightarrow \la + \la_0$ and consider the operator function $A^{-1}(\la_0) [A(\la_0) + (S(\la+\la_0) - S(\la_0))] = I +S_1(\la)$, where $S_1(\la) \in \sigma_\infty$.
Let $\{e_m\}_1^\infty$ be an orthogonal basis in $H$ and $\{P_m\}$ be the set of orthogonal projectors such that $P_m(H) = \Span \{e_k\}_1^m$. Obviously, $P_m \rightarrow I$ and $Q_m = I - P_m \rightarrow 0$ if $m \rightarrow \infty$ in the strong operator topology (i.e. $\|P_m x-x\| \rightarrow 0$ for each $x \in H$). Let $\overline{\Omega}_1$ be a closed domain in $\Omega$. Since the operator $S(\la)$ is compact we have for any fixed $\la \in \overline{\Omega}_1, \|Q_m S(\la) \| \rightarrow 0$ if $m \rightarrow \infty$. The operator function $Q_mS(\la)$ is holomorphic, hence for any $\delta > 0$ there exists an $\ep > 0$ and $m+0 = m_0 (\la)$ such that for all $m \geq m_0$ and all $\mu$ for which $|\mu - \la| \leq \ep$ we have $\|Q_mS(\mu)| \leq \delta$. We can cover the domain $\overline{\Omega}_1$ by such discs and then choose a finite subcover. Thus for all $\mu \in \overline{\Omega}_1$ and all $m \geq m_1$, we obtain
\begin{equation}
\|Q_mS(\mu)\| \leq \delta < 1
\end{equation}
where $m_1$ does not depend on $\mu$ but only on $\Omega_1$.

To find the inverse operator $A^{-1} (\la)$ we have to solve the equation
\begin{equation}
[I+S(\la)]x = f, \quad f\in H.
\end{equation}
Take any $m\geq m_1$ and denote $P_m = P, \quad Q_m = Q$. We can rewrite (6) in the form
\begin{equation}
Pc + PS(\la) x = Px + PS(\la) Px + PS(\la) Qx = Pf,
\end{equation}
\begin{equation}
Qx + QS(\la) x = Qx + QS (\la) Qx + QS (\la) Px = Qf.
\end{equation}

From (8) it follows that
$$[I +QS(\la)] Qx = Q[f-S(\la) Px].$$
Remembering that $Q=Q_m$ and taking into account (5) we obtain
$$Qx = [I+QS(\la)]^{-1} Q[f-S(\la)Px].$$
Using this equality we can rewrite (7):
\begin{equation}
P[I+S(\la) - S(\la)[I+QS(\la)]^{-1}QS(\la)] Px = P[f-PS(\la)[I+QS(\la)]^{-1}Qf]
\end{equation}
If $x = \sum x_k e_k$ then $Px= x_1e_1 + \cdots + x_m e_m$. This means that the equation (9) represents an algebraic system of $m$ equations with unknown variables $x_1, \cdots, x_m$. Note, that the equation (6) has a unique solution if $\la = \la_0$, hence the equation (9) has too. This means that the determinant $d(\la)$ of the algebraic system (9) is non-zero at $\la_0$. Now $S(\la)$ is holomorphic in $\Omega_1$ and so is $d(\la)$, moreover $d(\la) \not\equiv 0$. From (9) we obtain
$$Px = \frac{1}{d(\la)} F(\la) P[f-PS(\la)(I+QS(\la))^{-1}Qf]$$
where $F(\la)$ is a holomorphic operator function in $\Omega_1$. Then
\begin{multline*}
	A^{-1} (\la)f = x = Px + Qx = [(I+QS(\la))^{-1}Qf] + {}\\
		\frac{1}{d(\la)} [F(\la)P - QS(\la) F(\la) P][f-PS(\la)(I+QS(\la))^{-1}Qf]
\end{multline*}
Hence the operator $A^{-1} (\la)$ exists with the exception of some finite number of poles in $\overline{\Omega}_1 \subset \Omega$. Since $\overline{\Omega}_1$ is an arbitrary closed sub-domain in $\Omega$. we obtain the first assertion of the theorem.

The complete proof of the second statement is technically difficult. We sketch here only the main idea, for more details we refer the reader to the original paper [Keldysh 1]. Suppose that the principal part of the resolvent $A^{-1}(\la)$ for the pole $\la = c$ has the form
$$\frac{R_0}{(\la-c)^{m-1}} + \frac{R_1}{(\la-c)^{m}} + \cdots + \frac{R_m}{\la-c},$$
where $R_s$ are some operators in $H$ and $R_0 \neq 0$. Then we can write
\begin{equation*}
\begin{gathered}
x = A(\la) A^{-1}(\la)x = \\
= \left[A(c) + \frac{1}{1!} A'(c) (\la-c) + \cdots\right] \left[\frac{R_0 x}{(\la-c)^{m-1}} + \cdots + \frac{R_m x}{\la-c} + R(\la)x \right],
\end{gathered}
\end{equation*}
where $R(\la)$ is holomorphic at $\la = c$. The left side of this equality has no pole, hence the coefficients of powers of $(\la-c^){-\nu}, \quad \nu = m + 1, m, \ldots, 1$, on the right side are equal to zero. It follows that
\begin{align*}
& A(c)R_0 x =0, \\
& A(c)R_1 x + \frac{1}{1!} A'(c) R_0 x = 0, \\
& \hbox to 3cm {\dotfill}\\
& A(c) R_m x + \frac{1}{1!} A'(c) R_{m-1} x + \cdots + \frac{1}{m!} A^{(m)}(c) R_0 x = 0.
\end{align*}
This means that for each $x \in H$ the sequence $R_0 x, R_1 x, \cdots, R_m x$ is a chain of EAV of length $m+1$. From the definition of a canonical system (3) it follows that $m = p_1$ and
\begin{equation}
R_0 x = c_1 y_1^0 + \cdots + c_\ell y_\ell^0,
\end{equation}
because $\{y_k^0\}_1^\ell$ is a basis in $\Ker A(c) = \Ker [I + S(c)] \quad (\dim \Ker [I+S(c)] = \ell < \infty$ since $S(c) \in \sigma_\infty$). Note that, if elements $y_1$ and $y_2$ generate chains of EAV of length $m_1$ and $m_2$, then $y_1 + y_2$ generates a chain of EAV of length $\min (m_1, m_2)$. Thus from the definition of a canonical system; the coefficients $c_j$ in (10) are equal to zero for all $j$ such that $p_j < p_1 = m$. Moreover $c_j = c_j(x)$ are continuous linear functionals on $x$, and by virtue of Riesz' theorem, we can write $c_j(x) = (x, z_j^0)$ and
\begin{equation}
R_0 = (\cdot, z_1^0)y_1^0 + \cdots + (\cdot, z_q^0)y_q^0,
\end{equation}
where $q$ is such a number that $p_1 = \cdots = p_q > p_{q+1}$. As before, from the equality $[A(\overline \la)]^* [A^{-1} (\overline{\la})]^* x = x$ we can conclude that $R_0^* x, R_1^* x, \ldots, R_m^* x$ is a chain of EAV of $A^*(\la)$. This means that elements $z_1^0, \ldots, z_q^0$ generate chains of EAV of length $m$, and may be taken as the first elements of canonical system (4). The representation (11) shows that we proved (2) for the (leading) coefficient of $(\la - c)^{-m-1}$. More detailed analysis allows us to get the necessary representations for $R_1, R_2, \ldots, R_m$ and to prove (2).
\end{pf}

%Лекция 03 (19.9.2017)
\newpage
\section{New proof of the completeness theorem in finite dimensional case. Representation of the resolvent as a meromorphic function of finite order growth}

\textbf{1.}
Now we are able to give a new approach to the proof of the theorem on completeness of Keldysh derived chains of operator polynomials. To prove it in finite dimensional space we need only the representation for the principal part of $A^{-1} (\la)$ in the neighborhood of a pole.

Let
\begin{equation}
A(\la) = A_0 + \la A_1 + \cdots + \la^n A_n,
\end{equation}
$A_n$ be invertible, and (for simplicity) assume that $A(\la)$ has only simple eigenvalues $\{\la_k\}$ (i.e. the algebraic multiplicity of each eigenvalue $\la_k$ equals $1$). In this case Keldysh chains have the representation
\begin{equation}
\yt = \{y_k, \la_k y_k, \ldots , \la_k^{n-1} y_k\}, \quad k = 1, 2, \ldots,
\end{equation}
where $\{y_k\}$ are corresponding eigenvectors of $A(\la)$. Suppose the system (2) is not complete. Then there exists a vector $f = \{f_1, \ldots, f_n\} \in H^n$ such that
\begin{equation}
(f, \yt) = (f_1, y_k) + \cdots + (f_n, \la_k^{n-1} y_k) = 0, \quad k = 1, 2, \ldots .
\end{equation}

Denote $A^* (\la) = [A(\overline{\la})]^* = A_0^* + \la A_1^* + \cdots + \la^n A_n^*$. If $\dim H < \infty$ then there is a Laurent expansion
\begin{equation}
[A^*(\la)]^{-1} = \frac{(\cdot, y_k) z_k}{\la - \overline{\la}_k} + R^*(\la)
\end{equation}
where $R^* (\la)$ is holomorphic in the neighborhood of $\overline{\la}_k$ and $z_k$ is the eigenvector of $A^*(\la)$ corresponding to $\overline{\la}_k$. Let us consider the meromorphic vector function
\begin{equation}
F(\la) = [A^*(\la)]^{-1} f(\la),
\end{equation}
where $f(\la) = f_1 + \la f_2 + \cdots + \la^{n-1} f_n$.

Using (4) we obtain
\begin{align*}
F(\la) &= [A^*(\la)]^{-1}[f(\overline{\la}_k) + \frac{1}{1!}(\la - \overline{\la}_k)f'(\overline{\la}_k) + \cdots] \\
& = \frac{(f(\overline{\la}_k), y_k)z_k}{\la - \overline{\la}_k} + F_1 (\la)\\
& = \frac{[(f_1, y_k) + \cdots + (f_n, \la_k^{n-1} y_k)]z_k}{\la - \overline{\la}_k} + F_1(\la)
\end{align*}
where $F_1 (\la)$ is a holomorphic function in the neighborhood of $\overline{\la}_k$. Now the equalities (3) show that $F(\la)$ has no poles at $\{\overline{\la}_k\}$. Hence $F(\la)$ is an entire vector function (i.e. holomorphic in the whole complex plane). For sufficiently large $|\la| > r_0$ we also have the estimate
\begin{align*}
\|F(\la)\| &\leq \|[A^*(\la)]^{-1}\| \|f(\la)\| \\
& \leq |\la|^{-n} \|A_n^{-1} \| \|(I + \la^{-1} A_{n-1} A_n^{-1} + \cdots +
	\la^{-n}A_0A_n^{-1})^{-1} \|\times{}\\
&\kern 4cm{}\times (\|f_1\| + \cdots + |\la|^{n-1} \|f_n\|)\\
& \leq M|\la|^{-1},
\end{align*}
where the constant $M$ does not depend on $\la$. Hence $F(\la)$ is bounded in the whole complex plane and from Liouville's theorem it follows that $F(\la) \equiv const$. Since $\|F(\la)\| \rightarrow 0$ when $\la\rightarrow \infty$ we have $F(\la) \equiv 0$.

From (5) we have
$$f(\la) = A^* (\la) F(\la) = 0,$$
therefore $f = \{f_1, \ldots, f_n\} = 0$. Hence the system (2) is complete.

\begin{note}
We can usually work with holomorphic vector or operator functions as well as with scalar holomorphic vector function in a disk $D_\ep (c) = \{\la: |\la-c| < \ep\}$ then
\begin{equation}
f(\la) = \sum_{k=0}^\infty f_k(\la-c)^k, \quad f_k \in H,
\end{equation}
and the series converges strongly for $\la \in D_\ep(c)$. Indeed, for each $g \in H$ the scalar function $(f(\la), g)$ is holomorphic in $D_\ep(c)$. Hence
$$(f(\la), g) = \sum_{k=0}^\infty c_k (\la-c)^k, \quad \la \in D_\ep(c)$$
The coefficients $c_k = c_k (g)$ are linear functionals defined for all $g \in H$, therefore $c_k = (f_k, g)$. Now from the Banach-Steinhaus theorem we can deduce that the series (6) converges strongly.

Also, if $F(\la) = F_0 + F_1 \la + \cdots$ is an entire bounded vector function then so is the scalar function $(F(\la),g)$ for each $g \in H$. From the Liouville theorem it follows that $(F(\la), g) = \equiv const$. Hence $(F_j, g) = 0, \quad j = 1,2,\ldots, $ for each $g \in H$ and $F(\la) = F_0 \equiv const$. (See [Hille and Fhillips], for example, for details.)
\end{note}
\begin{note}
The proof of completeness does not change significantly if the eigenvalues $\{\la_k\}$ are not simple or semi-simple. It is an easy exercise to reduce from the definition of EAV and representation (2.2)\footnote{Here and further the notation (n.m) means the reference (m) from lecture n.} that the equalities
$$(f, \yt^0) = 0, \quad (f, \yt^1) = 0, \ldots, (f, \yt^m) = 0$$
are equivalent to the following:
\begin{align*}
&R_0^* f(\overline{\la}_k) = 0, \\
& R_1^* f(\overline{\la}_k) + \frac{1}{1!}R_0^* f'(\overline{\la}_k) = 0\\
& \hbox to 3cm{\dotfill}\\
& R_m^* f(\overline{\la}_k) + \cdots + \frac{1}{m!} R_0^* f^{(m)} (\overline{\la}_k) = 0.
\end{align*}

Therefore if vector $f\in H^n$ is orthogonal to all derived chains $\{\yt^h\}$ then the function (5) is an entire vector function. This observation allows us to finish the proof of completeness as before.
\end{note}

\begin{note}
Keldysh defines \textit{the system of EAV $\{y_k^h\}$ to be $n$-multiple complete in $H$ if the derived chains $\{\yt^h\}$ form a complete system in $H^n$}. But after considering some concrete operator pencils we will see that we have to investigate the properties of derived chains $\{\yt^h\}$ \underline{not} in the space $H^n$, but in some space $\HH$ which is embedded in $H^n$.
\end{note}

\textbf{2.}
The new proof of the completeness theorem can be generalized for infinite dimensional spaces $H$. The essence of the matter is contained in the subsequent theorem on the growth of the resolvent of operator pencils. First we have to recall some definitions.

An entire scalar function $f(\la)$ is said to be \textit{a function of finite order} if there exists a constant $p > 0$ such that the inequality
$$|f(\la)| < e^{{|\la|}^p}$$
is valid for all sufficiently large $|\la| > r_0 = r_0(p)$. The infimum of such numbers $p$ is called \textit{the order of the entire function $f(\la)$}.

We say an entire function $f(\la)$ has a \textit{finite type with order} $p$ if there exists a constant $k>0$ such that
$$|f(\la)| < e^{{k|\la|}^p}$$
for all sufficiently large $|\la| > r_1 = r_1(k)$. The infimum of such numbers $k$ is called \textit{the type of $f(\la)$} with order $p$.

It is an easy exercise to verify that the order $p$ and the type $\sigma$ of an entire function $f(\la)$ are determined by equalities
$$p = \varlimsup_{r\rightarrow \infty} \frac{\ln \ln M_f(r)}{\ln r},
\quad \sigma = \varlimsup_{r\rightarrow \infty} \frac{\ln M_r(r)}{r^p},$$
where $M_f(r) = \max\limits_{|\la| = r} |f(\la)|$.

The same definition of growth is applied to holomorphic vector or operator function. The only difference is that instead of $|f(\la)|$ we have to consider $\|f(\la)\|$.

If $A$ is a compact operator $(A \in \sigma_\infty)$, then operator $C = (A^*A)^{1/2}$ is compact too. The eigenvalues of the operator $C$ are called $s$-numbers of the operator $A$. We will assume that the sequence of $s$-numbers of $A$ is enumerated in decreasing order, so that $\|A\| = s_1(A) \geq s_2(A) \geq \cdots$.

We will write $A \in \sigma_p$ if $\sum\limits_{k=1}^\infty s_k^p(A) < \infty$.

\begin{theorem}[Fundamental theorem on the growth of the resolvent]
Let operator pencil (1) be such that $A(\la_0)$ is invertible for some $\la_0 \in \mathbb{C}$. Suppose also that there exists a number $p>0$ such that one of the following conditions is fulfilled:
\begin{align}
& A_j \in \sigma_{p/j} \quad j = 1,2, \ldots, n;\\
& s_k (A_j) = o (k^{-j/p}) \quad j = 1,2, \ldots, n;\\
& s_k (A_j) = O(k^{-j/p}) \quad j = 1,2, \ldots, n;
\end{align}
Then $A^{-1} (\la)$ is a meromorphic operator function whose order does not exceed $p$. This means that $A^{-1} (\la)$ admits the representation $A^{-1}(\la) = D(\la)/\Delta(\la)$, where $D(\la)$ and $\Delta(\la)$ are the operator and scalar functions, respectively, of order $p$ or less. Moreover, the type of $D(\la)$ and $\Delta(\la)$ is finite if (9) holds and equal to zero if either (7) or (8) is fulfilled.
\end{theorem}

In some respects this theorem is due to M. Keldysh because he was the first to prove a general result of this kind, although he took advantage of an important result due to T. Carleman estimating the Fredholm resolvent of a Hilbert-Schmidt operator. New approaches to the proof for linear pencils $A(\la) = I - \la A$ were proposed by V. Lidskii, V. Matsaev, I. Gohberg and M. Krein. Some contributions were also made by M. Gasimov and G. Radzievskii. An independent approach to the proof of such theorems for operator pencils with differential operators was developed by F. Brouder and S. Agmon (see comments and references in [Shkalikov 1, \S2]).

To prove this theorem we have to make a tour of some selected topics in the theory on non-self-adjoint operators. Certainly we will not be able to give proofs of all the results which we will use.

\textbf{Item 1.} Let $A$ be a bounded operator in $H$. It is well-known, and easy to see, that
$$H = \Ker A \oplus \overline{\Imm A^*} = \Ker A^* \oplus \overline{\Imm A}$$

Each bounded operator $A$ can be represented in the form
\begin{equation}
A = UC,
\end{equation}
where $C = (A^*A)^{\frac{1}{2}}$ and a partial isometry, such that $\Ker U = \Ker C$ and $U: \overline{\Imm A^*} = \overline{\Imm C} \rightarrow \overline{\Imm A}$ is an isometric one-to-one map. This representation is called \textit{the polar representation}. It is not complicated to prove (see proof in [Gohberg, Krein, Ch.1]), but very useful. In particular, if $A$ is a compact operator, then $C$ is too. Hence
$$C = \sum_{k=1}^\infty s_k(A) (\cdot, e_k) e_k,$$
where $\{e_k\}_1^\infty$ is the orthonormal system of eigenvectors of $C$. Then form (10) we get the \textit{Schmidt representation}
\begin{equation}
A = \sum_{k=1}^\infty s_k(A) (\cdot, e_k) f_k,
\end{equation}
where $\{f_k\}_1^\infty = \{Ue_k\}_1^\infty$ is also an orthonormal system, because $U$ is an isometric operator for $x \in \overline{\Imm C}$ and $e_k \in \Imm C$.

\textbf{Item 2.} Let $A,B$ be compact operators and $D$ a bounded operator in $H$. The following properties of $s$-numbers are fulfilled:
\begin{enumerate}
\item $s_k(A) = s_k(A^*), \quad k = 1,2,\ldots;$
\item $s_k(DA) \leq \|D\| s_k (A), \quad s_k(AD) \leq \|D\| s_k(A), \quad k=1,2,\ldots;$
\item $s_{k+1}(A) = \min\limits_{F\in R_k} \|A - F\|, \quad k=0,1,\ldots,$
where $R_k$ is the set of all operators with range of dimention $k$ or less. In particular, if $F\in R_r$, then
$$s_k(A+F) \leq s_{k-r}(A), \quad k=r+1, r+2, \ldots.$$
\item $s_{k+m+1} (A+B) \leq s_k(A) + s_m(B), \quad k,m = 1,2,\ldots,$\\
hence for a set of compact operators $A_1, A_2, \ldots, A_q$ one has the inequalities $s_k(A_1 + A_2 + \cdots + A_q) \leq s_{k_1}(A_1) + s_{k_1}(A_2) + \cdots + s_{k_1}(A_q), \quad k=1,2,\ldots,$ where $k_1 = [(k-1)/q]+1$ (by $[a]$ we denote the integer part of the number $a$);
\item $s_{k+m-1}(AB) \leq s_k(A)s_m(B), \quad k,m = 1,2,\ldots,$\\
and if $A_1, \ldots, A_q$ are compact operators then
$$
s_k(A_1 A_2\cdots A_q) \leq s_{k_1}(A_1) s_{k_1}(A_2)\cdots s_{k_1}(A_q),
$$
where $ k=1,2,\ldots$, $k_1 = [(k-1)/q] + 1$;
\item $|\la_1(A) \la_2(A) \cdots \la_k(A)| \leq s_1(A) s_2(A) \cdots s_k(A), \quad k=1,2,\ldots,$\\
where $\la_j(A)$ are eigenvalues of the operator $A$ numbered as many times as their algebraic multiplicity and in order of decreasing absolute value;
\item $\sum\limits_{j=1}^k |\la_j(A)|^p \leq \sum\limits_{j=1}^k s_j^p(A), \quad p>0, \quad k = 1,2,\ldots;$
\item $\prod\limits_{j=1}^k (1+r|\la_j(A)|) \leq \prod\limits_{j=1}^k (1+rs_j(A)), \quad r>0, \quad k =1,2,\ldots;$
\item $\sum\limits_{j=1}^k s_j(A+B) \leq \sum\limits_{j=1}^k s_j (A) + \sum\limits_{j=1}^k s_j (B), \quad k = 1,2,\ldots;$
\item $\sum\limits_{j=1}^k s_j(AB) \leq \sum\limits_{j=1}^k s_j(A) s_j(B), \quad k=1,2,\ldots,$
\end{enumerate}
and if $A_j \in \sigma_{p_j}, \quad j =1,\ldots, q$, then using the H\"older inequality one can deduce $A_1 A_2\ldots A_q \in \sigma_p$, where $p^{-1} = p_1^{-1} + p_2^{-1} + \cdots + p_q^{-1}$.

These results on $s$-numbers are due to H. Weyl, K. Fan, A. Horn and D.Allachverdiev. The proof of all these basic properties of $s$-numbers can be found in the book [Gohberg, Krein, Ch.2]. Only the proof of the property 6 (H. Weyl's theorem) is rather complicated. We propose a new sort proof due to A. Kostyuchenko.

Denote by $H_k$ the span of EAV of the operator $A$ corresponding to the first $k$ eigenvalues (counted according to their algebraic multiplicity) and by $A_k$ denote the restriction of $A$ to its invariant subspace $H_k$. We can choose a Jordan basis $\{e_j\}_1^k$, so that either $Ae_j = \la_j e_j$ or $Ae_j = \la_j e_j + e_{j-1}, \quad j=1,\ldots, k$. Write $A_k = P_k AP_k$, where $P_k$ is the orthoprojector on $H_k$. Then by Schmidt's orthogonalization of $\{e_j\}_1^k$ we get the basis $\{f_j\}_1^k$ for which operator $A_k$ evidently has triangular form and $(A_k f_j, f_j) = \la_j(A)$ (the basis $\{f_j\}_1^k$ is called the Schur basis for $A_k$). Hence
\begin{align*}
|\det A_k|^2 & = |\la_1\la_2 \cdots \la_k|^2\\
& = \det A_k^* \det A_k\\
& = \det A_k^* A_k\\
& = [s_1(A_k)\cdots s_k(A_k)]^2 \leq [s_1(A)\cdots s_k(A)]^2.
\end{align*}
The property 6 follows using property 2 of $s$-numbers.

\textbf{Item 3.} An operator $A$ is said to be \underline{nuclear} if $A \in \sigma_1$, i.e. $\sum s_k(A) < \infty$. Using the Schmidt representation (11) the following remarkable fact can be established (see [Gohberg , Krein, Ch. III, Sec.8]).

\begin{proposition}
$A \in \sigma_1$ if and only if for any orthonormal basis $\{\ph_k\}_1^{\infty}$ in the space $H$ the series
\begin{equation}
\sum_{k=1}^\infty (A\ph_k, \ph_k)
\end{equation}
converges. Moreover, the sum (12) does not depend on the choice of the basis $\{\ph_k\}_1^{\infty}$.
\end{proposition}

The sum (12) for operator $A \in \sigma_1$ is denoted by $\Tr A$. It is worth mentioning the following properties the following properties of the functional $\Tr A$ (we can consider class $\sigma_1$ as a normed space with norm $\|A\|_1 = \sum s_k(A)$ and then $\Tr A$ is a linear functional in this space):
\begin{enumerate}
\item $\Tr A^* = \overline{\Tr A}$;
\item $\Tr (AB) = \Tr (BA), \quad A,B \in \sigma_1$;
\item $\Tr A = \sum\limits_{k=1}^\infty \la_k(A)$.
\end{enumerate}
The first property is trivial, the second one can be easily obtained from the Schmidt representation. The third one is the assertion of a well-known theorem of Lidskii and is nontrivial. Further, we will need the inequality
\begin{equation}
|\Tr A| \leq \|A\|_1 = \sum s_k(A).
\end{equation}
Certainly, (13) follows immediately from property 7 of $s$-numbers, and the Lidskii theorem. But an elementary proof can also be proposed. Taking in (12)  $\ph_k = e_k$ and using the Schmidt representation (11) we easily get (13).

\textbf{Item 4.} If $A \in \sigma_1$ then the determinant of operator $I-A$ is defined by the formula
$$\det(I-A) = \prod_{k=1}^{\infty} (1-\la_k(A)).$$
This product obviously converges, because of property 7 for $s$-numbers
$$\sum_{k=1}^\infty |\la_j(A)| \leq \|A\|_1.$$

We consider also the characteristic determinant of the operator $A$
$$\det (I-\mu A) := D_A(\mu) := \prod_{k=1}^\infty (1-\mu \la_k(A)).$$

\begin{theorem}
$\det (I-A)$ is a continuous functional in the space of nuclear operators with norm $\|A\|_1$.
\end{theorem}

\begin{pf}
We have to show that for any $\ep > 0$ there exists a $\delta > 0$ such that the inequality $\|A-B\|_1 < \delta$ implies
$$|\det (I-A) - \det (I-B)| < \ep.$$
We have (with $\la_k = \la_k(A)$),
\begin{align*}
[\ln \det(I-\mu A)]' & = \frac{D_A'(\mu)}{D_A (\mu)} = -\sum_{k=1}^\infty \frac{\la_k}{1- \mu \la_k}\\
& = -\Tr [A(I - \mu A)^{-1}].
\end{align*}
The last equality is valid because of Lidskii's theorem. (Evidently, for any fixed $\mu \neq \la_k$ the operator $A(I - \mu A)^{-1}$ is nuclear.) Now we get the representation
\begin{equation}
D_A(\mu) = \exp\left[-\int_\Gamma \Tr [A(I-\zeta A)^{-1}]d\zeta\right],
\end{equation}
where $\Gamma$ is any smoothcontour, which connects the points $0$ and $\mu$, and does not contain the points $\{\la_k^{-1}\}$. Obviously, operator $I - \zeta B$ is invertible for all $\zeta \in \Gamma$ if $\|B_A\| < \delta_1$ and $\delta_1$ is sufficiently small. Hence, there exists a constant $M$ such that
\begin{equation}
\max_{\zeta \in \Gamma} \|(I-\zeta A)^{-1}\| \leq M, \quad \max_{\zeta \in \Gamma} \|(I-\zeta B)^{-1}\| \leq M.
\end{equation}
Notice that
\begin{align*}
A(I - \mu A)^{-1} - B(I - \mu B)^{-1} & = (I - \mu A)^{-1} [A(I - \mu B) - (I - \mu A)B] (I-\mu B)^{-1}\\
& = (I - \mu A)^{-1} (A-B) (I-\mu B)^{-1}.
\end{align*}
Now, using property 2 of $s$-numbers and the estimates (13), (15), we obtain
$$\left| \int_\Gamma \Tr[A(I-\zeta A)^{-1} - B(I - \zeta B)^{-1}] d\zeta \right| \leq \gamma M^2 \|A-B\|_1,$$
where $\gamma$ is the length of $\Gamma$. From representation (14) it follows that
\begin{align*}
|D_A(\mu) & - D_B(\mu)| = \\
&= \left|D_A(\mu) \left(1- \exp\left[ \int_\Gamma \Tr [A(I-\zeta A)^{-1} - B(I-\zeta B)^{-1}] d\zeta \right]\right) \right|\\
&\leq \gamma M^2 |D_A(\mu)| \|A-B\|_1.
\end{align*}
The last inequality (we can put $\mu=1$) proves the thorem.
\end{pf}

\begin{proposition}
Let $A \in \sigma_1$ and $\{Q_n\}$ be a sequence of orthoprojectors such that $Q_n \rightarrow 0$, when $n\rightarrow \infty$. Then
\begin{equation}
\|Q_n A Q_n\|_1 \rightarrow 0, \quad \|Q_n A\|_1 \rightarrow 0, \quad \|AQ_n\|_1 \rightarrow 0.
\end{equation}
\end{proposition}

\begin{pf}
Using the Schmidt representation (11) again we get
$$A = \sum_{k=1}^N s_k (\cdot, e_k) f_k + \sum_{k = N+1}^\infty s_k(\cdot, e_k) f)k = A_N + A_\ep,$$
where $\|A_\ep\|_1 = \sum\limits_{k=N+1}^\infty s_k(A) < \ep$.
Since $A_N$ is finite dimensional, we have $\|Q_n A_N Q_N\|_1 \rightarrow 0$, $\|Q_nA_N\|_1 \rightarrow 0$ and (16) follows.
\end{pf}

The theorem on continuity of the determinant and the last proposition give us some important results.
\begin{corollary}
Let $A \in \sigma_1$ and $\{P_n\}$ be orthogonal projectors such that $P_n \rightarrow I$ when $n \rightarrow \infty$. Then
$$\det (I-A) \lim_{n \rightarrow \infty} \det (I - P_n A P_n).$$
\end{corollary}
\begin{pf}
We have to notice only that
$$\|A - P_nA P_n\|_1 = \|Q_n A P_n + P_n A Q_n + Q_n A Q_n\|_1 \rightarrow 0$$
\end{pf}
\begin{corollary}
If $A \in \sigma_1, B\in \sigma_1$ then
\begin{equation}
\det (I-A)(I-B) = \det (I-A)(I-B).
\end{equation}
\end{corollary}
\begin{pf}
The equality (17) is known in finite dimensional space, hence if $P_n \rightarrow I$, then
\begin{align*}
\det (I-A)(I-B) & = \lim_{n \rightarrow \infty} \det (I-P_n A P_n - P_n B P_n + P_n AB P_n)\\
& = \lim_{n \rightarrow \infty} \det [(I-P_n A P_n)(I-P_nBP_n) - P_n A Q_n B P_n] \\
& = \lim_{n \rightarrow \infty} \det(I- P_n A P_n) (I-P_nBP_n)\qquad\text{[since $\|A Q_n\|_1 \rightarrow 0$]}\\
& = \lim_{n \rightarrow \infty} \det (I- P_n A P_n) \cdot \lim_{n \rightarrow \infty} \det(I-P_n B P_n)\\
& = \det (I-A) \det(I-B).
\end{align*}
\end{pf}
\begin{corollary}
If $A \in \sigma_1, B\in \sigma_1$ and $I-B$ is invertible then
\begin{equation}
\det (I-A)(I-B)^{-1} = \frac{\det (I-A)}{\det(I-B)}.
\end{equation}
\end{corollary}
\begin{pf}
Since $(I-B)^{-1} = I + B(I-B)^{-1}$, we have
$$\det (I-B)^{-1} = \prod_{k=1}^\infty (1+\la_k(1-\la_k)^{-1}) = \prod_{k=1}^\infty (1-\la_k)^{-1} = \frac{1}{\det (I-B)}.$$
Hence the equality (18) follows from (17).
\end{pf}
\textbf{Item 5.} Now we establish the assertion of the fundamental theorem for a linear pencil
$$A(\la) = I - \la A, \quad A \in \sigma_p, \quad p\leq 1.$$

\begin{theorem}
If $A \in \sigma_p, \quad p \leq 1$ then
$$(I - \la A)^{-1} = \frac{D(\la)}{\det (I - \la A)}$$
and
\begin{equation}
|\det (I - \la A)| \leq \prod_{k=1}^\infty (1+ |\la| s_k(A)),
\end{equation}
\begin{equation}
\|D(\la)\| = \|(I-\la A)^{-1} \det(I-\la A) \| \leq \prod_{k=1}^\infty (1+ |\la| s_k(A)).
\end{equation}
\end{theorem}
\begin{pf}
The estimate (19) follows from property 8 of $s$-numbers
$$\left| \prod_{k=1}^\infty (1 - \la \la_j(A)) \right| \leq \prod_{k=1}^\infty (1 + |\la| |\la_j(A)|) \leq \prod_{k=1}^\infty (1+|\la| s_j(A)).$$
To prove the estimate (20), choose arbitrary vectors $\ph, \psi$ such that $\|\ph\| = \|\psi\| = 1$ and consider the operator
$$A_1 = A + \xi (\cdot, \psi) \ph, \quad \xi > 0.$$
According to property 3 of $s$-numbers we have
\begin{align*}
& s_{j+1} (A_1) \leq s_j(A), \quad j = 1,2,\ldots,\\
& s_1(A_1) = \|A_1\| \leq \|A\| + \xi = s_1(A) + \xi.
\end{align*}
Hence
\begin{equation}
|\det(I - \la A_1)| \leq [1 + |\la(s_1(A)+\xi)] \prod_{j=1}^\infty (1+ |\la| s_j(A)).
\end{equation}
Since $(I-\la A_1) (I-\la A^{-1}) = I - \la K$, where $K$ is a one dimensional operator
$$Kf = \xi ((I- \la A)^{-1}f,\psi) \ph,$$
we have from corollary 3,
$$\det (I - \la A_1)(I - \la A)^{-1} = 1 - \la \la_1 (K) = \frac{\det(I-\la A_1)}{\det(I-\la A)}.$$
Solving the equation $Kf = \la_1 f$, we find $d = \ph$ and $\la_1 = \la_1 (K) = \xi ((I- \la A)^{-1}\ph, \psi)$. Therefore
$$1 - \la \xi ((I - \la A)^{-1} \ph, \psi) = \frac{\det (I- \la A_1)}{\det (I - \la A)}$$
and taking into account (21), we obtain
$$|((I - \la A)^{-1}\ph, \psi)| \leq \frac{1}{\la \xi} + \left[ \frac{1}{\la \xi} + \frac{s_1}{\xi} + 1 \right] \frac{\prod\limits_{k=1}^\infty (1 + |\la| s_k (A))}{|\det (I - \la A)|}.$$
Now let $\xi > 0$ tend to infinity and notice that $\|R\| = \sup\limits_{\|\ph\| = \|\psi\| = 1} |(R\ph, \psi)|.$ Then the last estimate gives (20).
\end{pf}
Now we can get the assertion of the fundamental theorem by applying Borel's theorem on growth of canonical products (see, for example, [B. Levin]). We formulate this result for simple canonical products.

\begin{theorem}[Borel's Theorem]
Let one of the following conditions be fulfilled:
\begin{enumerate}
\item $\sum s_k^p < \infty, \quad p\geq 1;$
\item $s_k = o(k^{-1/p}), \quad p < 1;$
\item $s_k = O(k^{-1/p}), \quad p < 1.$
\end{enumerate}
Then $\Delta (\la) = \prod\limits_{k=1}^\infty (1 + s_k \la)$ is an entire function or order $p$ or less and finite type. Moreover, the type of $\Delta(\la)$ is equal to zero if either condition 1 or condition 2 holds.
\end{theorem}

\textbf{Item 6.} We start to prove the fundamental theorem for the general case using the note that, without loss of generality, one can assume $A_0 = I$ (as in the theorem on holomorphic operator functions). Denote
$$F(\la) = -A_1 \la - A_2 \la^2 - \cdots - A_n \la^n.$$
Evidently, for each fixed $\la$ and any $\ep > 0$ we have$F(\la) \in \sigma_{p+\ep}$. Choose the smallest integer $\ell$ such that $p/\ell < 1$. Then taking into account the property 5 of $s$-numbers we obtain $F^\ell (\la) \in \sigma_1$. Thus for any fixed $\la \in \mathbb{C}$ the function
$$\Delta (\la) = \det (I - F^\ell(\la))$$
is well defined. Moreover, using corollary 1 and Weierstrass' theorem on uniformly convergent sequences of holomorphic functions, we find that $\Delta (\la)$ is an entire function (the uniform convergence of the functions $\det (I - P_n F^\ell (\la) P_n), \quad P_n \rightarrow I,$ follows from the theorem on continuous dependence of the determinant and from the continuity of the function $F^\ell (\la)$ in the nuclear norm).

From the simple relation
\begin{equation}
A^{-1} (\la) = [I- F(\la)]^{-1} = [I + F(\la) + \cdots + F^{\ell - 1} (\la)] [I-F^\ell (\la)]^{-1}
\end{equation}
we conclude that the growth of the meromorphic function $A^{-1} (\la)$ is the same as the growth of $[I - F^\ell(\la)]^{-1}$, because the left divisor in (22) is polynomial and does not have any influence on the order or type of $A^{-1} (\la)$. Using the theorem on the estimate of the resolvent in the latter item, we have
$$[I - F^\ell (\la)]^{-1} = \frac{D(\la)}{\Delta(\la)}$$
and
\begin{equation}
\|D(\la)\| \leq \prod_{k=1}^\infty [1+s_k(F^\ell (\la))], \quad |\Delta(\la) \leq \prod_{k=1}^\infty [1+s_k (F^\ell (\la))]
\end{equation}
(to get these estimates we make substitution in (19), (20): $\la \rightarrow 1, \quad A \rightarrow F^\ell (\la)$).
Denote
\begin{align*}
F^\ell (\la) & = (-1)^\ell \{\la^\ell A_1^\ell + \la^{\ell+1}[A_1^{\ell-1}A_2 + A_1^{\ell - 2} A_2 A_1 + \cdots + A_2 A_1^{\ell-1}] + \\
& + \la^{\ell+2} [A_1^{\ell-1} A_2 + \cdots ] + \la^{n\ell} A_n^\ell \} = (-1)^\ell \sum_{j=\ell}^{n\ell} \la^j B_j.
\end{align*}
First, assume that condition (9) holds. Then taking into account the properties 4, 5 of $s$-numbers we find $s_k(B^j) = O (k^{-j/p}), \quad j = \ell, \ell+1, \ldots, n\ell, \quad k = 1,2, \ldots .$ Using the property 4 again we obtain
\begin{equation}
s_k(F^\ell(\la)) \leq \sum_{j=\ell}^{n\ell} |\la|^j s_{k_1} (B_j) \leq M\sum_{j=\ell}^{n\ell} |\la|^j k^{-j/p},
\end{equation}
where $k_1 = \left[ \frac{k-1}{n(\ell - 1)+1} \right] + 1, \quad M = const.$ Hence
$$\prod_{k=1}^\infty [1 + s_k (F^\ell (\la))] \leq \prod_{k=1}^\infty (1 + M\sum_{j=\ell}^{n\ell} |\la|^{j_k-j/p}) \leq \prod_{j = \ell}^{n\ell} \prod_{k=1}^\infty (1 + M |\la|^{j_k-j/p}).$$
According to Borel's theorem the function
$$f_j (\mu) = \prod_{k=1}^\infty (1 + M \mu k^{-j/p}), \quad j = \ell, \ell + 1, \ldots, n\ell$$
has the order $p/j$ and finite type. Then we find from the definition that the function $f(\la) = f_\ell(\la^\ell) \cdots f_{n\ell}(\la^{n\ell})$ has order $p$ and finite type and so do the functions $D(\la)$ and $\Delta (\la)$.

The proof does not change if (9) is replaced by (8). If the condition (7) holds then using the properties 5. 10 of $s$-numbers we may deduce that $B_j \in \sigma_{p/j}, \quad j = \ell, \ldots, n\ell$. Taking into account the first estimate in (24) and noting that index $k_1$ repeats $\kappa = n(\ell - 1) + 1$ times when $k$ runs through the integers, we obtain
$$\prod_{k_1 = 1}^\infty [1 + s_k(F^\ell(\la))] \leq \left( \prod_{j = \ell}^{n\ell} \prod_{k=1}^\infty [1 + |\la|^j s_k (B_j)] \right)^\kappa.$$
Recalling that $B_j \in \sigma_{p/j}$ and applying Borel's theorem, we find from (23) that $D(\la)$ and $\Delta (\la)$ have the order $\leq p$ and type 0 with order $p$. This proves the theorem. \hfill$\scriptstyle\blacksquare$

\newpage
%Лекция 04
\section{Keldysh-Lidskii theorem on the completeness}

To prove the subsequent theorems on completeness we need to recall some classical results of the theory of entire functions (see, for example, [Levin, Ch.1]).

\begin{theorem}[Phragmen-Lindel\"of Theorem]
Let $f(\la)$ be a holomorphic function \footnote{We defined the order and the type of an entire function, but the same definition is applied to functions which are holomorphic in a sector.}  of order $p$ in a sector $\Omega_\alpha (\ph) = \{\la : |\ph - \arg \la| \leq \alpha \}$ and
\begin{equation}
|f(\la) | \leq M
\end{equation}
on the sides of the sector $\Omega_\alpha (\ph)$. If $\alpha < \pi/2p$ then the estimate (1) holds throughout all the sector $\Omega_\alpha (\ph)$.
\end{theorem}

\begin{note}
One can deduce from this theorem a slight generalization. Instead of the estimate (1), assume that the following holds:
\begin{equation}
|f(\la)| \leq M (1+|\la|^m)
\end{equation}
on the sides of the sector $\Omega_\alpha (\ph)$. Then the same estimate holds throughout all of the sector $\Omega_\alpha (\ph)$, probably with a new constant $M_1$ instead of $M$. To prove this fact we may consider a function $f(\la)/p(\la)$, where $p(\la)$ is a polynomial of degree in with zeros lying outside of the sector $\Omega_\alpha (\ph)$.
\end{note}

The next result is a corollary of the theorem giving a lower estimate of entire functions due to E. C. Titchmarsh.

\begin{theorem}[Theorem on a Ratio of Holomorphic Functions]
Let $F(\la) = \frac{F_1(\la)}{F_2(\la)}$, where $F_j(\la), \quad j = 1,2$, are entire functions of the order $p_j$ and type $\sigma_j$ with order $p_j$. If $F(\la)$ is also an entire function then it is a function of order $p = \max(p_1, p_2)$ or less and of the type $\sigma = \sigma_1 + \sigma_2$ or less with order $p$.
\end{theorem}

Now we can formulate and prove a general theorem on completeness. The main assumption of this theorem is that, on some rays in the complex plane, the growth of the resolvent of a pencil does not exceed polynomial growth. It may seem at first sight that such a condition is unnatural and is difficult to establish for some specific linear or polynomial pencils. But we will dispel any such illusion later on. Now we mention only that there are a lot of papers devoted to estimates of the resolvent for boundary value problems containing a spectral parameter for ordinary differential operators, as well as for partial differential operators. Some of them are due to G. Birkhoff and Ja. Tamarkin, S. Agmon and L. Nierenberg, M. Agranovich and M. Vishik (for estimates of the resolvent in non-Hilbert spaces see the book [Tribel] and references there).

\begin{theorem}[General Theorem on Completeness] Let the operator pencil
$$A(\la) = A_0 + \la A_1 + \cdots + \la^n A_n$$
be such that $A_j \in \sigma_{p/j},\quad j = 1,\ldots, n$, for some $p>0$
and $\Ker A_n^* = 0$. Also, let there exist a finite set of a rays $\{\gamma_k\}_{k=1}^q$ dividing the complex plane into $q$ sectors with angles less than $\pi/2p$ and such that the resolvent of the pencil $A(\la)$ exists on these rays for sufficiently large $|\la| > r_0$ and has the estimate
\begin{equation}
\|A^{-1} (\la)\| \leq M (1+|\la|^m),
\end{equation}
where $M$ and $m$ are constants. Then the system of Keldysh derived chains of $A(\la)$ is complete in $H^n$.
\end{theorem}
\begin{pf}
We have noticed (see Note 3.2) that if a vector $f = \{f_1, \ldots, f_n \} \in H^n$ is orthogonal to all Keldysh derived chains then
$$F(\la) = [A^*(\la)]^{-1} (f_1 + \la f_2 + \cdots + \la^{n-1} f_{n-1})$$
is an entire vector function. It follows from the fundamental theorem on the estimate of the resolvent and the theorem on a ratio of entire functions that $F(\la)$ has the order $p$ or less. Taking into account the equality
$$\|[A^*(\la)]^{-1}\| = \|[A^{-1} (\overline\la)]^*\|$$
and the estimate (3) which holds asymptotically on the rays $\{\gamma_k\}_1^q$, we obtain the estimate
\begin{equation}
\|F(\la)\| \leq M_1 (1+|\la|^{m+n-1})
\end{equation}
on the rays $\{ \overline\gamma_k \}_1^q$. The angle between neighboring rays is less than $\pi/2p$. Hence, in virtue of the Phragmen-Lindel\"of theorem we obtain the estimate (4) in each sector contained between neighboring rays. Thus the estimate (4) holds in the whole complex plane and, from Liouville's theorem, we conclude that
$$F(\la) = F_0 + \la F_1 + \cdots + \la^r F_r, \quad r\leq m+n-1.$$
On the other hand
$$(A_0^* + \la A_1^* + \cdots + \la^n A_n^*) (F_0 + \la F_1 + \cdots + \la^r F_r) = f_1 + \cdots + \la^{n-1} f_n.$$
The right side of this equality is a polynomial of degree $n-1$ but the left side is a polynomial of degree $n+r$ or less. Hence we find $A_n^* F_r = 0, \quad A_n^* F_{r-1} = 0, \ldots, A_n^* F_0 = 0$. Since the kernel of the operator $A_n^*$ is trivial we have $F(\la) \equiv 0$ and then $f = \{f_1, \ldots, f_n\} = 0$. The theorem is proved.
\end{pf}
We will deduce some useful corollaries from this general theorem. But first we recall some definitions and prove some auxiliary results.

The numerical range $\theta(A)$ of an operator $A$ is the set of all complex numbers $(Au,u)$, where $u$ takes values in the unit sphere: $\|u\| = 1$. A theorem due to Hausdorff asserts that $\theta(A)$ is a convex set. It is known also (see, for example, [Kato, Ch.5]) that the closure of $\theta (A)$ contains the spectrum of $A$ and for all $\mu \not\in \theta(A) \cup \sigma(A)$ the following estimate holds
\begin{equation}
\|(A - \mu I)^{-1}\| \leq \frac{1}{\dist(\mu, \theta(A))}.
\end{equation}

An operator $T$ is called \textit{sectorial} if its numerical range $\theta(T)$ is a subset of a sector $|\ph - \arg \la| \leq \alpha$ for some $\ph \in [0, 2\pi)$ and $\alpha \in [0, \pi/2]$. The numbers $\ph$ and $\alpha$ are called the \textit{vertex} and \textit{semi-angle} of the sectorial operator $T$. Further, when dealing with a sectorial operator, we will assume (for simplicity) that its vertex is equal to zero.

An operator $T$ is called accretive (dissipative) if its numerical range lies in the right-half plane, $\Ree \la \geq 0$ (in the left-half plane $\Ree \la \leq 0$).

\begin{lemma}
If $T$ is a sectorial operator with semi-angle $\alpha$ then outside the sector $\Omega_{\alpha + \ep} = \{ \la: |\arg \la| \leq \alpha + \ep \}, \quad (\ep > 0)$ the following estimates hold:
\begin{align}
& \|(I-\la T)^{-1} \| \leq \frac{1}{\sin \ep},\\
& \|(I-\la T)^{-1} T \| \leq \frac{1}{|\la|\sin \ep}.
\end{align}
Moreover, if $\Ker T^* = 0$, then for any fixed vector $x$
\begin{equation}
\|(I-\la T)^{-1} x \| \rightarrow 0
\end{equation}
when $\la \rightarrow \infty$ outside the sector $\Omega_{\alpha + \ep}$.
\end{lemma}
\begin{pf}
If $\la \not\in \Omega_{\alpha + \ep}$ then according to (5)
$$\|(I-\la T)^{-1} \| = |\la|^{-1} \|(\la^{-1} I - T)^{-1} \| \leq \frac{1}{|\la| \dist (\la^{-1}, \Omega_\alpha)} \leq \frac{1}{\sin \ep}$$
and the first estimate (6) follows.

Also, if $\la \not\in \Omega_{\alpha + \ep}$, then for all $y \in H$
\begin{align*}
\|(I-\la T)y\| \|Ty\| & \geq | ((I - \la T)y, Ty) | = | (y,Ty) - \la (Ty, Ty) |\\
& \geq \|Ty\|^2 \dist(\la, \Omega_\alpha) \geq \|Ty\|^2 |\la| \sin \ep
\end{align*}
and $\|Ty\| \leq (|\la| \sin \ep )^{-1} \|(I - \la T)y\|$. Hence for all $x = (I - \la T)y$ we have
$$\|T(I-\la T)^{-1} x \| \leq (|\la| \sin \ep)^{-1} \|x\|$$
and the second estimate (7) follows.

If $\Ker T^* = 0$ then $\overline{\Im T} = H$, therefore for any $x \in H$, and given any $\delta > 0$, there exists a vector $y = Tz$ such that $\|y-x\| < \delta$. Then for $\la \not\in \Omega_{\alpha + \ep}$ we have
$$\| (I-\la T)^{-1} x\| \leq \frac{1}{\sin \ep} \left( \frac{\|z\|}{|\la|} + \delta \right).$$
Since $\delta$ may be chosen arbitrary small, we obtain (8).
\end{pf}
\begin{lemma}
Let $A = (I+S)T$, where $S \in \sigma_\infty$, $\Ker A^* = 0$, and let $T$ be a sectorial operator\footnote{An operator $A$, having such a representation, we call a compact perturbation of a sectorial operator.} with semi-angle $\alpha$. Then outside a sector $\Omega_{\alpha + \ep}, \quad \ep>0$, and for sufficiently large $|\la| > r_0 = r_0(\ep)$, the following estimate holds:
\begin{equation}
\|(I - \la A)^{-1}\| \leq M = M(\ep).
\end{equation}
\end{lemma}
\begin{pf}
It follows from the assumption $\Ker A^* = 0$ that $\Ker (I+S^*) = 0$ and $\Ker (I+S) = 0$ (because $S \in \sigma_\infty$). Hence $I+S$ is invertible and $(I+S)^{-1} = I - S(I+S)^{-1} = I + S_1, \quad S_1 \in \sigma_\infty$. Note also, that $\Ker T^* = 0$. Further, we have
\begin{align}
\begin{split}
\| (I - \la A)^{-1} \| & = \| [(I+S)(I-\la T)(I+S_1(I-\la T)^{-1})]^{-1} \| \\
& \leq \|(I+S)^{-1} \| \|(I-\la T)^{-1} \| \|(I+S_1(I-\la T)^{-1})^{-1} \|.
\end{split}
\end{align}
Notice, we have if $V = (\cdot, \ph) \psi$ is a one dimensional operator (or finite dimensional) then it follows from (8) that
$$\| V(I-\la T)^{-1} \| = \| (I-\overline{\la}T^*)^{-1} V^*\| \rightarrow 0$$
when $\la \rightarrow \infty$ outside $\Omega_{\alpha + \ep}$. (Obviously, the numerical range of $T^*$ lies in the sector $\Omega_\alpha$ and $\Ker T = \Ker T^* = 0$, see Note 3 below.) Operator $S_1$ is compact, hence it may be approximated in the operator norm with any accuracy by a finite dimensional operator. Thus $\| S_1 (I-\la T)^{-1} \| \rightarrow 0$ when $\la \rightarrow \infty$ outside $\Omega_{\alpha + \ep}$ and the estimate (9) follows from (6) and (10).
\end{pf}
\begin{note}
Lemmas 1 and 2 are valid if the sectorial operator $T$ is replaced by a self-adjoint operator $C$ (not necessarily non-negative). The only difference is that in this case all estimates hold outside the sector $\Lambda_\ep = \{\la : |\arg \la| \leq \ep \quad or \quad |\pi - \arg \la| \leq \ep \}$.
\end{note}

Now we are able to present some corollaries from the general theorem on completeness.

\begin{corollary}[Keldysh-Lidskii theorem]
Let $T$ be a sectorial operator with semi-angle $\alpha$ and $\Ker T^* = 0$. If $T \in \sigma_p$ and $p < \pi/ 2\alpha$ then the system of EAV of $T$ is complete.
\end{corollary}
\begin{pf}
Since the estimate (6) holds, we find that all assumptions of the general theorem are fulfilled for the linear operator pencil $A(\la) = I - \la T$.
\end{pf}
\begin{note}
Actually, one can omit the assumption $\Ker T^* = 0$ in the latter corollary, because for an accretive operator $T$ (and of course, for a sectorial operator $T$) we have $\Ker T^* = \Ker T$ (this equality follows from the representation $T = T_R + i T_I$, where $T_R = (T + T^*)/2, \quad T_I = (T-T^*)/2i \geq 0$). Hence, for a sectorial operator we have the representation $H = \overline{\Imm T} \oplus \Ker T$. Since the restriction of $T$ to its invariant subspace $H_1 = \overline{\Imm T}$ has a complete system of EAV in $H_1$, we find that the assertion of Corollary 1 is valid without the assumption that $\Ker T^* = 0$.
\end{note}

\begin{corollary}
Let $A = (I+S)T, \quad S\in \sigma_\infty, \quad T$ be a sectorial operator with semi-angle $\alpha$, and $\Ker A^* = 0$. If $T \in \sigma_p$ and $p < \pi/2\alpha$ then the system of EAV of the operator $A$ is complete.
\end{corollary}
\begin{pf}
By Lemma 2, we find that all assumption of the general theorem are fulfilled for the linear pencil $A(\la) = I - \la A$.
\end{pf}
\begin{corollary}[Theorem of Keldysh]
Let
\begin{equation}
A(\la) = I + S_0 + S_1 C \la + S_2 C^2 \la^2 + \cdots + S_{n-1} C^{n-1} \la^{n-1} + (I+S_n) C^n \la^n,
\end{equation}
where $S_j, \quad j = 0, 1, \ldots, n$, are compact operators, $\Ker (I+S_n) = 0$, and $C = C^* > 0$. If $C \in \sigma_p$ for some $p>0$ then the system of Keldysh derived chains of the pencil $A(\la)$ is complete in $H^n$. In particular, the system of EAV of a compactly perturbed positive self-adjoint operator \\
$$A = (I+S)C, \quad S \in \sigma_\infty, \quad \Ker (I+S) = 0, \quad C > 0, \quad C\in \sigma_p,$$
is complete $H$.
\end{corollary}
\begin{pf}
Without loss of generality we can assume that $S_n = 0$, otherwise we have to consider the operator pencil $(I+S_n)^{-1} A(\la)$, which has the representation (11) with $S_n = 0$. If $S_n = 0$ then
\begin{equation}
A^{-1} (\la) = (I + \la^n C^n)^{-1} (I + \sum_{k=0}^{n-1} S_k \la^k C^k (I + \la^n C^n)^{-1})^{-1}.
\end{equation}
If $\{ \omega_j \}_1^n$ are the roots of the equation $\omega^n + 1 = 0$, then
$$I + \la^n C^n = \prod_{j=1}^n (1+ \omega_j \la C)$$
and for $k = 0,1,\ldots, n-1$, we have
\begin{align*}
\| S_k \la^k & C^k (I + \la^n C^n)^{-1} \| \\
& = [S_k (I + \omega_{k+1} \la C)^{-1}] [\prod_{j=1}^k \la (I + \omega_j \la C)^{-1}] [\prod_{j=k+1}^n (I + \la \omega_j C)^{-1}]\\
& = A_1 (\la) A_2(\la) A_3(\la).
\end{align*}
Denote by $\Omega_\ep^n$ the union of $n$ sectors in the complex plane with semi-angles $\ep$ and vertex $-\omega_k, \quad k = 1,\ldots,n$. Then, according to Lemma 2 $\| A_1 (\la) \| \rightarrow 0$ if $\la \rightarrow \infty$ outside of $\Omega_\ep^n$ and according to Lemma 1 $\| A_j(\la) \| \leq M, \quad j=2,3$ outside the domain $\Omega_\ep^n$. Hence,
$$\| (I+ \sum_{k=0}^{n-1} S_k \la^k C^k (I + \la^n C^n)^{-1})^{-1} \| \leq M_1$$
for $\la \not\in \Omega_\ep^n$ and $|\la|$ sufficiently large. Using Lemma 1 again and the representation (12) we obtain
$$\| A^{-1} (\la) \| \leq M_2$$
if $\la \not\in \Omega_\ep^n$ and $|\la| > r_0$ is sufficiently large. Now we can choose $\ep$ such that $\ep < \pi/2p$, and all assumptions of the general theorem are fulfilled.
\end{pf}

\begin{note}
Taking Note 2 into account we may assume in Corollary 3 that $C$ is an arbitrary self-adjoint operator with $\Ker C = 0$, instead of $C > 0$.
\end{note}

It is also worth mentioning that we can refine the assertion of Corollary 1 by replacing the condition $p <\pi/2\alpha$ with $p \leq \pi/2\alpha$. For this purpose we need to apply the following fact from theory of entire functions (see, for example, [Levin, Ch.1]).

\begin{theorem}[Refined Version of The Phragmen-Lindel\"of Theorem]
Let $f(\la)$ be a holomorphic function of order $p$ and of minimal type \footnote{A holomorphic function $f(\la)$ is said to have a minimal type with order $p$ if it has the type zero with order $p$.} in a sector $\Omega_\alpha = \{ \la: |\arg \la | \leq \ep \}$. If $p \leq \pi/2\alpha$ and
$$|f(\la)| \leq M (1+|\la|^m)$$
on the sides of the sector $\Omega_\alpha$, then the same estimate holds throughout the sector $\Omega_\alpha$, probably with a new constant $M_1$ instead of $M$.
\end{theorem}
\begin{theorem}[Refined Theorem on Completeness of EAV of a Sectorial Operator\footnote{For accretive operators this theorem follows from a deep theorem of M.~Krein on completeness of EAV of compact accretive operators with nuclear real component.}]
Let $T$ be a sectorial operator with semi-angle $\alpha$ and either $T \in \sigma_p$, or $s_k(T) = o(k^{-1/p}), \quad k = 1,2,\ldots$. If $p \leq \pi / 2\alpha$ then the system of EAV of the operator $T$ is complete.
\end{theorem}
\begin{pf}
Let $\Omega_{\alpha, \ep} = \{ \la : \la = \mu -\ep, |\arg \mu| \leq \alpha \} = \Omega_\alpha - \ep$. Then for all $\la \not\in \Omega_{\alpha, \ep}$
$$\dist (\la^{-1}, \Omega_\alpha) \geq |\la|^{-1} \sin \left( \frac{\sin \ep}{|\la|} \right) \sim \frac{\sin \ep}{|\la|^2}.$$
Hence, using the estimate (5), we find for all $\la \not\in \Omega_{\alpha, \ep}$
\begin{equation}
\| (I-\la T)^{-1} \| = |\la|^{-1} \| (\la^{-1} I - T)^{-1} \| \leq \frac{M}{\sin \ep} |\la|,
\end{equation}
where $M$ does not depend on $\ep$ and $\la$.

Taking into account the Note 2, we can assume that $\Ker T^* = 0$. If $f$ is orthogonal to the EAV of $T$, then the vector function $F(\la) = (I-\la T)^{-1} f$ is an entire function of order $p$ and minimal type (according to the assertion of the fundamental theorem on the estimate of the resolvent and the theorem on a ratio of entire functions). Since the estimate (13) holds outside the sector $\Omega_{\alpha, \ep}$ with semi-angle $\alpha$ and $p \leq \pi/2\alpha$, we obtain from the refined version of the Phragmen-Lindel\"of Theorem, that $F(\la)$ is a linear function. As at the end of the general theorem on completeness, we can show that $F(\la) \equiv 0$ and $f = 0$.
\end{pf}

\newpage
%Лекция 05
\section{Half-range minimality  and completeness problems for dissipative pencils}

Let us return to the subject of Section~1 and consider the Cauchy problem

\begin{equation}
A(-i \frac{d}{dt}) u(t) \ = \ A_0 u - iA_1 \frac{du}{dt}+ ... + (-i)^n A_n \frac{d^n u}{dt^n} =0,
\end{equation}
\begin{equation}
(-i)^j u^{(j)} (0) = \phi_j, \ j=0,1,...,n-1,
\end{equation}
where $u(t)$ is a function with values in Hilbert space $H$. Note the following simple result for finite dimensional space $H$.

\begin{proposition}
If $\dim H < \infty$ and $\Ker A_n^*=0$ then the Cauchy problem (1), (2) has a unique solution for any given initial vectors $\{\phi_j\}_0^{n-1}$ .
\end{proposition}
\begin{pf}
To prove the existence of the solution we consider two approaches. First, according to the theorem 1.1, the system of Keldysh derived chains $\{\tilde{y}_k^h\}$ is a basis in $H^n$. Hence, there exist coefficients $\{c_k^h\}$ such that $$\phi=\{\phi_0,\phi_1,...,\phi_{n-1}\}=\sum_{h,k}{c_k^h \tilde{y}_k^h}.$$
Then the function (see formulas (1.5), (1.6))
\begin{equation}
u(t)=\sum_{h,k} {c_k^h e^{i\lambda_k t} (y_k^h + \frac{it}{1!} y_k^{h-1} +...+ \frac{(it)^h}{h!} y_k^0) }
\end{equation}
is the solution of (1), (2).
Second, denoting $$\tilde{u}(t)=\{u(t), -iu'(t),...,(-i)^{n-1}u^{(n-1)}(t)\},$$
we can rewrite (1), (2) in the form
\begin{equation}
-i\frac{d\tilde{u}}{dt}=\A \tilde{u}(t), \ \ \ \ \ \A=\A_1^{-1} \A_0,
\end{equation}
\begin{equation}
\tilde{u}(0)=\phi=\{\phi_0,\phi_1,...,\phi_{n-1}\},
\end{equation}
where $\A_0, \A_1$ are the operators defined in (1.7). \\
For any bounded operator $L$ we an define the operator $$e^L=I+\frac{1}{1!}L+\frac{1}{2!}L^2+...,$$
where the series converges in the uniform operator topology. Hence, the solution of (4), (5) can be represented by the formula $$\tilde{u}(t)=e^{i\A t} \phi .$$
The first component of the function $\tilde{u}(t)$ represents the solution of (1), (2). The uniqueness of the solution of (4), (5) (or (1), (2)) is a well-known fact in the theory of ordinary differential equations.
\end{pf}
The second approach can be applied for infinite dimensional space $H$, if all operators $A_j$, $j=0,1,...,n$, are bounded and $A_n$ is invertible (in this case the operator $\A$ is bounded). But this does not cover some important practical problems of mathematical physics. As we will see soon, for most interesting equations operator $\A$ is unbounded and to define the exponent we have to recall some semigroup operator theory. As a rule, the operator $ \A $ does not generate a $C_0-$semigroup in $H^n$ and we need to look for another space, where it has better properties. But, all our attempts to prove that $\A$ is the generator of $C_0-$semigroup in some space will be fruitless if there exists a subsequence $\lambda_{k_s} \in \sigma(\A)$ such that $\Imm \lambda_{k_s} \rightarrow \infty$ (the condition $\Imm \lambda_k \leq const$ for all $\lambda_k \in \sigma(\A)$ is a necessary condition for $\A$ to be a generator $C_0-$semigroup). Hence, if the spectrum of the pencil $A(\lambda)$ (
 recall that it has the same spectrum as $\A$) does not satisfy the condition
\begin{equation*}
\Imm \lambda_k \leq const,\ for\ all\ \lambda_k \in \sigma(A)
\end{equation*}
then the Cauchy problem for the equation (1) is not correctly set.\\
It is a well-known fact from the theory of partial differential equations that the Cauchy problem is correctly set for some types of hyperbolic and parabolic equations but is not so for elliptic equations. Examining some concrete problems, one discovers that when equation (1) originates with an elliptic problem (in this case the order $n=2l$ is even), it makes sense to impose only $l$ conditions at $t=0$. For example,
\begin{equation}
(-i)^j u^{(j)} (0) = \phi_j,\ j=0,1,...,l-1.
\end{equation}
If the equation (1) is considered on the finite interval $t\in [0,T]$, then one has to impose another $l$ condition at $t=T$, for example
\begin{equation}
(-i)^j u^{(j)} (T) = \psi_j,\ j=0,1,...,l-1.
\end{equation}
(These conditions may be different. For example,  $(-i)^{j+l} u^{(j+l)} (T) = \psi_j,\ j=0,1,...,l-1$).
The case $T=\infty$ is of special interest. In this case the conditions (7) are replaced by the condition
\begin{equation}
\lim_{t \rightarrow \infty}{u(t)} =0,
\end{equation}
of
\begin{equation}
||u(t)||=const,\ \ \ 0<t<\infty.
\end{equation}
Actually, neither (8) nor (9) is the proper condition at infinity, when the spectrum of the operator pencil $A(l)$ contains some real eigenvalues. But these important details will be discussed in the next lecture.
Suppose that $\dim H < \infty$ and consider the problem (1), (6), (8). For simplicity, suppose also that eigenvalues of $A(\lambda)$ are semi-simple. Using the Fourier method we look for a solution represented by a series
\begin{equation}
u(t)=\sum_{\Imm\lambda_k > 0}{c_k e^{i\lambda_k t} y_k}.
\end{equation}
We put in this sum only the eigenvectors $y_k$, corresponding to the eigenvalues $\lambda_k$ with positive imaginary part, otherwise the condition (8) does not hold. Also, we have to satisfy the conditions (6). From (10) and (6) we obtain
\begin{equation}
\phi=\{\phi_0,\phi_1,...,\phi_{l-1}\}=\sum_{\Imm\lambda_k > 0}{c_k \{y_k,\lambda_k y_k,...,\lambda_k^{l-1} y_k\}}.
\end{equation}
The vectors $\hat{y}_k=\{y_k,\lambda_k y_k, ..., \lambda_k^{l-1} y_k\}$ we call \textit{the Keldysh derived chains of length l}. The definition of the derived chains $\hat{y}_k^h$, corresponding to the associated vectors is similar to the definition of $\tilde{y}_k^h$ given in Section~1.

Denote by $\mathcal{E}_0^+ \ (\mathcal{E}_R^+)$ the system $\{\hat{y}_k\}$ of Keldysh derive chains of length $l$ corresponding to all eigenvalues $\lambda_k$ with $\Imm\lambda_k>0\ (\Imm\lambda_k \geq 0) $ .
\begin{proposition}
If $\dim H<\infty$ then the problem (1), (6), (8) has a unique solution for all given vectors $\{\phi_j\}_0^{l-1}$ if and only if the system $\mathcal{E}_0^+=\{\hat{y}_k\}_{\Imm\lambda_k>0}$ is a basis in the space $H$.
\end{proposition}
\begin{pf}
According to the Proposition 1, any solution $u(t)$ of (1) has the representation (3).If all the eigenvalues are semisimple (this is not essential) and the condition (8) holds, then $u(t)$ must have the representation (10). Hence, the conditions (6) are equivalent to (11). But a vector $\phi \in H^l$ can be uniquely represented by the series (11) if and only if the system $\mathcal{E}_0^+$ is a basis.
\end{pf}
Obviously, for the problem (1), (6), (9) is also valid, but we have to replace the system $\mathcal{E}_0^+$ by the system $\mathcal{E}_R^+$.
Hence we come to the problem which we call \textit{the half-range basis problem}. Similarly, we can consider the problem of completeness of the system  $\mathcal{E}_R^+$ (\textit{half-range completeness}) and the problem of minimality of  $\mathcal{E}_0^+$ (\textit{half-range minimality}).\\
Using the method of G. Radzeivskii [1] (1974) we can easily prove the following result.

\begin{theorem}
Let $\dim H < \infty$ and the operator pencil $A(\lambda)$ satisfy the following conditions: \\
a) $\Ker A_n=0$; \\
b) $\Imm (A(\lambda)x,x) \leq 0$ for all $x \in H$ and $\lambda \in \mathbb{R}$; \\
c) there exists a point $\lambda_0 \in \mathbb{R}$ such that $0 \notin \theta(A(\lambda_0))$. \\
Then the system $\mathcal{E}_R^+$ corresponding to the operator pencil $A(\lambda)$ is complete in $H^l$.
\end{theorem}
\begin{pf}
Suppose that there exists a vector $f=\{f_1,...,f_l\} \in H^l$, which is orthogonal to the system $\mathcal{E}_R^+$, i.e.
\begin{equation}
(f_1,y_k)+(f_2, \lambda_k y_k) + ... + (f_l, \lambda_k^{l-1} y_k)=(f(\bar{\lambda}_k), y_k)=0 \ .
\end{equation}
for all $\lambda_k$ with $\Imm\lambda_k \geq 0$, where $$f(\lambda)= f_1+\lambda f_2 +...+ \lambda^{l-1} f_l .$$
Consider the meromorphic scalar function $$F(\lambda)=([A^*(\lambda)]^{-1}f(\lambda), f(\bar{\lambda})).$$
From the representation of $[A^*(\lambda)]^{-1}$ in the neighborhood of a pole $\bar{\lambda_k}$ and the equalities (12) it follows that $F(\lambda)$ is a holomorphic function in the lower half-plane. Since $\Ker A_n = 0$, we have $||A^{-1}(\lambda)|| = O(|\lambda|^n)$,
$$|F(\lambda) = O(|\lambda|^{-n+2(l-1)}) = O(|\lambda|^{-2}).$$
Denoting $g(\lambda) = [A^*(\lambda)]^{-1} f(\lambda)$, we can rewrite $F(\lambda)$ in the form
$$F(\lambda) = (g(\lambda), A^*(\bar{\lambda}) g(\bar{\lambda})) = (A(\lambda) g(\lambda), g(\bar{\lambda})).$$
Since $g(\lambda) = g(\bar{\lambda})$ for all $\lambda \in \mathbb{R}$, it follows from condition b) that $\Imm F(\lambda) \leq 0$ for $\lambda \in \mathbb{R}$. If $\Imm F(\lambda) = 0$ for all $\lambda \in \mathbb{R}$, then it follows from the Riemann-Schwartz principle that $F(\lambda)$ has a symmetric holomorphic continuation in the upper half-plane.
Since $F(\lambda) $ has no poles on the real axis, it is an entire function. Taking into account the estimate (13), we deduce from Liouville's theorem that $F(\lambda) \equiv 0$.\\
Suppose $\Imm F(\lambda_1) \neq 0$ for some $\lambda_1 \in \mathbb{R}$. Then $\phi(\lambda) = \Imm F(\lambda)$ is an harmonic bounded function in the lower half-plane and $\phi(\lambda) \leq 0$ for $\lambda \in \mathbb{R}$. Since $\phi(\lambda)$ is not identically zero we have (according to the maximum principle) $\phi(\lambda) <0$ for $\Imm \lambda <0$. Now recall the Caratheodory theorem (see [Levin, Ch1.]): \textit{If $F(\lambda)$ is a holomorphic function in the open lower plane and $\Imm F(\lambda)<0$ for $\Imm \lambda <0$ then} $$|F(\lambda)| > \frac{1}{5} |F(-i)|\ |\lambda|^{-1} .$$
This contradicts the estimate (13). Hence $F(\lambda)\equiv 0$. \\
Now we can write
\begin{equation*}
f(\lambda)=f(\lambda_0) + \frac{1}{1!} f'(\lambda_0)(\lambda - \lambda_0) + ... + \frac{1}{(l-1)!} f^{(l-1)} (\lambda_0) (\lambda - \lambda_0)^{(l-1)}
\end{equation*}
and using condition c), we can find $f(\lambda_0) = 0,f'(\lambda_0)=0,..., f^{(l-1)} (\lambda_0) = 0$, i.e. $f(\lambda) \equiv 0$. This proves the theorem.
\end{pf}
\begin{note}
It may seem at first sight that one can omit condition c) of Theorem 1. But it is essential. For example, the self-adjoint quadratic operator pencil
$$
A(\lambda)=
\begin{pmatrix}
0 & 1 \\
1 & 0
\end{pmatrix}
+ \lambda
\begin{pmatrix}
0 & -3i \\
3i & 0
\end{pmatrix}
-\lambda^2
\begin{pmatrix}
0 & 2 \\
2 & 0
\end{pmatrix}
$$
in two-dimensional space H has two eigenvalues $\lambda_1=i$ and $\lambda_2=\frac{i}{2}$ in the upper half-plane and the same eigenvector corresponds to both eigenvalues. This example was given by G. Radzievskii [2] and, independently, a similar example was given by A. Kostyuchenko. Certainly, condition c) automatically holds for monic operator polynomials.
\end{note}
Denote by $\mathcal{E}_0^- \ (\mathcal{E}_R^-)$ a system of Keldysh derived chains of length $l$ corresponding to the eigenvalues $\lambda_k$ with $\Imm\lambda_k<0 (\Imm \lambda_k \leq 0)$.
\begin{corollary}
If the condition a) - c) of Theorem 1 are fulfilled and the operator pencil $A(\lambda)$ has no real eigenvalues, then both systems $\mathcal{E}_0^+$ and $\mathcal{E}_0^-$ form bases in the space $H^l$.
\end{corollary}
\begin{pf}
Under our assumptions $E_0^-=E_R^-$ and the completeness of the system $\mathcal{E}_0^-$ can be proved using the same methods. Hence, $\dim( Span\ \mathcal{E}_0^\pm )= \kappa^\pm \geq lm$, where $m=\dim H$. Since $\kappa^+ + \kappa^- = nm = 2lm$, we have $\kappa^+=\kappa^-=lm=\dim H^l$.
\end{pf}
If the operator pencil $A(l)$ has real eigenvalues, we cannot deduce from Theorem 1 that the systems $\mathcal{E}_0^+$, $\mathcal{E}_0^-$ are minimal in the space $H^l$. In fact they \underline{are} minimal and we will prove even a more general fact for infinite dimensional space $H$.\\
\textbf{Definition 5.6.} The system $\{e_k\}_1^\infty$ in the Hilbert space $H$ is called \textit{linearly independent} if any finite system is linearly independent.\\
\begin{note}  Certainly, if the system $\{e_k\}_1^\infty$ is minimal then it is linearly independent. The converse assertion is not true. For example, the system of functions $\{x^k\}_{k=0}^\infty \in L_2 [0,1]$ is linearly independent, but it is not minimal in $L_2[0,1]$.
\end{note}
\begin{definition}
The point $\lambda_k \in \sigma(A)$ is called \textit{a point of the discrete spectrum} of the operator pencil $A(\lambda)$ if it is an isolated point of $\sigma(A)$ and the resolvent $A^{-1}(\lambda)$ has the representation (2.2) in some neighborhood of $\lambda_k$. The set of all such points we denote by $\sigma_d (A)$.
\end{definition}

The following theorem is also due to G. Radzievskii [2] (1987).

\begin{theorem}
Let the operator pencil $A(\lambda)$ satisfy the conditions b), c) of Theorem 1 and all operators $A_j,\ j=0,1,...,n$, be bounded. Let $\mathcal{E}_0^+$ be the system $\{\hat{y}_k\}$ consisting of all Keldysh derived chains of length $l$, corresponding to $\lambda_k \in \sigma_d (A)$ with $\Imm \lambda_k>0$. Then the system $\mathcal{E}_0^+$ is a linearly independent system.
\end{theorem}
\begin{pf}
Suppose that
\begin{equation}
\sum_{k=1}^{N}{c_k \hat{y}_k}=\sum_{k=1}^{N}{c_k\{y_k, \lambda_k y_k,..., \lambda_k^{l-1} y_k\}} =0,
\end{equation}
where $\lambda_k \in \sigma_d(A)$ and $\Imm \lambda_k > 0$. Consider the meromorphic scalar function
\begin{equation*}
f(\lambda)=(A(\lambda)d(\lambda), d(\bar{\lambda})), \ \ \ \ \ d(\lambda)=\sum_{k=1}^N {\frac{c_k y_k}{\lambda-\lambda_k}}.
\end{equation*}
Notice that $\infty$ is a regular point for $d(\lambda)$ and at infinity is the Laurent expansion $$d(\lambda)=a_1 \lambda^{-1} + a_2 \lambda^{-2} + ... \ ,$$
where $$a_j = \sum_{k=1}^N {c_k \lambda_k^{j-1} y_k},\ \ j=1,2,...\ .$$
Evidently, $||A(\lambda)|| \leq M|\lambda|^{2l}, 2l=n$. Thaking into account (14), we have $a_1=a_2=...=a_l=0$. Hence
\begin{equation*}
|F(\lambda)|=O(|\lambda|^{2l-2(l+1)})= O(|\lambda|^{-2}), \lambda \rightarrow \infty .
\end{equation*}
It appears that the function $F(\lambda)$ may have poles at the points $\{\lambda_k\}_1^N$  and $\{\bar{\lambda_k}\}_1^N$. However, since $A(\lambda_k)y_k = 0$, all points $\{\lambda_k \}_1^N$ are regular
and $F(\lambda)$ is holomorphic function in the closed upper half-plane. It follows from condition b) that $\Imm F(\lambda) \leq 0$ for all $\lambda \in \mathbb{R}$. Now we can repeat the arguments in the proof of Theorem 1 and deduce $F(\lambda) \equiv 0$. Using condition c), we obtain $d^{(j)}(\lambda_0)=0$, for all $j=0,1,...\ $. Then it follows from the uniqueness theorem for holomorphic functions, that $d(\lambda) \equiv 0$. Hence, $c_k=0, \ k=1,...,N$.
\end{pf}
We can easily deduce from Theorems 1 and 2 the following assertion.

\begin{corollary}
If conditions a) - c) of Theorem 1 are fulfilled then for any given vectors $\{\phi_j \}_0^{l}$ there exists a solution of problem (1), (6), (9). Under conditions b), c) of Theorem 1 the solution of problem (1), (6), (8) is unique.
\end{corollary}
\begin{note} If the pencil $A(\lambda)$ has real eigenvalues then under the same assumptions we cannot 	guarantee the existence of solution (1), (6), (8) and the uniqueness of the solution (1), (6), (9). Hence, in this case we have to replace condition (8) or (9) by a more refined one. To do this we have to consider some concrete problems in mathematical physics.
\end{note}

\newpage
%Лекция 06
\section{Mandelstam radiation principle (non-resonant case) and half-range problems}

The linear differential equations which arise in the theory of electromagnetic waves in elasticity theory can often be reduced to the following form:
\begin{equation}
-A\frac{\partial^2u}{\partial z^2} - iB\frac{\partial u}{\partial z} + Cu + \frac {\partial^2 u}{\partial t^2} =0,
\end{equation}
where the function $u(t,z)=u(t,x,y,z)$ takes values in a Hilbert space $H$, the variable $t$ denotes the time, operators $A,\ B$ are symmetric and $C$ is self-adjoint in $H$. Such kinds of equations arise in the \textit{wave-guide} regions $Q=\Omega \times R^+$, where $\Omega$ is a bounded domain in the $x-y$ plane and the direction of $z$ is orthogonal to this plane (see figure 1). Then the role of $H$ is played by the space $L_2 (\Omega)$. Plane wave-guide regions may also be considered. In this situation $\Omega$ is an interval (see figure 2).
The solutions of the wave equation (1), which are periodic in time, i.e.
\begin{equation}
u(t,z)=v(z)e^{i\omega t},\ \ \ v(z)=v(z,x,y)
\end{equation}
are of considerable interest. The constant $\omega$ is called the \textit{angular frequency}. Substituting (2) into (1) we obtain the equation of \textit{stable oscillations} with given frequency $\omega$
\begin{equation}
-A\frac{d^2 v}{dz^2}- iB \frac{dv}{dz} +Cv - \omega^2 Iv = 0.
\end{equation}
Let $w_k$ be the eigenvectors, corresponding to the eigenvalues $\lambda_k$, of the related operator pencil
\begin{equation}
L_\omega(\lambda) = \lambda^2 A + \lambda B +C - \omega^2 I.
\end{equation}
The elementary solutions $v_k(z)=w_k e ^{i\lambda_k z}$ of equation (3) are called \textit{propagating waves} if $\lambda_k \in \mathbb{R}$ and the \textit{evanescent waves} if $\Imm\lambda_k>0$. The waves with $\Imm\lambda_k < 0$ have no physical meaning. The eigenvectors $w_k$ are called \textit{the amplitudes} and $\lambda_k$ are called \textit{the wave-numbers}. \\
\textbf{Example.} The simplest but also important equation of type (3) is Helmholtz' equation in the semi-strip $Q = [0,1] \times [0,\infty)$
\begin{equation}
-\Delta v - \omega^2 v = 0,\ \ v=v(x,z), \ \ \Delta = \frac{\partial^2}{\partial x^2} + \frac {\partial^2}{\partial z^2},
\end{equation}
\begin{equation}
v(0,z) = v(1,z)=0,\ 0\leq z<\infty.
\end{equation}
Denote by $C$ the operator $Cw=-w''$ with domain of definition
$$D(C)=\{w\ |\ w \in W_2^2 [0,1],\ w(0)=w(1)=0\},$$
where $W_2^2[0,1]$ is Sobolev space consisting of the functions $w(x)$ such that $w$ and $w'$ are absolutely continious and $w'' \in L_2[0,1]$. Then $C$ is a positive self-adjoint operator (see, for example, [Najmark]) in the space $H=L_2[0,1]$ and the problem (5), (6) can be rewritten in the form
\begin{equation}
-\frac{d^2 v}{dz^2} + Cv - \omega^2 v = 0.
\end{equation}
We also have to impose the initial condition
\begin{equation}
v(x,0)=v(0)=\phi,\ \ \phi=\phi(x) \in L_2[0,1],
\end{equation}
and to define a restiction on the behaviour of the solution $v(z)$ when $z\rightarrow \infty$.\\
The elementary solutions of equation (7) have the representation
$$v_k(z)=w_k(x)e^{i\lambda_k z},\ \ k=\pm 1, \pm 2,...,$$
where $\lambda_k = \sqrt{\omega^2-\pi^2 k^2}$ and $w_k(x) = \sin{\pi kx}$ are the solutions of the eigenvalue problem $$(\lambda^2 I +C -\omega^2 I) w = 0 .$$
If $\omega<\pi$ then the pencil $\lambda^2 I + C - \omega^2 I$ has no real eigenvalues and we can represent the soluiton of the problem (7), (8) in the form
$$v(z)=\sum_{k=1}^{\infty}{c_k v_k(z)}=\sum_{k=1}^{\infty}{c_k e^{-\sqrt{\pi^2 k^2 -\omega^2}\ z}} \sin{\pi kx} ,$$
where $c_k = \frac{1}{\pi}(\phi(x), \sin{\pi kx})_{L_2}$. Obviously, this solution satisfies the condition
\begin{equation}
v(z) \rightarrow 0 \ \ when \ z \rightarrow \infty.
\end{equation}
If $\omega>\pi$ then the equation (7) has a finite set of real wave-numbers $\lambda_k$ (see figure 3). Now we can not find for any $\phi \in H = L_2[0,1]$ the solution of the problem (7), (8) satisfying the condition (9). But if we replace the condition (9) by the condition $||v(z)|| \leq const$, we can not guarantee the uniqueness of the solution. For example, if $\phi(x)=\sin{\pi x}$ then both functions
$$v^+(z)=e^{i\sqrt{\omega^2 - \pi^2}\ z} \sin{\pi x},\ \ v^-(z) = e^{-i\sqrt{\omega^2 - \pi^2}\ z}\sin{\pi x}$$
are bounded and satisfy (7), (8). Hence, to select the unique solution, we consider not \textit{all} propagating waves of the equation (7) but only half of them. How does one choose this half? In our particular case the answer is easy: for example, we can choose the half of propagating waves corresponding to all positive wave-numbers $\lambda_k$. The real wave-number $\lambda_k$ characterizes \textit{the phase velocity} of the corresponding wave. The wave with $\lambda_k>0$ runs to positive infinity, but the wave with $\lambda_k<0$ runs to negative infinity. Thus we can select the propagating waves according to their phase velocities and claim that the waves with positive phase velocity have physical meaning, but the waves with negative phase velocity do not. Then we come to the following principle.\\ \\
\textbf{Sommerfeld Radiation Princliple.} The solution of equation (7) must have the representation $$v(z)=v_0(z)+v_1(z),$$
where $||v_0(z)|| \rightarrow 0\ $ if $\ z \rightarrow \infty\ $ and $v_1(z)$ is a finite superposition of propagating waves with positive phase velocity, i.e.
$$v_1(z)=\sum_{\lambda_k >0}{c_k w_k e^{i \lambda_k z}},$$
where $\{w_k\}$ are the amplitudes and $\{c_k\}$ are some coefficients. \\
It is not difficult to prove the uniqueness and existence of the solution of (7), (8), satisfying the Sommerfeld radiation principle. But more than 50 years ago physicists discovered that the Sommerfeld radiation principle does not work for more complicated equations type (3). To establish a new principle we have to introduce the \textit{group velocity} of the wave. Using Rellich's theorem (see Ch. 7.2 of Kato [1]) we find that the real eigenvalues $\lambda_k=\lambda_k(\omega)$ of the operator pencil (4) are holomorphic functions of $\omega$ with the exception of some exclusive frequencies $\omega = \xi_k, \ \xi_k \rightarrow \infty$, which are called the \textit{resonance frequencies} ($\omega$ is called a resonance frequency if there exists at least one real eigenvalue $\lambda_k$ which is not semisimple). Hence, for all non-resonant frequencies the functions $\lambda_k ' (\omega)$ are well defined. The number $\frac{1}{\lambda_k ' (\omega)}$ is called \textit{the group velocit
 y} of the propagating wave $v_k(z)=w_k e^{i \lambda_k z}$. Now we can formulate the other principle of wave selection. \\ \\
\textbf{Mandelstam Radiation Principle.} The solution of equation (3) must have the representation $$v(z)=v_0(z) + v_1(z),$$
where $||v_0(z)|| \rightarrow 0 $ when $z \rightarrow \infty$ and $v_1(z)$ \textit{is a finite superposition of propagating waves with positive group velocity.}
Our next goal is to obtain the different representations for the group velocity $1/\lambda_k ' (\omega)$. Further, we suppose that $C$ is a self-adjoint operator and the domains of definition of the symmetric operators $A,B$ contain the domain of $C$.
\begin{proposition}
If $\omega>0$ is not a resonant frequency of the self-adjoint operator pencil (4) then
\begin{equation}
\frac{1}{\lambda_k ' (\omega)} = \frac{(L_\omega ' (\lambda_k) w_k, w_k)}{2\omega(w_k,w_k)}
\end{equation}
where $L_\omega ' (\lambda_k) = 2\lambda_k A + B$ and $w_k$ is the corresponding amplitude.
\end{proposition}
\begin{pf}
\footnote{Similar assertions were discovered for matrix polynomials by I. Gohberg, P. Lancaster and L. Rodman [1] (1979), by I. Vorovich and V. Babeshko [1] (1979), by A. Kostyuchenko and M. Orazov [1] (1981), by A. Zilbergleit and Ju. Kopilevich [1] (1983).}
Let $\lambda_k$ be a real eigenvalue corresponding to the amplitude $w_k$ of the pencil $L_\omega(\lambda)$ with a non-resonant frequency $\omega>0$. For fixed $\lambda$ in the real neighborhood of the point $\lambda_k$ consider the eigenvalue problem
\begin{equation}
(L_0(\lambda)-\xi^2 I)y(\lambda)=0\ \ \ \ L_0(\lambda)= A\lambda^2 +B\lambda +C,
\end{equation}
viewing $\xi$ as a spectral parameter. According to Rellich's theorem (see Ch. 7.2 of Kato [1]) the eigenvalue $\xi_k = \xi(\lambda)$, such that $\xi(\lambda_k)=\omega>0$, is a holomorphic function of $\lambda$ in the real neighborhood of $\lambda_k$ and there is a corresponding holomorphic eigenvector $y_k=y(\lambda),\ y(\lambda_k)=w_k$. It follows from (11) that
$$
[L_0 ' (\lambda) - 2\xi(\lambda) \xi ' (\lambda)] y(\lambda) + [L_0(\lambda) - \xi^2 (\lambda) I ] y'(\lambda) = 0.
$$
Substituting $\lambda = \lambda_k,\ \xi(\lambda_k) = \omega,\ y(\lambda_k) = w_k$, we obtain from this equation
$$([L_0 '(\lambda_k) w_k - 2\omega \xi'(\lambda_k)w_k + L_\omega (\lambda_k) y'(\lambda_k)], w_k)=$$
$$=(L_\omega ' (\lambda_k)w_k,w_k) - 2\omega \xi'(\lambda_k)(w_k,w_k)=0. $$
Taking into account that
$$ \lambda_k ' (\omega) = \frac{d\lambda_k(\xi)}{d\xi} \bigg|_{\xi=\omega} = \bigg[ \frac{d\xi(\lambda)}{d\lambda} \bigg|_{\lambda=\lambda_k} \bigg]^{-1} = [\xi'(\lambda_k)]^{-1}, $$
we obtain the relation (10).
\end{pf}
\begin{proposition} [Kostyuchenko and Shkalikov ${[1]}$, 1983]
If $\lambda_k$ is a simple real eigenvalue of the operator pencil $L_\omega(\lambda)$ then the principal part of the resolvent $L_\omega^{-1} (\lambda)$ at the pole $\lambda_k$ has the representation $$\frac{\varepsilon_k(\ \cdot \ ,w_k)w_k}{\lambda-\lambda_k},$$
where
\begin{equation}
\varepsilon_k=\frac{1}{(L'(\lambda_k)w_k, w_k)}.
\end{equation}
\end{proposition}
\begin{pf}
Assuming that $\lambda_k$ is a simple eigenvalue of $L_\omega(\lambda)$, we imply that it is also a point of the discrete spectrum. Hence the principal part of $L_\omega^{-1}(\lambda)$ has the representation $$\frac{(\ \cdot\ , z_k)\ w_k}{\lambda-\lambda_k},$$
where $z_k \in \Ker L_\omega^*(\lambda_k)$. Since $L_\omega^*(\lambda_k) = L_\omega(\lambda_k)$, we can find a number $\varepsilon_k$ such that $z_k = \varepsilon_k w_k$. From the identity
\begin{align*}
w_k &= L_\omega(\lambda)L_\omega^{-1}(\lambda) w_k \\
&= \big[ L_\omega(\lambda_k) + (\lambda - \lambda_k) L_\omega ' (\lambda_k) + ... \big] \big[ \frac{\varepsilon_k(w_k,w_k)w_k}{\lambda - \lambda_k} + R(\lambda_k)w_k + ... \big]  \\
&= \varepsilon_k(w_k,w_k) L_\omega ' (\lambda_k)w_k + (\lambda - \lambda_k)[...] + ...
\end{align*}
we find $$(w_k,w_k)= \varepsilon(w_k,w_k)(L_\omega ' (\lambda_k)w_k, w_k) $$
and the equality (12) follows.\\
\begin{note}
The assertion of Proposition 2 is valid not only for seladjoint operator pencil $L_\omega(\lambda)$ but for arbitrary operator pencil $A(\lambda)$, satisfying the condition $\Imm (A(\lambda)x,x) \leq 0$ for all $x \in H$ and all $\lambda \in \mathbb{R}$. Indeed, according to note 4.2, we have in this case $\Ker A(\lambda_k)= \Ker A^*(\lambda_k)$, hence we can repeat the arguments in the proof of Proposition 2.
\end{note}
\end{pf}
For the simple eigenvalue $\lambda_k$ the number $\ \varepsilon_k = sign(L'(\lambda_k)w_k,w_k)$ is called \textit{the sign charachteristic} of the corresponding eigenvector $w_k$ (see, for example, Ch. 10 of Gohberg, Lancaster and Rodman [2]). According to the definition given in the papers Daffin [1] and Langer and Krein [1], the simple eigenvalue $\lambda_k$ is called \textit{an eigenvalue of positive (negative) type} if the corresponding number $\epsilon_k > 0 (<0)$.\\
Taking into account Note 1 and the equality $sign\ \lambda_k ' (\omega) - sign\ \varepsilon_k$, we can reformulate the Mandelstam radiation principle for differential equations of an arbitrary order, if the corresponding operator pencil is dissipative.\\
As we have mentioned, the concrete problems of mathematical physics involve unbounded operators. In this lecture we will consider the problem on solvability of equations of arbitrary order, but only in finite dimensional space $H$.\\
As in Section~5, let us consider the problem $(n=2l)$
\begin{equation}
A(-i\frac{d}{dz})v(z) = A_0 v - iA_1\frac{dv}{dz} + ... + (-i)^nA_n \frac{d^n v}{dz^n} = 0,
\end{equation}
\begin{equation}
(-i)^{j_v (j)} (0) = \phi_j, j=0,1,...,l-1,
\end{equation}
where $A_0,...,A_n$ are operators acting in finite dimensional space $H$. Assume, that the related operator pencil $$A(\lambda) = A_0 + \lambda A_1 + ... + \lambda^n A^n $$ is dissipative, i.e.
\begin{equation}
\Imm (A(\lambda)x,x) \leq 0  \ \ for\ all\ x \in H \ and\ all\ \lambda \in \mathbb{R}.
\end{equation}
Assume also that all real eigenvalues of $A(\lambda)$ are simple. Now, let us introduce the systems $E^\pm$ which we call the first and the second part of eigen and associated vectors of $A(\lambda)$ respectively:
\begin{equation*}
E^+ = \{w_k^h\}_{\Imm \lambda_k>0} \cup \{w_k\}_{\lambda_k \in \mathbb{R}, \varepsilon_k>0}\ ,
\end{equation*}
\begin{equation*}
E^- = \{w_k^h\}_{\Imm \lambda_k<0} \cup \{w_k\}_{\lambda_k \in \mathbb{R}, \varepsilon_k<0}\ .
\end{equation*}
Define also systems the systems $\mathcal{E}^+ \ (\mathcal{E}^-)$ consisting of Keldysh derived chains of length $l$, corresponding to vectors $w_k^h \in E^+ \ (E^-)$. Remembering the definition of the systems $\mathcal{E}_0^\pm, \mathcal{E}_R^\pm$ given in Section~5, we notice that $\mathcal{E}_0^\pm \subset \mathcal{E}^\pm \subset \mathcal{E}_R^\pm$. \\
The following problem is of our interest: to find a solution of equation (13) satisfying the initial conditions (14) and the Mandelstam radiation principle at infinity. We say this is \underline{the half-range Cauchy problem}. Observe that (in finite dimensional case only!) a solution $v(z)$ of (13) satisfies the Mandelstam radiation principle if and only if $$v^{(j)}(0) \in \Span E^+,\ j=0,1,...,n-1.$$
Repeating the arguments in the proof of proposition 5.2 we obtain the following result.
\begin{proposition}
For any given initial vectors $\{\phi_j\}^{l-1}$ there exists a unique solution of the problem (13), (14), satisfying the Mandelstam radiation principle at infinity if and only if the system $\mathcal{E}^+$ is a basis in $H^l$.
\end{proposition}
\begin{theorem} [Shkalikov ${[2]}$, 1985]
Let $\dim H<\infty$, the pencil $A(\lambda)$ satisfies the condition (15), $\Ker A_n = 0$ and let there exist a point $\lambda_0 \in \mathbb{R}$, such that
\begin{equation}
0\notin \theta(A(\lambda_0)).
\end{equation}
then the systems $\mathcal{E}^+$ and $\mathcal{E}^-$ are complete in $H^l$.
\end{theorem}
\begin{pf}
Consider, for example, the system $\mathcal{E}^+$. Suppose there exists a vector $f=\{f_1,...,f_l\} \in H^l$, which is orthogonal to the system $\mathcal{E}^+$, i.e.	
$$ \lambda_k(f_1,y_k) + (f_2,\lambda_k y_k) + ... + (f_l, \lambda_n^{l-1} y_k) = 0 $$
for all $\lambda_k$ with $\Imm \lambda_k > 0$ and real $\lambda_k$ is of positive type. Consider the meromorphic scalar function
$$
F(\lambda) = ([A^*(\lambda)]^{-1} f(\lambda), f(\bar{\lambda})),\ \ f(\lambda) = f_1 + \lambda f_2 + ... + \lambda^{l-1} f_{l-1},
$$
where $A^*(\lambda) = [A(\bar{\lambda})]^*$. \\
Repeating the arguments in the proof of Theorem 5.1, we conclude that $F(\lambda)$ has no poles in the lower half-plane as well as in the real points $\lambda_k$ of positive type. Hence, on the real axis the function $F(\lambda)$ may have poles only at points $\lambda_k$ of negative type and according to Proposition 2 the principal part of $F(\lambda)$ at pole $\lambda = \lambda_k$ is equal to
\begin{equation}
\frac{\varepsilon_k(f(\lambda_k),w_k)(w_k,f(\lambda_k))}{\lambda - \lambda_k} = \frac{\varepsilon|(f(\lambda_k),w_k)|^2}{\lambda - \lambda_k} .
\end{equation}
Consider the contour $\Gamma = C_R \cup I_1 \cup C_\varepsilon^1 \cup ... \cup I_q \cup C_\varepsilon^q \cup I_{q+1}$, which is depicted in Figure 4 ($C_R$ is a large semi-circle of radius $R$, $C_\varepsilon^k, k=1,...,q$, are small semicircles of radii $\varepsilon$ with centers in poles $\lambda_1, \lambda_1,...,\lambda_q$ and $I_1,...,I_{q+1}$ are intervals on the real axis). Then
\begin{equation}
\int_\Gamma {F(\lambda)d\lambda} = \big( \int_{C_R} + \sum_{k=1}^{q+1} \int_{I_k} + \sum_{k=1}^q \int_{C_{\varepsilon}^k} \big) F(\lambda)d\lambda = 0	
\end{equation}
Since $\Ker A_n = 0$, we deduce (see the estimate (5.13)) that $|F(\lambda)| = O(|\lambda|^{-2})$ when $\lambda \rightarrow \infty$, hence
\begin{equation}
\int_{C_R} {F(\lambda)d\lambda} \rightarrow 0,\ \ when \ R\rightarrow \infty.
\end{equation}
According to (17), we have
\begin{equation}
\sum_{k=1}^q \int_{C_{\varepsilon}^k} F(\lambda) \rightarrow \pi i \sum_{k=1}^q \varepsilon_k |(f(\lambda_k), w_k)|^2 \ \ if \ \varepsilon \rightarrow 0,
\end{equation}
and
\begin{equation}
\sum_{k=1}^q \int_{I_k} F(\lambda) d\lambda \rightarrow V.P. \int_{-\infty}^{\infty} F(\lambda) d \lambda \ \ if \ \varepsilon \rightarrow 0, \ R \rightarrow \infty.
\end{equation}
Hence, from (18) - (21) we have
\begin{equation}
\pi i \sum_{k=1}^q {\varepsilon_k |(f(\lambda_k),w_k)|^2} + V.P. \int _{-\infty} ^{\infty} F(\lambda) d \lambda = 0.
\end{equation}
Now notice that $\varepsilon_k<0$ in (22) for all $k$ and $\Imm F(\lambda) \leq 0$ if $\lambda \in \mathbb{R}$ (see Theorem 5.1). Then we immediately obtain from (22) that $(f(\lambda_k),w_k) = 0,\ k=1,...,q$, and $\Imm F(\lambda) \equiv 0$ for $\lambda \in \mathbb{R}$. Hence, $F(\lambda)$ is a real fuction on the real axis and has no real poles. Using the Riemann-Schwartz symmetry principal and Liouville's theorem, we obtain $F(\lambda) \equiv 0$. Then condition (16) allows us to conclude $f(\lambda) \equiv 0$.
\end{pf}
\begin{theorem}[Radzievskii ${[2]}$, 1987]
Let the conditions of Theorem 1 hold with the exception of conditions $\Ker A_n = 0$ and $\dim H<\infty$. Then the system $\mathcal{E}^+$ and $\mathcal{E}^-$ are linearly independent.
\end{theorem}
\begin{pf}
Following the proof of Theorem 5.2, consider the meromorphic function
$$
F(\lambda) = (A(\lambda) d(\lambda), d(\bar{\lambda})),\ \ \ \ \ d(\lambda) = \sum_{k=1}^N {\frac{c_k w_k}{\lambda - \lambda_k}},
$$		
where $w_k$ are the eigenvectors corresponding to eigenvalues $\lambda_k$ with $\Imm \lambda_k>0$ and real $\lambda_k$ of positive type. It was proved in Theorem 5.2 that $F(\lambda)$ is holomorphic in the open upper half-plane. Obviously, the principal part of $F(\lambda)$ at the real pole $\lambda_k$ is equal to
$$\frac{(A'(\lambda_k)w_k,w_k)|c_k|^2}{\lambda - \lambda_k}.$$
By virtue of Note 1 it follows that
$$sign(A'(\lambda_k)w_k,w_k) = sign\ \varepsilon_k .$$
Hence, all residues of the function $F(\lambda)$ corresponding to real poles are positive, $\Imm F(\lambda) \leq 0 $ for all $\lambda \in \mathbb{R}$ and $|F(\lambda)| = O(|\lambda|^{-2})$ if $\lambda \rightarrow \infty$. Taking the contour $\Gamma$ depicted in Figure 5 and repeating the arguments of Theorem 1, we conclude that $F(\lambda) \equiv 0$. Now condition (16) shows that $d(\lambda) \equiv 0$, hence the system $\mathcal{E}^+$ is linearly independent. The same arguments apply to the system $\mathcal{E}^-$.
\end{pf}
As a corollary of Theorems 1 and 2 and Proposition 1 we obtain the following result.
\begin{theorem}
Let $\dim H < \infty$ and the pencil $A(\lambda)$ satisfy the conditions of Theorem 5.1. Then for any given initial vectors $\{\phi_j\}_0^{l-1}$ there exists a unique solution of the problem (13), (14) satisfying the Mandelstam radiation principle at infinity.
\end{theorem}

\begin{note}
If $\Ker A_n = 0$ and $\dim H<\infty$ then the duality principle is valid: if the systems $\mathcal{E}^+$ and $\mathcal{E}^-$ are complete in $H^l$ then they are linearly independent and vice versa. Indeed, let $x^+\ (x^-)$ denote the number of vectors of the system $\mathcal{E}^+\ (\mathcal{E}^-)$ and let $\dim H = m$. It follows from the definition of the systems $\mathcal{E}^\pm$ and Theorem 1.1 that
\begin{equation}
x^+ + x^- = mn = 2ml.
\end{equation}
If $\mathcal{E}^+$ and $\mathcal{E}^-$ are complete then $x^+ \geq ml$ and $x^- \geq ml$. Now (23) implies $x^+ = x^- = ml$, hence the system $\mathcal{E}^+ \ (\mathcal{E}^-)$ is a basis. On the contrary if $\mathcal {E}^+$ and $\mathcal{E}^-$ are linearly independent then $x^+ \leq ml,\ x^- \leq ml$ and we obtain again from (23) that $\mathcal{E}^+\ (\mathcal{E}^-)$ is a basis.
\end{note}

\begin{figure}[t]
\begin{center}
\begin{tikzpicture}[scale=1.5]
\draw [thin, ->] (0,0)--(4.5,0);
\draw [thin, ->] (0,0)--(0,1.2);
\draw [thick, ->] (0,0)--(-0.8,-0.8);
\draw [thin] (0,0.6)--(4,0.6);
\draw [thin] (0,-0.3)--(4,-0.3);
\node [right] at (4.6,0) {$z$};
\node [below right] at (0,1.2) {$y$};
\node [left] at (-0.9,-0.8) {$x$};
\end{tikzpicture}
\par\medskip Figure~1
\end{center}
\end{figure}

\begin{figure}[t]
\begin{center}
\begin{tikzpicture}[scale=1.5]
\draw [thin, ->] (0,0)--(5.2,0);
\draw [thin, ->] (0,0)--(0,1.4);
\draw [thin] (0,0.8)--(5,0.8);
\node [below] at (5,0) {$z$};
\node [left] at (0,1.4) {$x$};
\node [below left] at (0,0) {$0$};
\node [left] at (0,0.8) {$1$};
\node [left] at (0,0.4) {$v(x,0)=\varphi$};
\node [below right] at (1.4,0) {$v(0,z)=0$};
\node [above right] at (1.4,0.8) {$v(1,z)=0$};
\node [right] at (1.2,0.5) {$\Delta v-\omega^2 v=0$};
\end{tikzpicture}
\par\medskip Figure~2
\end{center}
\end{figure}

\begin{figure}[t]
\begin{center}
\begin{tikzpicture}[scale=1.5]
\draw [thin, ->] (-1.6,0)--(3.5,0);
\draw [thin, ->] (0,-0.5)--(0,1.1);
\end{tikzpicture}
\par\medskip Figure~3\par
(wave-number distribution of Helmholtz equation.)
\end{center}
\end{figure}

\begin{figure}[t]
\begin{center}
\begin{tikzpicture}[scale=1.5]
\draw [thin] (-2.5,0)--(2.5,0);
\draw [thin, ->] (0,-2.1)--(0,0.8);
\node [above] at (-2.5,0) {$I_1$};
\node [above] at (-1.5,0) {$I_2$};
\node [above] at (1.7,0) {$I_{q+1}$};
\node [below] at (-2,0) {$C^1_\varepsilon$};
\node [below] at (-1,0) {$C^2_\varepsilon$};
\node [below] at (1.2,0) {$C_q^\varepsilon$};
\end{tikzpicture}
\par\medskip Figure~4
\end{center}
\end{figure}

\begin{figure}[t]
\begin{center}
\begin{tikzpicture}[scale=1.5]
\draw [thin] (-2.5,0)--(2.5,0);
\draw [thin, ->] (0,-0.5)--(0,1.5);
\node [below] at (-2.5,0) {$I_1$};
\node [below] at (1.7,0) {$I_{q+1}$};
\node [above] at (-1.7,0) {$C^1_\varepsilon$};
\node [above] at (1.4,0) {$C_q^1$};
\node [right] at (-1,1.5) {$C_R$};
\node [right] at (3,1.2) {$\lambda$-plane};
\end{tikzpicture}
\par\medskip Figure~5
\end{center}
\end{figure}

\newpage
%Лекция 07
\section{Generalized Mandelstam radiation principle (resonant case). The factorization of a quadratic pencil}

We have assumed in proving theorems on half-range completeness and half-range minimality that all eigenvalues of the operator pencil A($\lambda$) are simple. Certainly, the simplicity of the non-real eigenvalues is not essential. For example, if vector $f= \bigl\{ f_1 , \dots , f_{\ell} \bigr\}$ is orthogonal to all Keldysh derived chains of length l corresponding to eigenvalues $\lambda_k$ with $\Imm~\lambda_k > 0 $ then the vector function
$$
[A^\ast(\lambda)]^{-1}(f_1 + \lambda f_2 + \dots + \lambda^{\ell-1} f_\ell)	
$$
is holomorphic in the lower half-plane~(see Note 3.2). Hence, the proof of Theorem~6.1 does not change. Similarly, in Theorem~6.2 we have only to replace the function $d(\lambda)$ by the function
\begin{equation}
d(\lambda) = \sum_{\Imm\lambda_k>0} ~\sum_{h=0}^{p_k} \frac{c_{k,h}					 ~w^h_k}{(\lambda - \lambda_k)^{p_k + 1 - h}}~+~\sum_{\lambda_k \in 					 \mathbb {R},  \varepsilon_k > 0} \frac{c_k w^0_k}{\lambda-\lambda_k}, 		 \end{equation}
where $w^0_k,w^1_k,\dots,w^{p_k}_k$ are the chains of canonical systems corresponding to the eigenvalues $\lambda_k$ ($\lambda_k$ is repeated as many times as its geometric multiplicity). But if the real eigenvalues of $A(\lambda)$ are not simple, then we come to a new problem, which is serious even in the finite dimensional case. In this situation we have to select the proper subset of elements from the canonical system
\begin{equation}
w^0_k,w^1_k,\dots,w^{p_k}_k, ~~k = N_1,\dots,N_2
\end{equation}
corresponding to the real eigenvalue $\lambda_k = c$ (the number $N_2-N_1+1$ is the geometric multiplicity of c).

Actually, we can not select the proper subset from \underline {any} canonical system (1).	First we have to choose a special canonical system and then to divide it. The problem on selection of elements from the canonical system (1) was originated in the paper of Kostyuchenko and Orazov [2] (1975). In this paper an important supplement was made to the remarkable theorem of Krein and Langer [1], which asserts: \textit{if $L(\lambda) = \lambda^2 I + \lambda B + C$, where $B = B^{*}$ is bounded and C is a positive compact operator then $L(\lambda)$ admits a factorization
$$
L(\lambda) = (\lambda I + B -Z)(\lambda I + Z),
$$
such that the spectrum of the operator Z lies in the closed upper-half plane and coincides with the spectrum of $L(\lambda)$ in the open upper-half plane. } A natural question arises: How to divide the Jordan chains of $L(\lambda)$, corresponding to the real eigenvalues $\lambda_k \in \sigma (L)$, to obtain the Jordan chains of Z? This problem can be solved by using a geometrical approach, because the problem on factorization of $L(\lambda)$ is equivalent to the existence of the maximal invariant subspace for the linearization $\mathscr{Z}$ in Klein space with indefinite metric G, where
$$
\mathscr{Z} =
\begin{bmatrix}
B & C^{\frac{1}{2}} \\ -C^{\frac{1}{2}} & 0
\end{bmatrix} ,
G = \begin{bmatrix}
I & 0 \\0 &-I
\end{bmatrix}.
$$
Investigations in this field have a long history. Important ideas on connection between factorization and existence of the maximal invariant endspace were developed in the paper Langer~[3]. For finite dimensional space a comprehensive treatment of this theory can be found in the books of Gohberg, Lancaster and Rodman [2-4]. However, we will use an analytic, rather than a geometric approach. First we will prove the existence of a special canonical system and then we will be able to divide it and to select the proper part. The following result is due to Kostyuchenko and Shkalikov [1] (1983).
\begin{theorem}[Theorem on the Existence of a Normal Canonical System]

Let $A(\lambda)$ be holomorphic self-adjoint\footnote{An operator function $A(\lambda)$ is called self-adjoint in a neighborhood of point c if there exist $\varepsilon > 0$ such that $A(\lambda) = [A(\lambda)]^\ast$ for all $\lambda : c - \varepsilon < \lambda < c + \varepsilon$. } operator function in a neighborhood of the real point c and c be the point of discrete spectrum of $A(\lambda)$, i.e., the resolvent $A^{-1}(\lambda)$ has a pole at $\lambda = =c$ and the principal part of $A^{-1}(\lambda)$ at this pole has the representation (cf.(2.2))
\begin{multline}
\sum_{k=N_1}^{N_2} \frac{(\cdot,z^0_k)w^0_k}{(\lambda - c)^{p_k + 1}} +
\frac{(\cdot,z^1_k)w^0_k + (\cdot,z^0_k)w^1_k}{\lambda - c)^{p_k}} +\dots\\
{} + \frac{(\cdot,z^{p_k}_k)w^0_k + (\cdot,z^{p_k-1}_k)w^1_k + \dots + (\cdot,z^0_k)w^{p_k}_k}{\lambda - c)}
\end{multline}
where
\begin{equation}
z^0_k,z^1_k,\dots, z^{p_k}_k,~~~~k = N_1,\dots,N_2
\end{equation}
is the adjoint canonical system to the canonical system (2) of EAV of the operator function $A(\lambda)$.

Then a canonical system (2) can be chosen in such a way that
\begin{equation}
z^h_k = \varepsilon_k w^h_k, ~~~k = N_1,\dots,N_2;~~h = 0,1,\dots,p_k,
\end{equation}
where $\varepsilon_k = \pm 1$.
\end{theorem}

\begin{pf}
Let $N_1 = 1,N_2 = N$ and $q \ge 1$ be such an integer that $p_1 = p_2 = \dots = p_q > p_{q+1} \ge \dots \ge p_N$. Denote by $L_h$ the operators of finite range coinciding with the coefficients of the powers $(\lambda - c)^{-p_1 - 1 - h}, h = 0,1,\dots,p_1$, in the representation (3). Evidently, the operators $L_h$ are self-adjoint. In particular, the operator $L_0$ is self-adjoint, hence we can find the vectors $e^1_1,\dots,e^0_q$ such that
$$
L_0 = \sum_{k = 0}^q \varepsilon_k(\cdot,e_k)e_k, ~~\varepsilon_k = \pm 1.
$$
Evidently, the vectors $\{ z^0_k\}^q_1$ and $\{ w^0_k\}^q_1$ lie in $\Span\{ e^0_k\}^q_1$. Then we obtain from (3) the following representation
\begin{equation}
L_1 = \sum_{k = 1}^q \varepsilon_k \biggl[ (\cdot,w^1_k)e^0_k + (\cdot,e^0_k)w^1_k + (\cdot,f_k)e^0_k \biggr] +
\sum_{k = q+1}^{q+q_1} (\cdot,z^0_k)w^0_k ,
\end{equation}	
where $f_k = z^1_k - w^1_k$ and $q_1$ is the number of chains with length equal to $p_1 - 1$.

Denote $H_1 = \Span\{ e^0_k\}^q_1, H_2 = \Span\{ w^0_k\}^{q+q_1}_{q+1}$. Since $\{ w^0_k\}^{q+q_1}_1$ are linearly independent (this follows from the definition of a canonical system), we have $H_1 \bigcap H_2 = \emptyset$. Hence, we obtain the unique representation $z^0_j = \varphi_j + \psi_j$, where $\varphi_j \in H_1, \psi_j \in H_2$. Consider the operators
$$
B_1 = \sum_{k = q+1}^{q+q_1}(\cdot,\psi_k)w^0_k,~~~
C_1 = \sum_{k=1}^q \varepsilon_k(\cdot,f_k)e^0_k +
\sum_{k = q+1}^{q+q_1}(\cdot, \varphi_k)w^0_k
.$$
The operator $L_1$ is self-adjoint and from the representation (6) we find that $B_1 + C_1$ is self-adjoint. Now notice, that in the orthogonal basis consisting of elements \footnote{ We may choose a canonical system (2) so that $\{w^0_k\}^N_1$ is an orthogonal system. Then the system $e^0_1,\dots,e^0_q,~w^0_{q+1},\dots,w^0_{q+q_1}$ is orthogonal.} $e^0_1,\dots,e^0_q,~w^0_{q+1},\dots,w^0_{q+q_1}$ we have the matrix representation $B_1 = \{b_{jk}\}$ and $b_{jk}$ may be not equal to zero only in the right lower quadrant, i.e. if $\min \{j,k\} > q$. On the contrary, all elements of the matrix $C_1 = \{c_{jk}\}$ in the same basis are equal to zero if $\min \{j,k\}>q$. Since $B_1 + C_1$ is self-adjoint, we have in this situation that both operators $B_1$ and $C_1$ are self-adjoint. Hence there exists a basis $\{e^0_k\}^{q+q_1}_{q+1}$ in the space $H_2$ such that
$$
B_1 = \sum_{k = q+1}^{q+q_1} \varepsilon_k (\cdot,e^0_k) e^0_k
$$

Taking into account the matrix representation of the operator $C_1 = C^\ast_1$, we can choose the elements $x_1,\dots,x_q$ such that
$$
C_1 = \sum_{k=1}^{q} (\cdot,x_k)e_k + (\cdot,e_k)x_k
.$$
Then denoting $e^1_k = w^1_k + \varepsilon_k x_k$, we obtain from representation (6) the following one
$$
L_1 = \sum_{k=1}^q \varepsilon_k[(\cdot,e^1_k)e^0_k + (\cdot,e^0_k)e^1_k] + \sum_{k=q+1}^{q+q_1} \varepsilon_k(\cdot,e^0_k)e^0_k
.$$

Next we can represent the operator $L_2$ in the form
\begin{equation}
\begin{split}
L_2 = \sum_{k=1}^q \varepsilon_k[(\cdot,w^2_k)e^0_k + (\cdot,e^1_k)e^1_k + (\cdot,e^0_k)w^2_k + (\cdot,f_k)e^0_k]  +  \\
+\sum_{k=q+1}^{q+q_1} \varepsilon_k[(\cdot,w^1_k)e^0_k + (\cdot,e^0_k)w^1_k + (\cdot,f_k)e^0_k]  +
\sum_{k=q+q_1+1}^{q+q_1+q_2} (\cdot,z^0_k)w^0_k
\end{split}	
\end{equation}
where

$$
f_k=
\begin{cases}
z^2_k - w^2_k,&\text{if $k=1,\dots,q,$}\\
z^1_k - w^1_k,&\text{if $k=q+1,\dots,q+q_1.$}\\
\end{cases}
$$

Now we can see from representation (7) that we can apply the same arguments as before and obtain the representation
\begin{multline*}
L_2 = \sum_{k=1}^q \varepsilon_k[(\cdot,e^2_k)e^0_k+(\cdot,e^1_k)e^1_k+(\cdot,e^0_k)e^2_k] +{}\\
{}+\sum_{k=q+1}^{q+q_1} \varepsilon_k[(\cdot,e^1_k)e^0_k+(\cdot,e^0_k)e^1_k] + \sum_{k=q+q_1+1}^{q+q_1+q_2} \varepsilon_k(\cdot,e^0_k)e^0_k.
\end{multline*}
The same arguments can be repeated for operators $L_3,\dots,L_{p_1}$. Then we obtain the assertion of the theorem.	
\end{pf}

\begin{definition}
A canonical system (2) satisfying to the relations (5) is called a {\bf normal canonical system}. The numbers $\varepsilon_k$ appearing in (4) are called  {\bf the sign characteristics} of the chains of the normal canonical system.
\end{definition}

Now we would like to establish a similar result for dissipative operator functions. The operator functions. The operator function $A(\lambda)$ is called dissipative in the neighborhood of the real point c if there exists an $\varepsilon > 0$ such that $\Imm(A(\lambda)x,x) \leqslant 0$ for all $x\in H$ and $\lambda : c-\varepsilon < \lambda < c+\varepsilon$. The following result was recently proved by Shkalikov [3](1988).

\begin{theorem}[Theorem on the Existence of the Regular Canonical System]
Let a real point c of discrete spectrum of the operator pencil $A(\lambda)$  and $A(\lambda)$ be dissipative in the neighborhood of c. Let (2) be a canonical system corresponding to the eigenvalue  c of $A(\lambda)$ and (4) be the adjoint canonical system. Then
\begin{equation}
\Span\{w^h_k\}^{N_2~~~~\alpha_k}_{k=N_1,~h=0} = \Span\{z^h_k\}^{N_2~~~~\alpha_k}_{k=N_1,~h=0} : = S^0
\end{equation}
where $\alpha_k = [\frac{p_k-1}{2}]$ (if $p_k = 0$ then $\alpha_k = -1$ and we assume that the vector $w^0_k$ does not belong to $S^0$). Moreover, a canonical system (2) can be chosen in such a way that for all indices k satisfying the conditions $p_k = 2\ell_k$ (i.e. for all chains of odd length) a representation
\begin{equation}
z^{\ell_k}_k = \varepsilon_k w^{\ell_k}_k + w,~~~where~\varepsilon_k = \pm 1,~~~w\in S^0
\end{equation}
is valid.
\end{theorem}
\begin{pf}
The proof of this theorem is rather complicated and can be found in the paper Shkalikov [3]. Here we omit it. For the linear operator pencils this Theorem follows from Propositions 8.6, 8.7 of the next Section.
\end{pf}
\begin{note}
It may happen, that $w^{\ell_k}_k = 0$, then $z^{\ell_k}_k = 0$ to (it is possible only if the order n of the pencil $A(\lambda)$ is greater than 2). In this case the equality (8) does not determine the sign $\varepsilon_k$ and it has to be determined from the equality
$$
\tilde z^{\ell_k}_k = \varepsilon_k \tilde w^{\ell_k}_k + \tilde w,
$$
where $\tilde z^{\ell_k}_k (\not =0!)$ denotes the Keldysh derived chain of length n corresponding to the element $z^{\ell_k}_k$. Thus, in addition to (9) the following equality holds (see Shkalikov [3])
\begin{equation}
(G \tilde z^{\ell_k}_k,z^{\ell_k}_k) = (G \tilde w^{\ell_k}_k,w^{\ell_k}_k) = -\lambda^n_k \varepsilon_k,
\end{equation}
where n is the order of $A(\lambda)$, $\varepsilon_k = \pm 1$, $\lambda_k = c \not = 0$ (otherwise we have to shift the spectral parameter) and
$$
G = \begin{bmatrix}
0 & 0 & \dots & 0 & A_0 \\
0 & 0 & \dots & A_0 & A_1 \\
\hdotsfor{5} \\
0 & A_0 & \dots & A_{n-3} & A_{n-2} \\
A_0 & A_1 & \dots & A_{n-2} & A_{n-1}\\
\end{bmatrix}.
$$
\end{note}
\begin{definition}
A canonical system (2) satisfying to the relations (9), (10) is called a {\bf regular canonical system}. The numbers $\varepsilon_k$ appearing in (9), (10) are called {\bf the sign characteristics} of the corresponding chains.
\end{definition}
\begin{note}
Hence for dissipative operator pencil we cannot introduce the sign characteristics for all chains, but only for chains of odd length. A natural question arises: Do the sing characteristics exist for the chains of even length? A simple example shows that the answer is negative. Consider the dissipative operator pencil
$$
A(\lambda)=I - P_0 - iC\lambda^2 ,~~~P_0 = (\cdot,e_0)e_0,~~\bigl\| e_0 \bigr\| = 1,~~C>0.
$$
Obviously, the point $\lambda = 0$ is a point of discrete spectrum of $A(\lambda)$ and the principal part of $A^{-1}(\lambda)$ at this point is equal to
$$
\frac{i(\cdot,e_0)e_0}{\lambda^2}.
$$
Hence, the direct and the adjoint canonical systems coincide in this case with chains $e_0, 0$ and $-ie_0, 0$ respectively.
\end{note}

The theorem on existence of the regular canonical system allows us to formulate the Mandelstam radiation principle for resonant frequencies, i.e., for the case when $A(\lambda)$ has real eigenvalues $\lambda_k$ which are not semi-simple. For this case we have not met the formulation of this principle in physical literature.
\begin{definition}
Let (2) be a regular canonical system corresponding to real eigenvalue $\mu (=\lambda_k)$.Let $\mathscr{E}_k = \pm 1$ be sign characteristics corresponding to Jordan chains of odd length and let $\mathscr{E} = 0$ for the Jordan chains of even length. We say the solution V(Z) of equation (6.13) satisfies the generalized Mandelstam radiation principle  at $\infty$ if $V(Z)$ admits the representation
$$
V(Z)=V_1(Z) + V_0(Z),
$$
where $V^{(j)}_0 \rightarrow 0, j = 0,1,\dots,(n-1)$, when $Z \rightarrow \infty$ and $V_1(Z)$ is a superposition of elementary solutions of the following type:
\begin{equation}
V^h_k(Z) = \ell^{i\lambda_k Z}(W^h_k + \frac{Z}{1!}W^{h-1}_k + \dots + \frac{Z^h}{h!}W^{0}_k),~~0\leqslant h\leqslant \biggl[\frac{p_k + \varepsilon_k}{2}\biggr]
\end{equation}
If $p_k = 0$ and $ \mathscr{E} _k = -1$ then $[1/2] = -1$ and we assume that no elementary solutions corresponding to this index k is involved in a superposition $V_1(Z)$.
\end{definition}

We save the same definition for the generalized Mandelstam radiation principle at $-\infty$; the only difference is that we have to replace the inequalities for h in (11) by $0\leqslant h \leqslant \biggl[ \frac{(p_k - \varepsilon_k)}{2}\biggr]$.

Denote by $E^+$ the first half of the eigen and associated vectors of $A(\lambda)$, namely,
$$
E^+ = \{ w^h_k \}_{\Imm\lambda_k > 0} \cup \{ w^h_k \}_{\lambda_k \in \mathbb {R}, ~0\leqslant h \leqslant \biggl[ \frac{(p_k + \varepsilon_k)}{2}\biggr]}.
$$
Thus, the system $E^+$ consists of all vectors from the canonical system (2), corresponding to all eigenvalue $\lambda_k$ with $\Imm \lambda_k > 0$ and of selected vectors from the regular canonical system, corresponding to real eigenvalue (this selection is produced according  to the sign characteristics). Similarly denote
$$
E^- = \{ w^h_k \}_{\Imm\lambda_k < 0} \cup \{ w^h_k \}_{\lambda_k \in \mathbb {R}, ~0\leqslant h \leqslant \biggl[ \frac{(p_k - \varepsilon_k)}{2}\biggr]}.
$$
Now we can determine $ \mathscr{E}^+ ( \mathscr{E}^-) = \{\hat w^h_k\}$, where $\hat w^h_k$ are the Keldysh derived chains of length $\ell$ corresponding to the vectors $w^h_k \in E^+ (E^-)$. Obviously, $\hat w^h_k = Tv^h_k$ where
$$
Tv(Z) = \{V(0),-iv'(0),\dots, (-i)^{\ell - 1}V^{(\ell - 1)}\}
$$
and $V^h_k(Z)$ are defined by (11).
\begin{theorem}[Theorem on Completeness]
If the conditions of Theorem 6.1 hold then the systems $\mathscr{E}^+$ and $\mathscr{E}^-$ are complete in the space $H^{\ell}$.
\end{theorem}
\begin{pf}
First, for simplicity we suppose that $A(\lambda)$ is a quadratic operator pencil, hence $\ell = 1$. In this case $\mathscr{E}^+ = E^+$. Assume that there exists a vector $f$, which is orthogonal to all elements of the system $E^+$. Then the scalar function
$$
F(\lambda) = ([A^*(\lambda)]^{-1}f,f), ~~A^*(\lambda) = [A(\bar \lambda)]^* ,
$$
is holomorphic in the lower half-plane (see the proof of Theorem 6.1) and according to (3) its principal part is equal to
\begin{equation}
\begin{split}	
\sum_{k=N_1}^{N_2}\frac{(f,z^0_k)(w^0_k,f)}{(\lambda - c)^{p_k+1}} + \frac{(f,z^1_k)(w^0_k,f) + (f,z^0_k)(w^1_k,f)}{(\lambda - c)^{p_k}} + \dots + \\
+ \frac{(f,z^{p_k}_k)(w^0_k,f) +\dots+ (f,z^0_k)(w^{p_k}_k,f)}{(\lambda - c)}.
\end{split}	
\end{equation}		
Since $(f,w^h_k) = 0$ for $h = 0,\dots,\biggl[ \frac{(p_k+\varepsilon_k)}{2} \biggr]$ and (8), (9) hold, we have $(f,z^h_k) = 0$ for $h \leqslant \biggl[ \frac{(p_k+\varepsilon_k)}{2} \biggr]$. Now it follows from (10) that all coefficients at powers $(\lambda - c)^{-s}, s\geqslant 1$, are equal to zero if $p_k$ is odd (i.e. the length of the corresponding chain is even) or $p_k$ is even but $\varepsilon_k = 1$. Hence, taking into account (9), we find that the expression (12) is equal to
$$
\sum_{p_k = 2\ell_k,\varepsilon_k = -1} \frac{(f,z^{\ell_k}_k)(w^{\ell_k}_k,f)}{\lambda - c} = -\sum_{p_k = 2\ell_k,\varepsilon_k = -1} \frac{\bigl| (w^{\ell_k}_k,f) \bigr|}{\lambda - c}.
$$
Thus the function $F(\lambda)$ may have only simple real poles with non-positive residues. Repeating the arguments in the proof of Theorem 6.1 we obtain $F(\lambda)\equiv 0$ and $f=0$.

If $\ell > 1$ then we have to consider the function
$$
F(\lambda)=(\biggl[ A^*(\lambda)\biggr]^{-1}f(\lambda),f(\bar \lambda)), ~~f(\lambda)=f_0 + f_1\lambda + \dots + f_{\ell-1}\lambda^{\ell-1},
$$
where a vector $\{f_0,f_1,\dots,f_{\ell-1}\}$ is orthogonal to all vectors belonging to $\mathscr{E}^+$. It is an easy exercise to show that this assumption as before implies the analyticity of $F(\lambda)$ in the closed lower half-plane with possible exception of a simple real pole with non-positive residues. This gives $F(\lambda)\equiv 0$ and $f(\lambda)\equiv 0$.
\end{pf}
\begin{theorem}[Theorem on Linear Independence]
If the conditions of Theorem 6.2 hold then the system $\mathscr{E}^+$ and $\mathscr{E}^-$ are linear independent in the space $H^{\ell}$.
\end{theorem}
\begin{pf}
As in the previous theorem we assume for simplicity that $A(\lambda)$ is a quadratic pencil. Consider the scalar function
$$
F(\lambda) = (A(\lambda)d(\lambda),d(\bar \lambda)),
$$
where $d(\lambda)$ is determined by (1), but the second term in (1) is replaced by
$$
d_1(\lambda) = \sum_{\lambda_k\in \mathbb {R}} ~\sum_{0\leqslant h\leqslant \beta_k} \frac{c_{k,h}~w^h_k}{(\lambda - \lambda_k)^{\beta_k + 1 - h}},~~~\beta_k = \biggl[ \frac{(p_k+\varepsilon_k)}{2} \biggr].
$$

Then $A(\lambda)$ is holomorphic in the upper half-plane and may have poles on the real axis. Obviously the principal part of the function $F(\lambda)$ at a real pole $\lambda = \lambda_k (=c)$ coincides with the principal part of the function
$$
F_1(\lambda) = (A(\lambda)d_1(\lambda),d_1(\lambda))
$$
at this pole.

After some technical calculations (they are not simple; see Shkalikov [3], lemma 4) we find that the principal part of $F_1(\lambda)$ at the real pole $\lambda = c$ is equal to
$$
\sum_{\lambda_k=c, \varepsilon_k>0}~~~\frac{\bigl|c_{k,\ell_k}\bigr|^2  (G \tilde w^{\ell_k}_k,\tilde w^{\ell_k}_k)(w^{\ell_k}_k,w^{\ell_k}_k)}{\lambda - c}.
$$

It follows from (10) that the function $F(\lambda)$ may have poles on the real axis only with negative residues. Repeating the arguments in the proof of Theorem 6.2 we obtain $F(\lambda)\equiv 0$ and $d(\lambda)\equiv 0$.
\end{pf}
\begin{theorem}[Theorem on solvability of half-range Cauchy Problem]
Let the pencil $A(\lambda)$ satisfy the conditions of Theorem 6.1. Then for given initial vectors $\{\varphi_j\}^{\ell-1}_0$ there exists a unique solution of problem (6.13), (6.14) satisfying the generalized Mandelstam radiation principle at $\infty$.
\end{theorem}
\begin{pf}
This theorem is a corollary of Proposition 6.3 and two previous theorems.
\end{pf}
\begin{note}
For the case when $\dim H < \infty$ and $\Ker A_n \not= 0$ the duality principle is valid as well as in non-resonant case (see Note 6.8). Namely, if both systems $\mathscr{E}^+$ and $\mathscr{E}^-$ are complete (or linearly independent) then they are basis' in $H^{\ell}$. Indeed, one can easily check that the equality $x_+ + x_- = 2\ell m$ holds in general case as well as in non-resonant case. Therefore, the same arguments can be applied to prove this fact.
\end{note}

The ideas presented in the last two lectures can be extended to obtain similar results for dissipative pencils of odd order as well as for pencils satisfying the condition
\begin{equation}
\Imm(\lambda A(\lambda)x,x)\leqslant 0~~~\forall x\in H,~\forall \lambda \in \mathbb{R}.
\end{equation}

First, suppose that $A(\lambda)$ satisfies condition(6.15) and $n = 2\ell + 1$. Observe that in this case (6.15) implies $A_n = A^*_n$. We can represent $A_n = A^+_n - A^-_n$ where $A^{\pm}_n \geqslant 0, A^+_n A^-_n = A^-_n A^+_n = 0$. Let $P^+$ and $P^-$ be orthoprojectors onto $\Imm~A^+_n:=H^+$ and $\Imm~A^-_n:=H^-$ respectively. Let V(Z) be a function with values in H. Define "the trace" operators $T_{\pm}$ by formula
$$
T_{\pm}V(Z) = \{V(0),-iV'(0),\dots,(-i)^{\ell-1}V^{(\ell-1)}(0),(-i)^{\ell}P^{\pm} V^{(\ell)}(0)\}.
$$
Now, define the system $\mathscr{E}^{\pm}$ in such a way that
$$
\mathscr{E}^{\pm} = \{T_{\pm}V^h_k(Z)\},
$$
where $V^h_k(Z)$ are elementary solutions of (6.13) such that $V^h_k(0)\in E^{\pm}$. For example, if $A^+_n > 0$ (i.e. $p^+ = I$) then $\mathscr{E}^+ (\mathscr{E}^-) = \{\hat w^h_k\}$, where $\hat w^h_k$ are Keldysh derived chains of length $\ell + 1~(\ell)$ corresponding to vectors $w^h_k \in E^+ (E^-)$.
\begin{theorem}
Let the pencil $A(\lambda)$ satisfy the conditions of Theorem 6.1 and let $n = 2\ell + 1$. Then the system $\mathscr{E}^+$ is a basis in the space $H^{\ell-1} \times H^+$ while the system $\mathscr{E}^-$ is a basis in $H^{\ell-1} \times H^-$.	
\end{theorem}	
\begin{pf}
First, let us prove the completeness of the systems $\mathscr{E}^+$ and $\mathscr{E}^-$. Consider, for example the system $\mathscr{E}^+$. Assume, that this system is not complete. In this case there exists a vector $\{f_0,f_1,\dots,f_{\ell-1},P^+ f_{\ell}\}$ such that the function
$$
F(\lambda) = ( [A^{\ast}(\lambda)]^{-1} f(\lambda),f(\bar\lambda))
$$
where
$$
f(\lambda) = f_0 + \lambda f_1 + \dots + \lambda^{\ell-1} f_{\ell-1} + \lambda^{\ell}P^+ f_{\ell}
$$
is holomorphic in the lower half plane and may have only simple poles on the real axis with non-positive residues (see the proof of Theorem 6.1 and the theorem on completeness from this lecture). Let us compute the residue at infinity. Since $A_n$ is invertible, we have
$$
F(\lambda) = \frac{1}{\lambda}(A^{-1}_n P^+f_{\ell},P^+f_{\ell}) + O\Bigl( \frac{1}{\lambda^2}\Bigr),~~~\lambda\rightarrow\infty.
$$
Hence the residue of $F(\lambda)$ at infinity is also non-positive. Now, we can repeat the arguments which we applied in Theorem 6.1. Then we obtain $F(\lambda)\equiv 0$ and $f(\lambda)\equiv 0$. To prove the basisness we can apply the duality principle (see Note 3).
\end{pf}
\begin{corollary}
Let $A(\lambda)$ satisfy condition (6.15), $\dim~H < \infty, n = 2\ell + 1$ and let $A_n > 0$. Then the system $\mathscr{E}^+ (\mathscr{E}^-)$ consisting of Keldysh derived chains of length $\ell + 1 (\ell)$ constructed from eigen and associated vectors $w^h_k \in E^+ (E^-)$ form a basis in $H^{\ell+1} (H^{\ell})$.
\end{corollary}
\begin{pf}
We have to note only that condition (6.16) holds for sufficiently large $\lambda_0 \in \mathbb {R}$, since $A_n >0$.
\end{pf}

Now, we consider the pencils $A(\lambda)$ satisfying condition (13). Notice that conditions (6.15) and (13) are different - each of them does not imply the other. First, consider the case $n = 2\ell$. In this case (13) implies $A_0 = A_0^{\ast}$, $A_n = A_n^{\ast}$.

Let (2) be a regular canonical system corresponding to real eigenvalue $\lambda_k$ and let $\delta_k = sgn \lambda_k \mathscr{E}_k$, where $\mathscr{E}_k$ are the sign characteristics of the corresponding Jordan chains (we define $\mathscr{E}_k = 0$ if the length of the corresponding Jordan chains is even). Let us introduce the systems
\begin{equation*}
\begin{split}
Y^+ = \{w^h_k\}_{\Imm\lambda_k >0} \bigcup \{w^h_k\}_{\lambda_k \in \mathbb {R},~0\leqslant h\leqslant \biggl[ \frac{(p_k + \delta_k)}{2} \biggr]} \\
Y^- = \{w^h_k\}_{\Imm\lambda_k <0} \bigcup \{w^h_k\}_{\lambda_k \in \mathbb {R},~0\leqslant h\leqslant \biggl[ \frac{(p_k - \delta_k)}{2} \biggr]}
\end{split}
\end{equation*}

Consider the spectral decompositions $A_0 = A^+_0 - A^-_0, A_n = A^+_n - A^-_n ~(A^{\pm}_0\geqslant 0, A^{\pm}_n\geqslant 0 )$ and denote by $Q^{\pm}$ and $P^{\pm}$ the orthoprojectors onto $\Imm A^{\pm}_0$ and $\Imm A^{\pm}_n$ respectively. For a vector valued function $V(Z)$ define the "trace" operator
$$
T_{\pm}V(Z) = \{Q^{\pm} V(0),-iv'(0),\dots,(-i)^{\ell-1}V^{\ell-1}(0),(-i)^{\ell}P^{\pm}V^{(\ell)}(0)\}
$$
and introduce the systems
\begin{equation}
Y_{\pm} = \{ T_{\pm} V^h_k(Z) \}
\end{equation}
where $V^h_k(Z)$ are elementary solutions of (6.13) such that $V^h_k(0) \in Y^{\pm}$. For example, if $A_0 >0$ and $A_n >0$ then $Y_- = \{ \hat w^h_k \}$ where $\hat w^h_k$ are Keldysh derived chains of length $\ell$ corresponding to vectors $\hat w^h_k \in Y^-$.
\begin{theorem}
Let the pencil $A(\lambda)$ satisfy condition (13), $n = 2\ell, ~\Ker~A_0 = 0$ and $\Ker~A_n = 0$. Let also condition (6.16) hold and $\dim~H <\infty$. Then the system $Y_+ (Y_-)$ forms a basis in the space $H^-_Q \times H^{\ell-1} \times H^+_P (H^+_Q \times H^{\ell-1} \times H^-_P)$ where $H^{\pm}_Q = \Imm~Q^{\pm}$ and $H^{\pm}_P = \Imm~P^{\pm}$.
\end{theorem}
\begin{pf}
Let us prove the completeness of the systems $Y_+$ and $Y_-$. Using the duality principle we obtain the basisness.

Suppose there exists a vector $\{Q^-f_0,f_1,\dots,f_{\ell-1},P^+f_{\ell}\}$ which is orthogonal to the system $Y^+$. Consider the function
$$
F(\lambda) = \frac{1}{\lambda} (\biggl[ A^{\ast}(\lambda) \biggr]^{-1}f(\lambda),f(\bar \lambda)),
$$
where $f(\lambda) = Q^-f_0 + \lambda f_1 + \dots + \lambda^{\ell-1}f_{\ell-1} + \lambda^{\ell}P^+f_{\ell}$. Denoting $g(\lambda) = \biggl[ A^{\ast}(\lambda) \biggr]^{-1} f(\lambda)$ we can write
$$
F(\lambda) = \lambda^{-2}(\lambda A(\lambda)g(\lambda), g(\bar\lambda)).
$$
Now, it follows from condition (13) that $\Imm~F(\lambda)\leqslant 0$ for $\lambda \in \mathbb {R}$. Repeating the arguments of Theorem 5.1 we obtain that $F(\lambda)$ is holomorphic in the lower half plane. We also have
$$
F(\lambda) = \lambda^{-1} (A^{-1}_n P^{+}f_{\ell},P^{+}f_{\ell})+O(\lambda^{-2})
$$
when $\lambda \rightarrow \infty$, i.e., the residue of $F(\lambda)$ at $\infty$ is non-positive too. The function $F(\lambda)$ may have other real poles which are simple and the corresponding residues are non-positive (see Theorem 6.1 and the theorem on completeness in this lecture). Applying the arguments of Theorem 6.1 we find $F(\lambda)\equiv 0$ and $f(\lambda)\equiv 0$.
\end{pf}

Theorem 8 is most interesting in the case when $A_0$ and $A_n$ are definite operators. We offer the reader to formulate a Corollary from Theorem 5 for this case.

Finally, if condition (13) holds and $n = 2\ell + 1$ we should introduce the operators
$$T_{\pm}V(Z) = \{ Q^{\mp}V(0),-iV'(0),\dots,(-i)^{\ell}V^{(\ell)}(0)\}$$
and thence define the systems $Y_{\pm}$ by equality (14). The following result is valid.
\begin{theorem}
Let the conditions of Theorem 5 hold but $n = 2\ell + 1$. Then the system $Y_+ (Y_-)$ forms a basis in $H^-_Q \times H^{\ell} (H^+_Q \times H^{\ell})$.
\end{theorem}
\begin{pf}
Repeat the arguments of Theorem 7.
\end{pf}

\newpage
%Лекция 08
\section{Dissipative and linearly dissipative operator pencils }

In this section we continue the study of dissipative operator pencils satisfying the condition
\begin{equation}
\Im(A(\la)x, x) \le 0 \quad \mbox{for all $\la \in \mathbb {R}$ and $x \in H$}
\end{equation}
or the condition
\begin{equation}
\Im(\la A(\la)x, x) < 0 \quad \mbox{for all $\la \in \mathbb {R}$ and $x \in H$.}
\end{equation}
But now we deal with infinite dimensional space H.

Further we will use the notation
$$
A^\la = \frac{A+A^*}{2}, A^I=\frac{A-A^*}{2i}
$$
Obviously, $A^\lambda$ and $A^I$ are self-adjoint operators and $A=A^\lambda+iA^I$. Notice, that if $A(\lambda)$ is a linear pencil
\begin{equation}
A(\la) = A_0 + \la A_1
\end{equation}
then (2) holds if and only if $A_0=A_0^*$ and $A_1^I \le 0$.\ Assume that $A_0$ is an invertible operator. Then the spectral problem for linear pencil (3) is equivalent to the spectral problem for the linear operator $A=A_0^{-1}A_1$. Obviously, condition (2) holds for linear pencil (3) if and only if $A$ is an $A_{0}$-dissipative operator, i.e. A is dissipative in the space $H$ with regular indefinite metric $(A_{0}x,x)$. The main goal of this lecture is to show that some properties of dissipative operators in a space with indefinite metric are similar to properties of self-adjoint operators.

First we recall some definitions from the theory of operators in Hilbert space withe indefinite metric. Let W be a symmetric operator (it can be unbounded, but we will consider only bounded operators). Denote by $[x,y] = (W x,y)$ the new scalar product in H which is indefinite if the operator $W$ is indefinite. This metric is called \textit{regular} if $W$ is bounded and invertible. The Hilbert space with regular indefinite metric generated by the operator $W$ is called a \textit{Pontrjagin space} if either operator $W_{+} = \frac{|W| + W}{2}$ or $W_{-}=\frac{|W| - W}{2}$ is finite dimensional.

%page 8.2

The space with regular indefinite metric is called a \textit{Krein space} if both operators $W_{+}$ and $W_{-}$ are infinite dimensional. The subspace $H_{1} \subset H$ is called $W$-\textit{non-positive (nonnegative, neutral)} if $[x,x] \le 0$  $(\ge0$,  $= 0)$ for all $x \in H_1$. A $W$-non-positive subspace $H_1$ is called \textit{maximal non-positive} if it is not contained in any other $W$-non-positive subspace $H_2$. The operator $A$ is called dissipative$
\footnote{In Section~3 we called an operator A dissipative if $A^{\la} \le 0$. But in the mathematical literature the word "dissipative"\ is also used to denote operators satisfying the condition $A^{I} \le 0$.}$
in the space with indefinite metric $[\cdot]$ (or $W$-dissipative) if
$$
\Im[Ax,x] \le 0 \quad \mbox{for all $x \in H$.}
$$
Now we establish some useful properties of $W$-dissipative operators.

\begin{proposition}
Let $A$ be a $W$-dissipative operator and $S_+^0 = span\{y_k^h\}$, where $y_k^h$ are $EAV$ of a linear pencil
$$
A(\lambda)=I + \lambda A
$$
corresponding to the eigenvalues $\lambda_k$ with $\Im\lambda_k > 0$. Then $S_+^0$ is a W-non-positive subspace.
\end{proposition}

\begin{pf}
If $y_k$ is an eigenvector of a pencil $A(\lambda)$ corresponding to an eigenvalue $\lambda_k$ with $\Im\lambda_k > 0$, then $u_k(t)=e^{i\lambda_kt}y_k$ is a solution of equation
\begin{equation}
-iAu^\prime (t) + u(t) = 0
\end{equation}
and $u_k(t) \rightarrow 0$\,\ when $t\to\infty$. Similarly, the function
$$
u(t) = \sum c_{k,h}u_k^h(t)\mbox{,}
$$
where $u_k^h(t)$ are elementary solutions corresponding to EAV $y_k^h \in S_+^0$, satisfies
the equation (4) and $u(t) \to 0$ when $t \to \infty.$ Using $(4)$, we obtain
\begin{equation}
\begin{aligned}	
[u(\xi), u(\xi)]^\prime &= [u^\prime(\xi), u(\xi)] + [u(\xi), u^\prime(\xi)] \\
&= [u^\prime(\xi), i Au^\prime(\xi)] +[i Au^\prime(\xi),u^\prime(\xi)] = - (V u^\prime(\xi), u^\prime(\xi)) \\
&= (u^\prime(\xi), Wu(\xi)) + (Wu(\xi),u^\prime(\xi)) \\
&= (u^\prime(\xi),iTu^\prime(\xi))+(iTu^\prime(\xi), u^\prime(\xi)) \\
&= -(T_j u^\prime(\xi), u^\prime(\xi))\mbox{,}
\end{aligned}
\end{equation}

%page 8.3

$V=2(WA)^I=[{WA}-{(WA)}^*]/i\le{0}$. Integrating the equality $(5)$ from $t$ to $\infty$ we obtain
\begin{equation}
[u(t), u(t)]=\int_t^{\infty}(Vu\prime{(\xi)},  u\prime{(\xi)})d\xi.
\end{equation}
In particular, $[u(0), u(0)]=[y,y]\le0$ for all  $y\in{S_+^0}$.
\end{pf}

\begin{note}
Equality $(6)$ has a physical sense. For some particular equations $(4)$ describing physical processes the form $[u(t), u(t)]$ plays the role of an energy functional and operator $V = 2(WA)^I \le 0$ is responsible for damping. Thus the equality $(6)$ shows that, for systems with damping,the energy functional decreases monotonically when $t \to \infty$.
\end{note}

\begin{proposition}
Let $A$ be a $W$-dissipative operator and
\begin{equation}
y_k^0, y_k^1, {\ldots},y_k^{p_k},\,\ k=N_1, \dots, N_2
\end{equation}
be the canonical system of EAV of the pencil $W + \lambda W A$ corresponding to a real eigenvalue $\lambda_k = c \ne 0$. Let
\begin{equation}
x_k^0, x_k^1, {\ldots}, x_k^{p_k},\,\ k = N_1, \dots, N_2
\end{equation}
be the adjoint canonical system of EAV of the pencil ${W} +{\lambda A^*W}$ corresponding to the same eigenvalue $c$. If ${[\gamma]}$ is the integer part of a number $\gamma$ and $a_k = [p_k/2]$ then the elements
\begin{equation}
y_k^0,\,\ y_k^1,\,\ {\ldots},\,\ y_k^{\alpha_k}\text{,} \quad { k = N_1}, \dots, {N_2}\text{,}
\end{equation}
\begin{equation}
x_k^0,\,\ x_k^1,\,\ {\ldots},\,\ x_k^{\alpha_k}\text{,} \quad { k = N_1}, \dots, {N_2}\text{,}
\end{equation}
belong to $Ker({WA})^I$.
\end{proposition}

\begin{pf}
Denote $T = WA.$ From the definition of EAV we have
$$
(cT + W)y_k^h = -Ty_k^{h-1}, \quad 0 \le h \le p_k \text{ ($y^{-1}:= 0$).}
$$
In particular,
$$
\Im((cT + W){y_k^0},{y_k^0}) = c({T^I}{y_k^0},{y_k^0)}= 0.
$$
%page 8.4	

Since $T^I\ge{0}$ we have $y_k^0\in Ker \, T^I$. Suppose we have proved that\,\ $y_k^0, \ldots, y_k^{h-1}\,\ \in\,\ Ker \,\ T^I$ for ${h-1}<\,\ \alpha_k$.\,\ Then $2h\,\ \le\,\ {p_k}$ and
\begin{equation}
\begin{aligned}	
(Ty_k^{h-1}, y_k^h) &= (T^*y_k^{h-1}, y_k^h) = -(y_k^{h-1}, (c T + W)y_k^{h+1}) \\			
&= (Ty_k^{h-2}, y_k^{h+1}) = \dots = -(y_k^0, (c T + W)y_k^{2h}) = 0.
\end{aligned}
\end{equation}
Therefore,
$$
0 = \Im\left((c T + W)y_k^{h} + Ty_k^{h-1}, y_k^{h}\right) = c(T^I y_k^h,y_k^h).
$$
From this equality we deduce as before that $y_k^h\in
Ker\,\ T^I.$\,\ Similarly, we can prove $x_k^{h} \in Ker \, T^{I}$ if $h\le\alpha_k$.
\end{pf}

\begin{proposition}
Let the assumption of   Proposition 8.2 hold. If $p_j < p_k$ then
\begin{equation}
(W x_{k}^h, x_j^r) = 0
\end{equation}
for all $r\le{[p_j/2]}, \quad h\le [p_k/2]$.
\end{proposition}

\begin{pf}
According to Proposition 8.2 we  have $T x_j^r = T^*x_j^r$, $T = WA$. Since ${h + i + 1} \le p_k$, we find ((cf. (11)) \begin{equation}
\begin{aligned}
(x_{k}^{h}, T x_{j}^r) &= - \left((cT^* + W\right )x_k^{h+1}, x_j^r) = (x_k^{h+1}, T x_j^{r-1}) \\
&= \dots = (x_k^{h+s+1},(c T + W)x_j^0) = 0.
\end{aligned}
\end{equation}
From the definition of EAV we obtain
$$
W_k^{h}x_k^{h} = -cT^*x_k^{h} - T^*x_k^{h-1}
$$
Then $(12)$ follows from $(13)$.
\end{pf}

Denote by $\sigma_d(A)$ the discrete spectrum of a pencil $A(\lambda)$, i.e. the set of isolated eigenvalues of $A(\lambda)$ of finite algebraic multiplicity

\begin{proposition}
Let $A(\lambda) = W + \lambda T$ and eigenvalues $\lambda_k \in \sigma_d(A)$ be enumerated according to their geometric multiplicity. If $(7)$, $(8)$ are mutually adjoint  canonical systems corresponding to eigenvalues $\lambda_k$ then the following biorthogonality relations hold:
\begin{equation}
(T y_{k}^{h}, x_{j}^{s}) = \delta_{k,j}\delta_{h, p_s-h}
\end{equation}

%page 8.5

If $\lambda_j \ne 0$ then
\begin{equation}
(W y_k^h, x_j^s - \bar\lambda_j^{-1}x_j^{s-1} + \ldots + (-1)^s\bar\lambda_j^{-s}x_j) = -\lambda_j\delta_{k,j}\delta_{h,p_s-h}
\end{equation}
\end{proposition}

\begin{pf}
We have
$$
A(\la) = [A(\lambda_k) + T(\lambda - \lambda_k)]y_k^h = -Ty_k^{h-1} + (\lambda - \lambda_k)Ty_k^h,\,\ 0\le h \le p_k
$$
(if $h = 0$, then we assume $y_k^{-1} = 0)$. Using the representation $(2.2)$ we obtain
\begin{equation}
\begin{aligned}
y_k^h &= A^{-1}(\lambda)A(\lambda)y_k^h \\
&= \sum^{N_2}_{j=N_1}\sum^{p_j}_{s=0}
\left[
\frac{(\cdot ,x_j^s)y_j^0 + \ldots +( \cdot , x_j^0 )y_j^s}{(\lambda - \lambda_j)^{p_{j}+1-s}
} + R(\lambda)\right]\left[{(\lambda - \lambda_k)Ty_k^h - Ty_k^{h-1}}\right]
\end{aligned}
\end{equation}
where $R(\lambda)$ is a holomorphic operator function at the point $\lambda_j$.  We may assume that $N_1 = 1$, $N_2 = N$, $p_1\ge p_2\ge \ldots\ge p_N$

Let $\lambda_k \ne \lambda_j$. Taking $h = 0$ and comparing coefficients of the powers $(\lambda - \lambda_j)^{-p_{j}-1 + s}$, $0\le s \le p_j$, we find
\begin{equation}
\sum_{p_j = p_1}{(Ty_k^0, x_j^0) y_j^0 = 0},
\end{equation}
\begin{equation}
\sum_{p_j = p_1}(Ty_k^0, x_j^1) y_j^0  +
\sum_{p_j = p_{1}}{(Ty_k^0, x_j^0) y_j^1} +
\sum_{p_j = {p_1-0}} (Ty_k^0, x_j^0) y_j^0 = 0.
\end{equation}
It follows from the definition of a canonical system that elements $ \{y_j^0\}_1^N$ are linearly independent. Hence, it follows from $17$, that
\begin{equation}
(Ty_k^0, x_j^0) = 0, \quad \text{for all $p_j = p_1$}.
\end{equation}
Now, it follows from (18), (19), that
$$
(Ty_k^0, x_j^0) = 0 \quad \text{if $p_j = p_1 - 1$;} \quad (Ty_k^0, x_j^1) = 0 \quad \text{if $p_j = p_1$.}
$$
Repeating the argument we find subsequently
$$
(Ty_k^0, x_j^s) = 0 \quad (Ty_k^1. x_j^s), \dots, (Ty_k^{p_k}, x_j^s) = 0
$$
for all $0 \le s \le p_j.$	

%page 8.6

The same arguments can be applied in the case $\la_k = \la_j$. Comparing the coefficients of the powers $(\la - \la_j)^\nu$ in (18) it is found that, for $h = 0, 1, \dots, p_k$,
$$
(Ty_j^h, x_j^s) = \delta_{h, p_j-h},$$
and (14) follows. Noticing that
$$
T^{*}x_j^s = -\bar \la_j^{-1}W^{*}(x_j^s - \bar \la_j^{-1}x_j^{s-1}+\dots+(-1)^sx_j^0),
$$
we obtain (2.8).
\end{pf}

Let (7) be a canonical system of EAV of a dissipative pencil $W + \la T$ corresponding to a real eigenvalue $\la_k = c$. Denote by $S^0$ the span of elements
\begin{equation}
y_k^0, y_k^1, \dots, y_k^{\beta_k}, \quad k = N_1, \dots,N_2, \quad \beta_k = [(p_k-1)/2]
\end{equation}
(if $p_k$ = 0, we assume that $\beta_k = -1$ and the element $y_k^0$ does not belong to $S^0$). Let us fix an index k, $N_1 \le k \le N_2$. If a number $p_k + 1$ is even we set $S_k = S^0$. If $p_k+1$ is odd, we denote by $S_k$ the span of elements (20) combined with elements $y_j^{\alpha_j}, \alpha_j = [p_j/2]$, where index $j$ runs through all values such that $p_j = p_k$. Similarly, by replacing the chains (7) with adjoint chains (8) we construct subspaces $(S^0)^{*} and S_k^{*}$.

\begin{proposition}
If $A(\la) = W + \la T$ is a dissipative operator pencil then for all nonzero real $\la_k \in \sigma_d(A)$
$$
S^0 = (S^0)^{*} \quad \text{and} \quad S_k = S_k^{*} \quad \text{for all $N_1 \le k \le N_2$.}
$$
\end{proposition}

\begin{pf}
Suppose that $y_k^h \in S_k$ and $x_k^h \notin S_k$. It follows from Proposition 3 that (10) are chains of EAV of a pencil $W+\la T$. Since (7) is a canonical system, we have a representations
\begin{equation}
x_k^h = \sum_{j=N_1}^{N_2} \sum_{s=0}^{h} c_{j, s}y_j^s, \quad 0 \le h \le \alpha_j = [p_j/2].
\end{equation}

%page 8.7

We have assumed that $x_k^h \notin S_k$, therefore, at least one of the numbers $c_{j, s}$ in (2.13) is not equal to zero for $s > \beta_k = [(p_k-1)/2]$, $p_j < p_k$. In this case, however, $x_j^{p_j-s} \in S_j^{*}$. Hence,
$$
\textsl{g} = x_j^r - \bar \la_j^{-1}x_j^{r-1}+ \dots+ (-1)^rx_j^0 \in S_j^{*}, \quad r=p_j-s
$$
According to proposition (2.3) we have $(Wx_j^h, \textsl{g}) = 0$. On the other hand, it follows from Proposition (2.4) and representation (2.13) that
$$
(Wx_j^h, \textsl{g}) = -\la_jc_{j, s}.
$$
As $\la_j \ne 0$, we have $c_{j, s}$. Hence, $x_k^h \notin S_k$ is not valid. The equality $S^0 = (S^0)^*$ is proved in a similar way.
\end{pf}

\begin{proposition}
Let the assumption of Proposition (2.2) hold. Then a canonical system (2.1), corresponding to a real eigenvalue $\mu$ can be chosen in such a way that
\begin{equation}
(Wy_j^{\alpha_j}, y_s^{\alpha_s}) = \epsilon_j\delta_{js}, \quad \epsilon_j = \pm 1, \quad \alpha_j = [p_j/2],
\end{equation}
for all indices $1 \le j, s \le N$ such that $p_j+1$ or $p_s+1$ are odd.
\end{proposition}

\begin{pf}
Assume that $p_1 = p_2 = \dots = p_q > p_{q+1} \ge \dots \ge p_r$, $p_{r+1} \ge p_{r+2} \ge \dots \ge p_N$, $p_j = 2\alpha_j$ if $1 \le j \le r$ and $p_j = 2\alpha_j+1$ if $r < j \le N$. Let $P_1$ be the orthoprojector onto subspace $S_1 = S_1^*$ (see the Proposition 6). Obviously, $\dim S_1 \circ S^0 = q$. It follows from the biorthogonality relations (15) that self-adjoint operator $P_1WP_1$ has exactly $q$ nonzero eigenvalues which correspond to an orthogonal basis $\{\phi_s\}_1^q$. Obviously, a canonical system (7) can be chosen in such way that the system $\{\phi_s\}_1^q$ will coincide with $\{y_k^{\alpha_1}\}_1^q$. Then after a proper norming the relations (22) will hold for $k = 1,2, \dots, q$. Considering orthoprojector $P_{q+1}$ onto subspace $S_{q+1} = S_{q+1}^*$ and self-adjoint operator $P_{q+1}WP_{q+1}$, we can repeat the arguments and choose the chains of length $1+p_{q+1}$ (not changing the first chains) so that relations (22) will hold for all $p_k  3D p_{q+1}$. The next step is evident. Hence the proof of Proposition 7 can be completed by induction.
\end{pf}

%page 8.8

\begin{note}
Let a canonical system (7) satisfy condition (22). Then for all indices $k$ such that $p_k = 2\alpha_k$ the elements $x_k^{\alpha_k}$ of the adjoint canonical system (8) have the representation
\begin{equation}
x_k^{\alpha_k} = -\la_k\delta_ky_k^{\alpha_k} = y, \quad \alpha_k = p_k/2,
\end{equation}
where $y \in S^0$. This representation follows from Proposition 6 and relations (15). $\square$	
\end{note}

Let $y_k = c$ be real eigenvalue of a dissipative pencil $A(\la) = W+\la T$, and let canonical system (7) satisfy the condition (22). Denote by $S_c^+$ ($S_c^-$) the span of elements (20) combined with $y_k^{\alpha_k}$ satisfying the relations (22) with $\delta_k = -1$ ($\delta_k = 1$). Let $S_0^+$ ($S_-^0$) be the span of all EAV $\{y_k^h\}$ of $A(\la)$ corresponding to eigenvalues $\la_k \in \sigma_d(A)$ with $\Im \la_k > 0$ ($\Im \la_k < 0 $). Denote by $S^+$ ($S^-$) the minimal subspace containing $S_+^0$ ($S_-^0$) and all subspaces $S_{\la_k}^+$ ($S_{\la_k}^-$) corresponding to real eigenvalues $\la_k \in \sigma_d(A)$. Finally, by $S$ we denote the minimal subspace containing all EAV $y_k^h$ corresponding to eigenvalues $\la_k \in \sigma_d(A)$.

\begin{proposition}
Let the assumption of Proposition 3 hold. Then $S_c^+$ is a $W$-non-positive subspace. If $S_c$ is the span of the elements (7) then $S_c^+$ is a maximal $W$-non-positive subspace in $S_c$.
\end{proposition}

\begin{pf}
It follows from Propositions 4, 5 and definitions that $S_c^+$ is a $W$-non-positive subspace. Assume that $S_c^+ \subset S^1 \subset S_c$, where $S^1$ is also $W$-non-positive, and that there exists an element $y \in S^1$ such that $y \notin S_c^+$. Obviously, $y \notin S_c^-$, because it follows from assumptions $y \in S_c^-$, $y \notin S_c^+$ that $(Wy, y) > 0$. Hence, $y \notin S_c^+$ and $y \notin S_c^-$. Then, using (15) we can find an element $y_k^h \in S^0$ such that $(Wy_k^h, y) = \gamma \ne 0$. We may assume that $\gamma > 0$, otherwise we have to replace $y$ by $\gamma y$. Denote $z = \rho y_k^h + y$. Then from Proposition 4 we obtain $(Wz, z) = 2\rho\gamma + (Wy, y) \to \infty$ if $\rho \to +\infty$. This is the contradiction.
\end{pf}

%page 8.9

\begin{theorem}
Let $W$ be an invertible self-adjoint operator in Hilbert space H and $A$ be a $W$-dissipative operator. Then $S^+$ ($S^-$) constructed from EAV of as pencil $I + \la A$ is a $W$-non-positive (nonnegative) subsoace. Moreover $S^+$ ($S^-$) is maximal $W$-non-positive (nonnegative) subspace in S.
\end{theorem}

\begin{pf}
It follows from biorthogonality relations (15) and the definitions that $S^+$ is $W$-nonpositive. Repeating the arguments of Proposition 9 we find that $S^+$ is maximal $W$-non-positive in the subspace $S$.
\end{pf}

Now we give application of Theorem 10 to the problem of half-range minimality for dissipative operator pencils.

Consider an operator pencil
\begin{equation}
A(\la) = A_0+\la A_1+\dots+\la^nA_n,
\end{equation}
where $A_j$, $j = 1,\dots,n$ are bounded operators in Hilbert space H and condition (1) holds. We assume that $A_0$ is invertible (if $\sigma(A) \ne \mathbb {R}$ then we can shift $\la \to \la + \la_0$, $\la_0 \in \mathbb {R}$, $\la_0 \notin \sigma(A)$; after this translation condition (1) will also hold). With pencil (24) we associate the linear operators
$$A =
\begin{bmatrix}
A_0^{-1}A_1 & A_0^{-1}A_2 & \ldots &A_0^{-1}A_{n-2}& A_0^{-1}A_{n-1}\\
-I & 0 & \ldots & 0 & 0\\
0 & -I & \ldots & 0 & 0 \\
\ldots & \ldots & \ldots & \ldots \\
0 & 0 & \ldots & -I& 0
\end{bmatrix},
$$
$$
G =
\begin{bmatrix}
0 & 0 & \ldots & 0 & A_0\\
0 & 0 & \ldots & A_0 & A_1\\
\ldots & \ldots & \ldots & \ldots & \ldots\\
0 & A_0 & \ldots & A_{n-3} & A_{n-2} \\
A_0 & A_1 & \ldots & A_{n-2} & A_{n-1} \\
\end{bmatrix}
$$

%page 8.10

Consider also the operators
$$
G_q = GA_q =
\begin{bmatrix}
T_0 & 0 \\
0 & T_1
\end{bmatrix}, \quad 0 \le q \le n,
$$
where
$$T_0 =
(-1)^q\begin{bmatrix}
0 & 0 & \ldots & 0 & A_0\\
0 & 0 & \ldots & A_0 & A_1\\
\ldots & \ldots & \ldots & \ldots & \ldots\\
0 & A_0 & \ldots & A_{n-q-3} & A_{n-q-2} \\
A_0 & A_1 & \ldots & A_{n-q-2} & A_{n-q-1} \\
\end{bmatrix},
$$
$$
T_1 =
(-1)^{q-1}\begin{bmatrix}
A_{n-q+1} & A_{n-q+2} & \ldots & A_{n-1} & A_n\\
A_{n-q+2} & A_{n-q+3} & \ldots & A_n & 0\\
\ldots & \ldots & \ldots & \ldots & \ldots\\
A_{n-1} & A_n & \ldots & 0 & 0 \\
A_n & 0 & \ldots & 0 & 0 \\
\end{bmatrix}
$$
We assume that at $q = 0$ ($q=m$) the block $T_1$ ($T_0$) is absent in the represantion of the matrix $G_q$. Now we set
$$
F_q =
\begin{bmatrix}
T_0^* & 0 \\
0 & T_1
\end{bmatrix}
$$
Denote by $W_q$, $0 \le q \le n$, the self-adjoint operator in $H^n$, such that its matrix coincides with $F_q$ over the main diagonal, with matrix $F_q^*$ under the main diagonal and with $(F_q + F_q^*)/2$ on the main diagonal. For example, if $n=2$, then
$$
W_0 =
\begin{bmatrix}
0 & A_0^* \\
A_0 & A_1^{\mathbb {R}}
\end{bmatrix}, \quad
W_1 =
\begin{bmatrix}
-A_0^{\mathbb {R}} & 0 \\
0 & A_2^{\mathbb {R}}
\end{bmatrix}, \quad
W_2 =
\begin{bmatrix}
A_1^{\mathbb {R}} & A_n \\
A_n^* & 0
\end{bmatrix}
$$
and if $n=4$, then
$$
W_0 =
\begin{bmatrix}
0 & 0 & 0 & A_0^*\\
0 & 0 & A_0^* & A_1^*\\
0 & A_0& A_1^{\mathbb {R}} & A_2^* \\
A_0 & A_1 & A_2 & A_3^{\mathbb {R}}
\end{bmatrix}, \quad
W_2 =
\begin{bmatrix}
0 & A_0^* & 0 & 0\\
A_0 & A_1^{\mathbb {R}} & 0 & 0\\
0 & 0& -A_3^{\mathbb {R}} & -A_4 \\
0 & 0 & -A_4^* & 0
\end{bmatrix}.
$$

%page 8.11

Further we assume $n=2l$. A similar result can be obtained for the case $n=2l+1$. One can check easily that the following important equalities hold ($l = n/2$)
\begin{equation}
W_{2q}A - (W_{2q}A)^*=
i\begin{bmatrix}
0_{l-1-q} & 0 & 0\\
0 & V_0 & 0\\
0 & 0 & 0_q
\end{bmatrix} := iJ_q
\end{equation}
for $q = 0,1, \ldots, l-1$, where $0_q$ denotes a square zero matrix of order $q$, and
\begin{equation}
V_0 =
\begin{bmatrix}
2A_0^I & A_1^I & 0 & \ldots & 0 & 0\\
A_1^I & 2A_2^I & A_3^I & \ldots & 0 & 0\\
0 & A_3^I & 2A_4^I & \ldots & 0 & 0\\
\ldots & \ldots & \ldots & \ldots & \ldots\\
0 & 0 & 0 & \ldots & 2A_{n-2}^I & A_{n-1}^I \\
0 & 0 & 0 & \ldots & A_{n-1}^I & 2A_n^I \\
\end{bmatrix}
\end{equation}
Analogiously, for $q=1,2,\ldots,l$ we have
\begin{equation}
W_{2q-1}A - (W_{2q-1}A)^*=
-i\begin{bmatrix}
0_{l-q} & 0 & 0\\
0 & V_1 & 0\\
0 & 0 & 0_{q-1}
\end{bmatrix} := -iJ_{2q-1},
\end{equation}
where
\begin{equation}
V_1 =
\begin{bmatrix}
2A_1^\Upsilon & A_2^\Upsilon & 0 & \ldots & 0 & 0\\
A_2^\Upsilon & 2A_3^\Upsilon & A_4^\Upsilon & \ldots & 0 & 0\\
0 & A_4^\Upsilon & 2A_5^\Upsilon & \ldots & 0 & 0\\
\ldots & \ldots & \ldots & \ldots & \ldots\\
0 & 0 & 0 & \ldots & 2A_{m-1}^\Upsilon & A_{m}^\Upsilon \\
0 & 0 & 0 & \ldots & A_{m}^\Upsilon & 0 \\
\end{bmatrix}
\end{equation}

%page 8.11a

The following result is a generalization of Proposition 5 for polynomial operator pencils. It is proved in the paper Shkalikov [3].

\begin{theorem}
Let a polynomial pencil $A(\la)$ be defined by (24) and $0 \notin \sigma(A)$. Let
\begin{equation}
y_k^0,y_k^1, {\ldots}, y_k^{p_k}\text{,} \quad { k = N_1}, \dots, {N_2}
\end{equation}
\begin{equation}
z_k^0, z_k^1, {\ldots}, z_k^{p_k}\text{,} \quad { k = N_1}, \ldots, {N_2}
\end{equation}
be mutually adjoint canonical systems corresponding to the eigenvalues $\la_k \in \sigma_d(A)$. Then the following biorthogonality relations hold:
\begin{equation}
\begin{aligned}
(G_q\tilde y_j^h,\tilde z_j^s - \bar \la_j^{-1}&\binom{n-q}{1}\tilde z_j^{s-1}+\ldots+(-\bar \la_j)^s \binom{n+s-q-1}{s}\tilde z_j^0) \\
&= (-1)^{q+1}\la_k^{n-q}\delta_{s, p_j-s},
\end{aligned}
\end{equation}
where $\binom{r}{s}$ are binomial coefficients and $\tilde y_j^h$, $\tilde z_k^s$ are Keldysh derived chains constructed from canonical systems (29), (30) respectively.
\end{theorem}

\begin{pf}
Let the operator $\check A$ be defined by the same matrix as $A$, except that the operators $A_j$ are replaced by operators $A_j^*$, $j=1,\ldots,m$. Then the following equalities can be easily verified by induction
$$
GA_q = (\check A^q)^*G, \quad q=0,1,\ldots \, .
$$
\begin{multline*}
\check A^r\tilde z_j^s = (-1)^r \bar \la_j^{-r}[\tilde z_j^S - \bar \la_j^{-1} \binom{r}{1}\tilde z_j^{S-1}+\ldots+{}\\
{}+(-\bar \la_j)^S \binom {r+S-1}{S}\tilde z_j^0], \quad r = \pm1, \pm 2, \ldots \,.
\end{multline*}
These equalities enable us to prove (31) only for some fixed $q$, for instance, $q=n$.

%page 8.12

It is known (see Section~1) that Keldysh derived chains $\tilde y_k^h$ are EAV of a linear pencil $I +\la A$. If $\{\tilde x_j^S\}$ is the adjoint system then according to Proposition 6
$$
(A\tilde y_k^h, \tilde x_j^S) = \delta_{k,j}\delta_{h, p_s-h}.
$$
Hence, the relations (31) are equivalent to the following equalities
\begin{equation}
G_n^*\tilde z_j^S = (-1)^{n+1}A^{*} \tilde x_j^S \quad \text{or} \quad G_{n-1}^*\tilde z_j^S = (-1)^{n+1}\tilde x_j^S
\end{equation}

For the result $(I + \la A)^{-1}$ we have the following representation (which can be verified by multiplication of $(I + \la A)$)
\begin{equation}
(I+\la A)^{-1} = \begin{bmatrix}
L^{-1}(\la)T_1 & \ldots & L^{-1}T_n\\
\la L^{-1} T_1 & \ldots & \la L^{-1}T_n\\
\ldots & \ldots & \ldots\\
\la^{n-1}L^{-1}(\la)T_1 & \ldots & \la^{n-1}L^{-1}(\la)T_n
\end{bmatrix} +
\begin{bmatrix}
0 & 0 & 0 & \ldots & 0\\
0 & I & 0 & \ldots & 0\\
\ldots & \ldots & \ldots\\
0 & \la^{n-2}I & \la^{n-3}I & \ldots & I
\end{bmatrix},
\end{equation}
where
\begin{gather*}
L^{-1}(\la) = A^{-1}A_0,\\
T_1 = I, \quad T_j(\la) = -(\la L_j + \la^2 L_{j+1}+\ldots+ \la^{n-j+1}L_n), \quad j = 2, \ldots, n,\\
L_j = A_0^{-1}A_j.
\end{gather*}
On other hand according to Theorem on holomorphic operator function (Section~2) the principal part of $(I+\la A)^{-1}$ in a neighborhood of the pole $\la_k$ has the form
\begin{equation}
\sum^{N_2}_{k=N_1}\sum^{p_k}_{h=0}
\frac{(\cdot ,\tilde x_k^h)\tilde y_k^0 + (\cdot, \tilde x_k^{h-1})\tilde y_k^1 + \ldots +( \cdot ,\tilde x_k^0 )\tilde y_k^h}{(\lambda - \lambda_k)^{p_k+1-h}
}
\end{equation}

Let us compute the coefficient of $(\la - \la_k)^{-(p_k+1-h)}$ of the vector $(I+\la A)^{-1}\tilde f$, $\tilde f = {f_1,\ldots, f_n}$, in a neighborhood of the pole $\la_k$. Using the representation (33) and (2.2) (for $\la_k = c$) after simple algebra, we find that this coefficient is given
\begin{equation}
\sum_{j=1}^m \sum_{q=0}^h (\frac{1}{q!} \, T_j^{(q)}(\la_k)f_j, z_k^{h-q})\tilde y_k^0 + \ldots
\end{equation}
(we have not written out the remaining terms, which happen to be linear combinations of the elements $\tilde y_k^1,\ldots, \tilde y_k^h$). On the other hand from (34) we find that this coefficient equals
\begin{equation}
(\tilde f, \tilde x_k^h)\tilde y_k^0+\ldots \, .
\end{equation}
Let $\tilde x_k^h = \{x_{k, 1}, x_{k, 2}^h, \ldots, x_{k, n}^h\}$. Since $\tilde f = {f_1,\ldots, f_n}$ is any vector in $H^n$, we find comparing (35) and (36)
\begin{equation}
\begin{aligned}
x_{k, j}^h &= \sum_{q=0}^h \frac{1}{q!} \, T_j^{* \, (q)}(\la_k) z_k^{h-q} \\
&= -\sum_{q=0}^h \sum_{s=0}^{n-j}
\frac{1}{q!} \, \frac{d^q(\la^{s+1})}{d\la} \Bigr |_{\la = \bar \la_k} L_{j+s}^* x_k^{h-q} \\
&= -\sum_{q=0}^h \sum_{s=0}^{n-j} \frac{1}{q!} \, \la_k^{s-q+1} \binom{s+1}{q}L_{j+s}^*z_k^{h-q}.
\end{aligned}
\end{equation}
It follows from the distinction of Keldysh derived chains that
\begin{equation}
\tilde z_k^h = \Bigr \{\sum_{q=0}^h \la_k^{-q}\binom{0}{q}z_k^{h-q}, \sum_{q=0}^h \la_k^{1-q}\binom{1}{q}z_k^{h-q}, \ldots, \sum_{q=0}^h \la_k^{n-1-q}\binom{n-1}{q}z_k^{h-q} \Bigr \}.
\end{equation}
Writing out the matrix $G_{n-1}^*$ and using (38) we find that the first coordinate of the vector $G_{n-1}^*\tilde z_k^h$ is equal to $(-1)^{n-1}z_k^h = (-1)^{n-1}x_{k,1}^h$ (the last equality is valid according to (37)). Hence (32) is satisfied for the first coordinate. Using (37), (38) we can also check the equality (32) for the subsequent coordinates. It proves Theorem 11.
\end{pf}
\begin{note}
Theorem 11 is valid for arbitrary operator pencil $A(\la)$. The dissipative condition (1) or (2) is not required.
\end{note}

We say a pencil $A(\la)$ is \textit{linearly dissipative} if either the operator $V_0$, which is defined by (26), satisfies the condition $V_0 \le 0$ or the operator $V_1$, which is defined by (28), satisfies the condition $V_1 \le 0$. Since conditions (25), (27) hold, a pencil $A(\la)$ is linearly dissipative if and only if a linearization $A$ is $W_{2q}$-dissipative operator for all $q = 0,1,\ldots,l-1$, or $A$ is $W_{2q-1}$-dissipative for all $q=1, 2,\ldots, l$.

\begin{proposition}
If $A(\la)$ is linearly dissipative, i.e. $V_0 \le 0$ ($V_1 \le 0$), then it is dissipative, i.e. condition (1), (condition (2)) holds.
\end{proposition}
\begin{pf}
Let, for example, $V_0 \le 0$. Consider the function $u(t) = \phi(t)x$, where $x \in H$ and $\phi(t)$ is a smooth rapidly decreasing function when $t \to \pm \infty$ ($\phi(t) \in S$). If
$$
\tilde u(t) =\{u(t), -iu^\prime,\ldots, (-i)^{n-1}u^{(n-1)}(t)\}
$$
then

%page 8.12a

\begin{equation}
\begin{aligned}
\int_{-\infty}^{\infty} (V\tilde u(t), \tilde u(t))dt &= 2\sum_{s=0}^l \int_{-\infty}^{\infty} =  (A_{2s}^Iu^{(2s)}(t), u^{(2s)}(t))dt \\
&= -i \sum_{s=0}^{l-1} \int_{-\infty}^{\infty} \Bigr [(A_{2s+1}u^{(2s+1)}(t), u^{(2s)}(t)) +{}\\
&\kern 4cm{}+ (A_{2s+1}u^{(2s)}(t), u^{(2s+1)}(t)) \Bigr ] dt \\
&= 2 \Im \int_{-\infty}^{\infty} (A(\la) \hat \phi(\la)x, \hat \phi(\la)x) d\la \le 0.
\end{aligned}
\end{equation}
Here we denoted
$$
\hat u(\la) = \int_{-\infty}^{\infty} e^{-i\la t}u(t)dt = \hat \phi(\la)x
$$
and took into account that
$$
(i\la)^s\hat u(\la) = \int_{-\infty}^{\infty}e^{-i\la t}u^{(s)}(t)dt, \quad  \int_{-\infty}^{\infty}(\hat u(\la), \hat v(\la))d\la =  \int_{-\infty}^{\infty}(u(t), v(t))dt.
$$
Hence the inequality (39) holds for all $x \in H$ and all $\phi \in S$. It is known (see, for example, Yosida [1]) that the Fourier transform maps S onto S continuously in both directions. Using this fact we easily obtain (1) from (39).
\end{pf}

Obviously, the converse assertion is not true. For example, the pencil $A(\la) = i(1-\alpha \la^2 + \la^4)I, \quad 0 < \alpha \le 2$, satisfies condition (1) but it is not linearly dissipative.

Further we will consider linearly dissipative pencils, satisfying the condition $V_0 \le 0$. A similar result can be obtained for the case $V_1 \le 0$. In lecture 7 we introduced the systems $E^\pm$ consisting of half of EAV of a dissipative pencil $A(\la)$ and the systems $\xi^\pm$ consisting of the Keldysh derived chains of length $l = n/2$ corresponding to the vector $y_k^h\in E^\pm$.

%page 8.12b

\begin{proposition}
Let $A(\la)$ be a linear dissipative pencil, i.e. $V_0 \le 0$, $\mathcal{L}^+$ ($\mathcal{L}^-$) be the minimal subspace containing all elements $\tilde y_k^h$ (Keldysh derived chains of length $n$) constructed from elements $y_k^h \in E^+$ ($E^-$). Then $\mathcal{L}^+$ ($\mathcal{L}^-$) is a $W_{2q}$-non-positive (nonnegative) subspace for all $q =0,1,\ldots, l-1$.
\end{proposition}

\begin{pf}
Let $\la_k = c$ be a real eigenvalue of $A(\la)$. Keldysh derived chains $\tilde y_k^h$ coincide with EAV of the linearization $A$ of pencil $A(\la)$. It follows from (25) that $A$ is $W_{2q}$-dissipative, $q=0,1,\ldots, l-1$. According to the definition of the systems $E^+$ the elements $\tilde y_k^{\alpha_k} \in E^+$ satisfy the equalities
\begin{equation}
\tilde y_k^{\alpha_k} = \ep_k\tilde z_k^{\alpha_k}+\tilde y, \quad \tilde y \in S_c^0, \quad \ep_k = 1, \quad \alpha_k = [p_k/2].
%\tag{40}
\end{equation}
Using proposition 3 we can show\footnote{The proof depends on direct calculations, but these calculations are rather complicated if $n > 2$. Here we omit them.} that
$$
(W_{2q}\tilde y, \tilde y) = (G_{2q}\tilde y, \tilde y) \quad \text{for all $y \in S_c$}.
$$
In particular,
$$
(W_{2q}\tilde y_k^{\alpha_k}, \tilde y_k^{\alpha_k}) = (G_{2q}\tilde y_k^{\alpha_k}, \tilde y_k^{\alpha_k}).
$$
Now it follows from Theorem 11 and (40) that
\begin{equation}
(W_{2q}\tilde y_k^{\alpha_k}, \tilde y_k^{\alpha_k}) = (G_{2q}\tilde y_k^{\alpha_k},\tilde y_k^{\alpha_k}) = (-1)^{2q+1}\la_k^{n-2q} < 0
\end{equation}
Recalling the definition of the space $S^+$ for $W_{2q}$-dissipative operator $A$ we find from (41) that $\mathcal{L}^+ = S^+$. Then the assertion of Proposition 14 follows from Theorem 10.	
\end{pf}

%page 8.12c

\begin{proposition}
Let $A(\la)$ be a linearly dissipative operator pencil and $\mathcal{L}^+$ ($\mathcal{L}^-$) be the same subspace as in proposition 14, and $0 \notin \Theta(A_0)$. Define the operator $P: H^n \to H^l$ by the equality
$$
P\tilde x = P\{x_1,x_2,\ldots,x_n\} = \{x_{l+1},\ldots,x_n\}.
$$
Then
\begin{equation}
\|P\tilde x\| \ge \ep \|\tilde x\|, \quad \ep > 0, \quad \text{for all $\tilde x \in \mathcal{L}^+$ ($\mathcal{L}^-$)}
\end{equation}
i.e. the operator $P: \mathcal{L}^+ \to P(\mathcal{L}^+)$ has a bounded inverse.
\end{proposition}

\begin{pf}
Denote
$$
\tilde x = Q\{x_1,\ldots, x_n\} = \{x_1,\ldots, x_l\}.
$$
Obviously, the estimate (42) holds if and only if
\begin{equation}
\|Q\tilde x\| \le M\|P\tilde x\| \quad \text{for all $\tilde x \in \mathcal{L}^+$ ($\mathcal{L}^-$)},
\end{equation}
where the constant $M$ does not depend on $\tilde x$. Assume that the estimate (43) does not hold. Then there exists an element $\tilde x \in \mathcal{L}^+$ such that
\begin{equation}
\|Q\tilde x\| = 1, \quad \|P\tilde x\| = o(1), \quad \text{i.e. } \|x_j\| = o(1), \quad j = l+1\ldots, n.
\end{equation}
Recalling the matrix representations of the operators $W_{2q}$ and $A$ we find from (44) that
$$
|(W_{2q}\tilde x, \tilde x)| = o(1), \quad |(W_{2q}A\tilde x, A\tilde x)| \le \|W_{2q}\| \, \|A\|.
$$
Obviously, $A(\mathcal{L}^+) \subset \mathcal{L}^+$ and according Proposition 14 $\mathcal{L}^+$ is non-positive.
Hence we can apply Schwartz' inequality and obtain
\begin{equation}
|(W_{2q}A\tilde x, \tilde x)|^2 \le |(W_{2q}A\tilde x, A\tilde x)| = o(1).
\end{equation}

%page 8.16

On the other hand, we have
$$
\begin{aligned}
&(A\tilde x, W_0\tilde x) = \\
&\begin{bmatrix}
\begin{bmatrix}
A_0^{-1}A_1x_1+\ldots \\
-x_1 \\
\ldots\\
-x_l \\
-x_{l+1} \\
\ldots \\
x_{n-1}
\end{bmatrix},
\left[\begin{aligned}
&A_0^{*}x_n \\
&A_0^{*}x_{n-1} + A_1^*x_n  \\
&\ldots \quad \ldots \quad \ldots \\
&A_0^*x_l  +A_1^*x_{l+1}  +\ldots\\
&A_0x_{l-1}  +\ldots \\
&\ldots \quad \ldots \quad \ldots \\
&A_0x_{1} +\ldots \\
\end{aligned}\right]
\end{bmatrix} = -(A_0x_l, x_l) + o(1).
\end{aligned}
$$
Since $0 \notin \Theta(A_0)$, we find from (45) that $\|x_l\| = o(1)$. Using this relation and (44) we obtain
$$
(A\tilde x, W_2\tilde x) = -(A_0x_{l-1}, x_{l-1})+o(1) \Rightarrow \|x_{l-1}\| = o(1).
$$
Repeating the arguments, we find $\|x_j\| = o(1)$, $j = l, l-1, \ldots, 1$. This contradicts the assumptions $\|Q\tilde x\| = 1$. Hence estimate (42) is valid.	
\end{pf}

\begin{theorem}
Let $A(\la)$ be a linear dissipative operator pencil ($V_0 \le 0$) and $0 \notin \Theta(A_0)$, Then system $\xi^+$ ($\xi^-$) is minimal in the space $H^l$, $l = n/2$.
\end{theorem}

\begin{pf}
Let an operator $H: h_1 \to H_2$ be bounded and have bounded inverse. Obviously, if a system $\{e_k\}$ is minimal in $H_1$ then system $\{Ke_k\}$ is minimal in $H_2$. It follows from Proposition 5 that the system $\{\tilde y_k^h\} \in \mathcal{L}^+$ is minimal in $H^n$. Hence, from Proposition 15 we find that $\{P\tilde y_k^h\}$ is minimal in $H^l$. If all eigenvalues of $A(\la)$ are semi-simple then the system $\{\la_k^{-l}P\tilde y_k^0\}$ coincides with $\xi^+$, hence $\xi^+$ is minimal. In the general case one has to check  the equalities
\begin{equation}
\tilde y_k^h = \la_k^{-l} \sum_{r=0}^{h} \binom{n}{r} \la_k^{-r}P\tilde y_k^{h-r}, \quad \{\tilde y_k^h\} = \xi^+.
\end{equation}
This can be easily done by using the formula (38) for Keldysh derived chains $\tilde y_k^h$ (see details in Shkalikov [1]). The equalities (46) imply that the systems $\xi^+$ and $\{P\tilde y_k^h\}$ are connected by a triangular transformation and from this fact one can easily deduce that $\xi^+$ is a minimal.
\end{pf}

%page 8.17

We proved in Section~7 that the systems $\xi^+$ is linearly independent if $A(\la)$ satisfies condition (1). The result of minimality of $\xi^+$ is much sharper and it has been proved only in the case when $A(\la)$ is a linearly dissipative pencil. In this connection the following natural question arises.

\textit{Open Problem}$\quad$ Does condition (1) and $0 \notin \Theta(A_0)$ imply the minimality of the system $\xi^+$ in $H^l$?

\begin{note}
Theorems on minimality can be obtained for the case $V_0 \le 0$ and $n = 2l+1$ as well as for linearly dissipative pencils satisfying the condition $V_1 \le 0$ (see Section~7).
\end{note}

\textit{Comments}$\quad$ $W$-dissipative operators were introduced in the book of Dalezki, Krein [1] and were studied by Kuzhel, Azizov, Iohidov, M. Krein, Langer, Gomilko, Radzievski and many other authors. References can be found in the recent book Azizov, Iohidov [1]. In this book the proof of Proposition 1 is given, although here we proposed a new proof. Proposition 3 is due to Radzievski. Proposition 5 is proved by Keldysh [1, 2]. Proposition 6 and 7 are proved in the paper Shkalikov [3].

Theorem 10 seems to be new. Theorem 11 is proved in the papers Shkalikov [3, 6]. The second part of this lecture, connected with application to the half-range minimality problem is based on the paper Shkaliov [3]. Here a new version of the proof of Theorem 16 is given. To prove Proposition 5, we borrowed ideas from the paper Langer [3], where a similar assertion is proved for self-adjoint monic operator pencils.

The existense of a maximal $W$-non-positive subspace $S_c^+$ such that $S_c^+ \supset S_c^o$ was proved by Gomilko [1] (1983) (this is related to our Proposition 9). But here we get more, in particular, an important information on connection of direct and adjoint Jordan chains is established.

\newpage
%Лекция 09
\section{Factorization of dissipative operator pencils}

In this lecture we solve the problem of factorization of dissipative pencils in finite dimensional space and obtain some results on factorization of linearly dissipative pencils in infinite dimensional space.

First we prove one important result due to H. Langer. Let
\begin{equation}
A(\la) = A_0 + \la A_1+\ldots+\la^{n-1}A_{n-1}+\la^nA_n
\end{equation}
be a pencil of bounded operators in Hilbert space H and let
\begin{equation}
\tilde A = \begin{bmatrix}
-A_n^{-1}A_{n-1} & -A_n^{-1}A_{n-2} & \ldots & -A_n^{-1}A_1& -A_n^{-1}A_0\\
I & 0 & \ldots & 0 & 0\\
0 & I & \ldots & 0 & 0 \\
\ldots & \ldots & \ldots & \ldots \\
0 & 0 & \ldots & I& 0
\end{bmatrix}
\end{equation}
This is well-known and can be easily checked that $A$ is a linearization of $A(\la)$, i.e. the spectra of $A(\la)$ and $\la I-\tilde A$ coincide and the corresponding Jordan chains of operator $\tilde A$ coincide with Keldysh derived chains of $A(\la)$. Notice, that operator $\tilde A$ differs from linearization $A$ used in Section~8. The choice (2) in this lecture is not incidental. One comes with serious technical difficulties to prove the subsequent theorem in terms of the old linearization A.

%page 9.2

For any operator $K: H^k \to H^{n-k}$ (acting from $H^k$ into $H^{n-k}$) we call the subspace
\begin{equation}
M =
\begin{Bmatrix}
\begin{bmatrix}
K\hat x\\
\hat x
\end{bmatrix}
\Bigr | \hat x \in H^k
\end{Bmatrix} \subset H^n
\end{equation}
the graph subspace of $K$. Obviously, the graph subspace of any bounded operator $K$ is closed subspace in $H^n$.

\begin{theorem}
(Langer [3] (1976)). The pencil (1) admits factorization
\begin{equation}
A(\la) = L(\la)K(\la)
\end{equation}
with a pencil $K(\la) = \la^kI-\la^{k-1}K_{k-1}-\ldots-\la K_1 - J_0$ of degree $k$ ($< n$) and a pencil $L(\la)$ of degree $n-k$ if and only if the linearization $A$ has an invariant subspace of a bounded operator $K: H^k \to H^{n-k}$

If $M$ is such a subspace and
\begin{equation}
K = \begin{bmatrix}
K_{11} & K_{12} & \ldots & K_{1k}\\
K_{21} & K_{22} & \ldots & K_{2k}\\
\ldots & \ldots & \ldots & \ldots\\
K_{n-k,1} & K_{n-k, 2} & \ldots & K_{n-k, k}\\
\end{bmatrix}
\end{equation}
then the operator $K_{ij}: H \to H$ are uniquely determined by the coefficients of the right divisor, in particular,
$$
K_{n-k,j} = K_{k-j}, \quad j=1,2,\ldots,k.
$$
Moreover, the spectra of $K(\la)$ and of ${\tilde A}|_{M}$ coincide (${\tilde A} |_{M}$ is the restriction of $\tilde A$ onto $M$).
\end{theorem}

\begin{pf}
\textbf{Step 1.} Obviously, this is enough to prove theorem for monic pencil with $A_n = I$, therefore further we assume $A_n = I$.

Let a subspace $M$ has the representation (3) and K is defined by (5). Denote
$$
K_{n-k,j} = K_{k-j}, \quad j =1,2,\ldots, k.
$$
Observe, that if $M$ is invariant with respect to $\tilde A$ then the operators $K_{ij}$ are uniquely determined by the operators $K_j$. Indeed, if
\begin{equation}
\tilde x = \{K_{11}x, K_{21}x, \ldots, K_{n-k,1}x, x,0,\ldots, 0\}
\end{equation}
then the ($n-k$)-th component of $\tilde A \tilde x$ is $K_{n-k-1,1}x$. On the other hand, it has to be equal to $K_{k-1}^2x+K_{k-2}x$, since $\tilde A x\in M$. Therefore, $K_{n-k-1, 1} = K_{k-1}^2+K_{k-2}$. Using the inclusions $\tilde A^s \tilde x \in M$ we may determine the operators $K_{n-k-s, 1}$. Taking in (3) $\tilde x = \{0, x, 0, \ldots, 0\}$ and applying the same arguments we may determine the operators $K{i, 2}$ by $K_j$, and thence in a similar way all the other operator.

\textbf{Step 2.} Let us show that the existence of an invariant subspace (3) implies the factorization (4). For the resolvent $(A-\la I)^{-1}$ we have the following representation
$$
\begin{aligned}
(A-\la I)^{-1} =
&-
\begin{bmatrix}
\la^{n-1} \\
\la^{n-2} \\
\ldots \\
\la \\
1
\end{bmatrix}
A^{-1}(\la)[1, \la, \ldots, \la^{n-2}, \la^{n-1}]
\begin{bmatrix}
I & A_{n-1} & \ldots & A_2 & A_1\\
0 & I & \ldots & A_3 & A_3\\
\ldots & \ldots & \ldots & \ldots & \ldots\\
0 & 0 & \ldots & I & A_{n-1}\\
0 & 0 & \ldots & 0 & I\\
\end{bmatrix} \\
&+
\begin{bmatrix}
0 & I & \la I & \ldots & \la^{n-3}I & \la^{n-2}I\\
0 & 0 & I & \ldots & \la^{n-4} I & \la^{n-3}I\\
\ldots & \ldots & \ldots & \ldots & \ldots & \ldots\\
0 & 0 & 0 & \ldots & I & \la I\\
0 & 0 & 0 & \ldots & 0 &  I\\
0 & 0 & 0 & \ldots & 0 &  0\\
\end{bmatrix}
\end{aligned}
$$
which can be verified by multiplication of $\tilde A - \la I$ from the left.

Let a vector $\tilde x$ be defined by (6). Then the ($n-k$)-th component of $(\tilde A-\la I)^{-1}\tilde x$ equals
\begin{equation}
-\la^kA^{-1}(\la)L(\la)x + x
\end{equation}
where
$$
L(\la) = [1, \la, \ldots, \la^{n-2}, \la^{n-1}]
\begin{bmatrix}
I & A_{n-1} & \ldots & A_2 & A_1\\
0 & I & \ldots & A_3 & A_3\\
\ldots & \ldots & \ldots & \ldots & \ldots\\
0 & 0 & \ldots & I & A_{n-1}\\
0 & 0 & \ldots & 0 & I\\
\end{bmatrix}
\begin{bmatrix}
K_{11} \\
K_{21} \\
\ldots \\
K_{n-k, 1} \\
I \\
0 \\
\ldots \\
0
\end{bmatrix}
$$
is the pencil of degree $n-k$. By the invariance of $M$ we have $(\tilde A-\la I)^{-1}\tilde x \in M$. Therefore, the vector (7) can be also represented in the form
$$
K_{k-1}y_1+K_{k-2}y_2+\ldots+K_0y_k
$$
where
$$
y_j = -\la^{k-j}A^{-1}(\la)L(\la)x, \quad j = 1,2, \ldots, k,
$$
coincide with $(n-k+j))$-th component of the vector $(\tilde A - \la I)^{-1}\tilde x$. Hence
$$
-\la^kA^{-1}(\la)L(\la)x+x = (-\la^kK_{k-1}-\la^{k-2}K_{k-2}-\ldots-K_0)A^{-1}(\la)L(\la)x
$$
or
$$
K(\la)A^{-1}(\la)L(\la) = I
$$
and factorization (4) follows.

\textbf{Step 3.} Suppose, conversely, that the factorization (4) holds. Define the following matrices $(k \le j \le n-1)$

%page 9.6

$$
K_j =
\begin{bmatrix}
K_{k-1} & K_{k-2} & \ldots & K_1 & K_0 & 0& \ldots  & 0\\
I & 0 & \ldots & 0 & 0 & 0 & \ldots & 0\\
0 & I & \ldots & 0 & 0 & 0 & \ldots & 0\\
\ldots & \ldots & \ldots & \ldots & \ldots & \ldots & \ldots & \ldots\\
0 & 0 & \ldots & 0 & 0 & 0 & \ldots & I\\
\end{bmatrix}
$$
with $j+1$ rows and $j$ columns and also
\begin{equation}
K =
\begin{bmatrix}
K_{k-1} & \ldots & K_1 & K_0 \\
I & \ldots & 0 & 0 \\
\ldots & \ldots & \ldots & \ldots \\
0 & \ldots & I & 0 \\
\end{bmatrix}
\end{equation}
Suppose we have proved the equality
\begin{equation}
\tilde A K_{n-1} \ldots K_k =
\begin{bmatrix}
K_{k-1} & \ldots & K_0 & 0& \ldots  & 0 & 0\\
I & \ldots & 0 & 0& \ldots  & 0 & 0\\
\ldots & \ldots & \ldots & \ldots & \ldots  & \ldots & \ldots\\
0 & \ldots & 0 & 0& \ldots  & I & 0\\
\end{bmatrix} K_{n-1}\ldots K_k
= K_{n-1}\ldots K_k K
\end{equation}
Then the subspace $M = K_{n-1}\ldots K_kH^k$ is invariant with respect to $\tilde A$ and it is easy to see that this subspace has the representation (3).

%page 9.7

To prove (8), observe that the first step in the partial division of a polynomial $\la^lB_l+\la^{l-1}B_{l-1}+\ldots+B_0$ by $K(\la)$ ($l \ge k$) from the right gives a remainder whose coefficients are the entries of the product
$$
[B_l, B_{l-1},\ldots, B_0]K_l.
$$
Therefore the factorization (4) yields
$$
[A_{n-1}-K_{k-1}, A_{n-2}-K_{k-2}, \ldots, A_{n-k}-K_0, A_{n-k-1}, \ldots, A_0]K_{n-1}\ldots K_k = 0.
$$
Now it is easy to see that the last equality is equivalent to (9).

\textbf{Step 4.} Evidently, the spectrum of $K(\la)$ coincide with spectrum of the operator $K$ defined by (8). Moreover, it follows from (9) that
$$
(\tilde A -\la I)K_{n-1}\ldots K_k = K_{n-1}\ldots K_k(K-\la I).
$$
This means that $\sigma(K)$ coincides with spectrum of $\tilde A |_M$, where $M = K_{n-1}\ldots K_kH^k$, and the last assertion of Theorem 1 follows.	
\end{pf}

Now we can easily obtain the results on factorization of dissipative matrix polynomial. We consider both cases of the dissipativity condition, i.e.
\begin{equation}
\Im(A(\la)x, x) \le 0 \quad \text{for all $x \in H$ and $\la \in \mathbb {R}$}
\end{equation}
and
\begin{equation}
\Im(\la A(\la)x, x) \le 0 \quad \text{for all $x \in H$ and $\la \in \mathbb {R}$}.
\end{equation}

\begin{theorem}
Let the condition (10) hold, $n=2l$, $\dim H < \infty$, $\Ker A_n = \{0\}$ and there exists $\la_0 \in \mathbb {R}$ such that $0 \notin \Theta(A(\la_0))$. Then $A(\la)$ admits factorization (4) with a pencil $K(\la) = \la^l I -\la^{l-1}K_{l-1}-\ldots-K_0$ such that the system of eigen and associate vectors of $K(\la)$ coincides with the system $E^+$ ($E^-$) of pencil $A(\la)$.
\end{theorem}

\begin{pf}
Without loss of generality we may assume that $0 \notin \sigma(A)$. Otherwise we can shift $\la \to \la+\la_0$, obtain the factorization (4) and then shift back $\la \to \la-\la_0$.

Let $\mathcal{L}^+$ be the minimal subspace containing all elements $\tilde y_k^h$ (Keldysh derived chains of length $n$) constructed from elements $y_k^h \in E^+$ of pencil $A(\la)$. Obviously, $\mathcal{L}^+$ is invariant with respect to linearization $\tilde A$ of $A(\la)$. Let us prove that $\mathcal{L}^+$ has representation (3) with $k=l$. Consider the operator $P: H^n \to H^l$ defined by the equality
$$
P\tilde x = P\{x_1,x_2,\ldots,x_n\} = \{x_{l+1},\ldots, x_n\}.
$$
Denote $P^+ = P |_{\mathcal{L}^+}$. If all eigenvalues of $A(\la)$ are semi-simple then the system $\{\la_k^{-l}Py_k^0\}$ coincides with the system $\xi^+$ which is basis according to Theorems on completness and linear independence from Section~7.The completeness of $\xi^+$ implies $\Im P^+ = H^l$ while the linear independence implies $\Ker P^+ = \{0\}$. Then it follows immediately that $\mathcal{L}^+$ has the representation (3).

In general case (when non-semi-simple eigenvalues exist) make use from representation (8.40) which shows that the systems $\xi^+$ and $\{P\tilde y _k^h\}$ ($y_k^h \in E^+$) are connected by a triangular transformation. Hence, the system $\{P\tilde y_k^h\}$ is basis and this yields again that $\mathcal{L}^+$ has the representation (3). Now apply Theorem 1 to complete proof.
\end{pf}

\begin{theorem}
Let the condition (11) hold, $\dim H < \infty$, $n=2l$, $A_0 > 0$ and $A_n > 0$. Then $A(\la)$ admits factorization (4) with a pencil $K(\la) = \la^l I -\la^{l-1}K_{l-1}-\ldots-K_0$ such that the system of eigen and associate vectors of $K(\la)$ coincides with the system $Y^-$ ($Y^+$) of pencil $A(\la)$.
\end{theorem}

\begin{pf}
Apply Theorem 7.8 and repeat the arguments of Theorem 2.
\end{pf}

\begin{note}
A similar results can be obtained in the case $n=2l+1$. Namely, if $A(\la)$ satisfies the condition (10) and $A_n > 0$ ($A_n < 0$) then $A(\la)$ admits the factorization (4) with $K(\la)$ of degree $l+1$ ($l$) and the system if EAV of $K(\la)$ coincides with $E^+$.

If $n=2l+1$, $A(\la)$ satisfies the condition (11), $A_0 < 0$, $A_n < 0$ ($>0$) then $A(\la)$ admits the factorization (4) with $K(\la)$ of degree $l$ ($l+1$) and the system of EAV of $K(\la)$ coincides with $Y^+$. Another versions of factorization theorems involving the systems $E^-$ and $Y^+$ can be also formulated. These assertions follow from Theorems 7.6 and 7.9.
\end{note}

The problem on factorization of self-adjoint or dissipative operators in Hilbert space is much more deep. Langer [3] proved that each maximal $G$-non-positive ($G$-nonnegative) subspace $M \subset H^n$ which is invariant under the linearization $\tilde A$ of monic self-adjoint pencil $A(\la)$ has the form (3) with $k=[(n+1)/2]$ ($k=[n/2]$), where
$$
G =
\begin{bmatrix}
0 & 0 & \ldots & 0 & I\\
0 & 0 & \ldots & I & A_{n-1}\\
\ldots & \ldots & \ldots & \ldots & \ldots \\
0 & I & \ldots & A_3 & A_2\\
I & A_{n-1} & \ldots & A_2 & A_1\\
\end{bmatrix}
$$
is the simmetrization of $\tilde A$. Hence the factorization problem is reduced to the problem on existence of maximal $G$-semi-definite subspaces invariant with respect to $G$-self-adjoint operator $\tilde A$. Using the results of Section~8 we can obtain a similar assertion for linearly dissipative operator pencils. But the problem  on existence of maximal semi-definite subspace invariant with respect to self-adjoint operator in Krein space is still open. The deepest results of operator theory in spaces with indefinite metric are connected with this problem. It is solved for some particular classes of operators in Krein space and these results generate the corresponding factorizaion theorems. To make acquaintance with these results we refer the reader to the remarkable paper Langer [3].

In this connection the subsequent result on factorization of dissipative pencils in Hilbert space is of interest. It is formulated in terms of solvability of the half-range Cauchy problem. But the last problem can be solved for some classes of differential equations associated with pencil $A(\la)$ (see Note 8 in the end of this Section).

We say $v(z)$ is a regular solution of the equation
\begin{equation}
A(-i\frac{d}{dz})v(z) = A_0v - iA_1\frac{dv}{dz}+\ldots+(-i)^nA_n\frac{d^nv}{dz^n} = 0
\end{equation}
on interval $(a, b)$ if $v(z)$ has $n$ continuous derivatives as a function with values in $H$ on $(a, b)$ and $v(z)$ satisfies the equation (12).

Let $A(\la)$ be dissipative operator pencil satisfying the condition (10). Let also the real spectrum of $A(\la)$ be discrete. Denote by $S^+(0, \infty)$ the linear manifold of all solutions $v(z)$ of equation (12) on the semi-axis $(0, \infty)$ satisfying the Mandelstam radiation principle at $\infty$ (in sense of Definition 7.4). Assume also $n=2l$ and consider the trace operator $J: S^+(0, \infty) \to H^l$ defined by the equation
\begin{equation}
J_\xi v(z) = \{v(\xi), -iv^\prime(\xi),\ldots, (-i)^{l-1}v^{(l-1)}(z)\}.
\end{equation}
For $v(z) \in S^+(0, \infty)$ we denote by $\tilde v(z)$ the following function with values in $H^n$
\begin{equation}
\tilde v(z) = \{v(z), -iv^\prime(z),\ldots, (-i)^{n-1}v^{(n-1)}\}.
\end{equation}

\begin{proposition}
Ley $A(\la)$ be linearly dissipative pencil ($V_0 \le 0$) and let the operators $W_q$, $q=0,2,\ldots, n$, be defined as in Section~8. Suppose $v(z) = v_1(z)+ v_0(z) \in S^+(0, \infty)$ where $v_0^j(z) \to 0$ for $j = 0,1,\ldots, n-1$ when $z \to \infty$, and $v_1(z)$ is a finite superposition of elementary solutions corresponding to real eigenvalues. Then for $q=0,2,\ldots,n$ the following relations hold
\begin{equation}
(W_q\tilde v_0(z), \tilde v_0(z)) \le 0,
\end{equation}
\begin{equation}
(W_q\tilde v_1(z), \tilde v_1(z)) \le 0,
\end{equation}
\begin{equation}
(W_q\tilde v_0(z), \tilde v_1(z)) = 0,
\end{equation}
\begin{equation}
(W_q\tilde v(z), \tilde v(z)) \le 0,
\end{equation}
\end{proposition}

\begin{pf}
Let the linearization $A$ of $A(\la)$ be defined as in Section~8. Then the function $v(z)$ satisfies the equation
$$
A\tilde v^\prime(z)+i\tilde v(z) = 0,
$$
since $v(z)$ satisfies (12). Using (8.25) we obtain
$$
\begin{aligned}
(W_q\tilde v_0(z), \tilde v_0(z))^\prime &= (W_q\tilde v_0^\prime, \tilde v_0(z)) + (W_q\tilde v_0(z), v_0^\prime(z)) \\
&= -i(W_r\tilde v_0^\prime (z), Av_0^\prime(z)) +i(W_rA\tilde v^\prime(z), \tilde v^\prime(z)) = -(J_q\tilde v_0^\prime(z), v_0^\prime(z)).
\end{aligned}
$$
Integrating this equality from $z$ to $\infty$ and taking into account that $v_0(z)$ vanishes at $\infty$ we obtain
$$
(W_q\tilde v_0(z), \tilde v_0(z)) = \int_z^\infty (J_q\tilde v_0^\prime(\xi), \tilde v_0^\prime(\xi)) d\xi \le 0,
$$
since $V_0 \le 0$. Hence the inequality (15) holds.

By our assumption the operator $A$ is $W_q$-dissipative for $q=0,1,\ldots,n$. Then the inequality (16) follows from Theorem 8.10.

According to Proposition 8.3 $v_1(z) \in \Ker J_q$ and $v_1^\prime(z) \in \Ker J_q$. Therefore
$$
(W_q\tilde v_0(z), \tilde v_1(z)) = -(J_q\tilde v_0^\prime(z), \tilde v_1^\prime (z)) = 0.
$$
Now (15)-(17) give the inequality (18).	
\end{pf}

\begin{proposition}
Let $A(\la)$ be linearly dissipative ($V_0 \le 0$), $n=2l$ and $0 \notin \Theta(A_0)$. Then for $v(z) \in S^+(0, \infty)$ the following estimate holds
\begin{equation}
\begin{aligned}
\|v(z)\| &+\|v^\prime(z)\|+\ldots+\|v^{(l-1)}\| \\
&\le M(\|v^{(l)}(z)\| +\|v^{(l+1)}\|+\ldots+\|v^{(n-1)}(z)\|).
\end{aligned}
\end{equation}
\end{proposition}

\begin{pf}
The estimate (19) is equivalent to the following estimate
\begin{equation}
\|Q\tilde v(z)\| \le M \|P\tilde v(z)\|, \quad \tilde v(z) \in S^+(0, \infty),
\end{equation}
where the operators Q and P are defined in Proposition 8.15. Taking into account Proposition 4 we may repeat all arguments from Proposition 8.15 and prove (20) as well as (8.37).		
\end{pf}

\begin{note}
Denote by $S^+$ the minimal subspace in $H$ containing all vectors $\tilde v(\xi)$ for fixed $\xi > 0$, such that $v(z) \in S^+(0, \infty)$. Obviously, $\mathcal{L}^+ \subset S^+$ where  $\mathcal{L}^+$ is defined in Proposition 8.14. It may happen that $\mathcal{L}^+ \ne S^+$, hence Proposition 5 and Proposition 8.15 are not identical.
\end{note}

\begin{theorem}
Let the conditions of Proposition 5 hold. If for some $\xi \ge 0$ the image of the trace operator $J_\xi$ defined by (13) is dense in $H^l$ then $A(\la)$ admits the factorization (4) with pencil $K(\la)$ of degree $l$. Moreover, if the whole spectrum of $A(\la)$ is discrete then $\sigma(K)$ lie in the upper half-plane and the system of eigen and associated vectors of $K(\la)$ coincide with $E^+$.
\end{theorem}

\begin{pf}
If $v(z) \in S^+(0, \infty)$ then the function $\tilde v(z)$ defined by (14) satisfies the equation
\begin{equation}
\tilde A \tilde v^\prime - i\tilde v = 0,
\end{equation}
where $\tilde A$ is the linearization (2). Let $S^+$ be a minimal subspace in $H^n$ containing all vectors $\tilde v(\xi)$ for fixed $\xi \ge 0$ such that $v(z) \in S^+(0, \infty)$. Obviously, $S^+$ is invariant with respect to $\tilde A$. According to Proposition 5 we have representation
$$
S^+ =
\begin{Bmatrix}
\begin{bmatrix}
K\hat x \\
\hat x
\end{bmatrix}, \quad \hat x \in \mathcal{L} \subset H^l
\end{Bmatrix}
$$
where $K: H^l \to H^l$ is a bounded operator, and $\mathcal{L}$ is a linear manifold in $H^l$ which has to be subspace, since $S^+$ is subspace. Notice, that
$$
A^l\tilde v(\xi) = \{v_1(\xi),\ldots, v_l(\xi), v(\xi), -iv^\prime(\xi),\ldots,(-i)^{l-1}v^{(l-1)}(\xi)\} \in S^+
$$
if $\tilde v(\xi) \in S^+$. Therefore $\mathcal{L} \supset \Im J_\xi$. By our assumption $\overline{\Im J_\xi} = H^l$, hence $\mathcal{L} = H^l$ and $S^+$ is the graph subspace of $K$. Now Theorem 1 implies the factorization (4) with $k=l$ and $\sigma (K) - \sigma (\tilde A |_{S^+})$. This yields the assertion of Theorem 7.	
\end{pf}

\begin{note}
The condition $\overline{\Im J_\xi} = H^l$ can be established for some classes of operator pencils. Such results are proved in ch. 8 of the paper Shkalikov [6]. In particular, $\overline{\Im J_\xi} = H^l$ for selfadjoint operator pencils of Keldysh type (see Theorem 8.9 of Shkalikov [6]).
\end{note}

\textbf{Comments.} $\quad$ The problem on factorization of self-adjoint polynomials has a long history and takes the origin from the paper of Krein and Langer [1]. Important factorization theorem was proved by Rosenblum and Rovnyak [1]. The paper Langer [3] became a millstone in factorization problems for pencils of degree $n > 2$. Kostyuchenko and Ozarov [1, 2] classified the real spectrum of the right divisor for selfadjoint quadratic pencils. Gohberg, Lancaster and Rodman [1-3] established theorems on factorization of selfadjoint matrix polynomial with classification of real spectrum. Nontrivial factorization theorems were proved by Markus and Matcaev (see details and comments in the book of Markus [1]). Interesting results on factorization of matrix and operator functions are contained in the books Bart, Gohberg and Kaashoek [1] and Litvinchuck and Spitkovskii [1].

Factorization theorems of this lecture for dissipative operator pencils seem to be new and are based on the paper Shkalikov [3].

\newpage
% Lecture 10
\section{Pontrjagin spaces. The proof of Azizov-Iohvidov-Langer theorem}

A classical Hilbert theorem asserts that any self-adjoint compact operator in Hilbert space $H$ has a complete orthonormal system of eigenvectors in $H$. Does this result admit a generalization on Pontrjagin space? This problem is the main subject of this  section.

There are a number of books on operator theory in Pontjagin and in Krein spaces. We point out the books of Bognar [B], Ando [A], Iohvidov, Krein and Langer [IKL], Azizov and Iohvidov [AI] and seveys of Iohvidov and Krein [IK] and Langer [L]. Nevertheless the proof of the subsequent theorem on Riesz basis property of eigenfunctions of self-adjoint operator in Pontrjagin space (which is due to Azizov and Iohvidov) readers can find in the only book [AI] (Theorem 4.2.12 of [AI]). Nowever, it is not easy to restore the proof from the text since if uses a of foregoing material. In our lectures we will try to elucidate the situation. We notice that such an attempt has been undertaken already in the paper of Binding and Seddighi [BS] although the latter paper dealt only with completeness, the problem on minimality and basisness had not been considered there. We hope also that this material will help readers in understanding some important concepts in the theory of operators in spaces with in
 definite metric.

\subsection{Pontrjagin theorem and the formulation of Azizov-Iohvidov-Langer theorem.}

Let $P_\varkappa = (H,G)$ be Pontrjagin space, i.e. $H$ is also supplied with the scalar product $(x,y)$ but $H$ is also supplied by the indefinite metric $[x,y] = (Gx, y)$ where $G$ is self-adjoint bounded and invertible operator having $\varkappa$ negative eigenvalues counting with multiplicities. According to spectral theorem for self-adjoint operators we can represent $G = G_+-G_-$ where $G_+G_- = G_-G_+ = 0$, $G_+ \geq 0$, $G_- \geq 0$ and our assumptions on $G$ are equivalent to the following: $|G| :=G_+ + G_- \gg 0$, $\operatorname{rank} G_- = \varkappa < \infty$.

An operator $A$ is said to be self-adgoint in $P_\varkappa$ if $[Ax,y] = [x,Ay]$ for all $x,y \in H$ and this is equivalent that $GA$ is self-adjoint in $H$. We present without proof the following fundamental result.

\begin{theorem}[Pontrjagin {[P]}]
A self-adjoint operator $A$ in $P_\varkappa$ has a maximal nonnegative and maximal non-positive subspaces $\L^+$ and $\L^-$ respectively which are invariant under $A$. Moreover, $\dim \L^- = \varkappa$\footnote{Each maximal non-positive subspace in $P_\varkappa$ has the dimensional $\varkappa$. This fact is trivial.}.
\end{theorem}

The proof of this theorem is not trivial and can be found in the books mentioned above. Moreover, this theorem is also valid for bounded dissipative operators in $P_\varkappa$ and even for maximal dissipative operators (unbounded). See the book of Azizov and Iohvidov [AI]. We should say that the most difficult results of the theory of operators in space with indefinite metric are connected with this Pontrjagin theorem.

\begin{definition}
Let $\L$ be a subspace in $P_\varkappa$. The subspace $\L^{[\bot]} = \{ y \;|\: [y,x] = 0 \quad \forall x\in \L \}$ is said to be orthogonal to $\L$ in $P_\varkappa$ (or $G$ - orthogonal in $H$). The subspace $\L^0 = \L \cap \L^{[\bot]}$ is called isotropic subspace of $\L$.
\end{definition}

It follows from Definition 1.1 that for any subspace $\L$ we have $(\L^{[\bot]})^{[\bot]} = \L$ (saying $\L$ to be a subspace we always suppose that $\L$ is closed). Hence, the isotropic subspaces of $\L$ and $\L^{[\bot]}$ coincide.

\begin{definition}
A subspace $\L$ in $P_\varkappa$ is called non-degenerated if its isotropic subspace $\L^0 = \{0\}$. Otherwise $\L$ is called degenerated.
\end{definition}

Let $c \in \mathbb{C}$ be an eigenvalue of an operator $A$. We denote by $\L_c$ the subspace consisting of all eigen and associated vectors of $A$ corresponding to $c$. We call $\L_c$ the root subspace corresponding to $c$.

\begin{theorem}
Let $A$ be a self-adjoint compact operator in $P_\varkappa$. Then there exists a Riesz basis composed of eigen and associated vectors of $A$ if and only if the root subspace $\L_0$ corresponding to the point 0 is non-degenerated. Moreover, if $\L_0$ is non-degenerated then such a Riesz basis can be chosen almost $G$-orthogonal (i.e. all but finitely many vectors of this basis are mutually $G$-orthogonal).
\end{theorem}

To prove this theorem is a basis goal of our lectures. Here we recall the definition of a Riesz basis, the concepts of completeness and minimality which will be used in the sequel.

\begin{definition}
A system $\{ e_k \}$ of Hilbert space $H$ is said to be a Riesz basis if there exists such bounded and invertible operator $T$ in $H$ that $\{ Te_k \}$ is the complete orthonormal system.
\end{definition}

\begin{definition}
A system $\{ e_k \}$ in $H$ is said to be complete if any vector $x \in H$ can be approximated with arbitrary accuracy by a finite linear combination of elements from $\{ e_k \}$.
\end{definition}

\begin{definition}
A system $\{ e_k \}$ in $H$ is said to be minimal if there exists a system $\{ f_k \}$ such that $(e_k, f_j) = \delta_{kj}$ where $\delta_{kj}$ is the Kronecker symbol.
\end{definition}

\begin{problem}
A system $\{ e_k \}_1^\infty$ is complete in $H$ if and only if the equalities
$(f, e_k) = 0, \quad k = 1,2,\ldots,$
imply $f=0$.
\end{problem}

\begin{problem}
A system $\{ e_k \}$ is minimal in $H$ if and only if any its vector can not be approximated with any accuracy by a linear combination of the other elements.
\end{problem}

\begin{problem}
A system $\{ e_k \}$ is a Riesz basis in $H$ if and only if there exists a new scalar product $(\cdot, \cdot)_1$ in $H$ such that the norms $\| \cdot \|_1$ and $\|\cdot\|$ are equivalent and the system $\{ e_k \}$ is complete and orthonormal in $(\cdot, \cdot)_1$.
\end{problem}

\subsection{Example.}

Let us consider one concrete example in order to see what is happening when the root subspace $\L_0$ is degenerated. This example will help to understand the situation in general.

In the space $H = \ell_2$ we consider the operators

$$
A=\left(\begin{array}{cccccc}{0} & {0} & {0} & {0} & {0} & {. .} \\
{1} & {0} & {\alpha_{3}} & {\alpha_{4}} & {\alpha_{5}} & {. .} \\
{\bar{\alpha}_{3}} & {0} & {\beta_{3}} & {0} & {0} & {. .} \\
{\bar{\alpha}_{4}} & {0} & {0} & {\beta_{4}} & {0} & {.} \\
{\bar{\alpha}_{5}} & {0} & {0} & {0} & {\beta_{5}} & {.} \\
{.} & {.} & {.} & {.} & {.} & {.}\end{array}\right)
$$

$$G =
\begin{pmatrix}
\begin{matrix}
0 & 1 \\
1 & 0
\end{matrix}
& \rvline & \bigzero \\
\hline
\bigzero & \rvline &
\begin{matrix}
1 & 0 & 0 & 0 & \ldots \\
0 & 1 & 0 & 0 & \ldots \\
0 & 0 & 1 & 0 & \ldots \\
0 & 0 & 0 & 1 & \ldots \\
\cdot & \cdot & \cdot & \cdot & \ldots
\end{matrix}
\end{pmatrix}
$$
where $\{ \alpha_k \}_3^\infty$, $\{ \beta_k \}_3^\infty$ satisfy the following conditions
\begin{equation}\label{2.1}
\sum |\alpha_k|^2 < \infty, \quad \beta_k > \beta_{k+1} \rightarrow 0, \quad if \;\; k \rightarrow \infty
\end{equation}

\begin{problem}
Show that $A$ is compact operator in $\ell_2$ if the condition ($\ref{2.1}$) holds. Moreover, $GA = (GA)^*$, i.e. $A$ is self-adjoint operator in $P_1 = (\ell_2, G)$.
\end{problem}

Let us find all eigen and associated vectors of $A$. Writing the equation $Ax = \la x$ for $x = (x_1, x_2, \ldots) \in \ell_2$ we obtain

\begin{equation}\label{2.2}
\{ 0, x_1 + \sum_{k=3}^\infty \alpha_k x_k, \bar \alpha_3 x_1 + \beta_3 x_3, \bar \alpha_4 x_1 + \beta_4 x_4, \ldots \} = \la\{ x_1, x_2, x_3, x_4, \ldots \}.
\end{equation}
If $\la \neq 0$ we have $x_1 = 0, \quad \la = \beta_k, \quad x_j = 0$ for $j \neq k$ (since $\beta_j \neq \beta_k$ for $j \neq k$) and $x_2 = \alpha_k \beta_k^{-1} x_k$. This means that nonzero eigenvalues of $A$ coincide with $\{ \beta_k \}$ and the corresponding eigenvalues have the representation
$$y^k = \{ 0, \frac{\alpha_k}{\beta_k}, 0, \ldots, 0, 1, 0, \ldots \}, k=3,4,\ldots,$$
where 1 occupies the k-th position.

Now, suppose $\la$ to be equal zero in (\ref{2.2}). Certainly the vector $y^2 = \{ 0,1,0,0, \ldots \}$ is the eigenvector corresponding to the eigenvalue 0.

\begin{proposition}
The root subspace $\L_0$ of $A$ consists of the only vector $y^2$ (up to multiplication by constant) if and only if
\begin{equation} \label{2.3}
\{ \alpha_k \beta_k^{-1}\} \not\in \ell_2
\end{equation}
\end{proposition}

\begin{pf}
Let the condition ($\ref{2.3}$) holds. If (\ref{2.2}) is fulfilled with $\la_0$ then $x_k = -\bar \alpha_k \beta_k^{-1} x_1$. Therefore, in the case $x_1 = 0$ (\ref{2.2}) has the only solution $x=0$ and in the case $x_1 \neq 0$ the solution $x \not\in \ell_2$. This means that $\L_0$ does not contain another eigenvectors.

A vector $x$ is an associated with $y^2$ if $Ax = 0 \cdot x + y^2$, i.e.
\begin{equation}\label{2.4}
\{ 0, x_1 + \sum_{k=3}^\infty \alpha_k x_k, \bar \alpha_3 x_1 + \beta_3 x_3, \bar \alpha_4 x_1 + \beta_4 x_4, \ldots \} = \{ 0, 1,0,0, \ldots \}
\end{equation}
It is easy to see agein that if (\ref{2.4}) has a solution $x$ then $x \not\in \ell_2$. Hence, (\ref{2.3}) implies $\dim \L_0 = 1$.

On the other hand, suppose $\{ \alpha_k \beta_k^{-1}\} \in \ell_2 $. Then $y^1 = \{ -1, 0, \frac{\bar \alpha_3}{\beta_3}, \frac{\bar \alpha_4}{\beta_4}, \ldots \}$ is the eigenvector of $A$ if the condition
$$\gamma:= 1- \sum_{k=3}^\infty \frac{|\alpha_k|^2}{\beta_k} = 0$$
holds and $y^1 = -\gamma^{-1} \{ -1, 0, \frac{\bar \alpha_3}{\beta_3}, \frac{\bar \alpha_4}{\beta_4}, \ldots \}$ is the associated with $y^2$ if $\gamma \neq = 0$.
\end{pf}

\begin{proposition}
A system of root vectors of $A$ is a Riesz basis in $\ell_2$ if and only if $\{ \alpha_k \beta_k^{-1}\} \not\in \ell_2 $ then the system of root vectors is not complete in $\ell_2$ and not minimal.
\end{proposition}
\begin{pf}
Let $\{ \alpha_k \beta_k^{-1}\} \in \ell_2 $. Then the system $\{ y^k \}_1^\infty$ is obviously complete and minimal (prove this!). Moreover, $\{ y^k \}_2^\infty$ is a Riesz basis in the subspace $\ell_2 \ominus \{e_1\}$, where $e_1 = \{1,0,0,\ldots\}$ (prove this!). Then $\{y^k\}_1^\infty$ is a Riesz basis in $\ell_2$.

Now, suppose (\ref{2.3}) to be hold. Then the system of root functions coincides with $\{y^k\}_2^\infty$. Obviously, it is not complete (the element $e_1$ is orthogonal to $\{y^k\}_2^\infty$) and it is not minimal! We prove this showing $y^2 \in \overline{\Span\{ y^k \}_3^\infty}$. Suppose there exists a vector $y = \{y_1, y_2, \ldots \} \in \ell_2$ such that $0 = (y, y^k) = y_2 \alpha_k \beta_k^{-1} - y_k = 0$.

Now it follows: if $y_2 \neq 0$ then $y \not\in \ell_2$ but if $y_2 = 0$ then $y = 0$. Hence, the system $\{ y^k \}_3^\infty$ is complete in $\overline{\Span\{ y^k \}_2^\infty}$.
\end{pf}

We notice that in the case $\{ \alpha_k \beta_k^{-1}\} \not\in \ell_2$ the root subspace $\L_0 = \{ y^2 \}$ is degenerated since $[y^2, y^2] = 0$. Because of that the system of root functions neither complete nor minimal.

\subsection{Criteria for $\mathcal{L}$ to be Pontrjagin subspace in $P_\varkappa$.}
Let us recall some well known facts on geometry of Pontrjagin space. The results which we present in this section are well known although the proofs sometimes are new. We always suppose $\mathcal{L}$ to be closed saying $\mathcal{L}$ to be a subspace. First, let us recall the following definitions.

\begin{definition}
A subspace $\mathcal{L}$ in $P_\varkappa$ is said to be a Pontrjagin subspace if it is Pontrjagin space with indefinite metric inherited from $P_\varkappa$.	 
\end{definition}

\begin{definition}
A subspace $\mathcal{L}$ is said to be a regular subspace in $P_\varkappa$ if the operator $G_{\mathcal{L}}=P_{\mathcal{L}} G | \mathcal{L}$ where $P_{\mathcal{L}} : H \rightarrow \mathcal{L}$ is the opthoprothector and $G|\mathcal{L}$ is the restriction of $G$ onto $\mathcal{L}$ is invertible. The operator $G_\mathcal{L}$ is called Gram operator.
\end{definition}

\begin{definition}
Let $\mathcal{L}$ be a subspace in $P_\varkappa$. An operator $Q:H\rightarrow \mathcal{L}$ is said to be $G$-orthogonal projector onto $\mathcal{L}\ $ if  $\ Q^2 = Q\ $  and  $\ x-Qx \in \mathcal{L}^{[\perp]}$.
\end{definition}

\begin{rem}
Certainly $G$-orthogonal projector not always exists. But if it exists then it is bounded. Indeed, it is defined on the whole $H$ and it is easy to prove that it is closed. Then by virtue of Closed graph theorem $Q$ is bounded.
$$ x_n \rightarrow x, \ \ y_n=Gx_n \rightarrow y. \ \ \ \ \ Qy_n = y_n \rightarrow y \ \Rightarrow \ Q^2 y_n = Qy_n.$$
\end{rem}

\begin{definition}
A subspace $\mathcal{L}$ is said to be projectively complete if
\begin{equation} \label{eq1}
\mathcal{L} + \mathcal{L}^{[\perp]} = H.
\end{equation}
\end{definition}

\begin{rem}
If ($\ref{eq1}$) holds then the sum is direct. Indeed, if $\mathcal{L}^0 = \mathcal{L} \cap \mathcal{L}^{[\perp]}$ then ($\ref{eq1}$) implies $\mathcal{L}^0 [\perp] \ \mathcal{L} + \mathcal{L}^{[\perp]} = H$, hence $G(\mathcal{L}^0) \perp H$. Since $G$ is invertible we have $\mathcal{L}^0 = \{0\}$.
\end{rem}

\begin{definition}
A subspace $\mathcal{L}$ in $P_\varkappa$ is said to be positive (uniformly positive) if $[x,x]>0$ ($\geq \varepsilon ||x||^2$ with some $\varepsilon >0$) for all $0 \neq x \in \mathcal{L}$. Negative and uniformly negative subspaces are defined in the same way.
\end{definition}

\begin{theorem}
The following statements in Pontrjagin space $P_\varkappa$ are equivalent:
\begin{enumerate}
\item $\mathcal{L}$ is positively complete;
\item There exists a $G$-orthogonal projector onto $\mathcal{L}$;
\item $\mathcal{L}$ is regular;
\item $\mathcal{L}$ is a Pontrjagin subspace;
\item $\mathcal{L}$ is non-degenerated;
\item $\mathcal{L}^{[\perp]}$ is non-degenerated.
\end{enumerate}
\end{theorem}

\textit{Note.} $\mathcal{L} = \{
\begin{pmatrix}
x\\
\alpha x
\end{pmatrix},
x \in H_2  \} $
is degenerated $\Leftrightarrow \ \alpha = \pm 1 $.

\begin{proof}
\textit{Step 1}. Let us prove 1) $\Leftrightarrow$ 2).\\
If $\mathcal{L}$ is projectively complete then according to Remark 3.2 the sum ($\ref{eq1}$) is direct. This implies the existence of a uniquely defined on the whole $H$ $G$-orthogonal projector onto $\mathcal{L}$. The implication 2) $\Rightarrow$ 1) follows from the definition.\\
\textit{Step 2}. 2) $\Rightarrow$ 3). According to Remark 3.1 $G$-orthogonal projector $Q$ onto $\mathcal{L}$ is bounded. Now for $x\in \mathcal{L}, ||x||=1$, we have
$$
||G_{\mathcal{L}} x|| \ ||QG|| \geq |(G_{\mathcal{L}} \ x, QGx)| = |(P_{\mathcal{L}} Gx, QGx)| = |(Gx, QGx)| = $$ $$ = |(Q^* Gx, Gx)| = |(GQx, Gx)| = (Gx,Gx) \geq ||G^{-1}||^2 .
$$
Here we used the fact $[Qx,y] = [x, Qy]$ which is equivalent to $(GQ)^* = GQ$ and follows from the definition of $Q$. From the last inequality we obtain $||G_{\mathcal{L}} x || \geq \varepsilon ||x||$. Since $G_{\mathcal{L}}$ is self-adjoint we have it is inveritable. \\
\textit{Step 2a}. 3) $\Rightarrow$ 2). Suppose that $G_{\mathcal{L}}$ is inveritable. Then $G_{\mathcal{L}} = G_{\mathcal{L}}^+ - G_{\mathcal{L}}^-$, where $G_{\mathcal{L}}^+$ and $G_{\mathcal{L}}^-$ are uniformly positive on the subspaces $\mathcal{L}^\pm = Im \  G_{\mathcal{L}}^\pm$. This means that the norm $||\cdot || $ in $\mathcal{L}^+ (\mathcal{L}^-)$ is equivalent to the norm $||\cdot||_1$ defined by the equality $||x||_1 = |[x,x]|^{1/2}$. Then for any fixed $y\in H$ a linear functional $\phi_y (x) = (x, Gy)$ is continious in $\mathcal{L}^+(\mathcal{L}^-)$ with respect to both norms $|| \cdot||$ and $||\cdot||_1$. By virtue of the classical Riesz' theorem there exists a vector $x_1 \in \mathcal{L}^+ (\mathcal{L}^-)$ such that $\phi_y(x) = [x,y] = [x,x_1^{\stackrel{+}{\small{(-)}}}] $. Now, define $Q^\pm y = x_1^\pm$. Then $Q^+$ and $Q^-$ are $G$-orthogonal projectors onto $\mathcal{L}^+$ and $\mathcal{L}^-$. Moreover, since $\ G_{\mathcal{L}}^+\ G_{\mathcal{L}}^
 - = G_{\mathcal{L}}^-\ G_{\mathcal{L}}^+ = 0\ $ we have $\ Q^+\ Q^- = Q^-\ Q^+ = 0\ $. Hence $Q = Q^+ + Q^-$ is $G$-orthogonal projector onto $\mathcal{L}$.\\
\textit{Step 3}. The equivalence 3) and 4) follows from the definitions. \\
\textit{Step 4}. 4) $\Rightarrow$ 5). We have proved 4) $\Rightarrow$ 1) and according to Remark 3.2 $\ $ 1) $\Rightarrow$ 6). \\
\textit{Step 5}. 5) $\Rightarrow$ 4). Let $\mathcal{L}$ be non-degenerated. This is equivalent that $Ker\ G_{\mathcal{L}} = \{0\} $. We have also $$G_{\mathcal{L}} = P_{\mathcal{L}}\ (G_+\ -\ G_{-})\ | \ \mathcal{L} = P_{\mathcal{L}}\ (G_+\ +\ G_{-})\ |\ \mathcal{L} - 2\ P_{\mathcal{L}}\ G_{-}\ |\ \mathcal{L} = |G|_{\mathcal{L}}\ -\ 2\ G_{\mathcal{L}}^-$$. The operator $|G|_{\mathcal{L}}$ is uniformly positive in $\mathcal{L}$ while $G_{\mathcal{L}}^{-}$ is of finite rank. Now it follows from Fredholm theorem that $G_{\mathcal{L}}$ is inveritable since $Ker \ G_{\mathcal{L}} = \{0\}$. \\
\textit{Step 6}. The implications 5) $\Leftrightarrow$ 6) are obvious.
\end{proof}

As a corollary we obtain the following important theorem.
\begin{theorem}
Any positive (negative) subspace in $P_\varkappa$ is uniformly positive (negative).
\end{theorem}

\begin{proof}
Let $\mathcal{L}$ be a positive subspace. Then it follows $G_{\mathcal{L}} > 0$, therefore $\mathcal{L}$ is non-degenerated. Now apply Theorem 3.1.
\end{proof}

\begin{rem}
First four statements in Theorem 3.1 are equivalent also in Krein space. Only on Step 5 we used the fact that $G^{-}$ is of finite rank.
\end{rem}

\subsection{Riesz basis theorem.}
First we present some simple lemmas. We omit their proofs because they can be found in any book concerning indefinite metric. \\
Let us denote $\hat{\lambda}_k : = \{\lambda_k, \bar{\lambda}_k\}$ and $$\hat{\mathcal{L}}_k = \mathcal{L}_{\lambda_k} + \mathcal{L}_{\bar{\lambda}_k}.$$

\begin{lemma}
Let $\{\lambda_k\}_0^\infty$ be a sequence of eigenvalues of $A$ and $\{\mathcal{L}_{\lambda_k}\}$ be the corresponding root subspaces. Then $$\hat{\mathcal{L}}_k\ [\perp] \ \hat{\mathcal{L}}_j \ \ \ for\ \ all\ \ \hat{\lambda}_k \neq \hat{\lambda}_j$$
and each subspace $\hat{\mathcal{L}_k}$ is non-degenerated with possible exeption $\hat{\mathcal{L}}_0 = \mathcal{L}_0$ corresponding to the eigenvalue $\lambda_0 = 0$.
\end{lemma}

\begin{lemma}
There are finitely many subspaces, say $\mathcal{L}_{\lambda_0}, \mathcal{L}_{\lambda_1}, ... , \mathcal{L}_{\lambda_p}$ containing associated vectors, moreover, the length of Jordan chain corresponding to each eigenvector does not exceed $2\varkappa+1$.
\end{lemma}

\begin{lemma}
The dimension of each non-positive subspace in $P_\varkappa$ does not exceed $\varkappa$.
\end{lemma}

Now let us prove Theorem 1.2.
\begin{pf}
\textit{Step 1}. Let The root subspace $\mathcal{L}_0$ corresponding to $0$ be degenerated and $\mathcal{L}_0^0\ $ be the isotropic subspace of $\mathcal{L}_0$. We denote
$$\mathcal{L} = \overline{span\{\hat{\mathcal{L}}_k\}}_{\lambda_k \in \sigma(A)}.$$
From Lemma 4.1 we obtain that $\mathcal{L}_0^0$ is also the isotropic subspace in $\mathcal{L}$. Then $D(\mathcal{L}_0^0) \perp \mathcal{L}$ and $\mathcal{L} \neq H$. Thus we proved: the non-degeneracy of $\mathcal{L}_0$ is the necessary condition for completeness (and, certainly, for basisness) of the system of root vectors.\\
\textit{Step 2}. Now, assume that $\mathcal{L}_0$ is non-degenerated. We can represent it in the form
$$ \mathcal{L}_0 = \mathcal{L}_0 ' [+] \mathcal{L}_0 '' $$
where $\mathcal{L}_0 ' \in Ker \ A $ while $\mathcal{L}_0 ''$ is of finite dimension and both subspaces $\mathcal{L}_0 '$ and $\mathcal{L}_0 ''$ are invariant under $A$ (because of Lemma 4.2 there are finitely many associated vectors, therefore we can choose $\mathcal{L}_0 ''$ of finite dimensional). Certainly, $\mathcal{L}_0 '$ and $\mathcal{L}_0 ''$ are non-degenerated. \\
In each subspace $\{\mathcal{L}_{\lambda_k}\}_{k>p}$ and in $\mathcal{L}_0 '$ we can choose $G$-orthogonal basis $\{y_{k,s}\}$. In the finite dimensional subspaces $\mathcal{L}_0 ''$ and $\{\hat{\mathcal{L}}_k\}_1^p$ we can choose any basis $\{y_{k,s}\}$ consisting of root functions. By virtue of Theorem 3.1 Gram operators $G_{\mathcal{L}_0 '}$, $G_{\mathcal{L}_0 ''}$, $\{G_{\hat{\mathcal{L}}_k} \}_1^\infty$ are inveritable. Certainly, $\{y_{k,s}\}$ for each fixed $k$ is minimal in $\mathcal{L}_0 '$, $\mathcal{L}_0 ''$, $\{\hat{\mathcal{L}}_k \}_1^\infty$ respectively. This implies the existence of systems $\{z_{k,h}\}$ such that
$$ (y_{k,s}, z_{k,h}) = [y_{k,s}, G_{\hat{\mathcal{L}}_k}^{-1} z_{k,h}] = \delta_{sh}.$$
By virtue of Lemma 4.2 we have also
$$ [ y_{j,s}\ G_{\hat{\mathcal{L}}_k}^{-1} z_{k,h}] =  ( y_{j,s}\ G G_{\hat{\mathcal{L}}_k}^{-1} z_{k,h}) = \delta_{jk} \delta_{sh} . $$
The last relation shows that the chosen system of root functions $\{y_{k,s}\}$ in $ \mathcal{L}\ $ is minimal.\\
\textit{Step 3}. As before, let $ \mathcal{L}^0$ is non-degenerated. Then $\mathcal{L}$ is also non-degenerated. Indeed, if $\mathcal{L}^0$ is the isotropic subspace of $\mathcal{L}$ then $\mathcal{L}^0$ is invariant under $A$ (show this!). By virtue of Lemma 4.3 $\mathcal{L}^0$ is finite dimensional. According to Lemma 4.1 $A|\mathcal{L}^0$ has no non-zero eigenvalues. Hence $\mathcal{L}^0 \subset \mathcal{L}_0$. But $\mathcal{L}^0$ can not be the isotropic subspace of $\mathcal{L}$ since there are no isotropic subspace in $\mathcal{L}_0$. \\
Thus, we have proved that $\mathcal{L}$ is non-degenerated. By virtue of Theorem 3.1 $\mathcal{L}^{[\perp]}$ is non-degenerated and, certainly, it is invariant under $A$. The operator $A|\mathcal{L}^{[\perp]} = : A' $ has no eigenvalues. According to Theorem 3.1 $P_{\varkappa'}:=\mathcal{L}^{[\perp]}$ is the Pontrjagin subspace with $\varkappa'\leq \varkappa$. If $\varkappa'>0$ then according to Theorem 1.1 $A'$ has $\varkappa'$-dimensional invariant subspace and, consequently, has eigenvalues. This is a contradiction. If $\varkappa'=0$ then Gram operator $G_{\mathcal{L}^{[\perp]}}$ is positive and inveritable (Theorem 3.1) hence $G_{\mathcal{L}^{[\perp]}}$ is strictly positive, hence $P_\varkappa'=\mathcal{L}^{[\perp]}$ is the Hilbert space with the norm which is equivalent to the previous one. Since $A$ is compact we have $A'$ is compact in $P_{\varkappa'}$. According to Hilbert theorem it has the eigenvalues. This is again the contradiction. \\
\textit{Step 4}. It was proved already that we can choose a system $\{y_{k,s} \} $ of root functions almost $G$-orthonormal. Let
$$\mathcal{L}' = \overline{span\{\mathcal{L}_{\lambda_k}\}}_{k>p} + (\mathcal{L}_0 ')^+ ,\ \ (\mathcal{L}_0 ' )^+ \subset \mathcal{L}_0 ' .$$
Obviously, $\mathcal{L} '$ is $G$-positive for $p$ sufficiently large if $(\mathcal{L}_0 ')^+$ is positive in $\mathcal{L}_0 '$. We have also that $\mathcal{L} '$ is of finite codimension. By virtue of Theorem 3.2 $\mathcal{L} '$ is uniformly positive hence for $y_{k,s}, \ y_{j,h} \in \mathcal{L} '$ we have ( $G_{\mathcal{L} '} >> 0$ )
$$ [y_{k,s}, y_{j,h}] = (G_{\mathcal{L} '}\ y_{k,s} ,\ y_{j,h}) = (G_{\mathcal{L} '}^{1/2} y_{k,s}, G_{\mathcal{L} '}^{1/2} y_{j,h}) = \delta_{kj} \delta_{sh} .$$
This means that the basis $\{y_{k,s}\}$ in $\mathcal{L}'$ is equivalent to orthonormal basis $\{G_{\mathcal{L} '}^{1/2} y_{k,s} \}$ in $\mathcal{L} '$. Since $\mathcal{L}'$ is of finite codimension in $H$ we obtain the assertion of Theorem 1.2.
\end{pf}

\renewcommand{\refname}{{\large\rm Bibliography for Part~I}}

%%%%%%%%%%%%%%%%%%%%%%%%%%%%%%%%%%%%%%%%%%%%%%%%%%%%%%%%%%%%%%%%%%%%%%%%%%%%%%%%%%%%%%%%%%%%%%%%%%%%%%%%
%%%%%%%%%%%%%%%%%%%%%%%%%%%%%%%%%%%%%%%%%%%%%%%%%%%%%%%%%%%%%%%%%%%%%%%%%%%%%%%%%%%%%%%%%%%%%%%%%%%%%%%%

\newpage
{\centerline{\huge  Part 2}}

\medskip

\section{Operator pencils arising in elasticity and hydrodynamics: The instability index formula}

\subsection*{Introduction}
The plan of the present section is the following. In subsection 1 we consider some concrete problems arising in elasticity and hydrodynamics. Further we prefer to work with abstract formulations of physical problems under consideration. For this purpose we provide general classes of operator pencils with unbounded operator coefficients related to problems of origin. The main object of the paper is an operator pencil of the form
$$A(\lambda) = \lambda^2 F + (D+i G)\lambda + T,$$
where $F$ and $T$ are self-adjoint and boundedly invertible operators, while $D\geqslant 0$ and $G$ are symmetric and $T$-bounded. The study of the pencil $A(\lambda)$ is realized in Section 3. In particular, we introduce the concepts of the classical and the generalized spectra and investigate the relations between them. We associate the linear pencil
$$
\mathbf{A}(\lambda):=\mathbf{T}-\lambda \mathbf{W}:=-\left(\begin{array}{cc}{D+\i G} & {T} \\ {-J} & {0}\end{array}\right)-\lambda\left(\begin{array}{cc}{F} & {0} \\ {0} & {J}\end{array}\right), \quad J=T|T|^{-1},
$$
with the quadratic pencil $A(\lambda)$. It turns out that the operator $\textbf{T}$ is dissipative in the space $\textbf{H} = H x H_1$, where $H_1$ coincides with the domain of the operator $|T|^{1/2}$ and equipped with the norm $
(\cdot, \cdot)_{1}=\left(|T|^{1 / 2} \cdot,|T|^{1 / 2} \cdot\right).
$ Generally, the spectrum $\sigma(\textbf{A})$ of the linearization $\textbf{A}(\lambda)$ coincides with neither the classical nor the generalized spectrum of $A(\lambda)$. However, we prove that $\sigma(\textbf{A})$ coincides with the generalized spectrum of $A(\lambda)$ in the open right half plane if the operator $\textbf{W}$ generates a Pontrjagin space metric. In this case $\sigma(\textbf{A})$ in the right half plane consists of finitely many eigenvalues, say $\kappa(A)$, and the number $\kappa(A)$ characterizes the index of instability of the equation
$$
A\left(\frac{d u}{d t}\right)=F \frac{d^{2} u}{d t^{2}}+(D+\i G) \frac{d u}{d t}+T u=0, \quad u=u(t)
$$
The problem on stability for such kind of equations has a long background and apparently was originated by Kelvin and Tait \cite{11::KT} (in the end of Section 3 we present a short historical review related to this problem). The main result of the paper is the instability index formula
$$
\kappa(A)=\nu(F)+\nu(T)-\varepsilon^{+}(A)
$$
where $\nu(F)$ and $\nu(T)$ are the numbers of the negative eigenvalues of the operators $F$ and $T$ respectively, while $\varepsilon^{+}(A)$ is expressed in terms of the lengths and the sign characteristics of Jordan chains corresponding to the pure imaginary eigenvalues of $A(\lambda)$. In particular, if all the pure imaginary eigenvalues of $A(\lambda)$ are of definite type then $\varepsilon^{+}(A)$ coincides with the number of the first type eigenvalues of $A(\lambda)$ (see the definitions in Section 3).\\

The results of Section 2 on root subspaces of linear dissipative pencils seem at the first sight to be isolated from the main subject of the paper. However, these results form a theoretical base to prove the index formula in Section 3. In our opinion, they have also an independent interest.\\

In Section 4 we return to the physical problems of origin and present the corollories of our abstract results. Here we also demonstrate how the index formula can be applied to estimate the number of the nonreal eigenvalues of a self-adjoint operator pencil.

\subsection{Classes of unbounded operator pencils}
Small oscillations of an elastic thin beam of unit length with external and internal damping	(so called Kelvin-Voigt material)	are described by the equation
\begin{equation}
\frac{\partial^{4} u}{\partial x^{4}}+\frac{\partial}{\partial t} \frac{\partial^{2}}{\partial x^{2}}\left(\alpha(x) \frac{\partial^{2} u}{\partial x^{2}}\right)+\frac{\partial}{\partial x}\left(g(x) \frac{\partial u}{\partial x}\right)+\beta(x) \frac{\partial u}{\partial t}+\rho(x) \frac{\partial^{2} u}{\partial t^{2}}=0
\end{equation}
Here $x	\in	[0,1],	t \in \mathbb{R}^{+}$, and $u(x,t)$ is the	 transverse displacement at position	$x$ and time
$t$. The function $\alpha (x)$ determines the internal damping and takes generally small values. The function $\beta (x)\geqslant 0$ determines the distribution of viscous damping, $\rho(x)>0$ defines the mass distribution and $g(x)$ is responsible for the forces of contraction or tension (see more details in \cite{11::PI}, for example).

As the equation is considered on the finite interval, we have to submit solutions of
(1.1) to some boundary conditions. For the sake of definitness we consider the case when both ends of the beam are clamped, i.e.
\begin{equation}
u(0, t)=\frac{\partial u(x, t)}{\partial x} /_{x=0}=u(1, t)=\frac{\partial u(x, t)}{\partial x} /_{x=1}=0
\end{equation}
Separating variables $u(x, t)=y(x) e^{\lambda t}$, we obtain the following spectral problem
\begin{equation}
\begin{aligned} \rho^{-1}(x)\left[y^{(4)}(x)\right.&\left.+\left(g(x) y^{\prime}(x)\right)^{\prime}\right] \\ &+\lambda \rho^{-1}(x)\left[\left(\alpha(x) y^{\prime \prime}(x)\right)^{\prime \prime}+\beta(x) y(x)\right]+\lambda^{2} y(x)=0 \end{aligned}
\end{equation}
\begin{equation}
y(0)=y^{\prime}(0)=y(1)=y^{\prime}(1)=0.
\end{equation}
Suppose that $\rho(x), \beta(x) \in C[0,1], g(x) \in C^{1}[0,1]$ and $\alpha(x)\in C^1[0,1]$. According to the physical sense we have $\rho(x)>0, \alpha(x) \geqslant 0, \quad \beta(x) \geqslant 0$ and either $g(x) \geqslant 0$ or $g(x) \leqslant 0$. Then quadratic eigenvalue problem (1.3), (1.4) is represented in the form
\begin{equation}
\left[\lambda^{2} I+\lambda\left(D_{\alpha}+D_{\beta}\right)+A+C\right] y(x)=0,
\end{equation}
where operators $D_{\alpha}, D_{\beta}, A, C$ act in Hilbert space $H=L_{2}([0,1], \rho(x))$ with the scalar product
$$
(y, z)=\int_{0}^{1} \rho(x) y(x) \overline{z(x)} d x
$$
and are defined by the equalities
\begin{equation}
\begin{aligned}(A y)(x) &=\rho^{-1}(x) y^{(4)}(x), \quad\left(D_{\alpha} y\right)(x)=\rho^{-1}(x)\left(\alpha(x) y^{\prime \prime}(x)\right)^{\prime \prime} \\\left(D_{\beta} y\right)(x) &=\rho^{-1}(x) \beta(x) y(x), \quad(C y)(x)=\rho^{-1}(x)\left(g(x) y^{\prime}(x)\right)^{\prime} \end{aligned}
\end{equation}
on the domains
$$
\mathcal{D}(A)=\mathcal{D}\left(D_{\alpha}\right)=\mathcal{D}(C)=\left\{y | y \in W_{2}^{4}[0,1], y(0)=y^{\prime}(0)=y(1)=y^{\prime}(1)=0\right\} ,
$$\\
$$\mathcal{D}\left(D_{\beta}\right)=L_{2}([0,1], \rho(x))=H .$$
We denote by $I$ the identity operator and by $W_{2}^{k}[0,1]\left(k \in \mathbb{N}^{+}\right)$ the Sobolev spaces.

Naturally, it is more fruitful to study an abstract operator pencil of the form (1.5) rather then problem (1.3), (1.4). We have only to extract the most essential properties of the operators (1.6). We observe that these operators satisfy the following conditions (the terminology of unbounded operator theory we borrow from the book \cite{11::Ka}):
\begin{itemize}
\item[\textit{i)}] $A=A^{*} \gg 0$ (i.e. $A$ is self-adjoint and uniformly positive), and $T:=A+C$ is self-adjoint and bounded below;
\item[\textit{ii)}] $D_{\alpha}$ and $D_{\beta}$ are nonnegative symmetric $A$-bounded operators.
\item[\textit{iii)}] the identity operator $I$ and the operator $C$ are $A$-compact (or $T$-compact) and hence $T$ has finitely many negative eigenvalues.

\end{itemize}

Some results on spectrum of problem (1.3), (1.4) in the case $\alpha(x)=\textit{const}$ were reported by Pivovarchik \cite{11::PI}. The comprehensive study of abstract pencil (1.5) with $D_{\alpha}=\alpha A, \quad C=0$ was carried out by Lancaster and Shkalikov \cite{11::LS}. Additional results in the case $D_{\alpha}=\alpha A, \quad C \neq 0$ were obtained in a recent paper by Shkalikov and Griniv \cite{11::SG}. New problems appear in the case $\alpha\neq const$, as pencil (1.5) in this situation has nontrivial essential spectrum. However, we leave an interesting problem on the spectrum localization of pencil (1.5) with $D_{\alpha} \neq \alpha A$ for another occasion. We will deal with pencil (1.5) (and more general ones) mainly in view of the application of our index formula.

A more interesting example for the application of the index formula comes from hydrodynamics. Namely, small transverse oscillations of ideal incompressible fluid in a pipe of finite length are described by the equation which is obtained from (1.1) if we add in the left hand side of (1.1) the "gyroscopic" term
$$
2 s v \partial^{2} u / \partial x \partial t .
$$
Here $v$ is the velocity of the fluid and $s$ depends on the mass of the pipe and the fluid (see \cite{11::ZKM}, for example). The physical meanings of the functions in (1.1) are subject to change in this situation. In particular, $g(x) = v^2$. Assuming $sv = const$ and repeating the previous arguments we come to the following quadratic spectral problem
\begin{equation}
\left[\lambda^{2} I+\lambda\left(D_{\alpha}+D_{\beta}+i G\right)+A+C\right] y=0 ,
\end{equation}
where
$$
G y=-2 s v i y^{\prime}, \quad \mathcal{D}(G)=\mathcal{D}(A)
$$
and$D_{\alpha}, D_{\beta}, A, C$ are defined as in (1.6). The last operators retain the properties i)-iii). The most essential properties of the operator $G$ are the following:
\begin{itemize}
\item[\textit{iv)}]  $G$ is a symmetric $T$-bounded operator;
\item[\textit{v)}] $G$ is a $T$-compact operator.
\end{itemize}

It is also of interest to consider equation (1.1) on the semi-axis  $x \in \mathbb{R}^{+}$ (see the papers of Pivovarchik \cite{11::P2} and Griniv \cite{11::Gr}). Assuming that the left end of a beam is clamped we define the operator coefficients in (1.5) by equalities (1.6) on the domains
$$
\mathcal{D}(A)=\mathcal{D}\left(D_{\alpha}\right)=\mathcal{D}\left(D_{\beta}\right)=\mathcal{D}(C)=\left\{y | y \in W_{2}^{4}[0, \infty], y(0)=y^{\prime}(0)=0\right\}
$$
Obviously, this definition is correct if we assume in addition that all the functions $\rho(x), \rho^{-1}(x), \alpha(x), \beta(x), g(x)$ are bounded on $\mathbb{R}^{+}$. In this case, the properties i)-ii) are
retained, however, the property iii) is not true any more. This makes the problem much more complicated. Nevertheless, under some additional assumptions on the behavior of the function $g(x)$ at $\infty$ (see \cite{11::Gr}) the important property
\begin{itemize}
\item[\textit{vi)}]  $T=A+C$ has finitely many negative eigenvalues
\end{itemize}
remaims valid.

Analogously, equation (1.1) with the additional "gyroscopic"	term	 can	 be	considered
on the semi-axis $\mathbb{R}^{+}$ with respect to the variable $x$. In this	 case	we	obtain	a	pencil	of
the form (1.7) whose coefficients satisfy the properties i)-ii), iv) and also the property vi) under additional assumptions on the behavior of the function $g(x)$ found in the paper \cite{11::Gr}.
\subsection{Root subspaces of linear dissipative pencils and their properties}
In this section we deal with a linear dissipative operator pencil
$$
A(\lambda)=T-\lambda W,
$$
where $W$ is a bounded self-adjoint operator, while $T$ is a closed dissipative operator in Hilbert space $H$. This means that $T$ is closed and
$$
\operatorname{Im}(T x, x) \geqslant 0 \quad \text { for all } x \in \mathcal{D}(T)
$$
and $\mathcal{D}(T)$ is the domain of $T$. Through all the section we also assume that there exists at least one point $\mu_0$ belonging to the open upper half plane $\mathbb{C}^{+}$ such that $A\left(\mu_{0}\right)$ has a bounded inverse, i.e. $\mu_{0} \in \rho(A)$.

If the operator $W$ has a bounded inverse then the spectrum $\sigma(A)$ and the root subspaces $\mathcal{L}_{\mu}(A)$ of the pencil $A(\lambda)$ coincide with those of the operator $A = W^{-1}T$. Hence, in this case spectral problems for the pencil $A(\lambda)$ are equivalent to those for dissipative operators in Krein or Pontrjagin spaces (see \cite[ch.II, $\S$2]{11::AI}). In the sequel we prefer to deal with the linear pencil $A(\lambda)$. The motivation for this becomes clear when considering the corresponding operator differential equations. Moreover, at least formally, we obtain more general results, as we do not always assume that $W$ generates a regular indefinite metric.

The basic goal of this section is to prove formula (2.17). This formula is based on the well-known fundamental result on the existence of a maximal $W$-nonnegative $A$-invariant subspace in Pontrjagin space and on the explicit construction of maximal $W$-nonnegative subspaces corresponding to real normal eigenvalues of the pencil $A(\lambda)$ or of the operator $A=W^{-1}T$. In the paper \cite{11::SI} the author considered dissipative operator pencils of an arbitrary order $n \geqslant 1$ and constructed for such pencils regular canonical systems corresponding to real normal eigenvalues. This construction allows us to define the sign characteristics for Jordan chains and to realize the construction of a maximal $W$-nonnegative subspace $\mathcal{L}_{\mu}^{+}$ in the root subspace $\mathcal{L}_{\mu}$ corresponding to a real eigenvalue $\mu$. The additional details for linear pencils were given in the unpublished manuscript \cite{11::S2}. We note also the papers of Kostyuchenko and Or
 azov \cite{11::K} (devoted to the case of a self-adjoint operator $T$) and Gomilko \cite{11::G} related to this topic. However, our construction is new and, perhaps simpler, even for self-adjoint pencils. In addition we obtain the information on the connection of the middle elements of mutually adjoint canonical systems. This information is essentially used when considering half range completeness and minimality problems (see \cite{11::S1}). Recently Ran and Temme \cite{11::RT} investigated an analogous problem from another point of view. Here
we present some results of \cite{11::S2} concerning this subject.

Let $\mu$ be an eigenvalue of the pencil $A(\lambda)=T -\lambda W$ and
\begin{equation}
y_j^0,y_j^1,\dots,y_j^{p_j},\qquad j=1,\dots,N,
\end{equation}
be a canonical system of eigen and associated elements (or Jordan chains) corresponding to $\mu$ (see \cite{11::Ke}). The linear span of all elements (2.1) is denoted $\mathcal{L}_{\mu}(A)$ or simply $\mathcal{L}_{\mu}$ and is called the root subspace corresponding to the eigenvalue $\mu$. An eigenvalue $\mu$ is said to be normal if $A(\lambda)$ is invertible in a punctured neighborhood of $\mu$ and the number $N = \Ker(T - \mu W)$ as well as the lengths $p_j + 1$ of Jordan chains (2.1) are finite. It is known \cite{11::Ke} that the principal part of the Laurent expansion of the function $A^{-1} (\lambda)$ at the pole $\mu$ has the representation
\begin{equation}
\sum\limits_{j=1}^{N}\sum\limits_{s=0}^{p_j}\frac{(\cdot, x^s_j)y^0_j+\dots+(\cdot, x^0_j) y_j^s}{(\lambda - \mu)^{p_j+1-s}},
\end{equation}
where the adjoint system
\begin{equation}
x_j^0,x_j^1,\dots, x_j^{p_j}, \qquad j=1,\dots,N,
\end{equation}
is uniquely determined by the choice of system (2.1). It turns out that the adjoint system
(2.3)  is a canonical system of Jordan chains corresponding to the eigenvalue $\overline{\mu}$ of the pencil $A^*(\lambda)=T^*-\lambda W$.

Further the upper index is always used for numeration of associated elements while the the lower one numerates eigenvalues and canonical chains simultaneously, i.e. each eigenvalue is counted as many times as its geometric multiplicity. The set of all eigenvalues of the pencil $A(\lambda)$ is denoted $\sigma_p(A)$. For the subset in $\sigma_p(A)$ consisting of the normal eigenvalues we reserve the notation $\sigma_d(A)$ (the discrete spectrum). Notice that canonical system of Jordan chains (2.1) is well defined for any $\mu \in \sigma_p(A)$ (possibly, consisting of infinitely many elements), however, adjoint system (2.3) is well defined only for $\mu \in \sigma_d(A)$. As usually the indefinite scalar product $(Wx, x)$ is denoted $[x,x]$.

Although some of the subsequent propositions are essentially known, we present their proofs here for the reader's convenience. New constructions axe started from Proposition 2.6.

\begin{proposition}
Let $\mathcal{L}^0_+$ be the minimal subspace containing the root subspaces corresponding to all $\mu \in \mathbb{C}^+ \cap \sigma_p(A)$. Then $\mathcal{L}^0_+$ is a $W$-nonnegative subspace.
\end{proposition}
\begin{proof}
(Cf. \cite[Ch.2, Corollary 2.22]{11::AI}). We present here another, shorter proof. Suppose eigenvalues are numerated as many times as their geometric multiplicity. Let us consider the functions
$$
u_j^h(t)=e^{i\mu_j t}\left( y^h_j+\frac{it}{1!}y_j^{h-1}+\dots+\frac{(it)^h}{h!}y_j^0\right) , \qquad h=0,1,\dots,p_j,
$$
where $y_j^0,\dots,y_j^{p_j}$ are Jordan chains corresponding to the eigenvalues $\mu \in \mathbb{C}^+$. It is	easily
seen that the functions $u_j^h(t)$ satisfy the equation
$$
i W u^{\prime}(t) +Tu(t)=0.
$$
Any linear combination $u(t) = \sum c_{j,h} u^h_j(t0)$ also satisfies this equation, therefore
\begin{multline*}
[u(\xi), u(\xi)]^{\prime}= (W u^{\prime}(\xi)) +(u(\xi), Wu^{\prime}(\xi))\\
=(i T u(\xi), u(\xi)) + ( u(\xi), i T u(\xi) )= -2 Im(T u(\xi), u(\xi)).
\end{multline*}
As all the functions $u_j^h(t)$ vanish at $\infty$, so does $u(t)$. Integrating the last equality from $t$ to $\infty$ we obtain
$$
[u(t), u(t)] = 2 \int\limits_{t}^{\infty} Im(T u(\xi), u(\xi)) \geqslant 0.
$$
In particular $[u(0),u(0)] \geqslant0$ for all $u(0) = \sum c_{j, h} y^h_j$. By the definition the set of these elements is dense in $\mathcal{L}^0_+$, hence, $\mathcal{L}^0_+$ is a $W$-nonnegative subspace.
\end{proof}
\begin{proposition}
Let (2.1) be a canonical system corresponding to a real eigenvalue $\mu$.
If $[\gamma]$ is the integer part of a number $\gamma$ then the elements
\begin{equation}
y_k^0,y_k^1, \dots, y_k^{\alpha_k}, \qquad k = 1,\dots,N, \qquad \alpha_k=\left[ \frac{p_k}{2}\right],
\end{equation}
belong to $\mathcal{D}(T^*)$ and $T^* y_k^h = T y_k^h$ for all $1 \leqslant k \leqslant N, \quad 0\leqslant h \leqslant \alpha_k$.
\end{proposition}
\begin{proof}
First we notice that $T^*$ is well defined, as the	operator $T$ is closed	 by	assumption
(see \cite[Ch.3, $\S$5.5]{11::Ka}). Now, let us prove the following:
If $x \in \mathcal{D}(T)$   and $Im(Tx,x) =0$ then $x \in \mathcal{D}(T^*)$ and $T^{*} x=Tx$. (Cf. \cite[Ch.2, Theorem 2.15]{11::AI}). To prove this fact, we introduce an indefinite product in the space $\textbf{H} = H \times H$ as follows
$$
\langle \{x_1, x_2\} , \{y_1, y_2\}\rangle = i (x_1, y_2) - i(x_2, y_1).
$$
As $T$ is dissipative, we have
$$
\langle\textbf{x},\textbf{x}\rangle = 2Im(Tx,x)\geqslant0 \qquad \text{for all} \textbf{x} = \{x,Tx\} \in \Gamma(T),
$$
where $\Gamma(T)$ is the graph of $T$. If $x \in \mathcal{D}(T)$ and $Im(Tx,x) = 0$ then by virtue of Cauchy--Schwarz--Bunyakovskii inequality we obtain
$$
|(x,Tz) - (Tx, z)| = |\langle\textbf{x},\textbf{z}\rangle| \leqslant\langle \textbf{x}, \textbf{x}\rangle ^{1/2} \langle \textbf{z}, \textbf{z}\rangle ^{1/2} = 0 \qquad \text{for all} \textbf{z} = \{z, Tz\} \in \Gamma(T).
$$
Hence, $(Tz,x) = (z,Tx)$ for all $z \in \mathcal{D}(T)$. From the definition of the adjoint operator we obtain $x \in \mathcal{D}(T^*)$ and $T^*x = Tx$.

Now let us prove the assertion of Proposition 2.2. As the elements of system (2.1) are Jordan chains, we have
\begin{equation}
(T - \mu W) y_k^h = W y_k^{h-1}, \qquad o \leqslant h \leqslant p_k (y_k^{-1}:=0).
\end{equation}
In particular,
$$
Im((T - \mu W) y_k^0,y_k^0) = Im(Ty_k^0,y_k^0) = 0 .
$$
Therefore, $y_k^0 \in D(T^*)$ and $T y_k^0 = T^*  y_k^0$.
Now we can end the proof by induction. Suppose that for some $h \leqslant \alpha_k$ we have proved that
$$
y_k^s \in D(T^*) \text{and} T y_k^s = T^* y_k^s \qquad \text{for} s = 0, 1, \dots, h-1.
$$
As $2h \leqslant p_k$, we find
\begin{multline*}
(W y_k^{h-1} , y_k^h ) = (y_k^{h-1}, (T - \mu W)y_k^h) = (W y_k^{h-2}, y_k^{h+1}) =\dots\\
=(y_k^0,(T - \mu W)y_k^{2h}) = 0.
\end{multline*}
Hence, $Im (W y_k^{h-1} + \mu W y_k^h, y_k^h) = 0$ and $Im (T y_k^h, y_k^h) = Im((T - \mu W)y_k^h - W y_k^{h-1}, y_k^h)$
As before we deduce that $y_k^h \in D(T^*)$ and $T y_k^h = T^* y_k^h$.
\end{proof}
\begin{proposition}
Let (2.1) be a canonical system corresponding to a real eigenvalue $\mu$ . Then
\begin{equation}
[y_k^h, y_j^s] = 0, \qquad j= 1, \dots, N, \qquad h\leqslant[(p_k-1)/2], \qquad s \leqslant [(p_k-1)/2].
\end{equation}
If $p_j \neq p_k$ then (2.6) hold for all $s \leqslant [p_j/2], h \leqslant [p_k/2]$.
\end{proposition}
\begin{proof}
Suppose $p_j \leqslant p_k$. Then it follows from our assumptions that $h + s + 1 \leqslant p_k$. Taking into account (2.5) and the equalities $T y_j^s = T^* y_j^s$ (Proposition 2.2) we find
\begin{multline*}
(y_k^h , W y_j^s ) =((T - \mu W) y_k^{h+1}, y_j^s) = (y_k^{h+1}, W y_j^{s-1}) =\dots\\
=(y_k^{h+s+1}, (T^* -\mu W)y_j^0) =0.
\end{multline*}
and the equalities (2.6) follow.
\end{proof}

\begin{proposition}\label{pr:24}
Let $\nu$ and $\mu$ be eigenvalues of the pencils $A(\lambda)$ and $A^*(\lambda)$ respectively. If $\nu\ne\overline{\mu}$ then the root subspaces $\mathcal{L}\nu(A)$ and $\mathcal{L}\mu(A^*)$ are $W$-orthogonal. In particular, truncated Jordan chains (2.4) corresponding to a real eigenvalue $\mu$ of the pencil $A(\lambda)$ are $W$-orthogonal to any root subspace $\mathcal{L}\nu(A)$ if $\nu\ne\mu$.
\end{proposition}
\begin{proof}
Let $y^{0}, \ldots, y^{p} \in \mathcal{L}_{\nu}(A), x^{0}, \ldots, x^{q} \in \mathcal{L}_{\mu}\left(A^{*}\right)$ be Jordan chains and $\nu\ne\overline{\mu}$. Using
\begin{equation}\label{eq:25}
\begin{aligned}\left(T y^{s}, x^{l}\right) &=\nu\left[y^{s}, x^{l}\right]+\left[y^{s-1}, x^{l}\right] \\ &=\left(y^{s}, T^{*} x^{l}\right)=\bar{\mu}\left[y^{s}, x^{l}\right]+\left[y^{s}, x^{l-1}\right], \quad\left(y^{-1}:=x^{-1}:=0\right) \end{aligned}
\end{equation}
In particular, from these equalities we have $[y^0, x^0] = 0$. Now, the proof of the first assertion is ended by induction with respect to the index $s + l$. The second assertion follows from Proposition 2.2.
\end{proof}

\begin{proposition}\label{pr:25}
Let (2.1) and (2.3) be mutually adjoint canonical systems corresponding to normal eigenvalues $\mu_j$ which are enumerated according to their geometric multiplicity. Then the following biorthogonality relations hold:
\begin{equation}\label{eq:27}
\left[y_{k}^{h}, x_{j}^{s}\right]=-\delta_{k, j} \delta_{h, p_{j}-s}
\end{equation}
where $\delta_{m,n}$ is the Kronecker symbol.
\end{proposition}
\begin{proof} (Cf.\cite{11::Ke}). We have
$$
A(\lambda) y_{k}^{h}=\left[A\left(\mu_{k}\right)-\left(\lambda-\mu_{k}\right) W\right] y_{k}^{h}=W y_{k}^{h-1}-\left(\lambda-\mu_{k}\right) W y_{k}^{h}, \quad 0 \leqslant h \leqslant p_{k},
$$
where as before it is assumed that $y_k^{-1}:= 0$. Using the representation (2.2) we obtain

\begin{equation}\label{eq:28}
\begin{array}{l}{y_{k}^{h}=A^{-1}(\lambda) A(\lambda) y_{k}^{h}} \\ {=\sum_{j=N_{1}}^{N_{2}} \sum_{s=0}^{p_{j}}\left[\frac{\left(\cdot, x_{j}^{s}\right) y_{j}^{0}+\ldots+\left(\cdot, x_{j}^{0}\right) y_{j}^{s}}{\left(\lambda-\mu_{j}\right)^{p_{j}+1-s}}+R(\lambda)\right]\left[-\left(\lambda-\mu_{k}\right) W y_{k}^{h}+W y_{k}^{h-1}\right]}\end{array}
\end{equation}
where $R(\lambda)$ is a holomorphic operator function at the point $\mu =\mu_j$ and $N_2-N_1+1$ is the geometric multiplicity of the eigenvalue $\mu$. We may assume that $N_1=1$, $N_2= N$,
$p_{1} \geqslant p_{2} \geqslant \ldots \geqslant p_{N}$.

Suppose that $\mu_k\ne \mu_j$. If we take $h=0$ and compare the coefficients of the powers $\left(\lambda-\mu_{j}\right)^{-p_{j}-1+s}, 0 \leqslant s \leqslant p_{j}$, we find
\begin{equation}{eq:29}
(-\sum_{p_{j}=p_{1}}\left[y_{k}^{0}, x_{j}^{0}\right] y_{j}^{0}=0,
\end{equation}
\begin{equation}{eq:210}
-\sum_{p_{j}=p_{1}}\left[y_{k}^{0}, x_{j}^{1}\right] y_{j}^{0}-\sum_{p_{j}=p_{1}}\left[y_{k}^{0}, x_{j}^{0}\right] y_{j}^{1}-\sum_{p_{j}=p_{1}-1}\left[y_{k}^{0}, x_{j}^{0}\right] y_{j}^{0}=0.
\end{equation}
We do not write out the other coefficients corresponding to the indices $s\ge 2$. We also notice that the third term in (2.10) should be omitted if there are no Jordan chains of length $(p_1 + 1) - 1$. It follows from the definition of a canonical system that the elements $\left\{y_{j}^{0}\right\}_{1}^{N}$ are linearly independent. Hence, from (2.9) we have
\begin{equation}{eq:211}
[y^0_k,x^0_j] = 0,\quad  \text{for all indices}\quad  j\quad \text{such that}	 p_j	= p_1.
\end{equation}
Now, it follows from (2.10) and (2.11) that
$$
[y^0_k,x^0_j] = 0,\quad  \text{if}\quad p_j=p_1-1;\quad [y^0_k,x^1_j] = 0,\quad \text{if}\quad p_j=p_1.
$$
Repeating the argument we find	$[y^0_k,x^s_j]$ for all indices $0\le s\le p_j$. Using the last
equalities and taking $h = 1,2,\ldots, p_k$, we find subsequently
$$
\left[y_{k}^{1}, x_{j}^{s}\right]=0, \ldots,\left[y_{k}^{p_{k}}, x_{j}^{s}\right]=0 \quad \text { for all } 0 \leqslant s \leqslant p_{j}.
$$
The same arguments can be applied in the case $\mu_k =\mu_j$. Comparing the coefficients of the powers
$(\lambda - \mu_j)^\nu$ in (2.10) it is found that, for $h = 0,1,\ldots,p_k$,
$$
-\left[y_{j}^{h}, x_{j}^{s}\right]=\delta_{h, p_{j}-s}
$$
and relations (2.7) follow.
\end{proof}

Let a canonical system (2.1) correspond to a real normal eigenvalue $\mu$. Denote by $S^0_\mu$ the span of elements
\begin{equation}\label{eq:212}
y_{k}^{0}, y_{k}^{1}, \ldots, y_{k}^{\beta_{k}}, \quad k=1, \ldots, N, \quad \beta_{k}=\left[\left(p_{k}-1\right) / 2\right]
\end{equation}
(if $p_k = 0$, we assume that $\beta_k = -1$ and the element $y^0_k$ does not belong to $\mathcal{S}^0_\mu$. Let us fix an index $k$, $1\le k\le N$. If the number $p_k + 1$ is even we set $S_{\mu k}:=\mathcal{S}^0_\mu$. If $p_k + 1$ is odd we denote by $\mathcal{S}_{\mu_k}$ the span of elements (2.12) combined with the elements $y^{\alpha_j}_j$, $\alpha_j= [p_j/2]$, where index $j$ runs through all the values such that $p_j = p_k$. Similary, by replacing chains
(2.1) with adjoint chains (2.3) we construct subspaces $(\mathcal{S}^0_\mu)^*$ and $\mathcal{S}^*_{\mu_k}$. We emphasize that, according to our agreement about the enumeration of eigenvalues, the subspaces $\mathcal{S}_{\mu_k}$ are generally different although $\mu_k =\mu$.

\begin{proposition}\label{pr:26}
For all nonzero real normal eigenvalues p the following equalities hold
$$
\mathcal{S}_{\mu}^{0}=\left(\mathcal{S}_{\mu}^{0}\right)^{*}, \quad \mathcal{S}_{\mu_{k}}=\mathcal{S}_{\mu_{k}}^{*} \quad \text { for all } 1 \leqslant k \leqslant N
$$
\end{proposition}

\begin{proof}
Suppose that $y_{k}^{h} \in \mathcal{S}_{\mu_{k}}$ and$x_{k}^{h} \notin \mathcal{S}_{\mu_{k}}$. It follows from Proposition	2.2 that
$$
x_{k}^{0}, x_{k}^{1}, \ldots, x_{k}^{\alpha_{k}}, \quad k=1, \ldots, N, \quad \alpha_{k}=\left[p_{k} / 2\right],
$$
are chains of	EAE	of the pencil $A(\lambda)$ as	well as of $A^*(\lambda)$. Since (2.1) is a canonical system, we have the representation
\begin{equation}\label{eq:213}
x_{k}^{h}=\sum_{j=1}^{N} \sum_{s=0}^{h} c_{j, s} y_{j}^{s}, \quad \text { if } 0 \leqslant h \leqslant \alpha_{j}=\left[p_{j} / 2\right].
\end{equation}
We have assumed that $x_{k}^{h} \notin \mathcal{S}_{\mu_{k}}$, therefore, at least one of the numbers $C_{j,s}$ in (2.12) is not equal to zero for $s>\beta_{j}=\left[\left(p_{j}-1\right) / 2\right]$, $p_{j}<p_{k}$. In this case, however, $x_{j}^{p_{j}-s} \in \mathcal{S}_{\mu_{j}}^{*}$, i.e. $p_{j}-s \leqslant\left[p_{j} / 2\right]$. Applying Proposition 2.3 with respect to the pencil$A^*(\lambda)$ we find
$$
\left[x_{k}^{h}, x_{j}^{p_{j}-s}\right]=0.
$$
On the other hand it follows from Proposition 2.5 and representation (2.13) that
$$
\left[x_{k}^{h}, x_{j}^{p_{j}-s}\right]=-c_{j, s}
$$
Hence, the assumption $x_{k}^{h} \notin \mathcal{S}_{\mu_{k}}$ is not valid. The equality $\mathcal{S}^{0}=\left(\mathcal{S}^{0}\right)^{*}$ is proved in a similar way.
\end{proof}

\begin{proposition}\label{pr:27}
A canonical system (2.1) corresponding to a reed normal eigenvalue p of the pencil A( A) can be chosen in such a way that
\begin{equation}\label{eq:214}
\left[y_{j}^{\alpha_{j}}, y_{l}^{\alpha_{l}}\right]=\varepsilon_{j} \delta_{j, l}, \quad \alpha_{j}=\left[p_{j} / 2\right], \quad \varepsilon_{j}=\left\{\begin{array}{cl}{0} & {\text { if } p_{j}+1 \text { is even }} \\ {\pm 1} & {\text { if } p_{j}+1 \text { is odd }}\end{array}\right.
\end{equation}
for all indices $1 \le j$, $l \le N$.
\end{proposition}

\begin{proof} Fix an index $k$ such that $p_k + 1$ is odd. Assume that there are $q$ chains of the length $p_k + 1$, i.e. $p_j = p_k$ for $j = k,k + 1,\ldots, k + q - 1$. According to the definition of $\mathcal{S}_{\mu_{k}}$ we have $\operatorname{dim} \mathcal{S}_{\mu_{k}} \ominus \mathcal{S}_{\mu}^{0}=q$. Let $P_k$ be the orthoprojector onto the subspace $\mathcal{S}_{\mu_{k}}$. It follows from the biorthogonality relations (2.7) that the self-adjoint operator $P_kWP_k$ has exactly q nonzero eigenvalues which correspond to an orthogonal basis $\left\{\varphi_{s}\right\}_{1}^{q}$.	 We	can	replace,	 if
necessary, chains (2.1) corresponding to indices $l= k,k + 1,\ldots, k + q - l$,by their linear combinations find obtain a new canonical system such that the system $\left\{\varphi_{l}\right\}_{1}^{q}$ coincides with	 $\left\{y_{s}^{\alpha_{k}}\right\}_{k}^{k+q-1}$. Then,	after a proper norming, the relations (2.14) hold for all indices $l,j = k, k +1,\ldots,k + q - 1$. We can repeat the same arguments for any other index $r$ such that $\mathcal{S}_{\mu_{r}} \neq \mathcal{S}_{\mu_{k}}$. Taking into account that the subspaces $\mathcal{S}_{\mu_{r}}$ and $\mathcal{S}_{\mu_{k}}$ are $W$-orthogonal (Proposition 2.3), we obtain relations (2.14) for all indices such that $1\le j$, $l\le N$.
\end{proof}

\begin{proposition}\label{pr:28} Let a canonical system (2.1) correspond to a real normal eigenvalue $\mu$ and satisfy relations (2.14). Then for all indices $j$ such that $p_j = 2\alpha_j$ the elements $x^{\alpha_j}_j$ of the adjoint system (2.3) have the representation
\begin{equation}\label{eq:215}
x_{j}^{\alpha_{j}}=-\varepsilon_{j} y_{j}^{\alpha_{j}}+y, \quad \varepsilon_{j}=\pm 1, \quad \text { where } y \in \mathcal{S}_{\mu}^{0}.
\end{equation}
In other words: there exists a canonical system (2.1) such that for Jordan chains of odd length the middle elements $x^{\alpha_j}_j$ of its adjoint system have representation (2.15).
\end{proposition}

\begin{proof} As $x_{j}^{\alpha_{j}} \in \mathcal{S}_{\mu_{j}}^{*}=\mathcal{S}_{\mu_{j}}$, we have
$$
x_{j}^{\alpha_{j}}=\sum_{p_{i}=p_{j}} c_{l} y_{l}^{\alpha_{k}}+y, \quad \text { where } y \in \mathcal{S}_{\mu}^{0}.
$$
Now, if canonical system (2.1) satisfies relations (2.14) then $c_{l}=-\varepsilon_{j} \delta_{j, l}$, and relation (2.15)  follow.
\end{proof}

A canonical system (2.1) which satisfies relations (2.14) or (2.15) is said to be \emph{regular}. The numbers $\varepsilon_j$ in (2.15) are said to be \emph{sign characteristics}.  We note that for linear self-adjoint pencils the sign characteristics are determined in a different way, namely, $\varepsilon_j = \pm 1$ for Jordan chains of any length (see \cite[ Ch.3]{11::GLR}, and \cite[Lemma 2]{11::KS}). Simple examples show that for dissipative pencils the definite sign characteristics can not be well defined for Jordan chains of even length. In this situation it is convenient to assume that the \emph{sign characteristics} $\varepsilon_j= 0$ \emph{for all chains of even length} $p_j + 1$. It is supposed that this agreement holds through the rest of the paper.

Let (2.1) be a regular canonical system corresponding to a normal reed eigenvalue $\mu$. Denote by $\mathcal{L}_{\mu}^{+}\left(\mathcal{L}_{\mu}^{-}\right)$ the span of elements (2.12) combined with $y_{j}^{\alpha_{j}}$ satisfying relations (2.14) with $\varepsilon_j = +1$ ($\varepsilon_j = -1$). Then according to the definition of the sign characteristics we have
\begin{equation}\label{eq:216}
\operatorname{dim} \mathcal{L}_{\mu}^{+}=\sum_{k=1}^{N}\left(\varepsilon_{k}^{+}+\left[\left(p_{k}-1\right) / 2\right]\right), \quad \text { where } \quad \varepsilon_{k}^{+}=\max \left(0, \varepsilon_{k}\right).
\end{equation}

\begin{proposition}\label{pr:29}
Let $\mu$ be a real normal eigenvalue of the pencil $A(\lambda) = T - \lambda W$. Then $\mathcal{L}^+_\mu$   is a maximal $W$-non-positive subspace in the root subspace $\mathcal{L}_\mu$.
\end{proposition}

\begin{proof} It follows from Propositions 2.2 and 2.7 that $\mathcal{L}^+_\mu$ is a $W$-nonnegative subspace. Assume that $\mathcal{L}^+_\mu \subset \mathcal{L}' \subset \mathcal{L}_\mu$, where $\mathcal{L}$ is also $W$-nonnegative subspace, and there exists an element $y \in \mathcal{L}'$ such that $y\not\in\mathcal{L}^+_\mu$. Obviously, $y\not\in\mathcal{L}^-_\mu$, as the assumptions $y\in\mathcal{L}^-_\mu$, $y\not\in\mathcal{L}^+_\mu$ imply $[y, y] < 0$. Therefore, $y\not\in\mathcal{L}^+_\mu\cup \mathcal{L}^-_\mu$. Now, using (2.7) we can find an element $y^k_k\in\mathcal{S}^0_\mu$ such that $[y^h_k,y] =\gamma\ne 0$. Denote $z=a y_{k}^{h}+\gamma y$. Then $[z, z]=|\gamma|^{2}(a+[y, y]) \rightarrow-\infty$ if $a\to-\infty$. On the other hand $[z,z]\ge0$, as $z \in \mathcal{L}^{\prime}$ and $\mathcal{L}^{\prime}$ is by assumption $W$-nonnegative. This contradiction ends the proof.
\end{proof}

Denote by $\mathcal{L}$ the minimal subspace containing the root subspaces $\mathcal{L}_\mu(A)$ corresponding to all the eigenvalues $\mu \in \mathbb{C}^{+}$ and all the root subspaces $\mathcal{L}_\mu(A)$ corresponding to normal real eigenvalues. Analogously, let $\mathcal{L}^+$ the minimal subspace containing $\mathcal{L}_\mu$ for all $\mu\in\mathbb{C}^+\cap\sigma_p(A)$ and all the subspaces $\mathcal{L}^+_\mu$ corresponding to the normal real eigenvalues. For a self-adjoint operator $C$ we introduce the (well-known) notations
$$
\pi(C)=\operatorname{rank} C^{+}, \quad \text { where } C^{+}=(|C|+C) / 2, \quad \nu(C)=\pi(-C)
$$
Further, we use the following fundamental result.

\textbf{Theorem on a maximal nonnegative invariant subspace.} \textit{Suppose $W$ generates a Pontrjagin space, i.e. $W$ is boundedly invertible and $\nu(W) <\i$nfty. If $A = W^{-1} T$ and $\rho(A)\cap\mathbb{C}^+\ne\varnothing$  then there exists a maximal $A$-invariant $W$-nonnegative subspace $H^+\subset H$, $\dim H^+= \nu(W)$, such that the spectrum of the restriction $A /_{H}+$ lie in $\overline{\mathbb{C}}^{+}$, and in $\mathbb{C}^{+}$ coincides with the spectrum of $A$.}

\begin{proof} In the case $T = T^*$ this is a well-known Pontrjagin theorem \cite{11::P}. For a maximal $W$-dissipative operator $A$ in Pontrjagin space the theorem was proved by Krein and Langer \cite{11::KL}, and by Azizov \cite{11::A} (see \cite{11::AI} and references therein).
\end{proof}

\begin{theorem}\label{thm:210}
The subspace $\mathcal{L}^+$ defined above is a maximal $W$-nonnegative subspace in $\mathcal{L}$. If $W$ generates a Pontrjagin space and all the real eigenvalues of the pencil $A(\lambda)$ are normal then $\mathcal{L}^+$ is a maximal $W$-nonnegative subspace in the whole space $H$.
\end{theorem}

\begin{proof} It follows from Proposition 2.4 and the definition that $\mathcal{L}^+$ is a $W$-nonnegative subspace. As $\mathcal{L}^+_\mu$ is a maximal $W$-nonnegative subspace in $\mathcal{L}_\mu$ for any $\mu \in \sigma_{d}(A) \cap \mathbb{R}$ (Proposition 2.9), we have that $\mathcal{L}^+$ possesses the same property in $\mathcal{L}$.

Now, let $W$ generate a Pontrjagin space and all the real eigenvalues of the pencil $A(\lambda)$ are normal. According to the generalized Pontrjagin theorem there exists a maximal $W$-nonnegative subspace $H^+$ in $H$, $\dim H^+ =\nu(W)$ , such that $\mathcal{L}_{+}^{0} \subset H^{+} \subset \mathcal{L}$, where $\mathcal{L}_{+}^{0}$ is defined in Proposition 2.1. As the subspace $H^{+} \cap \mathcal{L}_{\mu}$ is $W$-nonnegative in $\mathcal{L}_\mu$ and $\mathcal{L}^+_\mu$ is a maximal nonnegative subspace in $\mathcal{L}_\mu$ (Proposition 2.9), we have: $\operatorname{dim}\left(H^{+} \cap \mathcal{L}_{\mu}\right) \leqslant \dim\mathcal{L}_{\mu}^{+}$ (see, for example, \cite[Ch.I, \S4]{11::AI} ). Then it follows that
$$
\operatorname{dim} H^{+}=\operatorname{dim} \mathcal{L}_{+}^{0}+\sum_{\mu \in \mathbb{R} \cap \sigma_{d}} \operatorname{dim}\left(H^{+} \cap \mathcal{L}_{\mu}\right) \leqslant \operatorname{dim} \mathcal{L}^{+}, \quad \sigma_{d}:=\sigma_{d}(A).
$$
On the other hand, it is known (\cite[ Ch.I, \S4]{11::AI}) that $\operatorname{dim} \mathcal{L}^{+} \leqslant \nu(W)=\operatorname{dim} H^{+}$. Hence, $\operatorname{dim} \mathcal{L}^{+}=\operatorname{dim} H^{+}$ and from this it follows that $\mathcal{L}^+$ is a maximal $W$-nonnegative subspace in the whole $H$.	
\end{proof}

\begin{corollary}\label{cor:211}
Let $W$ be boundedly invertible, $\nu(W) <\infty$, and a11 the real eigenvalues of $A(\lambda)$ be normal. Then the following formula is valid
\begin{equation}\label{eq:217}
\kappa(A)+\sum_{\mu_{k} \in \mathbb{R} \cap \sigma_{d}}\left(\varepsilon_{k}^{+}+\left[\left(p_{k}-1\right) / 2\right]\right)=\nu(W), \quad \varepsilon_{k}^{+}=\max \left(0, \varepsilon_{k}\right)
\end{equation}
Here $\kappa(A)$ is the total algebraic multiplicity of all eigenvalues in $\mathbb{C}^+$ and $\varepsilon_k(p_k + 1)$ are the sign characteristics (the lengths) of Jordan chains of regular canonical systems corresponding to real normal eigenvalues $\mu_k$.
\end{corollary}
\begin{proof}
It follows from formula (2.16) and Theorem 2.10.	
\end{proof}
\begin{rem}\label{rem:212}
Formula (2.17) is not applicable if the pencil $A(\lambda)$ has real eigenvalues which are embedded into the essential spectrum. In this case we do not know how to determine the sign characteristics and how to realize the explicit construction of a maximal $W$-nonnegative subspace in the the root subspace $\mathcal{L}_\mu$. However, the following inequality is always valid (cf. \cite[Ch.2. Theorem 2.26]{11::AI})
\begin{equation}\label{eq:218}
\kappa(A)+\sum_{\mu_{k} \in \mathbb{R} \cap \sigma_{p}}\left[\left(p_{k}-1\right) / 2\right] \leqslant \nu(W), \quad \sigma_{p}:=\sigma_{p}(A)
\end{equation}
This inequality is much more simple and follows directly from Propositions 2.1, 2.3 and 2.4. It expresses the fact that the linear span of all root subspaces $\mathcal{L}_\mu$ corresponding to $\mu \in \sigma_{p}(A) \cap \mathbb{C}^{+}$  and all the truncated root subspaces $\mathcal{S}_{\mu}^{0}$ corresponding to $\mu \in \sigma_{p}(A) \cap \mathbb{R}$ forms a $W$-nonnegative subspace (not necessarily a maximal one). Indeed, using (2.17) we can improve (2.18) and write the following inequality
\begin{equation}\label{eq:219}
\kappa(A)+\sum_{\mu_{k} \in \mathbb{R} \cap \sigma_{p}}\left(\varepsilon_{k}^{+}+\left[\left(p_{k}-1\right) / 2\right]\right) \leqslant \nu(W), \quad \sigma_{p}:=\sigma_{p}(A)
\end{equation}
where $\varepsilon_{k}^{+}=\max \left(0, \varepsilon_{k}\right)$ if $\mu_{k} \in \sigma_{d}$ and $\varepsilon_{k}^{+}=0$ if $\mu_{k} \in \sigma_{p} \backslash \sigma_{d}$.
\end{rem}

\subsection{Quadratic dissipative pencils and the instability index formula}
In this section we study a quadratic operator pencil of the form
\begin{equation}\label{eq:31}
\begin{array}{ll}{\text { (3.1) }} & {A(\lambda)=\lambda^{2} F+(D+i G) \lambda+T}\end{array}
\end{equation}
Further it is always assumed that the coefficients in (3.1) are operators in Hilbert space $H$ satisfying the following conditions:
\begin{enumerate}
\item[i)]	$F$ is	a self-adjoint bounded and boundedly	 invertible	 operator;
\item[ii)] $T$ is	defined on the domain  $V(T)$, $T = T^*$	and	 $T$ is	 boundedly	invertible;
\item[iii)] $D$ and $G$ are symmetric $T$-bounded operators (i.e. $D$ and $G$ are symmetric, $\mathcal{D}(D)\subset \mathcal{D}(T)$ and $\mathcal{D}(G)\subset \mathcal{D}(T)$. Moreover, $D\ge 0$.
\end{enumerate}

These assumptions imply that $A(\lambda)$ is a quadratic dissipative pencil with respect to the imaginary axis in the following sense (see \cite{11::S1})
$$
\operatorname{Im}(\zeta A(i \zeta) x, x)=\zeta^{2}(D x, x) \geqslant 0 \quad \text { for all } x \in \mathcal{D}(T) \text { and } \zeta \in \mathbb{R}
$$
One may expect that the quadratic dissipative pencil (3.1) can be transformed into a linear dissipative pencil. Indeed, such a linearization will be realized below. However, working with unbounded pencils we come to some new problems which do not arise when considering pencils with bounded coefficients. In particular, the spectrum of a linearization may not coincide with the spectrum of the original pencil.

According to our assumptions $A(\lambda)$ is well defined for each $\lambda \in \mathbb{C}$ on the domain $\mathcal{D}(T)$. Hence, the first natural definition of the resolvent set $\rho(A)$ is the following: $\zeta \in \rho(A)$ if $A(\zeta)$ with the domain $\mathcal{D}(T)$ has a bounded inverse. To give another definition, we consider the scale of Hilbert spaces $H_{\theta}, \theta \in \mathbb{R}\left(H_{0}=H\right)$ generated by the self-adjoint operator $S^{2}:=|T|:=\left(T^{2}\right)^{1 / 2}$. Namely, if $\theta> 0$ we set $H_{\theta}=\left\{x | x \in \mathcal{D}\left(S^{\theta}\right)\right\}$ with the norm $\|x\|_{\theta}=\left\|S^{\theta} x\right\|$. If $\theta< 0$, the space $H_\theta$ is defined as the closure of $H$ with respect to the norm $\|x\|_{\theta}=\left\|S^{\theta} x\right\|$.

Let us associate the pencil
$$
\hat{A}(\lambda)=\lambda^{2} \hat{F}+\lambda(\hat{D}+i \hat{G})+J
$$
with the pencil $A(\lambda)$. Here
$$
\hat{F}=S^{-1} F S^{-1}, \hat{D}=S^{-1} D S^{-1}, \hat{G}=S^{-1} G S^{-1}, J=T^{-1}|T|
$$
Obviously $\hat F$ and $J$ are bounded. From the next Proposition it follows that $\hat D$ and $\hat G$ are also bounded in $H$.
\begin{proposition}
Let $S$ be an uniformly positive self-adjoint operator and $B$ be a symmetric operator such that $\mathcal{D}(B) \supset\mathcal{ D}(S^2)$. Then the operator $S^{\theta-2} B S^{-\theta}$ defined on the domain $\mathcal{D}(S^{\theta-2})$ is bounded in $H$ for all $0\le \theta\le 2$. Equivalently, $B$ is bounded as an operator acting from $H_\theta$ into $H_{\theta-2}$.
\end{proposition}
\begin{proof}
As $B$ is closable, the assumption $\mathcal{D}(B) \supset \mathcal{D}(S^2)$ implies that $B: H_{2} \rightarrow H$ is a
bounded operator (this follows immediately from the closed graph theorem). Hence, the adjoint operator $B^*: H\rightarrow H_{-2}$ is also bounded. As $B^{*} \supset B$, we have that $B: H\rightarrow H_{-2}$ is bounded. Now, applying the interpolation theorem (see \cite[Ch.l]{11::LM}, for example) we find that B$B: H_{\theta} \rightarrow H_{\theta-2}$ is bounded for all $0\le \theta\le 2$.
\end{proof}

Let $\sigma(\hat{A})$ be the spectrum of the pencil $\hat{A}(\lambda)$ with bounded operator coefficients in the space $H$. It is easily seen that $\sigma(\hat{A})$ coincides with the spectrum of $A(\lambda)$ considered as the operator function in the space $H_{-1}$ on the domain $\mathcal{D}(A) = H_1$. Both our definitions of the spectra are better understood (especially for the specialists working with partial differential operators) if we say the following: $\sigma(A)$ is the spectrum of the pencil $A(\lambda)$ considered in the ``classical'' space $H$ while $\sigma(\hat{A})$ is its spectrum in the generalized space $H_{-l}$.

Generally, $\sigma(A) \neq \sigma(\hat{A})$. What is the connection between the classical and the generalized spectra? Some light is cast on this problem by the next propositions. It will be convenient to define in the complex plane the open set $\rho_{m}(A):=\rho(A) \cup \sigma_{d}(A)$. The set $\rho_{m}(\hat{A})$ is defined analogously. In the other words $\rho_m(A)$ and $\rho_m(\hat{A})$ are the domains where the operator functions $A^{-1}(\lambda)$ is finite meromorphic in the spaces $H$ and $H_{-1}$, respectively.

\begin{proposition}\label{prop:32} In the domain pm(A) fl pm(.4) all the eigenvalues and Jordan chains of .4(A) in the spaces H and H-\ coincide.
\end{proposition}

\subsection{Applications}

In this section we shall apply the obtained abstract results to concrete problems considered in Section 1.
\begin{theorem}
Formula (3.11) or its simplifications (3.12) or (3.13) are valid for operator pencil (1.7) associated with the problem of small oscillations of ideal incompressible fluid in a pipe of finite length if the condition $KerT = \{0\}$ is fulfilled $(T := A+C)$. For a pipe of infinite length the assertion of Theorem 3.7 is valid if $g(x)$ is such a function that $KerT = \{0\}$ and $\nu(T) < \infty$.
\end{theorem}
\begin{proof}
The conditions \textit{i)-ii)} and \textit{iv)} of Section 1 imply conditions \textit{i)-iii)} of Section 3 if it is assumed in addition that $KerT = \{0\}$. Moreover, for a pipe of finite length the assumptions of Corollary 3.8 are fulfilled. For a pipe of infinite length the operators $G$ and $I$ are not $T$-compact and we must use Theorem 3.7. In the last case we can not guarantee the absence of pure imaginary eigenvalues belonging to the non-discrete spectrum.
\end{proof}

If $KerT \ne \{0\}$ then $\lambda= 0$ is an eigenvalue of pencil (3.1). In this case the analogue of formula (3.11) can also be obtained. For this purpose one has to modify the results of Section 2 for the case $KerW \ne \{0\}$. Technically this is not a trivial work. However, the estimates for the number $\kappa(\hat{A})$ can be obtained easily if $KerT \ne \{0\}$.

\begin{theorem}
Suppose that a pencil $A(\lambda)$ is defined by (3.1) and its operator coefficients satisfy the assumptions \textit{i)-iii)} of Section 3 with the possible exception that the operators $F$ and $T$ are not necessarily boundedly invertible. Suppose that there exists a point $\mu$, $Re \mu > 0$ such that $\hat{A}(\mu)$ is boundedly invertible. Then
\begin{equation}
\kappa(\hat{A}) \le \nu(F) + \nu(T).
\end{equation}
\end{theorem}
\begin{proof}
Let us consider the pencil
$$
A_{\tau}(\lambda) = \lambda^2(F+\tau I) + (D  + iG)\lambda + T +\tau I, \quad \tau > 0.
$$
Obviously, $\nu(T+\tau I) = \nu(T)$, $\nu(F+\tau I) = \nu(F)$, if $\tau \in (0, \tau_0)$ and $\tau_0$ is sufficiently small. By virtue of Theorem 3.7 we have
\begin{equation}
\kappa(\hat{A_{\tau}}) \le \nu(F) + \nu(T) \quad \text{for all } 0 < \tau < \tau_0.
\end{equation}
Repeating the arguments from the proof of the Theorem 3.7 and taking into account that $\mu \in \rho(\hat{A})$ for some $\mu$ with $Re \mu > 0$ we obtain that the spectrum of $\hat{A}(\lambda) := \hat{A_0}(\lambda)$ in the open right half plane consists only of normal eigenvalues. These eigenvalues continuously depend on $\tau$ (see \cite{11::Ka}, Ch. 7). Then (4.2) implies (4.1).
\end{proof}

The results of Sections 2 and 3 can also be applied to self-adjoint pencils. Lancaster and Shkalikov \cite{11::LS} considered an operator pencil $L(\lambda)$ defined by (1.5) with $C = 0$, $D_{\alpha} = \alpha A$ and obtained the following estimate
\begin{equation}
\eta/2 \le \min_{\substack{k \in \mathbb{R}}} \pi(L(k)), \quad \pi(L) := \nu(-L),
\end{equation}
where $\eta$ is the number of non-real eigenvalues of the pencil $L(\lambda)$ counting with algebraic multiplicities. Using an analytic approach Shkalikov and Griniv proved a sharper estimate for the case $C = 0$ and reproved (4.3) for $C \ne 0$ (if $C$ is an $A$-compact operator). Here we refine the corresponding results from \cite{11::LS} and \cite{11::SG}.
\begin{theorem}
Let
$$
L(\lambda) = \lambda^2F + \lambda D + T,
$$
where $T = T^*$ and $F$,$D$ are symmetric and $T$-bounded operators. Let $S^2 = |T| + I$ and the scale of Hilbert spaces $H_\theta$ be generated by the operator $S \gg 0$. Suppose that there exist real points $a$ and $b$ belonging to $\rho(\hat{L})$ such that
$$
\pi(L(a)) < \infty, \quad \nu(L(b)) < \infty.
$$
Then the non-real spectrum of $L(\lambda)$ in the space $H_{-1}$ consists of finitely many, say $\eta$, non-real eigenvalues, and the following estimate is valid
\begin{equation}
\eta/2 \le \pi(L(a)) + \nu(L(b)) - \delta^+(L),
\end{equation}
where $\delta^+(L)$ is the number of real eigenvalues $\mu_k$ of $L(\lambda)$ counting with multiplicities such that
$$
(b-a)\left(\frac{\mu_k-a}{b-\mu_k}\right)(L^\prime(\mu_k)y, y) > 0 \quad \text{for all} \quad y \in KerL(\mu_k).
$$
\end{theorem}
\begin{proof}
We use the same idea as in \cite{11::LS} where estimate (4.4) was obtained in a slightly different situation not taking into account the number $\delta^+(L)$. It was shown in Section 3 that the Spectrum of $L(\lambda)$ in the space $H_{-1}$ coincides with the spectrum of $\hat{L}(\lambda) = S^{-1}L(\lambda)S^{-1}$ in the space $H$. The pencil $\hat{L}(\lambda)$ has the bounded operator coefficients $\hat{F}$, $\hat{D}$, $\hat{T}$. After the substitution $\lambda = (b\xi+a)(\xi+1)^{-1}$ we obtain the quadratic pencil
$$
\tilde L(\xi) := (\xi+1)^2 \hat{L}(\lambda(\xi)) = \xi^2\tilde{F} + \xi \tilde{D} + \tilde {T}, \quad \tilde{F} = \hat {L}(b), \quad \tilde T = \hat L(a).
$$
Let us consider the linearization of $\tilde {L}(\xi)$
$$
L(\xi) = -\begin{pmatrix}
\tilde{D} & \tilde{T}\\
\tilde{T} & 0 \\
\end{pmatrix}
- \xi \begin{pmatrix}
\tilde{F} & 0\\
0 & -\tilde{T} \\
\end{pmatrix}.
$$
Suppose that $\xi_k$ is a simple (or semi-simple) real eigenvalue of $\tilde{L}(\xi)$ with a corresponding eigenvector $y_k$. then the sign characteristic $\varepsilon_{k}$ (see Section 2) is defined as follows
$$
\begin{aligned}
\varepsilon_{k} &= \left( \begin{pmatrix}
\tilde{F} & 0\\
0 & -\tilde{T} \\
\end{pmatrix}
\begin{pmatrix}
\xi_k y_k\\
y_k\\
\end{pmatrix},
\begin{pmatrix}
\xi_k y_k\\
y_k\\
\end{pmatrix}
\right)_{H \times H} = \xi(\tilde{L}^\prime(\xi_k)y_k, y_k)	\\
&= \xi_k \lambda^\prime(\xi_k)(\hat{L}^\prime(\lambda_k)y_k, y_k) = (b-a)(\mu_k-a)(b-\mu_k)^{-1}(\hat{L}^\prime(\lambda_k)y_k, y_k).
\end{aligned}
$$
Now apply Corollary 2.11.
\end{proof}

We note that the estimate (4.4) is also new for matrix pencils.

\renewcommand{\refname}{{\large\rm Bibliography for Section~11}}

\newpage
\section{Factorization of elliptic pencils and the Mandelstam hypo\-thesis}

\subsection* {Introduction}

This section is a modified and  extended version of section 6, where the main attention was paid to the finite dimensional case. Here we   deal with pencils which present the abstract models  of concrete
essentially infinite dimensional problems.

Some problems of mathematical physics (one of them will be discussed
below) can be written abstractly in the form

\begin{equation} \label{0.1}
\mathcal{A}(u)= - F\frac{d^2u}{dy}+iG\frac{du}{dy}+(H-\omega^2R)u=0.
\end{equation}
Here $F, G, H,$ and $R$ are symmetric operators on a suitable Hilbert space
$\mathcal{H}$
satisfying certain additional conditions which ensure the elliptic
nature of this equation, and $\omega$ is a physical parameter (frequency)
which appears after the separation of the time variable.

Physical meaning have solutions of equation (\ref{0.1})
which are bounded as\break $y\to\infty$ and satisfy the
so-called radiation principle. Different approaches to formulate the
radiation principle have been widely discussed in physical and
mathematical literature (see, for example, {\sc Sveshnikov} \cite{12::Sv},
the books of {\sc Zilbergleit} and {\sc Kopilevich} \cite{12::ZK},
{\sc Vorovich} and {\sc Babeshko} \cite{12::VB}).
The formulation of the radiation principles is based on the
preliminary spectral analysis of the pencil
\begin{equation} \label{0.2}
T_{\omega}(\lambda) =\la^2F+\la G+H-\omega^2R.
\end{equation}
We say $\{\la_k,v_k\}$ with $v_k \ne 0$ is an eigenpair of the pencil $T_{\omega}(\lambda)$ if
$T_{\omega}(\la_k)v_k=0.$
Any eigenpair $\{\la_k, v_k\}$ generates the solution
\begin{equation}\label{0.3}
u_k=e^{-i\la_k y}v_k
\end{equation}
of equation (\ref{0.1}). Those solutions which correspond to the real
eigenvalues $\la_k$ are of particular interest, they are called
propagating waves. Among propagating waves there are the outgoing and
incoming ones. It was understood after the author's discussions with
physicists,
that the Mandelstam hypothesis can be formulated as
follows (see \cite{12::BS}, \cite{12::ZK}, although the problem is not
clearly formulated there):
{\em given an element $x\in\mathcal{ H}$ there is a unique solution $u(y)$ of
equation (\ref{0.1}) such that $u(0)=x$, and as $y\to\infty$ the solution
$u(y)$ asymptotically coincides with a linear combination of outgoing
waves.}

This problem is also related to those settled by {\sc Reyleigh} on the
wave diffraction on  a periodic surface. Some of them are treated in the book
of {\sc Wilcox} \cite{12::W}. This connection, however, is not easily seen,
and its demonstration is left for a future occasion.

Our first aim is to define an abstract model of strongly elliptic equations
in wave-guide domains whose symbols are quadratic selfajoint pencils.
The main goal
is to prove the factorization theorems for these pencils and investigate
the properties of a right divisor. The results obtained  enable us, in
particular, to approve the Mandelstam hypothesis.

Our starting point was a celebrated paper of {\sc Krein} and {\sc Langer}
\cite{12::KL} which deals with pencils of the form
$$
L(\la) =I+\la B+\la^2C.
$$
Here $I$ is the identity operator, $B$ is bounded and selfadjoint,
while $C$ is  positive and compact. The fundamental theorem of \cite{12::KL}
yields the factorization
$$
L(\la)=(I-\la Z_1)(I-\la Z).
$$
Among possible divisors there is an operator $Z$ whose
spectrum $\sigma(Z)$ lies in the closed upper (or lower) half plane and
coincides with the spectrum of $L(\la)$ in the open half plane.
A further analysis of an operator $Z$ occuring in this factorization
was given in the papers of {\sc Kostyuchenko} and {\sc Orazov} \cite{12::KO1} and
{\sc Kostyuchenko} and {\sc Shkalikov} \cite{12::KS}. However, while
attempting to apply the method of {\sc Krein} and {\sc Langer} to
attack the factorization problem for qudratic pencils with
{\em unbounded coefficients}, one faces new serious
obstacles. Moreover, {\em a further analysis of divisors} has to be carried
out after the factorization is already proved. In particular, to prove
the Mandelstam hypothesis we have to show that {\em among possible
factorizations}
$$
T_\omega(\lambda)=(\la-Z_1)F(\la-Z),
$$
{\em there is the only operator $Z$ which generates a} $C_0$ {\em (or holomorphic)
semigroup in an appropriate Hilbert space}.

The plan of this paper is the following. In Section 1 we define strongly
elliptic pencils as relatively compact pertubations of uniformly positive
ones. For pencils with discrete spectrum our definition is
equivalent to the asymptotic inequality
$$
T_\omega(\lambda)\ge\varepsilon(\la^2+H), \qquad \mbox {for}\
\la \in\mathbb{R}, \ |\la|>r_0,
$$
provided $r_0$ is sufficiently large. This assumption can be easily checked
for concrete elliptic systems, since it is equivalent to the G$\mathop{\rm o}\limits^
{\mbox{\tiny\rm o}}$rding
inequality (this is shown in Section 3).
Following the paper \cite{12::S1} we define the "classical" and the
"ge\-ne\-ralized" spectra of $T_{\omega}(\lambda)$. We show that the classical and
ge\-ne\-ralized
spectra of a strongly elliptic pencil coincide in the union of a ball
centered in the origin and a sufficiently small double sector containing
the real axis. Moreover, in this domain the spectrum consists
of finitely many normal eigenvalues.
For large values of $|\la|$ inside a double sector we prove the resolvent
estimates which play an important role in the sequel. They look similar
to the classical a priori estimates for regular elliptic boundary value
problems obtained by {\sc Agmon}, {\sc Douglas} and {\sc Nirenberg}
\cite{12::ADN}, \cite{12::AN} and {\sc Agranovich} and {\sc Vishik} \cite{12::AV}.
Nevertheless, estimates obtained in Section 1 are of different nature,
in particular, they can be used for elliptic systems on non-smooth domains.
One can feel the difference while considering the example in Section 6.

In Section 2 we give more details about the real spectrum of $T_{\omega}(\lambda)$.
In particular, we show that the outgoing waves correspond to those
eigenpairs which have the positive sign characteristics
$$
\varepsilon_k=(T'_\omega(\lambda_k)v_k,v_k).
$$

In Section 3 we prove the factorization theorem for positive strongly
elliptic pencils. We could obtain this theorem (although is not easy)
using classical results on the factorization of non-negative operator
functions on the real line (see the exposition of this theory in the books
of {\sc Foias} and {\sc Nagy}  \cite{12::FN} and of {\sc Rosenblum} and
{\sc Rovnjak} \cite{12::RR}). However, we
preferred to give a new approach based on the semigroup theory, as it seems
more natural for the problem in question.
Moreover, we believe that this method can be modified to fit arbitrary
strongly elliptic pencils not positive ones only.

In Section 4 we prove the factorization theorem for strongly elliptic pencils
(not necessarily positive) under an additional assumption (the so-called
Keldysh-Agmon condition). The proof is based on the preliminary analysis
of the half-range completeness and minimality problem for the pencil
$T_\omega(\lambda)$. To solve these problems
we borrow the ideas from the papers \cite{12::KS}
and \cite{12::SS}. In this exposition, however, we get rid of some superfluous
assumptions and presented the material in a different and shorter way.
In particular, in contrast to the cited papers, now we can apply our
results in the case when the operator $H$ is generated by an elliptic
operator (or system) on a non-smooth domain.

The results of Section 3 and 4 are used in Section 5 to approve the Mandelstam
hypothesis. Finally, in Section 6, we demonstrate how the obtained results
can be applied to the elliptic system of differential equations of
elasticity theory.

The second part of the paper (Sections 4--6) is a revised version of
results on elliptic pencils presented by the author in the unpublished
manuscripts \cite{12::S2}, \cite{12::S3}.
 %Using the opportunity the author
%thanks Professors {\sc V.~A.~Kondratiev, Yu.~I.~Kopilevich,
%A.~G.~Kostyuchenko,
%P.~Lancaster} and {\sc A.~S.~Zilbergleit} for valuable discussions.
%I am also indebted to {\sc dr.~R.~O.~Hryniv} who took the job of
%looking through the manuscript and correcting mistakes.

\subsection{Elliptic pencils and their spectrum}

\setcounter{equation}{0}
{\bf Definition of regular elliptic and strongly elliptic pencils}. In what
follows we always assume that the coefficients of equation (\ref{0.1})
or a quadratic pencil $T_{\omega}(\la)$  of the form (\ref{0.2}) are operators
on a separable Hilbert space $\mathcal{H}$ having the following properties (we borrow
the terminology from the book of {\sc Kato} \cite{12::Ka}):

{\em $F$ is a bounded and uniformly positive operator $(0\ll F\ll\infty)$;

$H$ is a selfadjoint uniformly positive operator with domain
$\mathcal{D}(H)\subset \mathcal{H} \ (H=H^*\gg 0)$;

$G$ is a symmetric operator ($G\subset G^*$) with domain
$\mathcal{D}(G)\supset \mathcal{D}(H^{1/2})$;

$R$ is an $H$-compact positive operator (i.e. $R>0,\ \mathcal{D}(R)\supset \mathcal{D}(H)$
and $RH^{-1}$ is compact on $\mathcal{H}$), and the closure of the operator
$\mathcal{H}M R\mathcal{H}M$ has  trivial kernel.}

It is worth noting that for any symmetric $H$-bounded operator $R$ the
closure of $\mathcal{H}M R\mathcal{H}M$ exists and is a bounded operator on $\mathcal{H}$ (see the remark
explaining the boundedness of the operator C defined in (\ref{12::1.6})).

The parameter $\omega$ plays a role in the sequel only in cases when we
appeal to physical considerations. For fixed $\omega$ it will be
convenient to denote $S=H-\omega^2R$ and consider the pencil
\begin{equation} \label{12::1.1}
T(\lambda) = \la^2F+\la G+S
\end{equation}
implying that
\begin{equation} \label{12::1.2}
\left\{
\begin{array}{l}
0\ll F\ll\infty,\quad G\subset G^*, \quad \mathcal{D}(G)\supset\mathcal{D}(H^{1/2}),\\
S=S^* \ \ \mbox{is a relatively compact perturbation of}\quad H=H^*\gg 0.
\end{array}\right.
\end{equation}

We use the scale of Hilbert spaces $\mathcal{H}_{\theta}$ generated
by the "main" operator \break $H$. Namely, for $\theta\ge 0$ the space
$\mathcal{H}_{\theta}$
coincides with $\mathcal{D}(H^{{\theta}/2})$ endowed with the norm $\|x\|_{\theta}=
\|H^{{\theta}/2}x\|$,
while $\mathcal{H}_{-\theta}$  is the dual space to $\mathcal{H}_{\theta}$ with respect to $\mathcal{H}$.
The following fact will be used in the sequel: If $S\gg 0$ then the scale of
Hilbert spaces generated by $S$ coincides with $\mathcal{H}_{\theta}$ for
$0 \le \theta \le 2$. This fact follows from the assumption
$\mathcal{D} (S)=\mathcal{D} (H)$ and the interpolation theorem (see, e.g., \cite{12::LM}, Ch 1).

Further, by writing $T(\la)$ instead of $T_{\omega}(\la)$ we always assume
that $T(\la)$ is of the form (\ref{12::1.1}) with coefficients satisfying
conditions (\ref{12::1.2}).

The definition of a regular elliptic boundary value problem (see \cite{12::AN},
\cite{12::AV}, \cite{12::LM}) is expressed algebraically in terms of principle
symbols of a differential equation and boundary operators (the so-called
ellipticity condition for the equation and the complementing Lopatinskii
condition for boundary operators). Suppose that we consider a regular
elliptic problem in a wave-guide domain $\Omega\times\mathbb{R}$ ($\Omega$ is a
smooth bounded domain in $\mathbb{R}^n$) and write it in abstract form (\ref{0.1})
(homogeneous boundary conditions are included in the domain of the main
operator $H$). It follows from the results of \cite{12::AN} and \cite{12::AV}:
a problem is regular elliptic if and only if $T(\la)$ is invertible for
$\la\in\mathbb{R}$ and $|\la|>r_0$, with $r_0$ large enough and for these
values of $\la$
\begin{equation} \label{12::1.3o}
\|HT^{-1}(\la)\|+|\la|\ \|H^{1/2}T^{-1}(\la)\|+|\la|^2\|T^{-1}(\la)\|\le
const.
\end{equation}

These arguments lead to the following definition (as we agreed the parameter
$\omega$ is omitted).
\begin{definition}
{\em A pencil $T(\la)$ or equation $T\left(i\frac{d}{dy}\right)u(y)=0$ is
said to be regular elliptic if estimate (\ref{12::1.3o}) holds for $\la\in\mathbb{R}$,
$|\la|>r_0$.}
\end{definition}
\medskip

In this paper, however, we deal mostly with equations which are abstract
generalizations of strongly elliptic equations (see, e.g., the book of
Fichera \cite{12::Fi}).
\begin{definition} {\em A pencil $T(\la)$ is said to be uniformly positive
if there exists a number $\varepsilon >0$ such that}
\begin{equation}
T(\la) \ge \varepsilon (\la^2+H) \qquad \mbox{ for all}\ \, \la \in \mathbb{R}. \label {12::1.3}
\end{equation}
\end{definition}

It follows from the definition that $S\gg 0$ if $T(\la)$ is uniformly
positive.
As $\mathcal{D} (S)=\mathcal{D}(H)$, both the operators $SH^{-1}$ and $HS^{-1}$ are
defined on the whole $\mathcal{H}$, and it follows from the definition that they are
closed. Hence, by the closed graph theorem these operators are bounded and
then there exist positive constants $c_0, c_1$ such that
$$
c_0\|Hx\|\le\|Sx\|\le c_1\|Hx\|, \qquad x\in\mathcal{D}(H).
$$
By virtue of the Heinz inequality (see \cite{12::Ka}, Ch.5.4
) we have
$c_0H\le S\le c_1H$.
Therefore, (\ref{12::1.3}) implies also
$$
T(\la) \ge \varepsilon_1(\la^2+S), \qquad \la \in \mathbb{R},
$$
with $\varepsilon_1=\varepsilon/c_1.$ Actually, we have just showed that the operator $H$
in Definition 1.2 can be replaced by any operator $S=S^*\gg 0$ such that
$\mathcal{D}(S)=\mathcal{D}(H)$.

\begin{definition} {\em A pencil $T(\la)$ of the form (\ref{12::1.1}) is said
to be strongly elliptic if there exists an $H$-compact positive operator
$V$ such that $T(\la)+V$ is uniformly positive.}
\end{definition}
\begin{proposition} Let $T(\la)$ be a strongly elliptic pencil. Then there
exist numbers $\varepsilon>0$ and $r_0 >0$ such that
\begin{equation}\label{12::1.4}
T(\la)\ge \varepsilon(\la^2+H) \qquad\mbox{for all}\ \, \la \in \mathbb{R} \ \,\mbox{and}
\ |\la|>r_0.
\end{equation}
\end{proposition}
\begin{proof} By the definition we have
\begin{equation}
T(\la)\ge \varepsilon (\la^2+H)-V, \label{12::1.5}
\end{equation}
where $V$ is an $H$-compact positive operator. Obviously, if $VH^{-1}$
is compact in $\mathcal{H}$ then $V$ is $H$-bounded with zero $H$-bound, i.e.
for any $\varepsilon >0$ there exist $c=c(\varepsilon)$ such that
$$
||V x||\le \varepsilon ||Hx||+c||x||,\quad c=c(\varepsilon),\quad
x\in \mathcal{D}(H).
$$
By virtue of the Heinz inequality we have
$$
V\le \varepsilon H+cI,
$$
where $I$ is the identity operator.
Taking in the last inequality $\varepsilon/2$ instead of $\varepsilon$ we
obtain (\ref{12::1.4}) from (\ref{12::1.5}).
\end{proof}

The inverse assertion of Proposition 1.4, generally, is not true. Examples
can be easily given by considering bounded operators $G$ and $S$
on $\mathcal{H}$. However, we can invert the statement of Proposition 1.4
assuming that $H$ has discrete spectrum or, equivalently, the
identity operator is $H$-compact.

\begin{proposition} Let $H^{-1}$ be compact in $\mathcal{H}$.Then condition
$\ref{12::1.4}$ implies that $T(\la)$ is strongly elliptic.
\end{proposition}
\begin{proof} According to (\ref{12::1.2}) $G$ is $\mathcal{H}P$-bounded operator and
$S-H$ is $H$-compact. Hence, $G$ is $H$-compact and for any $\la\in\mathbb{R}$,
we have
$$
T(\la)=\la^2F+H+K(\la),
$$
where $K(\la)H^{-1}$ is compact. Given $\varepsilon >0$ there exists
$c=c(\la)>0$ such that
$$
|(K(\la)x,x)|\le\varepsilon(Hx,x)+c(x,x).
$$
If $\varepsilon=1/2$ and $c_0$ is the maximum of $c(\la)$ on the interval
$(-r_0,r_0)$ then
$$
T(\la)+c_0I\ge\frac12(\la^2F+H), \qquad \mbox{for}\ \, \la\in(-r_0,r_0),
$$
and together with (\ref{12::1.4}) this implies that $T(\la)+c_0I$
is uniformly positive.
\end{proof}

{\bf Location of the spectrum and the resolvent
estimates}.
In \cite{12::S1} three different approaches
are proposed to define the spectrum of a pencil
with unbounded coefficients. In particular, the "classical" and the
"generalized" spectrum of $T(\la)$ are defined as follows.
We say $\mu$ {\em belongs to the classical spectrum of the pencil $T(\la)$
if $T(\mu)$ is not boundedly invertible in} $\mathcal{H}$. This concept is natural
but not always convenient (see \cite{12::S1}).
To define the generalized spectrum, let us consider
$F$, $G$ and $S$ as the operators acting on the space $\mathcal{H}_{-1}$ with
domain $\mathcal{D} =\mathcal{H}_1$ (recall that $\mathcal{H}_{\theta}=\mathcal{D}(H^{{\theta}/2})$ is the
scale of Hilbert
spaces generated by the operator $H$). Since all these operators are
$H$-bounded and symmetric, they are well defined in $\mathcal{H}_{-1}$
with domain $\mathcal{H}_1$ (see details in \cite{12::S1}). Now we can consider
$T(\la)$ as an ope\-ra\-tor function in the space $\mathcal{H}_{-1}$ defined on the domain
$\mathcal{D}(T)=\mathcal{H}_1$. We say that $\mu$ {\em belongs to the generalized spectrum of
the pencil $T(\la)$ if $T(\mu)$ is not boundedly invertible in} $\mathcal{H}_{-1}$.
The complement of the generalized spectrum is said to be the
generalized resolvent set of $T(\la)$.
It can be easily
checked (see \cite{12::S1}) that $\mu$ belongs to the generalized spectrum
of $T(\la)$ if and only if $\mu$ belongs to the spectrum of the pencil
$$
L(\la)=\la^2A+\la B+C,
$$
with bounded in $\mathcal{H}$ coefficients
\begin{equation}
A=\mathcal{H}M F\mathcal{H}M,\quad B=\mathcal{H}M G\mathcal{H}M,\quad C=\mathcal{H}M S\mathcal{H}M. \label {12::1.6}
\end{equation}

We have to explain why $C$ is bounded. The operator $SH^{-1}$ is defined
on the whole $\mathcal{H}$ and is closed. Then $SH^{-1}$ and its adjoint
$H^{-1}S$ are bounded, and according to the interpolation theorem
the operator $C$ is bounded, too.

Generally, we can not claim that the generalized and the classical spectra
of $T(\la)$ coincide. In the subsequent theorems we clarify the relationship
between these concepts.
\begin{theorem}
Let $\rho_{cl}(T)$ and $\rho_{gen}(T)$ be the classical and the generalized
resolvent sets of a strongly elliptic pencil $T(\la)$. Then
$$
\rho_{cl}(T)\subset\rho_{gen}(T).
$$
The real line belongs to $\rho_{cl}(T)\cap\rho_{gen}(T)$ with the possible
exception of finitely many normal eigenvalues whose algebraic multiplicity
coincide in both sences. If $H^{-1}$ is compact then the classical and the
generalized spectra coincide in the whole $\mathcal{C}$ and consist of normal
eigenvalues.
\end{theorem}
\begin{proof}
The last assertion of the theorem and the coincidence of the algebraic
multiplicities of the normal eigenvalues in both sences are proved in
\cite{12::S1}, \S 3.

Let $\lambda\in\rho_{cl}(T)$. Then
$$
T(\lambda):\mathcal{H}_2\to\mathcal{H} \quad \mbox{and}\quad T(\overline\lambda):\mathcal{H}_2\to\mathcal{H}
$$
are isomorphisms, hence, so are the operators
$$
T^*(\lambda)=T(\overline\lambda):\mathcal{H}\to\mathcal{H}_{-2}\quad \mbox{and}\quad
T^*(\overline\lambda)=T(\lambda):\mathcal{H}\to\mathcal{H}_{-2}.
$$
From the interpolation theorem (see \cite{12::LM}, Ch1) we obtain that
$T(\lambda):\mathcal{H}_1\to\mathcal{H}_{-1}$ is an isomorphism, i.e. $\la\in\rho_{gen}(T)$.

Let us prove the second statement. The assumption $\mathcal{D}(G)\supset\mathcal{D}(H^{1/2})$
implies that $GH^{-1/2}$ is defined on the whole $\mathcal{H}$ and it follows from
the definition that it is closed. Hence, $GH^{-1/2}$ is bounded and its
norm $\le c$. Then for any $\varepsilon>0$ we have
$$
\|Gx\|^2\le c^2\|H^{1/2}x\|\le c^2(Hx,x)\le 2c^2(\varepsilon\|Hx\|^2+
\varepsilon^{-1}\|x\|^2),\quad x\in\mathcal{D}(\mathcal{H}).
$$
This means that $G$ is $H$-bounded with zero $H$-bound, and so is the
operator
$$
K(\lambda)=\lambda^2F+\lambda G+S-H+V
$$
for any $H$-compact operator $V$ and $\lambda\in\mathcal{C}$.
We can  choose a positive operator $V$ such that
$$
T(\lambda)+V\ge\varepsilon^2(\lambda^2+H),\quad\lambda\in\mathbb{R}.
$$
It follows from the stability Theorem V.4.11 of \cite{12::Ka} that
$T(\lambda)+V=H+K(\lambda)\gg 0$ is self-adjoint for any fixed
$\lambda\in\mathbb{R}$. Therefore, $T(\lambda)+V$ is boundedly invertible in $\mathcal{H}$
for all $\lambda\in\mathbb{R}$ (and, hence, in
a neighborhood of any point $\lambda\in\mathbb{R}$). We have the representation
$$
T(\lambda)=\left[I-V(T(\lambda)+V)^{-1}\right](T(\lambda)+V),
$$
where $V(T(\lambda)+V)^{-1}$ is a holomorphic operator function in a
neighborhood of $\mathbb{R}$ whose values are compact operators. It follows from
the theorem on holomorphic operator function (see \cite{12::GGK}, Ch. XI) that
the spectrum of $T(\lambda)$ in a neighborhood of $\mathbb{R}$ consists of
finitely many isolated eigenvalues of finite algebraic multiplicity.
According to Proposition 1.4 all the real eigenvalues are located in a
finite interval $[-r_0,r_0]$.
This ends the proof.
\end{proof}

For $\delta>0$ and $0<\phi\le\pi/2$ we denote
$$
B_{\delta}=\{\lambda:|\lambda|\le \delta\},\quad
\Lambda^-_{\phi}=\{\lambda:|arg\lambda|<\phi\},\quad
\Lambda^-_{\phi}=\{\lambda:|\pi-arg\lambda|<\phi\}
$$
and $\Lambda_\phi=\Lambda^+_\phi\cup\Lambda^-_\phi$.

\begin{theorem}
Let $T(\la)$ be strongly elliptic. Then there exist positive
numbers $\phi$ and $\delta$ such that the union
$\Lambda_{\phi}\cup B_{\delta}$ with the possible exeption of finitely many
normal eigenvalues belongs to the classical resolvent set
of $T(\la)$ (and hence to $\rho_{gen}(T)$). Moreover, the
estimate
\begin{equation}
|\la|^2\, \|T^{-1}(\la)\|+ |\la|\, \|\mathcal{H}PT^{-1}(\lambda)\|+
\|\mathcal{H}P T^{-1}(\lambda)\mathcal{H}P\| \le const \label{12::1.7}
\end{equation}
$\mbox{holds for all}\ \la \in\Lambda_{\phi},|\la|>r_0$ if $r_0$ is large enough.
\end{theorem}
\begin{proof} Let us prove (\ref{12::1.7}) for $\lambda\in \Lambda_{\phi}^+$,
the same arguments can be applied for $\lambda\in\Lambda_{\phi}^-$.
If $\la=re^{i\theta}$, then
\medskip
\begin{equation}
T(\la) = T(r) +r^2(e^{2i\theta}-1)F+r(e^{i\theta}-1)G. \label{12::1.8}
\medskip
\end{equation}
This equality and Proposition 1.4 yield the estimate
\begin{equation}
Re\, (T(\la) x,x) \ge \varepsilon(r^2(x,x)+(Hx,x)), \quad x\in \mathcal{H}_2,
\ \la\in\Lambda_{\phi},\  |\la|>r_0, \label {12::1.9}
\end{equation}
for sufficiently small $\phi$ and large $r_0$. We noticed already that
the coefficients of the pencil $L(\lambda)=\mathcal{H}M T(\la)\mathcal{H}M$ are bounded
operators.
From (\ref{12::1.8}) we have
$$
\|L(\la)\|\,\|y\|\ge Re (L(\la)y,y)\ge \varepsilon r^2(y,y), \quad y\in\mathcal{H}_1,\
\la\in\Lambda_{\phi},\ |\la|>r_0.
$$
By continuity this inequality holds for all $y\in\mathcal{H}$ and implies that
zero does not belong to the
numerical range of $L(\lambda)$. Then $L(\la)$ is invertible,
and
\begin{equation}
||L^{-1}(\la)||\le\varepsilon^{-1}. \label {12::1.10}
\end{equation}
From this we have that
$$
T(\la) = \mathcal{H}P L^{-1}(\la)\mathcal{H}P: \mathcal{H}_1\to \mathcal{H}_{-1}
$$
is an isomorphism, and $T^{-1}(\la)$ exists in $\mathcal{H}_{-1}.$
Now, $L^{-1}(\la)= \mathcal{H}P T^{-1}(\lambda)\mathcal{H}P$ and from (\ref{12::1.10}) we obtain the
estimate of the third term in (\ref{12::1.6}).

It follows from Proposition 1.5 that
$$
\|T^{1/2}(r)x\|\ge\varepsilon\|(r^2+H)^{1/2}x\|,\quad r>r_0,\quad x\in\mathcal{D}(H).
$$
By virtue of Theorem 1.6 $T(r)$ is invertible for $r>r_0$, hence,
\begin{equation}\label{12::1.10a}
\begin{array}{c}
\|T^{-1/2}(r)\|\le\varepsilon^{-1}\|(r^2+H)^{-1/2}\|\le\varepsilon^{-1}r,\\
\phantom {S}\\
\|GT^{-1/2}(r)\|\le c\|H^{1/2}T^{-1/2}(r)\|\le c\varepsilon^{-1}.
\end{array}
\end{equation}
We have
$$
T(\lambda)=T^{1/2}(r)(I+G(\lambda))T^{1/2}(r),\quad
G(\lambda)=T^{-1/2}(r)(T(\lambda)-T(r))T^{-1/2}(r).
$$
It follows from representation (\ref{12::1.8}) and estimates (\ref{12::1.10a}) that
$\|G(\lambda)\|\le1/2$ if $\lambda\in\Lambda^+_{\phi}$ and $\phi$ is
sufficiently small. Hence, $T(\lambda)$ is invertible in $\mathcal{H}$ for
$\lambda\in\Lambda^+_\phi$, and
$$
\|T^{-1}(\lambda)\|\le 2\|T^{-1/2}(r)\|^2\le2\varepsilon^{-2}r^2,
$$
$$
\|H^{1/2}T^{-1}(\lambda)\|\le 2\|H^{1/2}T^{-1/2}(r)\|\ \|T^{-1/2}(r)\|\le
2\varepsilon^{-2}r^{-1}.
$$
This completes the proof.
\end{proof}

\begin{rem} We say $T(\la)$ is {\it positive} if $T(\lambda)>0$ for all
$\lambda\in\mathbb{R}$. We claim:\break {\em A positive strongly elliptic pencil is uniformly
positive}. Indeed, if $T(\lambda)>0$ for all $\lambda\in\mathbb{R}$ then
$\lambda$ is not an eigenvalue of $T(\lambda)$, and according to Theorem 1.6
$T(\lambda)$ is boundedly invertible in $\mathcal{H}$ as well as in $\mathcal{H}_{-1}$.
Therefore $T(\lambda):\mathcal{H}_1\to\mathcal{H}_{-1}$ is a continuous bijection for
$\lambda\in[-r_0,r_0]$, hence, so is $T^{-1}(\lambda):\mathcal{H}_{-1}\to\mathcal{H}_1$.
This yields the estimate $\|\mathcal{H}PT^{-1}(\lambda)\mathcal{H}P\|\le const$, which implies
$T(\la)\ge\varepsilon H$. Bearing in mind Proposition 1.4, we find that
$T(\lambda)$ is uniformly positive.
\end{rem}

\subsection{The real spectrum of a strongly elliptic pencil}
\setcounter{equation}{0}
\setcounter{theorem}{0}

We noticed in the Introduction, that the real eigenvalues of a
pencil $T(\la)$ play a significant role in physical considerations, as
they correspond to waves propagating the energy at the infinity (or
from the infinity). We already proved that strongly elliptic pencils may have
only finitely many real eigenvalues. In this section
we obtain additional valuable information.

First, recall that a point $\mu \in \mathcal{C}$ is said to be a normal
eigenvalue of $T(\la)$ if it is an isolated point of the spectrum of
$T(\la)$ and the principal part of the Laurent expansion of the resolvent
$T^{-1}(\lambda)$ in a neighborhood of $\mu$ admits a representation of the form
\begin{equation}
\sum_{k=1}^N\sum_{s=0}^{p_k}\ \
\frac{(\cdot,z_k^{p_k-s})\phantom{.} x_k^s}{(\la-\mu)^{p_k+1-s}}.
\label{12::2.1}
\end{equation}
Here
\medskip
\begin{equation}
x_k^0,\dots,x_k^{p_k}, \qquad k=1,\dots,N, \label{12::2.2}
\end{equation}
is a canonical system of eigen and associated elements of $T(\la)$ and
\begin{equation}
z_k^0,\dots,z_k^{p_k}, \qquad k=1,\dots,N, \label{12::2.3}
\end{equation}
is the adjoint canonical system which is uniquely defined by the
choice of system (\ref{12::2.2}).

Let $\mu\in\mathbb{R}$.
Since the classical and the generalized spectra of $T(\la)$ coincide in a
neighborhood of $\mathbb{R}$, the
elements of systems (\ref{12::2.2}) and (\ref{12::2.3}) belong to $\mathcal{H}_2$.
It follows from \cite{12::KS}, Lemma 2.1 that there exists a canonical system
(\ref{12::2.2}) such that
$$
x_k^s=\varepsilon_kz_k^s, \quad k=1,\dots,N, \quad s=0,\dots,p_k,
$$
where $\varepsilon_k = \pm 1$.
Such a canonical system is called {\em normal} and the numbers
$\varepsilon_k$
are called {\em the sign characteristics} of the corresponding Jordan chains.

A real eigenvalue $\mu$ is said to be of {\em positive (negative) type} if
$$
(T'(\mu)y,y)>0 \ (<0) \quad \mbox{for all} \ y\in Ker\,T(\mu).
$$
\begin{proposition}
If $\mu$ is a semi-simple real eigenvalue of $\ T(\la)$ and\break
$V=\sum \varepsilon(\cdot, y_k^0)\phantom{.} y_k^0$ is the residue operator of $T^{-1}(\lambda)$ at
the pole $\mu$ then
\begin{equation}
(L'(\mu)Vy,Vy) =(y,Vy) \qquad \mbox{for all} \ y\in
Ker\,T(\mu). \label{12::2.4}
\end{equation}
In particular, $\mu$ is of positive (negative) type if and only if all
the sign characteristics are positive (negative).
\end{proposition}

\begin{proof} For $y\in Ker\,T(\mu)$ we have $Vy\in Ker T(\mu)$ and
$$
y=T(\la)T^{-1}(\lambda) y=(T(\mu)+(\la-\mu)T'(\mu)+\dots)
(V(\la-\mu)^{-1} +R(\mu)+\dots)y=
$$
$$
=T'(\mu)Vy+T(\mu)R(\mu)y+o(1),
$$
where $o(1)\to 0$ as $\lambda\to \mu$ and $R(\mu)$ is a bounded operator
on $\mathcal{H}$. Taking the scalar product
with $Vy$ and letting $\la\to\mu$, we obtain (\ref{12::2.4}).
\end{proof}
\begin{theorem} Let a pencil $T_{\omega}(\lambda)$ of the form (\ref{0.2}) be strongly
elliptic. Then
for all $\omega >0$ with possible exception of some values
$\omega_k\to\infty$ (the so-called resonant frequences) there is an even
number, say $2\kappa$, of real eigenvalues of $T_{\omega}(\lambda)$ counting geometric
multiplicities. They all are of definite type and exactly $\kappa$ of
them are of positive (negative) type.
\end{theorem}
\begin{proof} Consider the pencil
$$
L_{\omega}(\la)=\mathcal{H}MT_{\omega}(\lambda)\mathcal{H}M = L(\la)-\omega^2 R_0, \quad R_0=\mathcal{H}M R\mathcal{H}M.
$$
The assumptions on the operators (see Section 1) ensure us that the
coefficients of $L_{\omega}(\la)$ are bounded operators on $\mathcal{H}$,
moreover, $R_0>0$. By virtue of\break Theorem 1.6 there are finitely many
normal eigenvalues of $L_{\omega}(\la)$ on the real axis.
To prove that they are of definite type with possible exception of
isolated values $\omega_k\to \infty$
we apply the known results of pertubation operator theory which are
based  on theorems due to {\sc Rellich} and {\sc Nagy} (see Ch 9 of
\cite{12::RN}), {\sc Krein} and {\sc Lyubarskii} \cite{12::KL},
{\sc Kostyuchenko} and {\sc Orazov} \cite{12::KO2}. A concentrated
exposition of this material can be found in the paper of
{\sc Shkalikov} and {\sc Hriniv} \cite{12::SH},\break Propositions 1.6-1.9.
The only reservation: the
condition  $R_0\gg 0$ assumed in [SH] can be replaced by $R_0>0$
provided the coefficients of the pencil $L_\omega(\la)$ are bounded.
The main idea of proving this result is the following. Let $\mu$ be a real
eigenvalue of the pencil $L_{\omega}(\lambda)$ with fixed $\omega=\omega_0$, and
let $\theta_0=\omega^2_0$. We notice that
$L_{\omega}(\la)$ is a linear selfadjoint pencil with respect to the
parameter $\theta=\omega^2$ and its eigenvalues $\theta_j(\la)$ according to the
Rellich-Nagy theorem depend analytically on $\la$ in a neighborhood
of an eigenvalue $\la=\mu$, namely,
$$
\theta_j(\la) =\theta_0+a_j(\la-\mu)^{p_j}+\dots
$$
with some $0\ne a_j\in \mathbb{R}$ and integer $p_j>0$.
Then
$\la_j(\theta)$ represent the branches of the inverse algebraic functions
$$
\la_{j,k}(\theta) = \mu +(a_j^{-1}(\theta-\theta_0))^{1/p_j} +\dots, \quad
k=0,\dots, p_j-1,
$$
and $p_j$ coincide with the lengths of the corresponding Jordan chains.
Hence, $\la_j(\theta)$ move locally either in the complex plane or leave
on the real axis depending monotonically on $\theta$, moreover,
the condition $R_0>0$ implies that the real branches $\la_j(\theta)$
are strictly monotone functions. Thus, all the real eigenvalues in a
small punctured neighborhood of $\mu$ are semi-simple. Further,
it turns out (see
Proposition 2.3 below) that the sign characteristics of the real
eigenvalues $\la_j(\theta)$ coincide with $sign\,\la'_j(\theta)$. Taking
into account Proposition 2.1, we obtain that all the real eigenvalues
of $T_{\omega}(\lambda)$ in a small right (left) neighborhood of $\mu$ are of
positive (negative) type. Hence, the resonant frequences are isolated
points.

Let us prove the other statements. Fix a non-resonant frequency $\omega$,
and fix a positive $H$-compact operator $V$ such that $T_{\omega}(\lambda)+V$ is
uniformly positive. Consider the pencil
\begin{equation}
T_{\omega}(\lambda)+\rho(V+I), \qquad 0\le \rho\le 1. \label{12::2.4m}
\end{equation}
Obviously, the closure of $\mathcal{H}M(V+I)\mathcal{H}M$ is positive in $\mathcal{H}$. Now apply
Proposition 1.9 from [SH] which says:
$$
E^+(\rho)+E^-(\rho) =const,
$$
where $E^+(\rho)$ and $E^-(\rho)$ are the number of real eigenvalues of
positive and negative type, respectively. This equality holds also for
the resonant values of  $\theta$ if the numbers $E^{\pm}(\rho)$ are defined
as in [SH]. Since for $\rho =1$ the pencil (\ref{12::2.4m}) is uniformly positive,
we have
$$
E^+(1)-E^-(1)=0, \quad \mbox{hence}\quad E^+(0)=E^-(0).
$$
This ends the proof.
\end{proof}

Let $\{\mu, f_0\}$ be a normal  eigen-pair corresponding to a simple or
semi-simple eigenvalue of $T_{\omega}(\lambda)$ with a fixed  $\omega=\omega_0>0$.
As we mentioned above the eigenvalue $\theta_0=\omega_0^2$ admits an analytic
continuation $\theta_j(\la)= \omega^2$ when $\la$ runs in a neighborhood
of $\mu$.
The value $\theta'(\la)\big|_{\la=\mu}$ is called the group velocity
(see, for example, \cite{12::ZK} or \cite{12::VB}) of the wave solution
$$
u(y)=e^{-i\mu y}f_0.
$$
\begin{proposition} If $\mu$ is a definite type eigenvalue of a pencil
$T_{\omega}(\lambda)$  and
$\{\mu, f_0\}$ is a corresponding normal eigen-pair then
\begin{equation}
(T'_{\la}(\mu)f_0,f_0) =\theta'_{\la}(\mu)(Rf_0,f_0), \quad \theta =\omega^2, \label {12::2.5}
\end{equation}
i.e. the sign characteristic of an eigen-pair coincides with the sign
of its group velocity.
\end{proposition}
\begin{proof}
Let $f(\la) = f_0+(\la-\mu)f_1+\dots $ be the eigen-element of $T_{\omega}(\lambda)$
corresponding to $\la=\la(\theta)$. Denoting $\theta_0=\omega_0^2$ we obtain
$$
[T(\mu)-\theta_0+T'(\mu)(\la-\mu)-(\theta-\theta_0)R+
\dots][f_0+(\la-\mu)f_1+ \dots]=0,
$$
therefore
$$
(T'(\mu)f_0,f_0)+o(1)=\frac{\theta-\theta_0}{\la-\mu}(Rf_0,f_0).
$$
Letting $\la\to\mu$, we get (\ref{12::2.5}).
\end{proof}
\begin{corollary} For any non-resonant frequency $\omega$
equation (\ref{0.1}) with strongly elliptic symbol (\ref{0.2})
possesses finitely many, say $2\kappa\ge 0$,  propagating waves and exactly
$\kappa$ of them are outgoing (incoming), i.e. have positive (negative)
group velocity or the sign characteristics.
\end{corollary}

\subsection{Factorization of positive strongly elliptic pencils}
\setcounter{equation}{0}
\setcounter{theorem}{0}

In this section we use abstract Sobolev spaces. Namely, by $W_m\mathbb{R}H$
we denote the space consisting of $\mathcal{H}_m$-valued functions $u(y)$
defined on $\mathbb{R}^+$, such that $u^{(j)}(y)$ exist in the generalized
sense
for $j\le m$ as $\mathcal{H}$-valued functions and the integral
$$
\int_0^{\infty}\left(\|u^{(m)}(y)\|_0^2+\|u(y)\|_m^2\right)\,dy=:\|u\|_{W_m}^2
$$
converges. The detailed information on abstract Sobolev spaces can be
found in the book of {\sc Lions} and {\sc Magenes} \cite{12::LM}.
We recall here some facts which we need below. According to {\em the
theorem on intermediate} derivatives we have
$$
u^{(j)}(y)\in L_2(\mathbb{R}^+,\mathcal{H}_{m-j}), \quad j=1,\dots,m,
\quad \mbox{if}\ \,u(y)\in W_m\mathbb{R}H.
$$
An important role in the sequel plays the trace theorem which we
formulate (as it needed) in the case $m=1$.\hfill\break
{\bf Trace Theorem.}
{\em A function $u(y) \in W_1\mathbb{R}H$ is continuous  and uniformly bounded
on $\mathbb{R}^+$ an $\mathcal{H}_{1/2}$-valued function and the trace operator
\begin{equation}
\mathcal{T}_r:\,W_1\mathbb{R}H\to \mathcal{H}_{1/2}, \qquad \mathcal{T}_ru=u(r), \label{12::3.1}
\end{equation}
is bounded for any fixed $r\in\mathbb{R}^+$, moreover, $\|\mathcal{T}_r\|\le c$ with a
constant $c$ not depending on $r\in\mathbb{R}^+.$}

If $u(y)\in W_2(\mathbb{R}^+,\mathcal{H})$ then $u'(y)\in W_2(\mathbb{R}^+,\mathcal{H}_1)$. As
$\mathcal{D}(G)\supset\mathcal{H}_1$, we have $Gu'(y)\in L_2(\mathbb{R}^+,\mathcal{H})$. Therefore
$$
\mathcal{A} u=T\left( i\frac d{dy}\right) u =-Fu''(y)+iGu'(y)+Su(y)
$$
is well defined in the space $L_2\mathbb{R}H$ with domain
$\mathcal{D}(\mathcal{A})=W_2\mathbb{R}H$.
Let $\mathcal{A}_0$ be the restriction of $\mathcal{A}$ on the domain
$$
\mathcal{D}(\mathcal{A}_0)=\Wq^{{\tiny\rm o}\phantom{.}\phantom{.}}\mathbb{R}H:=\{y\mid\,y\in W_1\mathbb{R}H ,\ y(0)=0\}.
$$
\begin{lemma}
Let $T(\la)$ be strongly elliptic. Then there exist a number $\varepsilon>0$ and an
$H$-compact self-adjoint
operator $V\ge 0$, such that
\begin{equation}
\varepsilon \|u\|_{W_1}-\|V^{1/2} u\|_{L_2} \le (\mathcal{A}_0 u,\,u)\le \varepsilon^{-1}
\|u\|_{W_1}, \quad u\in\mathcal{D}(\mathcal{A}_0). \label{12::3.2}
\end{equation}
If in addition $T(\la)$ is positive then the left hand side estimate holds
with $V=0$.
\end{lemma}
\begin{proof} Denote
$$
\hat{u}(\la)=\int_0^{\infty} u(y)e^{i\la y}\,dy, \qquad \la\in\mathbb{R}.
$$
It follows from the Plancherel theorem that
$$
\Bigl(\hat{f}(\la), \hat{g}(\la)\Bigr) = \Bigl( f(y), g(y)\Bigr),
\qquad\mbox{for} \ f,g\in L_2\mathbb{R}H ,
$$
where the scalar product $\Bigl(,\Bigr)$ is taken in $L_2\mathbb{R}H$. As $T(\la)$ is
strongly
elliptic, there is an $H$-compact operator $V$ such that
$$
T(\la) +V \ge \varepsilon (\la^2+H),\quad \la\in\mathbb{R}.
$$
We can suppose that $V=V^*$, otherwise the Fridrichs extension of $V$
should be considered.
Bearing in mind that for all functions
$u(y)\in \Wq^{{\tiny\rm o}\phantom{.}\phantom{.}}\mathbb{R}H=\mathcal{D}(\mathcal{A}_0)$
$$
-i\la \hat{u}(\la) = \int_0^{\infty} u'(y)\,e^{i\la y}\,dy,
$$
we find that
$$
\begin{array}{rl}
\Bigl((\mathcal{A}_0 +V)u,\,u\Bigr)&=\Bigl(Fu',\,u'\Bigr)-i\Bigl(Gu,\,u'\Bigr)+
\Bigl((S+V)u,\,u\Bigr)\\
\phantom{.}\\
&=\Bigl((T(\la) +V)\hat{u}(\la),\,\hat{u}(\la)\Bigr)\ge \varepsilon \Bigl((\la^2+H)
\hat{u}(\la),\,\hat{u}(\la)\Bigr)\\ \phantom{.}\\
&=\varepsilon\left[\Bigl(u',\,u'\Bigr)+\Bigl(\mathcal{H}P u,\,\mathcal{H}P u\Bigr)\right] =
\varepsilon \|u\|_{W_1}^2,
\quad u \in\mathcal{D}(\mathcal{A}_0).
\end{array}
$$
This implies the left hand side estimate of (\ref{12::3.2}).
The right one is trivial and
follows from the inequality
$$
|\Bigl(Gu,\,u'\Bigr)|\le c \Bigl(\mathcal{H}P u,\,\mathcal{H}P u\Bigr)\Bigl(u',\,u'\Bigr)
\le c \|u\|_{W_1}^2.
$$
To get the last statement of Lemma, recall Remark 1.8.
\end{proof}

Let $\Wr^{{\tiny\rm o}\phantom{.}\phantom{.}\phantom{.}\phantom{.}\phantom{.}}\mathbb{R}H$ be the dual space to
$\Wq^{{\tiny\rm o}\phantom{.}\phantom{.}}\mathbb{R}H$ with respect to $L_2\mathbb{R}H$. For any
$v\in \Wq^{{\tiny\rm o}\phantom{.}\phantom{.}}\mathbb{R}H$ and $u\in W_2(\mathbb{R}^+,\mathcal{H})$
\begin{equation}\label{12::3.3}
(\mathcal{A} u,\,v)=(Fu',\,v')-i(Gu,v')+(K\mathcal{H}P u,\,\mathcal{H}P v),
\end{equation}
where $K=\mathcal{H}M S\mathcal{H}M =K^*$ is a bounded operator on $\mathcal{H}$. For any fixed
$u\in W_1\mathbb{R}H$ the right hand side of (\ref{12::3.3}) represents a
continuous linear functional on $\Wq^{{\tiny\rm o}\phantom{.}\phantom{.}}\mathbb{R}H$.
According to the definition of a dual space,
any such a functional admits a representation $(f,v)$, with
$f\in \Wr^{{\tiny\rm o}\phantom{.}\phantom{.}\phantom{.}\phantom{.}\phantom{.}}\mathbb{R}H$.
Hence $\mathcal{A}$ admits the extension
\begin{equation}
\mathcal{A} :\,W_1\mathbb{R}H\to \Wr^{{\tiny\rm o}\phantom{.}\phantom{.}\phantom{.}\phantom{.}\phantom{.}}\mathbb{R}H. \label{12::3.4}
\end{equation}
{\em A function $u(y)\in W_1\mathbb{R}H$ is called a generalized solution
of the equation}
\begin{equation}
T\left(i\frac d{dy}\right)\,u(y)=0 \label{12::3.5}
\end{equation}
{\em if $u(y)$ belongs to the kernel of operator (\ref{12::3.4}).}

\begin{lemma} Let $T(\la)$ be a positive strongly elliptic pencil.
Then for any\break
$x\in\mathcal{H}_{1/2}$ there is a unique generalized solution $u(y)$ of
equation (\ref{12::3.5}) such that $u(0)=x$.
\end{lemma}
\begin{proof}
This statement is familiar from PDO theory; its abstract version
is proved in the same way, one should use only the Friedrichs theorem
instead of the Lax-Milgram lemma. Namely, taking into account Lemma 3.1
and the Friedrichs theorem (see \cite{12::RN}, Ch8), we obtain that
$\mathcal{A}_0$ admits the only self-adjoint extension $\mathcal{A}_F\gg 0$ such that
$\mathcal{D}(\mathcal{A}_F^{1/2}) = \Wq^{{\tiny\rm o}\phantom{.}\phantom{.}}\mathbb{R}H$. Hence,
\begin{equation}
\mathcal{A}_F\,:\ \Wq^{{\tiny\rm o}\phantom{.}\phantom{.}}\mathbb{R}H \to \Wr^{{\tiny\rm o}\phantom{.}\phantom{.}\phantom{.}\phantom{.}\phantom{.}}\mathbb{R}H \label{12::3.6}
\end{equation}
is an isomorphism. Since the trace operator
$\mathcal{T}_0$ defined in (\ref{12::3.1}) is surjective, for any $x\in\mathcal{H}_{1/2}$ there is
a function $v_1(y)\in W_1\mathbb{R}H$ such that $v_1(0) = x$. Then
$\mathcal{A} v_1(y)\in \Wr^{{\tiny\rm o}\phantom{.}\phantom{.}\phantom{.}\phantom{.}\phantom{.}}\mathbb{R}H$ and taking into account
that mapping (\ref{12::3.6}) is an isomorphism, we find a function
$v_2(y)\in \Wq^{{\tiny o}\phantom{.}\phantom{.}}\mathbb{R}H$ such that\break
$\mathcal{A}_Fv_2(y)=\mathcal{A} v_1(y).$ Hence, the function $u(y)=v_1(y)-v_2(y)$ is
a generalized solution of the equation $\mathcal{A} u(y)=0$ and $u(0)=v_1(0)=x.$ The
uniqueness follows from the condition $Ker\,\mathcal{A}_F =0$.
\end{proof}

\begin{lemma} Let $u(y)$ be a solution of equation (\ref{12::3.5}) on the
semiaxis $\mathbb{R}^+$ in the following sense:
$$
u(y),u'(y)\in C\mathbb{R}HP , \ \ u''(y)\in C\mathbb{R}H,
$$
and equation (\ref{12::3.5}) holds as an equality in $\mathcal{H}$. If $u(y)\in W_1\mathbb{R}H$, then
\begin{equation}
\|S^{1/2}u(y)\| = \|F^{1/2}u'(y)\|, \quad y\ge 0. \label {12::3.7}
\end{equation}
\end{lemma}
\begin{proof}
Consider in $\mathcal{H}^2=\mathcal{H}\times\mathcal{H}$ the operator
\begin{equation}
{\bf T}= \left(\begin{array}{cc}-F^{-1/2}GF^{-1/2}& -F^{-1/2}S^{1/2}\\
S^{1/2}F^{-1/2}&0\end{array}\right),
\label{12::3.8}
\end{equation}
acting in $\mathcal{H}^2=\mathcal{H}\times\mathcal{H}$ (the linearization of
$T(\la)$). Obviously, ${\bf T} $ is symmetric (and even selfadjoint)
in the Krein space $\mathcal{K} =\{\mathcal{H}^2,{\bf J}\}$ with the fundamental
symmetry ${\bf J} =\left(\begin{array}{cc} I&0\\0&-I\end{array}\right)$.
It is easy to see that equation (\ref{12::3.5}) is equivalent to the following
one
$$
{\bf T}\,{\bf u}(y)=i{\bf u}'(y), \qquad {\bf u}(y)=\left(\begin{array}{c}
F^{1/2}u'(y)\\-iS^{1/2}u(y)\end{array}\right).
$$
Using this equation we find (differentiation is allowed by our
assumptions)
$$
({\bf J}{\bf u}(y),\,{\bf u}(y))' =
({\bf J}{\bf u}'(y),\,{\bf u}(y))+({\bf J}{\bf u}(y),\,
{\bf u}'(y))=
$$
$$
=-i({\bf J}{\bf T}{\bf u}(y),\,{\bf u}(y))
+i({\bf u}(y),\,{\bf J}{\bf T}{\bf u}(y))=0.
$$
Therefore, $({\bf J}{\bf u}(y),\,{\bf u}(y))= const$. The
condition $u(y)\in W_1\mathbb{R}H$, obviously, implies
$({\bf J}{\bf u}(y),\,{\bf u}(y))=0$ and (\ref{12::3.7}) follows.
\end{proof}
\begin{theorem} Let $T(\lambda)$ be a strongly elliptic positive  pencil.
Then there exists a closed operator $Z$ in the space $\mathcal{H}$ with domain
$\mathcal{D}(Z)\subset \mathcal{H}_1$, such that
\begin{equation}
T(\lambda)x=(F\la-Z_1)(\la-Z)x \qquad\mbox{for all} \ x\in\mathcal{D}(Z), \label {12::3.9}
\end{equation}
where $Z_1=-(G+FZ)$ and the equality is understood in
$\mathcal{H}_{-1}$. Moreover,

(a) Z has a representation $Z=KS^{1/2}$ where $K$ is a partial
isometry in $\mathcal{H}$ whose image $\mathbb{R}e(K) =\mathcal{H}$;

(b) -iZ generates a holomorphic semigroup in the spaces $\mathcal{H}_{\theta},
0\le\theta\le 1/2$;

(c) the generalized solutions of equation (\ref{12::3.5}) satisfy the equation
$$
u'(y)=-iZu(y).
$$
Factorization (\ref{12::3.9}) with these properties is unique.
\end{theorem}
\begin{proof} Let $x\in\mathcal{H}_{1/2}$. By virtue of Lemma 3.2 there is a
generalized solution of equation (\ref{12::3.5}) such that $u_x(0) = x.$ Define
the operator function $U(t)$ on $\mathbb{R}$ as follows
$$
U(t)x=u_x(t), \qquad t\ge 0.r
$$
Note that according to Lemma 3.2 the restriction of the trace operator
$\mathcal{T}_0$ to $Ker\mathcal{A}\in W_1(\mathbb{R}^+,\mathcal{H})$ is a bounded isomorphism onto $\mathcal{H}_{1/2}$.
Hence the inverse operator
$$
\mathcal{T}_0^{-1}:\, \mathcal{H}_{1/2}\to Ker \mathcal{A}, \qquad \mathcal{T}_0^{-1}x=u_x(t)
$$
is bounded, as well as the operator
$U(t)=\mathcal{T}_t\,\mathcal{T}_0^{-1}$ acting in $\mathcal{H}_{1/2}$ (for any $t\ge 0$). It
follows from the definition of the operator $U(t)$ and from the trace
theorem (see the formulation at the beginning of this section) that
$$
\begin{array}{rl}
&U(t+s)=U(t)U(s),\qquad U(0)=I,\qquad  \|U(t)\|\le const,\\ \phantom{.}\\
&\slim_{\phantom{.}\phantom{.}\phantom{.}\phantom{.} t\to s} U(t)=U(s), \quad 0\le s\le t,
\end{array}
$$
where the strong limit is understood in $\mathcal{H}_{1/2}$. This means that
$U(t)$ is a  uniformly bounded $C_0$-semigroup in the space $\mathcal{H}_{1/2}$
(see, e.g., \cite{12::Yo}).
If $U(t)=e^{-iZt}$ where $-iZ$ is the generator of $U(t)$, then property (c)
of Theorem 3.4 is satisfied, and by Lemma 3.2 it defines $Z$ uniquely.

It is known from semigroup theory
that $Z$ and $Z^2$ (as well as the other powers) are closed operators
in $\mathcal{H}_{1/2}$ whose domains $\mathcal{D}(Z)$ and $\mathcal{D}(Z^2)$ are densely defined
in $\mathcal{H}_{1/2}$.

Let $x\in\mathcal{D}(Z^2)\subset\mathcal{H}_{1/2}$ and $u_x(t)$ is the corresponding
generalized solution of (\ref{12::3.5}).
In view of the semigroup properties the functions
$u'_x(t), u''_x(t)$ are continuous  in $\mathcal{H}_{1/2}$ on $\mathbb{R}^+$ and
$$
u'_x(t)=iZu_x(t)\qquad
u''_x(t)=-Z^2u_x(t).
$$
The operator $G:\mathcal{H}\to\mathcal{H}_{-1}$ is bounded, therefore, $Gu'_x(t)$ is
continuous in $\mathcal{H}_{-1}$. Since $u_x(t)$ is a generalized solution, we
have the equality
\begin{equation}
-Fu''_x(t)+iGu'_x(t) = -Su_x(t) \label{12::3.10}
\end{equation}
which is understood as an equality in $\Wr^{{\tiny\rm o}\phantom{.}\phantom{.}\phantom{.}\phantom{.}\phantom{.}}\mathbb{R}H$.
The left hand side is a continuous function in $\mathcal{H}_{-1}$, hence, so is the
function $Su_x(t)$. Equivalently, $u_x(t)$ is continuous in $\mathcal{H}_1$.
In particular, $x=u_x(0)\in\mathcal{H}_1$ and (\ref{12::3.10}) gives
\begin{equation}
(FZ^2+GZ+S)x=0, \qquad \mbox{for}\ \, x\in\mathcal{D}(Z^2), \label{12::3.11}
\end{equation}
where the equality is understood in $\mathcal{H}_{-1}$.

Our further aim is to extend (\ref{12::3.11}) to a larger domain. Notice, if
$x\in\mathcal{D}(Z^2)$ then the conditions of Lemma 3.3 are fulfilled and we have
$$
\|F^{1/2}u_x'(t)\|=\|S^{1/2}u_x(t)\|, \qquad t\ge 0.
$$
In particular, we have the equality $\|F^{1/2}Zx\|=\|S^{1/2}x\|$
which gives (for\break $x\in\mathcal{D}(Z^2)$) the representation $Z=KS^{1/2}$,
where $K$ is a partial isometry in $\mathcal{H}$. Since $\mathcal{D}(Z^2)$ is dense in $\mathcal{H}_{1/2}$
and $\mathcal{H}_{1/2}$ is dense in $\mathcal{H}$, we have $\mathbb{R}e(K)=\mathcal{H}$. Hence, $Z$ is
boundedly invertible in $\mathcal{H}$  and $Z^{-1}=S^{-1/2}K^*$.
This enables us to extend $Z$ from $\mathcal{H}_{1/2}$ onto $\mathcal{H}$ with domain
$D_{\mathcal{H}}(Z)=\mathbb{R}e\,(S^{-1/2}K^*)$. Further (and in (\ref{12::3.9})) we omit the
index $\mathcal{H}$ and imply that $Z$ acts in $\mathcal{H}$ and its domain $\mathcal{D}(Z)$
is understood as described. Certainly, $\mathcal{D}(Z)\subset\mathcal{H}_1$ and it
coincides with $\mathcal{H}_1$ if and only if $K$ is a unitary operator. Now,
both terms $Gx$ and $Sx$ are in $\mathcal{H}_{-1}$ for $x\in\mathcal{D}(Z)$, so
equality (\ref{12::3.11}) can be extended to all $x\in\mathcal{D}(Z)$. This is
equivalent to the factorization (\ref{12::3.9}), moreover, for $Z_1$ we have
the representation $Z_1=S^{1/2}K^*$ as well as $Z_1=-(G+FZ)$. Then we obtain
$$
(\la-Z)^{-1}=T^{-1}(\lambda)(F\la -S^{1/2}K^*)
$$
where the both sides are understood as operators in $\mathcal{H}$. Applying
Theorem 1.7 to the right hand side of the last identity we obtain the right hand side of the last identity we obtain
\begin{equation}
\|(\la-Z)^{-1}x\|\le \frac{c\|x\|}{1+|\la|} \label{12::3.12}
\end{equation}
in a double sector $\Lambda_\phi$ containing the real axis. Let us prove
that (\ref{12::3.12}) holds also for all $\la$ from the upper half plane $\mathcal{C}^+$.
Since $-iZ$ is a
generator of a $C_0$-semigroup in $\mathcal{H}_{1/2}$ we have (see [Yo, Ch. 9])
$$
\|(\la-Z)^{-1}x\|\le\|(\la-Z)^{-1}x\|_{1/2}\le \frac{c_x}{1+|\la|}
\quad \mbox{for all} \ \, x\in\mathcal{H}_{1/2} \ \mbox{and} \ \,\la\in\mathcal{C}^-\setminus
\Lambda_\varphi,
$$
with a constant $c_x$ depending on $x$, and the estimate holds in the
whole upper half plane $\mathcal{C}^+$ outside an arbitrary small double sector
$\Lambda_\varphi$ containing the real axis.
Applying the Phragmen-Lindel\"of theorem (see [Bo], for example) we
obtain estimate (\ref{12::3.12}) for all $\la\in\mathcal{C}^-$ and $x\in\mathcal{H}_{1/2}$ with the same
constant $c$ as it was in (\ref{12::3.12}). By continuity (\ref{12::3.12}) can be
extended for all $x\in\mathbb{R}$. This implies that $-iZ$ generates a
holomorphic semigroup in $\mathcal{H}$.

Actually, $-iZ$ generates a holomorphic semigroup in the space
$\mathcal{H}_{1/2}$, too. To prove this, we consider the pencil
$$
T_{\varphi}(\la) =T(e^{i\varphi}\la)
\quad  \mbox{and} \quad
\mathcal{A}^{\varphi} u =T_{\varphi}\left(i\frac d{dy}\right).
$$
For sufficiently small $|\varphi|$ we can reprove Lemma 3.1 changing
$(\mathcal{A}_0u,\,u)$ in (\ref{12::3.2}) by $Re (\mathcal{A}^{\varphi}_0u,\,u)$. This is possible,
since the Friedrichs extension exists for the sectorial operators (see
\cite{12::Ka}, Ch. 6). Repeating the arguments we find that there exists an
operator $-iZ_{\varphi}$ which generates a $C_0$-semigroup in
$\mathcal{H}_{1/2}$, $Z_\varphi$ possesses property (c) and realizes a factorization of
the form (\ref{12::3.9}) for the pencil $T_{\varphi}(\la)$.
From this we obtain $Z_{\varphi}=e^{i\varphi}Z$. Then the minimal
resolvent growth estimate of the form (\ref{12::3.12}) holds for $Z$ in
a small double sector containing the real axis.
Hence, the $C_0$-semigroup generated by $-iZ$ is, actually, a
holomorphic semigroup. Now, applying the interpolation theorem
we get assertion (b). This ends the proof.
\end{proof}

\subsection {Elliptic pencils satisfying the Keldysh-Agmon condition}
\setcounter{equation}{0}
\setcounter{theorem}{0}

\subsection {The resolvent growth condition} In this section we will
use the condition which in general form can be formulated as follows.

{\bf The resolvent growth condition}. {\em Assuming that $T^{-1}(\lambda)\mathcal{H}P x(\la)$ is
holomorphic in the upper (lower)  half plane $\mathcal{C}^+\ (\mathcal{C}^-)$ where
$$
x(\la) =x_0+\la x_1+\dots +\la^n x_n
$$
is an $\mathcal{H}$-valued polynomial, we have
\begin{equation}
\|\mathcal{H}PT^{-1}(\lambda)\mathcal{H}P x(\la)\| \le C|\la|^m  \quad \mbox {for all}\ \la\in
\mathcal{C}^+(\mathcal{C}^-),\quad |\la|>r_0, \label{12::4.1}
\end{equation}
with some constants $c$ and $m$. }

This condition is by no means obvious to verify and we formulate the
other one which can be checked out more easily.

{\bf Keldysh-Agmon condition}. {\em $T(\la)$ is of the form (\ref{12::1.1}) and

(a) the operator $H$ has discrete spectrum (i.e. $H^{-1}$ is compact and
its eigenvalues are subject to the estimates
\begin{equation}
\la_j(H)\ge cj^p, \quad j=1,2,\dots, \label {12::4.2}
\end{equation}
with some constants $c$ and $p$;

(b) either $p\ge 2$ or $p<2$ but there are rays $\gamma_j=\{\la\big|\,
\ arg\,\la=\theta_j\}, j=1,\dots,N,$ in the upper (lower) half plane
$\mathcal{C}^+(\mathcal{C}^-)$ such that
$$
0<\theta_j<\theta_{j+1}<2\pi/p, \ j=1,\dots, N-1; \ \
\max(\theta_1,\theta_{j+1} -\theta_j, \pi-\theta_N)<2\pi/p,
$$
and
$$
\|\mathcal{H}PT^{-1}(\lambda)\mathcal{H}P\|\le c(1+|\la|^m), \quad \mbox{for}\ \la\in \gamma_j,
$$
with some constants $c$ and $m$. }

\begin{proposition} If $T(\la)$ is a strongly elliptic pencil then the
Keldysh-Agmon condition implies the resolvent growth condition,
moreover, one can take in (\ref{12::4.1}) $m=n$.
\end{proposition}
\begin{proof}
First, notice that (\ref{12::4.2}) implies that the generalized and the classical
spectra of $T(\lambda)$ coincide (Theorem 1.6).
The essense of the matter is that condition (a) together
with $\mathcal{D}(G)\supset\mathcal{D}(\mathcal{H}P)$ imply that $\mathcal{H}PT^{-1}(\lambda)\mathcal{H}P$ is an
$\mathcal{H}$-valued
meromorphic operator function of order $2/p$. The proof is based on the
results of {\sc Keldysh} \cite{12::Ke}, {\sc Agmon} \cite{12::Ag},
{\sc Matsaev} \cite{12::Mat} et. al. (see historical remarks and details
in \cite{12::S3}, \S 2). Now, if
$$
F(\la) =\mathcal{H}PT^{-1}(\lambda)\mathcal{H}P x(\la)
$$
is holomorphic in $\mathcal{C}^+$ and $x(\la)$ is a polynomial then condition (b)
and the Phragmen-Lindel\"of theorem imply that $F(\la)$ has a polynonial
growth in $\mathcal{C}^+$. According to Theorem 1.7
\begin{equation}
|F(\la)| < c(1+|\la|^n), \qquad \la\in\mathbb{R},\  |\la|>r_0,\  n=\deg x(\la).
\label{12::4.3}
\end{equation}
Since $F(\la)$ is of order zero in $\mathcal{C}^+$, by virtue of the
Phragmen-Lindel\"of theorem the estimate (\ref{12::4.3}) holds asymptotically for all
$\la\in\mathcal{C}^+$.
\end{proof}

\subsection{Half-range completeness and minimality} In what follows
we consider for simplicity a generic situation when $T(\la)$
has only semi-simple real eigenvalues of definite type.
For a pencil of the form (\ref{0.2})
this is true according to Theorem 2.2 for all values of $\omega$ with the possible
exception of isolated resonant frequences $\omega_k\to\infty$.

Let $T(\la)$ have discrete spectrum and let the eigenvalues of $T(\la)$ be
numerated according to their geometric multiplicity (i.e. every eigenvalue
$\la_k$ is repeated $n=\operatorname{nul}\,T(\la_k)$ times). In this case we have
a one-to-one correspondence between the eigenvalues $\la_k$ and
canonical Jordan chains of the form (\ref{12::2.2}). As we agreed, all the real
eigenvalues are supposed to be semi-simple. The eigen-elements
corresponding to every real eigenvalue are assumed to form a normal
canonical system (see Section 2). Take all the chosen Jordan chains of
$T(\la)$ corresponding to the eigenvalues from the open upper (lower)
half-plane and
all the eigen-elements  corresponding to the real eigenvalues  of
positive (negative) type. Denote the system consisting of all these
elements by $E^+ (E^-)$ and call it the first (second) half of the
root elements of $T(\la)$.

Let us recall the well-known definitions. A system $\{e_k\}^{\infty}_1$
is said to be {\it minimal} in Hilbert space $\mathcal{H}$ if there exists an
adjoint system $\{e_k^*\}^{\infty}_1$  such that $(e_k,e^*_j)=\delta_{kj}$,
where $\delta_{jk}$ is the Kronecker symbol. Equivalently, $\{e_k\}_1^{\infty}$
is minimal if any element $e_k$ is not contained in the closed linear
span of the other ones. A system $\{e_k\}_1^{\infty}$ is said to be
{\em complete} in $\mathcal{H}$ if there is no non-zero element in $\mathcal{H}$ which is
orthogonal to all the elements of the system.

\begin{theorem} The first and the second half of the root elements of
a pencil $T(\la)$ form minimal systems in $\mathcal{H}$ provided $T(\la)$ is
strongly elliptic and has discrete spectrum.
\end{theorem}
\begin{proof} Let us work with the system $E^+$, for example. By
virtue of Propostition 1.4 there is a number $r_0>0$
such that the pencil $T_1(\la)=T(\la-r_0)$ has only positive
eigenvalues on the  real axis (to prove the minimality  of $E^-$ one
should consider the pencil $ T_1(\la)=T(\la+r_0))$. The Jordan chains
$x_k^0,\dots,x_k^p$ of the pencil $T(\la)$ are changed after this
transformation in the following way
$$
\xi_k^0=x_k^0\ ,
\ \xi_k^1=x_k^1-r_0^{-1}x_k^0,\ \dots\ ,\xi_k^p=x_k^p-r_0^{-1}x_k^{p-1}
-\dots-r_0^{-p}x_k^0,
$$
while the sign characteristics of the pairs $\{\la_k,y_k^0\}$ and
$\{\la_k+r_0, y_k^0\}$ are the same. Hence, it suffices to prove the
minimality for the case when $T(\la)$ has only positive real eigenvalues
and $S>0$. Let us consider the system
\begin{equation}
{\bf x}^s_k=\left(\begin{array}{c}F^{1/2}(\la_kx_k^s+x_k^{s-1})\\S^{1/2}x_k^s
\end{array}\right),
\qquad x_k^s\in E^+. \label {12::4.4}
\end{equation}
It is an easy exercise to show that the ${\bf x}_k^s$ are the root elements of
the operator
${\bf T}$ defined by (\ref{12::3.8})
(${\bf T}$ is the linearization of
$T(\la)$). As we mentioned ${\bf T} $ is a symmetric operator
in the Krein space $\mathcal{K}=\{\mathcal{H}^2,{\bf J}\}$ with the fundamental
symmetry ${\bf J} =\left(\begin{array}{cc} I&0\\0&-I\end{array}\right)$. From this
we have the biorthogonality relationships (see, e.g., \cite{12::AI}, Ch.1)
$$
({\bf J}{\bf x}_k^s, {\bf x}_j^h) = 0 \qquad \mbox{for all} \ \ j,k,s,h
$$
except for the case $k=j$ and $\la_k\in\mathbb{R}$. For $\la_k\in\mathbb{R}$ we have
$$
\begin{array}{rl}
&({\bf J}{\bf x}_k, {\bf x}_k)= \la^2_k(Fx_k, x_k)-(Sx_k,x_k)\\ \phantom{.}\\
&=2\la_k^2(Fx_k, x_k) +\la_k(Gx_k, x_k) =\la_k(T'(\la_k)x_k, x_k) =
\la_k\varepsilon_k,
\end{array}
$$
where $\varepsilon_k$ is the sign characteristic of the pair
$\{\la_k,y_k\}$. Hence,
\begin{equation}
({\bf J}{\bf x}_k^s, {\bf x}_j^h)=\delta_{kj}\la_k\varepsilon_k, \qquad
x_k^s, x_k^h \in E^+, \label {12::4.5}
\end{equation}
where $\varepsilon_k =0$ for the nonreal $\la_k$ and $\la_k\varepsilon_k>0$ for
$\la_k\in\mathbb{R}$ and $x_k\in E^+$. Let $x$ be a finite linear combination
of elements (\ref{12::4.4})
\begin{equation}
{\bf v}: = \left(\begin{array}{c}v_1\\v_2\end{array}\right)=\sum c_k^s
\left(\begin{array}{c} F^{1/2}(\la_kx_k^s+x_k^{s-1})\\S^{1/2}x_k^s
\end{array}\right) =:\sum c_k^s{\bf x}_k^s.
\label {12::4.6}
\end{equation}
From (\ref{12::4.5}) we obtain $\|v_1\|\ge\|v_2\|,$ therefore $\|{\bf v}\|\le 2\|v_1\|.$
Recall that the system ${\bf x}_k^s$ is minimal as the system of the root
elements of the operator ${\bf T}$ with discrete spectrum. Then the
inequality $\|v_1\|\le \|{\bf v}\|\le 2\|v_1\|$  implies (if we use the
second definition of minimality) that
$$
\{F^{1/2}(\la_kx_k^s-x_k^{s-1})\}, \qquad x_k^s\in E^+,
$$
is a minimal system in $\mathcal{H}$. Hence $E^+$ is minimal in $\mathcal{H}$, too.
\end{proof}
\begin{theorem} The first and the second half of the root elements of
a pencil $T(\la)$ form  complete systems in $\mathcal{H}_1$ provided $T(\la)$ is
strongly elliptic and the Keldysh-Agmon condition holds.
\end{theorem}

\begin{proof} As before, we deal with the system $E^+$. Suppose that there
is an element $f\in\mathcal{H}_1$ such that
\begin{equation}
(f,x_k^s)_1=(\mathcal{H}P f,\mathcal{H}P x_k^s)=0 \qquad \mbox{for all} \ x_k^s\in
E^+. \label{12::4.7}
\end{equation}
Choose a number $r_0$ such that $T(\la)>0$ for $\la>r_0$ and consider
the function
\begin{equation}
F(\la) =\frac1{\la-r_0}(\mathcal{H}PT^{-1}(\lambda)\mathcal{H}P g, g), \quad g =\mathcal{H}P f\in\mathcal{H}.\label{12::4.8}
\end{equation}
The principal part of $F(\la)$ in a neighborhood of a real pole
$\la_k$ has the representation
$$
\sum\frac{\varepsilon_k(g, \mathcal{H}P x_k)(\mathcal{H}P x_k, g)}{(\la-r_0)(\la-\la_k)}
$$
where $\varepsilon_k$ are the sign characteristics corresponding to the eigen-pair
$\{\la_k, x_k\}$.
Due to (\ref{12::4.7}) all the terms with $\varepsilon_k>0$ in the last expression are
equal to zero. Since $\la_k-r_0<0$, all the residues of $F(\la )$ at the
poles $\la_k\in\mathbb{R}$ are non-negative. The residue at the additional
pole $\la=r_0$ is non-negative, too. Taking into account the
representation (\ref{12::2.1}) of $T^{-1}(\lambda)$ in a neighbourhood of a non-real pole
$\la_k\in \mathcal{C}^+$ and assumption (\ref{12::4.7}), we find that $F(\la)$ is
holomorphic in $\mathcal{C}^+$. By the Schwarz symmetry principle
it is holomorphic in $\mathcal{C}^-$. Proposition 4.1 gives us
$F(\la)=O(\la^{-1})$
when $\la\to\infty$ uniformly in $\mathcal{C}$.

Let us show that the residue of $F(\la)$ at $\infty$ equals zero.
Given $\varepsilon >0$ we can find $g_0\in\mathcal{H}_1$ such that
$\|g-g_0\|_1<\varepsilon$. If we put $g_0$ in (\ref{12::4.8}) instead of $g$ then by
virtue of Theorem 1.7 the corresponding function vanishes at $\infty$
as $O(\la^{-3})$ when $\la\to\pm\infty$ uniformly in $\mathcal{C}$. Therefore,
$F(\la)=o(\la^{-1})$ as $\la\to\infty$ uniformly in $\mathcal{C}$, i.e. the residue
at $\infty$ is equal to zero.
Now, recall
that all the residues of $F(\la)$ at the finite poles are
non-negative. This is possible only if all they are equal to zero,
in particular,
$$
(\mathcal{H}P T(r_0)\mathcal{H}P g,\,g)=0.
$$
This implies $g=0$.
\end{proof}
\begin{corollary} The first and the second half of the root elements
of a strongly elliptic pencil $T(\la)$ satisfying the Keldysh-Agmon condition
form complete and minimal systems in spaces $\mathcal{H}_{\theta}$ for all
$0\le\theta\le 1$.
\end{corollary}
\begin{proof} It follows from the definitions: if a system is minimal
(complete) in $\mathcal{H} (\mathcal{H}_1)$ then it has the same property in $\mathcal{H}_{\theta}$
for $\theta>0\ (\theta<1)$. Now apply Theorems 4.2 and 4.3.
\end{proof}

\subsection{Factorization} The obtained results enable us to
construct a divisor of an elliptic pencil.
\begin{theorem} Let $T(\la)$ be a strongly elliptic pencil satisfying the
Keldysh-Agmon condition. Then
\medskip
\begin{equation}
T(\la)x = (\la -Z_1)F(\la-Z)x \label {12::4.9}
\end{equation}
where

(a) $Z$ and $Z_1$ admit a representation
$$
r-Z=K_0\mathcal{H}P, \qquad r-Z_1=\mathcal{H}P K_1
$$
with bounded and boundedly invertible in $\mathcal{H}$ operators $K_0$ and $K_1$,
provided $r\in\mathbb{R}$ is not an eigenvalue of $T(\la)$. In particular, $Z$ is
a closed operator on $\mathcal{H}$ with domain $\mathcal{D}(Z)=\mathcal{H}_1$\ ;

(b) the spectra of $T(\la)$ and $\la-Z$ coincide in the upper half-plane,
while on the real axis $\la-Z$ inherits only the positive type
eigen-pairs of $T(\la)$, i.e. the system of the root functions of $Z$
coincides with the first half of the root functions of $T(\la)$;

(c) $iZ$ generates a holomorphic semigroup in all spaces
$\mathcal{H}_{\theta},0\le\theta\le1$.

Equality (\ref{12::4.9}) holds for all $x\in\mathcal{H}_1$ and is understood in sense of
operators  acting from $\mathcal{H}_1$ to $\mathcal{H}_{-1}$. Factorization (\ref{12::4.9}) with
property (b) is unique.
\end{theorem}
\begin{proof} As in Theorem 4.2 we may assume that $T(\la)$ has only
positive eigenvalues, otherwise we have to work with $T(\la-r_0),\ r_0\gg
1$.

Let us consider the set of all finite linear combinations of elements
(\ref{12::4.4}). The elements of this set have representation (\ref{12::4.6}). If the
system $E^+$ is complete in $\mathcal{H}_1$ then the system $\mathcal{H}P(E^+)$ is
complete in $\mathcal{H}$. Therefore, Theorem 4.3 implies that the linear span of the
elements
$\{v_2\}$ in (\ref{12::4.6}) form a dense subset in $\mathcal{H}$ as well as the elements
$\{v_1\}$. Define the operator $K$ by
\medskip
\begin{equation}
Kv_1 = v_2. \label{12::4.10}
\end{equation}
It was shown in Theorem 4.2 that $\|v_2\|\le\|v_1\|$. Hence, $K$ is
densely defined on $\mathcal{H}$ and can be extended as a contraction on the
whole $\mathcal{H}$. The image of $K$ is dense in $\mathcal{H}$.

Denote by $E^0$ the subsystem of $E^+$ consisting of all elements
$x_k^s\in E^+$ corresponding to the non-real eigenvalues. Let $\mathcal{H}_0$
be the closure in $\mathcal{H}$ of the linear span generated by $E^0$.  Denote $\kappa
=\operatorname{codim}\mathcal{H}_0$
($\kappa$ coincides with the number of positive type eigenvalues counting
with geometric multiplicity). It is clear from (\ref{12::4.5}) that
$\|Kv_1\|=\|v_1\|$ for $v_1\in\mathcal{H}_0$, hence, $K(\mathcal{H}_0)$ is a closed
subspace in $\mathcal{H}$. By virtue of Corollary 4.4 the system $\mathcal{H}P(E^+)$ is
minimal and complete in $\mathcal{H}$. This implies that $\operatorname{codim}K(\mathcal{H}_0)=\kappa$.
Hence, there is a unitary operator $U$ in $\mathcal{H}$ such that the
restriction of $U$ onto $\mathcal{H}_0$ coincides with $K$, i.e. $U-K$ is of
finite rank. We noticed already that the image of $K$ is dense in
$\mathcal{H}$. Now, it follows from the Fredholm theorem
that $K$ is boundedly invertible on $\mathcal{H}$.

Denote $Z=F^{-1/2}K^{-1}S^{1/2}$, where $S=T(0)>0$. From (\ref{12::4.6}) and
(\ref{12::4.10}) we have
\begin{equation}
Zx_k^s=\la_kx_k^s + x_k^{s-1}, \qquad x_k^s\in E^+. \label {12::4.11}
\end{equation}
Since $x_k^s$ are the root elements of $T(\la)$, we have
$$
(FZ^2+GZ+S)x_k^s=0\qquad\mbox{for all} \ \, x_k^s\in E^+.
$$
The linear span of $E^+$ is dense in $\mathcal{H}_1$, hence,
\begin{equation}
-(FZ+G)Z=S, \label{12::4.12}
\end{equation}
where the equality is understood in the sense of operators acting from
$\mathcal{H}_1$ to $\mathcal{H}_{-1}$. Denoting $Z_1=-(FZ+G)F^{-1}$ we obtain from
(\ref{12::4.12}) the factorization
$$
T(\la)=(\la-Z_1)F(\la -Z).
$$
As $\mathcal{D}(S)=\mathcal{D}(H)$ we have $ S^{1/2} =K_2\mathcal{H}P$ with a bounded and
boundedly  invertible operator $K_2$. Hence, $Z=K_0\mathcal{H}P$ with
$K_0=F^{-1/2}K^{-1}K_2$. We have also
$$
Z_1=SH^{-1/2}K_0^{-1}F^{-1}=\mathcal{H}P K_2^*K_2K_0^{-1}F^{-1}=:H^{1/2}K_1.
$$
Thus (a) is proved. The assertion (b) follows from (\ref{12::4.11}). The
uniqueness of a factorization with property (b) follows from the
completeness of the system $E^+$. It remains to prove (c). To this
end we obtain from (\ref{12::4.9})
\begin{equation}
(\la -Z)^{-1} = T^{-1}(\lambda) (\la -\mathcal{H}P K_1)F. \label{12::4.13}
\end{equation}
Applying Theorem 1.7 we obtain
\begin{equation}
\|(\la -Z)^{-1}\| \le C|\la|^{-1}, \qquad \la\in\Lambda_{\phi},\quad |\la|>r_0.
\label{12::4.14}
\end{equation}
Moreover, $(\la -Z)^{-1}$ is holomorphic in $\mathcal{C}^-$. By virtue of (\ref{12::4.13})
and Proposition 4.1 $(\la-Z)^{-1}$ has a polynomial growth in $C^-$.
Consequently, (\ref{12::4.14}) holds for all $\la\in\Lambda_{\phi}\cup \mathcal{C}^-,
|\la|>r_0$. Thus, $iZ$ generates a holomorphic semigroup in $\mathcal{H}$.
Since $Z:\,\mathcal{H}_1\to\mathcal{H}$ is an isomorphism, $iZ$ possesses the same property
in $\mathcal{H}_1$. Applying the interpolation theorem we obtain assertion (c).
This ends the proof.
\end{proof}

\subsection{The Mandelstam hypothethis}
\setcounter{equation}{0}
\setcounter{theorem}{0}

In this section we solve the problem
\begin{equation}\label{12::5.1}
T\left(i\frac d{dy}\right)\,u(y) =0
\end{equation}
\begin{equation}\label{12::5.2}
u(0) =f
\end{equation}
\begin{equation}\label{12::5.3}
u(y)=u_+(y)+u_0(y),\quad u_0(y)\to 0\ \,\mbox{as}\ y\to\infty,
\end{equation}
where $u_+(y)$ is a linear combination of outgoing waves (\ref{0.3}).

Below we clarify the understanding of this problem and prove the
solvability in the classical sense and the uniqueness in the generalized
sense. We may say that (\ref{12::5.1})-(\ref{12::5.3}) is the half-range Cauchy problem
because instead of two initial conditions at $y=0$ we set only one,
but force a solution to behave at $\infty$ in a special way.

Further we denote by $C_2(a,b;\mathcal{H})$ the space of continuous on $(a,b)$
$\mathcal{H}_2$-valued functions whose derivatives $v'(y)$ and $v''(y)$ exist in
$\mathcal{H}_1-$ and $\mathcal{H}$-norm and belong to $C(a,b;\mathcal{H}_1)$ and $C(a,b;\mathcal{H})$,
respectively (the continuity at the ends of $(a,b)$ is not assumed!)

\begin{theorem} Let $T(\la)$ be strongly elliptic and assume that the
Keldysh-Agmon condition holds.  Then for any $\theta\in[0,1]$ and any
$f\in\mathcal{H}_{\theta}$ there exists a function $u(y)\in C_2(0,\infty;\mathcal{H})$
satisfying equation (\ref{12::5.1}),
having representation (\ref{12::5.3}) with exponentially decaying $\|u_0(y)\|_2$
and satisfying initial condition (\ref{12::5.2}) in the following sense
\begin{equation}
\lim_{y\to+0}\|u(y) -f\|_{\theta} =0. \label{12::5.4}
\end{equation}
\end{theorem}
\begin{proof}
We find a solution of the problem in question by means of the operator
$Z$ which was constructed in Theorem 4.5. Namely, denote
\begin{equation}
u(y) =\frac 1{2\pi i} \left( \int\limits_{\gamma} +\int\limits_{\Gamma}\right)
e^{i\la y}(\la-Z)^{-1}f\,d\la, \label{12::5.5}
\end{equation}
where $\gamma$ surrounds only real eigenvalues of $Z$, while $\Gamma$ lies in the
upper half-plane and is asymptotically directed along the rays $arg\,\la=\delta$
and $arg\,\la=\pi-\delta$ with sufficiently small $\delta>0$. By virtue of
Theorem 4.5 $iZ$ generates a holomorphic semigroup in $\mathcal{H}_{\theta}$, hence
integral (\ref{12::5.5}) is well defined and (\ref{12::5.4}) holds (see [Yo, Ch. 9]).
Moreover, the functions $Z^ku^{(j)}(y)$ are well defined for $y>0,
k,j\ge0$ and are continuous in $\mathcal{H}_{\theta}\subset\mathcal{H}$. Since
$Z:\,\mathcal{H}_1\to\mathcal{H}$ is an isomorphism, we obtain that $u^{(j)}(y)$
are continuous for $y>0$ in $\mathcal{H}_1$. The equality
$$
-(G+FZ)Zx=Sx
$$
holds for all $x\in\mathcal{H}_1$, in particular, for $x\in\mathcal{H}_2$.
As $u'(y)=iZu(y)$ we obtain that $iGu'(y)-Fu''(y)\in C(0,\infty;\mathcal{H})$,
equation (\ref{12::5.1}) is satisfied in $\mathcal{H}$ and $u(y)\in C_2(0,\infty;\mathcal{H})$.
Representation (\ref{12::5.3}) with an exponentially decaying function
$u_0(y)$ follows from (\ref{12::5.5}).
\end{proof}
\begin{theorem} A generalized solution u(y) of problem
(\ref{12::5.1})--(\ref{12::5.3}), such that $u(y)\in L_1(0,\varepsilon;\mathcal{H}_1)$
with some $\varepsilon>0$, is unique.
\end{theorem}
\begin{proof}
In Section 3 we assumed that generalized solutions $u(y)$ belong to
$ W_2^1(0,\infty;\mathcal{H})$.
Here our assumptions are weaker: we assume only
$u(y)\in W_2^1(\varepsilon,\infty;\mathcal{H})$ and $u(y)\in L_1(0,\varepsilon;\mathcal{H}_1)$
for any $\varepsilon>0$. Certainly, if $u(y)\in W_2^1(0,\infty;\mathcal{H})$
then $u(y)\in L_2(0,\varepsilon;\mathcal{H}_1)$ and $u(y)\in L_1(0,\varepsilon;\mathcal{H}_1)$.
By the definition of a generalized solution, the equation
$$
-Fu''(y)=-iFu'(y)+Su(y)
$$
is satisfied in the sense of
$\Wr^{{\tiny\rm o}\phantom{.}\phantom{.}\phantom{.}\phantom{.}\phantom{.}}(\varepsilon,\infty;\mathcal{H})$.
The right hand side belongs to $L_2(\varepsilon,\infty;\mathcal{H}_{-1})$, hence,
so does the left hand side.

Suppose that $\|u(y)\|\to 0$ as $y\to 0$. For $\lambda\in \mathcal{C}^+$ we have
\begin{equation}\label{12::5.6}
\begin{array}{c}
0=\int\limits_\varepsilon^\infty T_\omega\left( i\frac{d}{dy}\right) u(y)
e^{i\lambda y}dy=\\ \phantom{.}\\
=e^{i\lambda\varepsilon}\left(Fu'(\varepsilon)-i(\lambda
F-G)u(\varepsilon)\right)+T_\omega(\lambda)\hat u_\varepsilon(\lambda)=0,
\end{array}
\end{equation}
where
$$
\hat u_\varepsilon(\lambda)=\int\limits_\varepsilon^\infty u(y)e^{i\lambda y}dy.
$$
We consider (\ref{12::5.6}) as an equality in $\mathcal{H}_{-1}$. Since $u(y)$ is locally
integrable at zero as a function with values in $\mathcal{H}_1$ we can take the
limit as $\varepsilon\to 0$ and obtain
$$
T_\omega(\lambda)\hat u_0(\lambda)=-Fu'(0)=g\in\mathcal{H}_{-1}.
$$
Therefore, $\hat u_0(\lambda)=T_\omega^{-1}(\lambda)g$. Let us consider the
function
$$
F(\lambda)=\frac{1}{\lambda+r_0}\left(T_\omega^{-1}(\lambda)g,g\right ).
$$
It follows from (\ref{12::4.10}) that $F(\lambda)$ is bounded in $\mathcal{C}^+$ and has
finitely many poles on $\mathbb{R}$ with positive residues provided $r_0$ is
sufficiently large. Repeating the arguments of Theorem 4.9 we obtain
$F(\lambda)\equiv 0$. Hence, $\hat u(\lambda)\equiv 0$ and $u(y)\equiv0$.
\end{proof}

\subsection{Application to the Lame system of the elasticity theory}
\setcounter{equation}{0}
\setcounter{theorem}{0}

Small oscillations of an elastic medium are described by the system of
equations (see the books of {\sc Landau} and {\sc Lifshitz}
\cite{12::LL} or {\sc Kupradze} et.al. \cite{12::Ku})
$$
\rho\frac{\partial ^2w}{\partial t^2}+Lw=0,
$$
where $w=w(t,x)=(w_1,w_2,w_3)$ is the displacement vector, $\rho=\rho(x)$
is the density of the medium, $L$ is the operator matrix with the entries
$$
L_{kj}(D)=(\hat\lambda+\hat\mu)D_kD_j+\delta_{kj}\hat\mu (D_1^2+D_2^2+D_3^2),
\quad D_k=i\frac{\partial}{\partial x_k},
$$
and $\hat\lambda,\hat\mu$ are the Lame constants. We suppose that the space
variable\break$x=(x_1,x_2,x_3)$ belongs to the wave-guide domain
$Q=[0,\infty) \times \Omega$ where $\Omega$ is a bounded domain in the plane
$(x_2,x_3)$. Separating the time variable $w=ue^{i\omega t}$ we obtain the
stationary equation with given frequency $\omega$
\begin{equation}
(L-\omega^2\rho)u=0.\label {6.1}
\end{equation}
We have to impose with this equation boundary and initial conditions.
We pose on the lateral surface of the half-cylinder $Q$ homogeneous
conditions, since $Q$ is a wave-guide domain.
For simplicity let us consider the Dirichlet
boundary conditions
\medskip
\begin{equation}
u(x_1,x_2,x_3)|_{(x_2,x_3)\in\partial\Omega}=0\quad\forall x_1\ge 0.
\label {6.2}
\end{equation}
At the base of $Q$ we assume that
\begin{equation}
u(0,x_2,x_3)=\phi(x),\label{6.3}
\end{equation}
where $\phi(x)$ is a given function. We rewrite equation (\ref{6.1})
in the form
\begin{equation}
T_{\omega}\left(i\frac{d}{dy}\right)u=
-F\frac{d^2u}{dy^2}+iG\frac{du}{dy}+(H-\omega^2R)u=0, \label{6.4}
\end{equation}
where $y=x_1$,
$$
F=\left(\begin{array}{ccc} \hat\la+2\hat\mu&0&0\\0&\hat\mu&0\\0&0&\hat\mu\end{array}\right),
\quad G=i(\hat\la+\hat\mu)\left(\begin{array}{ccc} 0&D_2&D_3\\mathcal{D}_2&0&0\\mathcal{D}_3&0&0\end{array}\right),
$$
$$
\begin{array}{c}
H=\left(\begin{array}{ccc}
\mu\mathcal{D}elta&0&0\\0&\hat\mu\mathcal{D}elta+(\hat\la+\hat\mu)D_2^2&(\hat\la+
\hat\mu)D_2D_3\\
0&(\hat\la+\hat\mu)D_2D_3&\hat\mu\mathcal{D}elta+(\hat\la+\hat\mu)D_3^2\end{array}\right),\\
\phantom{.}\\
R=\rho(x)I,\quad \mathcal{D}elta=-(D_2^2+D_3^2),
\end{array}
$$
and $I$ is the identity matrix. We suppose that the operators $F$,$G$,$H$
act in the Hilbert space $\mathcal{H}=\left[L_2(\Omega)\right]^3$.

We have to specify a domain of the main operator $H$. Taking into account
boundary conditions (\ref{6.2}) we define
$$
\mathcal{D}(H)=\left\{v|\ v\in\left[W_2^2(\Omega)\right]^3,v|_{\partial\Omega}=0
\right\},
$$
$$
\mathcal{D}(G)=\left\{v|\ v\in\left[W_2^1(\Omega)\right]^3,v|_{\partial\Omega}=0
\right\}=:\left[\Wo^{{\tiny\rm o}\phantom{.}\phantom{.}}(\Omega)\right]^3,
$$
where $\left[W_2^k(\Omega)\right]^3$ are the Sobolev spaces of vector
functions on $\Omega\subset\mathbb{R}^2$.

We notice that the operator $H$ is positive, since
$$
(Hv,v)=\int\limits_\Omega Hv(x)\overline{v(x)}dx=\hat\mu\left(
\sum\limits_{j=1}^3\|D_2v_j\|^2+\|D_3v_j\|^2\right)+
$$
$$
+(\hat\la+\hat\mu)\|D_2v_2+D_3v_3\|^2,
$$
where $\|f\|^2=\int\limits_\Omega|f|^2dx$. Taking into account boundary
condition (\ref{6.2}) and the Friedrichs inequality we obtain
$(Hv,v)\ge\varepsilon\|v\|^2$ with some $\varepsilon>0$.
The operator $G$ is symmetric, as
$$
(Gv,v)=2(\hat\la+\hat\mu)Re(iD_2v_2+iD_3v_3,v_1).
$$

The operator $F$, obviously, is uniformly positive and bounded provided\break
$\rho(x)\ge\varepsilon>0$ is a measurable bounded function on $\Omega$.

It is well-known (see \cite{12::Ag} or \cite{12::Tr}, Ch. 5) that $H+cI$ is invertible in
$[L_2(\Omega)]^3$ provided $c\ge 0$ and $\Omega$ is a smooth domain. Therefore,
$H$ is a self-adjoint operator if $\Omega$ is smooth. This is not always true,
if $\Omega$, for instance, has corner points (see examples in the paper of
{\sc Kondratiev} and {\sc Shkalikov} \cite{12::KoS}). In
this case let us consider the Friedrichs extention $H_F$ of the operator $H$.
It is known (see \cite{12::RN}, Ch. 8) that it is the only extension which possesses
the property $\mathcal{D}(H_F^{1/2})=\mathcal{D}(G)$.

Denoting $H_F=H$ we remark that all the assumptions on the operator
coefficients claimed at the beginning of Section 1 are fulfilled.

We note that in the case of a non-smooth domain $\Omega$
there is no precise information on $\mathcal{D}(H_F)$, however, we do know that
$$
\mathcal{D}(H^{1/2}_F)=\mathcal{D}(G)=\left[\Wp^{{\tiny\rm o}\phantom{.}\phantom{.}}(\Omega)\right]^3.
$$
Actually, the domain of $H_F(=H)$ is not involved in our considerations,
the know\-ledge of $\mathcal{D}(H^{1/2})=\mathcal{H}_1$ is the only important information
which we need.

Now let us prove that the pencil corresponding to equation (\ref{6.4}) is
regular elliptic
in the case of a smooth domain and strongly elliptic otherwise.

\begin{proposition}
The pencil $T_{\omega}(\lambda)$ generated by the Lame system and the
Dirichlet boundary conditions is strongly elliptic.
\end{proposition}
\begin{proof}
We have
\begin{equation}\label{6.5}
\begin{array}{c}
(T_{\omega}(\lambda)v,v)=\la^2\left[(\hat\la+2\hat\mu)\|v_1\|^2+
\hat\mu(\|v_2\|^2+\|v_3\|^2)\right]+\\ \phantom{.}\\
+2\la(\hat\la+\hat\mu)Re\left[(iD_2v_2+iD_3v_3,v_1)\right]+
(\hat\la+\hat\mu)\|D_2v_2+D_3v_3\|^2+\\ \phantom{.}\\
+\hat\mu\sum\limits_{j=1}^3\left(\|D_2v_j\|^2+\|D_3v_j\|^2\right)-
\omega^2\|\rho^{1/2}v\|^2\ge\\ \phantom{.}\\
\ge\hat\mu\la^2\sum\limits_{j=1}^3\left(\|v_j\|^2+\|D_2v_j\|^2+
\|D_3v_j\|^2\right)-\omega^2\|\rho^{1/2}v\|^2\ge\\ \phantom{.}\\
\ge\varepsilon[\la^2(v,v)+(Hv,v)]-\omega^2(\rho v,v),\quad \la\in\mathbb{R},
\quad\|v\|^2=\|v_1\|^2+\|v_2\|^2.
\end{array}
\end{equation}
It is known (see \cite{12::Tr}, Ch 4.10) that the embedding
$I:[\Wo^{{\tiny\rm o}\phantom{.}\phantom{.}}(\Omega)]^3\to[L_2(\Omega)]^3$ is compact for
any bounded domain
$\Omega$ (we pay attention that if we consider, say, Neuman boundary
conditions, then we have to assume in addition that $\Omega$ is a
Lipshitzian domain).
By virtue of
Proposition 1.4 we obtain that $T_{\omega}(\la)$ is strongly elliptic.
\end{proof}

We remark
that Proposition 6.1
can be also proved in the case of an unbounded domain
$\Omega$ if we assume
$\rho(x)\to 0$ as $|x|\to\infty$.

\begin{proposition} The pencil $T_\omega(\lambda)$ is regular
elliptic if a domain $\Omega$ is smooth. Moreover, estimate (\ref{12::1.3o}) holds
asymptotically outside any double sector containing the imaginary axis and
the Keldysh-Agmon condition holds.
\end{proposition}
\begin{proof}(Cf.\cite{12::KO2}).
Denoting $-iD_k=\xi_k$, let us calculate the principal
characteristic symbol of the Lame system (the principal symbol does not
depend on $\omega$ and we can assume $\omega=0$). We have
$$
\det T_0(\lambda)=det\left[\lambda^2 \left(
\begin{array}{ccc}\hat\lambda+2\hat\mu&0&0\\0&\hat\mu&0\\0&0&\hat\mu
\end{array}\right)+\lambda(\hat\lambda+\hat\mu)\left(\begin{array}{ccc}
0&\xi_2&\xi_3\\\xi_2&0&0\\\xi_3&0&0\end{array}\right)+\right.
$$
$$
\left. +\left(\begin{array}{ccc}\hat\mu|\xi|^2&0&0\\0&\hat\mu|\xi|^2+
(\hat\lambda+\hat\mu)\xi_2^2&(\hat\lambda+\hat\mu)\xi_2\xi_3\\0&
(\hat\lambda+\hat\mu)\xi_2\xi_3&\hat\mu|\xi|^2+(\hat\lambda+\hat\mu)\xi_3^2
\end{array}\right)\right]=\hat\mu^2(\hat\lambda+2\hat\mu)(\lambda^2+|\xi|^2)^3,
$$
where $|\xi|^2=\xi_2^2+\xi_3^2$.

Hence, the ellipticity condition in the sense of \cite{12::AN} and \cite{12::AV} holds
for all $\lambda$ not belonging to the imaginary axis.
It is well known (see, e.g., \cite{12::LM}) that the Dirichlet boundary
condition satisfies the Lopatinskii condition for all elliptic systems.
Hence, the problem (\ref{6.1}), (\ref{6.2}) is regular elliptic and
according to the results of \cite{12::AN} and \cite{12::AV}
estimate (\ref{12::1.3o}) holds outside arbitrary small sector containing the
imaginary axis. Since $T_\omega(\lambda)$ is a seladjoint pencil, estimate
(\ref{12::1.3o}) implies
$$
||HT^{-1}_{\omega}(\la)||+||T^{-1}_{\omega}(\la)H||\le const,
$$
and, by virtue of the interpolation theorem, we have
\begin{equation}
||H^{1/2}T^{-1}_{\omega}(\la)H^{1/2}||\le const,\quad |\la|>r_0,
\label{6.6c}
\end{equation}
at any ray in $\mathcal{C}$ with exception of the imaginary exis. According to the
Weyl asymptotic formula for eigenvalues of the elliptic operators,
we have the estimate (\ref{12::4.2}) with $p=1$. Hence, if $\Omega$
is a smooth domain then the Keldysh-Agmon condition for the Lame system
is valid.
\end{proof}

In the case of a non-smooth domain we are able to prove the validity of the
Keldysh-Agmon condition only under additional constraints on the Lame
constants.
\begin{proposition}
Let $\Omega$ be a bounded domain in $\mathbb{R}^2$. If $\hat\mu>\sqrt{2}\hat\la$ then
estimate(\ref{12::1.8}) is satisfied in a double sector $\Lambda_\phi$ with
some $\phi>\pi/4$ and the Keldysh-Agmon condition holds.
\end{proposition}
\begin{proof}
Let us estimate the quadratic form $(T_\omega(\la)v,v)$ at the ray
$\la=e^{i\pi/4}\zeta,\quad \zeta>0$. Suppose $\omega =0$. Bearing in
mind (\ref{6.5}) we obtain
$$
Re\ e^{-i\pi/4}\left(T_0(e^{i\pi/4}\zeta)v,v\right)\ge\\
$$
$$
\frac{\sqrt{2}}2\hat\mu
\left(\zeta^2\sum\limits_{j=1}^{3}\|v_j\|^2+\|D_2v_j\|^2+\|D_3v_j\|^2\right)+
$$
$$
+\frac{\sqrt{2}}2(\hat\la+\hat\mu)\left(\zeta^2\|v_1\|^2-2\sqrt{2}|\zeta|\ \|D_2v_2
+D_3v_3\|\ \|v_1\|+\|D_2v_2+D_3v_3\|^2\right).
$$
Taking into account the inequality
$$
2(ac+bc)\le\sqrt{2}(a^2+b^2+c^2),\quad a,b,c>0
$$
we can estimate the second summand as follows
$$
\ge-(\hat\la+\hat\mu)(2-\sqrt{2})\frac{\sqrt{2}}2\left(
\zeta^2\|v_1\|^2+\|D_2v_2\|^2+\|D_3v_3\|^2\right).
$$
Therefore,
\begin{equation}
Re\ e^{i\pi/4}\left(T_\omega(e^{i\pi/4}\zeta)v,v\right)\ge\varepsilon
\left(\|v\|_1+\zeta^2\|v\|\right)-\omega^2(\rho v,v).
\label{6.6d}
\end{equation}
with some $\varepsilon>0$ provided $\hat\mu>\sqrt{2}\hat\la$.
Obviously, a similar estimate (if $\pi /4$ is replaced by $\theta$) holds at
any ray $\lambda=re^{i\theta}$ in a double sector $\Lambda_\phi$ provided
$0<\phi-\pi/4$ is small enough.

According to Theorem 1.7 estimate (\ref{6.6d}) gives the estimate of the resolvent
(\ref{6.6c}). Since (\ref{12::4.2}) holds in our case with $p=1$, we see that
the Keldysh --Agmon condition is satisfied.
\end{proof}

For simplicity we formulate the main result of this section not in the whole
generality.

\begin{theorem}
Let $\Omega$ be a bounded domain in $\mathbb{R}^2$ and assume that either
$\Omega$ is smooth
or Lame constants satisfy the condition $\hat\mu>\sqrt{2}\hat\la$.
Then for any function $\phi(x)\in\left[\Wo^{\tiny\rm o\phantom{.}\phantom{.}}(\Omega)\right]^3$
there is a unique classical solution $u(y)$ in the half-cylinder
$Q=\mathbb{R}^+\times\Omega$ of the stationary Lame system (\ref{6.4}), (\ref{6.2}) with given
non-resonant frequency $\omega$, such that this solution satisfies the
Mandelstam radiation principle as $x_1=y\to\infty$ and the initial condition
is understood in the following sense
$$
\lim_{y\to0}\|u(y,x_2,x_3)-\phi(x_2,x_3)\|_1=0.
$$
\end{theorem}
\begin{proof}
It follows from results of Section 5.
\end{proof}
\medskip

Our conjecture (which we can not prove at the moment) is that the condition
$\hat \mu>2\hat\lambda$ in Theorem 6.4 is superfluous. This condition is
used only in the proof of the existence. Apparently it is essential for the
validity of the Keldysh--Agmon condition and, hence, for the half range
completeness. However, it has not to be essential for the existence
of a solution. The reason is that for sufficiently small frequencies
$\omega$ the pencil $T_\omega(\lambda)$ is positive and Theorem 3.4 can be
applied to prove the existence.

\renewcommand{\refname}{{\large\rm Bibliography for Section~12}}

\newpage
\section{Scattering of waves by periodic gratings and factorization problems}
\subsection{Introduction}

In this paper we consider two scattering problems for the Helmholtz equation. The study of the first one was originated by Lord Rayleigh [Rl, R2]. The problem is to give an analysis of the scattering of a monochromatic plane wave incident on a grating with a periodic curve in $\mathbb{R}^2$ (Rayleigh considered a sinusoidal grating profile). The second problem is a three-dimensional analogue of the first one: to give an analysis of the scattering of a space wave by a periodic surface in $\mathbb{R}^3$.

The scattering of acoustic and electromagnetic waves by periodic gratings plays a significant role in physics and engineering, which caused a vast literature devoted to these problems. The works are mostly connected with the scattering in $\mathbb{R}^2$, because even in this case the problem is quite non-trivial and involves an intricate and lengthy analysis.

A mathematical core of the problem is to prove existence and uniqueness of the solution which describes the scattering. In the plane case and in the case of non-resonant frequencies the proof was given by Badyukov [Bl, B2], who reduced the problem to an integral equation of the Fredholm type using the Hankel function expansion for the kernel of the Helmholtz equation. Wilcox and Guilliot [WG] independently obtained similar results using Rayleigh-Bloch wave expansions (which, essentially, coincide with those in [Bl, B2]).

Alber \cite{13::A} and Wilcox \cite{13::W1} developed an alternative method for solving the scattering problem based on analytic continuations. Further developments, numerical studies, historical remarks and references can be found in the book edited by
Petit \cite{13::P}, in the monographs of Gunter \cite{13::G}, Wilcox \cite{13::Wl}, \cite{13::W2}, Galashnikova and Il'inskii \cite{13::GI}, Nazarov and Plamenevskii \cite{13::NP}, in the papers of Babich \cite{13::B}, Il'inskii and Mikheev \cite{13::IM}, Beljaev, Mikheev and Shamaev \cite{13::BMS}.

For the case of resonant frequencies, we have not found in the literature rigorous results on the solvability of the scattering problems. In this case an additional problem arises: how to select outgoing waves and to pose the radiation conditions? It turns out that the formulation of the radiation condition for the plane scattering problem in the case of resonant frequencies remains the same as in the non-resonant case. However, this circumstance is rather incidental from the mathematical point of view. It is explained by the fact that the Jordan chains of the spectral problem corresponding to the Helmholtz equation have the simplest structure: their lengths equal 2. This is not true if the scattering problem is considered for the system of elasticity (see \cite{13::KO}).

The aim of this paper is to propose a new approach to treat scattering problems. This approach is quite general and allows to consider scattering problems which were not treated before. It is based on the possibility to reformulate these problems in terms of abstract ordinary differential equations with operator coefficients on a Hilbert space. It turns out that the solvability of a scattering problem is equivalent to the solvability of the appropriate differential equation on the semiaxis with the radiation conditions at $+\infty$. To solve the last problem we apply the factorization theorems for the operator symbol of the corresponding equation. This paper can be considered as a continuation of the paper [Sh2], but it can be read independently.

An outstanding role in development of the factorization theory is played by the works of I. Gohberg and his co-authors. A particular mention deserves his pioneering work with M. Krein \cite{13::GK1}. Further developments and references can be found in the monograph of Gohberg and Feldman \cite{13::GF}, Gohberg and Krupnik \cite{13::GK}, Gohberg, Lancaster and Rodman \cite{13::GLR}, Markus \cite{13::M}, Gohberg, Goldberg and Kaashoek \cite{13::GGK}.

In our case the factorization problem has to be solved for selfadjoint operator pencils with unbounded operator coefficients. This leads to new difficulties. In particular, even if a linear right divisor $\lambda - Z$ of a pencil is found, one has to investigate the properties of Z: does this operator generate a holomorphic $C_0$-semigroup? The study of factorization of operator pencils with unbounded coefficients was originated in the authors' paper \cite{13::KS}. Here we give a short review of results on the factorization of elliptic pencils which are essentially used in the sequel.

In contrast with the previous works, an abstract approach of this paper does not use specific properties of the Helmholtz equation and can be applied to scattering problems in electrodynamics described by the Maxwell equation (see \cite{13::GI}) or the system of elasticity. The Helmholtz equation itself can be modified; the frequency $k^2$ can be replaced by a periodic function $k^2c^2 (x)$ that corresponds to the scattering in a non-homogeneous medium. All our arguments remain valid; the only change is to replace the exponents by the eigenfunctions of the Sturm-Liouville operator with potential $k^2c^2(x)$ and with quasi-periodic boundary conditions. We remark also that the resonant case is not an obstacle for the method.

The plan of the paper is the following. In Section 1 we pose the scattering problem for the plane Helmholtz equation and formulate the radiation condition. In Section 2 we give a review of results on factorization of elliptic pencils, properties of divisors and solvability of the corresponding operator equations on the semiaxis with the radiation conditions at $+\infty$. In the subsequent sections we give the detailed analysis of the scattering problems for the two- and three-dimensional cases. To our best knowledge, the space problem has not been considered in the literature before.

\subsection{Scattered Waves and the Radiation Condition for the Two-dimensional Helmholtz Equation}

Let $(x, y)$ be the coordinates in $\mathbb{R}^2$ and $\Gamma$ be a $2\pi$-periodic curve given by a smooth function $y = a (x)$. Let
\begin{equation}\label{13::1.1}
v_{\varphi}(x, y)=e^{-i k(x \sin \varphi+y \cos \varphi)}
\end{equation}
be a monochromatic wave incident on the grating $\Gamma$. The reflection of this wave generates the scattered waves which are to be found. The number $\varphi$ coincides with the angle between the axis $Oy$ and the direction of the wave (see Figure 1). The wave $v_{\varphi}(x, y)$ satisfies the Helmholtz equation
\begin{equation}\label{13::1.2}
\Delta u+k^{2} u=0, \quad u=u(x, y)
\end{equation}
and the quasi-periodic boundary conditions
\begin{equation} \label{13::1.3}
\begin{aligned} u(0, y) &=e^{i t-i \nu} u(2 \pi, y) \\ u_{x}^{\prime}(0, y) &=e^{-i \nu} u_{x}^{\prime}(2 \pi, y) \end{aligned}
\end{equation}
where $\nu=2 \pi k \sin \varphi$.

Figure 1

Naturally, the scattered waves also satisfy equation (1.2) and boundary conditions (1.3). We have to declare the law of reflection. Assume that the wave $v_{\varphi}$ reaches all points of the grating. This means that $\cot \varphi>\max a^{\prime}(x)$. The full reflection means
\begin{equation} \label{13::1.4}
\left.u(x, y)\right|_{\Gamma}=\left.v_{\varphi}(x, y)\right|_{\Gamma}
\end{equation}
or
\begin{equation}
u(x, a(x))=v_{\varphi}(x, a(x)),
\end{equation}
where $u(x,y)$ is a solution of the scattering problem in the domain
$$
\Omega=\{x, y | x \in \mathbb{R}, y>a(x)\},
$$
i.e., a solution of equation \eqref{13::1.2} subject to boundary conditions \eqref{13::1.3}.

The problem given in the unbounded domain $\Omega$ by equation \eqref{13::1.2} and initial condition \eqref{13::1.4} is not well-posed, since the frequency $k^2 > 0$ belongs to the continuous spectrum of the Laplace operator in $\Omega$ , with the Dirichlet boundary condition on $\partial \Omega$. To extract physically reasonable solutions in such cases one claims additional conditions. It is well known that for the Helmholtz equation on the exterior of a bounded domain the condition

$$
\frac{\partial u}{\partial r}-i k u=o(r) \quad \text { as } \quad r=\sqrt{x^{2}+y^{2}} \rightarrow \infty
$$
guarantees existence and uniqueness of a solution. This is the so-called Sommerfeld radiation condition.

For unbounded domains with periodic boundaries the radiation conditions have a more intricate form. To formulate them, we remark that the elementary quasi- periodic solutions of equation \eqref{13::1.2} have the representation

\begin{equation} \label{13::1.5}
u_{n}^{\pm}(x, y)=e^{\pm i \lambda_{n} y} e^{i \mu_{n} x}, \mu_{n}=\frac{v}{2 \pi}+n, \lambda_{n}=\sqrt{k^{2}-\mu_{n}^{2}}.
\end{equation}

Here $n \in \mathbb{Z}$, and the main branch of the square root function is chosen, i.e., $\lambda_{n}>0$ for $k>\left|\mu_{n}\right|
$ and $\operatorname{Im} \lambda_{n}>0$  for $\left|\mu_{n}\right|>k$. The solutions $u_{n}^{-}(x, y)$ corresponding to the non-real values $ \lambda_{n}$ grow exponentially in $\Omega$ as $y \to \infty$ and have no physical sense in scattering problems. The solutions corresponding to the real numbers $ \lambda_{n}$ are called propagating waves and play the most important role. The solutions $u_{n}^{+}(x, y)\left(u_{n}^{-}(x, y)\right)$ which correspond to $\lambda_{n}>0$($\lambda_{n}<0$) are called outgoing (incoming) waves. The physical sense prompt us that the scattered waves must include only outgoing waves and exponentially decaying waves corresponding to solutions $u_{n}^{+}(x, y), \operatorname{Im} \lambda_{n}>0$. Actually, it was Rayleigh [\cite{13::Rl}, \cite{13::R2} who assumed that scattered waves consist only of outgoing and decaying waves. The problem of a choice of physically reasonable propagating waves has been widely discussed in the literature since 30's. In particu
 lar, Mandelstam noticed that the Rayleigh hypothesis does not work for some equations of electrodynamics and proposed to choose the waves with positive group velocity (see details in \cite{13::Sh2}). For the Helmholtz equation the Rayleigh and the Mandelstam hypotheses coincide.

Now we can formulate the scattering problem as follows: {\it to find a solution of equation \eqref{13::1.2} subject to quasi-periodic conditions \eqref{13::1.3}, initial condition \eqref{13::1.4} and the radiation condition
\begin{equation} \label{13::1.6}
u(x, y)=\sum_{-k \leq \mu_{n} \leq k} c_{n} e^{i \mu_{n} x} e^{i \lambda_{n} y}+o(1)
\end{equation}
where the sum contains only outgoing waves corresponding to $\lambda_{n}>0$  and $o(1)$ is a decaying function as $y \to \infty$.} Here $c_n$ are the amplitudes of the outgoing waves. They must be determined by initial condition \eqref{13::1.4}. Further we will see that $o(1)$ in \eqref{13::1.6} is represented as a convergent series of the exponentially decaying waves (although this is not necessarily true in the three-dimensional problem).

A frequency $k^2$ is called {\it resonant} (in some books it is called {\it cut-off}) if $\lambda_n^2=k^2-\mu_n^2=0$ for some $n \in \mathbb{Z}$. If $\nu / \pi \notin \mathbb{Z}$ then the equality $\lambda_{n}=0$  may hold for the only value $n_{0} \in \mathbb{Z}$ In this case the pair of elementary solutions
$$
e^{i k x}, y e^{i k x}, \quad \mu_{0}=v / 2 \pi+n_{0}=k \geq 0,
$$
corresponds to the wave number $\lambda_{n}=0$. If $\nu / \pi \in \mathbb{Z}$ and the equality $k^2-\mu_n^2=0$holds for some $n \in \mathbb{Z}$, then it holds for two values $n_{0}, n_1 \in \mathbb{Z}$. In this case the resonant elementary solutions have the form
$$
e^{i k x}, y e^{i k x} ; \quad e^{-i k x}, y e^{-i k x}.
$$
The first functions of these pairs of solutions are degenerated waves (independent of $y$); the second ones are associated solutions. There is one-to-one correspondence between elementary solutions of problem \eqref{13::1.2}, \eqref{13::1.3} and Jordan chains of the pencil (operator symbol) corresponding to this problem. It follows from the subsequent general results that in the presence of a Jordan chain of even length only the first half of the functions of this chain has to be taken into consideration. In our case the lengths of Jordan chains equal 2. Therefore, in this case only the eigenfunctions must be involved in the group of the scattered waves participating in the radiation condition. Hence the radiation condition in the resonant case is given as before by formula \eqref{13::1.6}.

\subsection{Factorization of elliptic operator pencils and solvability of the corresponding equations on the semi-axis}

In this section we deal with an operator pencil
$$
T(\lambda)=\lambda^{2} F+\lambda G+H-V
$$
on a Hilbert space $\mathfrak{H}$. It is always assumed that the "main" operator $H$ is self-adjoint and uniformly positive while the other ones are symmetric.

{\bf Definition 2.1} $T(\lambda)$ is called {\it elliptic} if the following conditions are fulfilled:
i)  $F$ is bounded and uniformly positive $(0 \ll F \ll \infty)$;
ii) $H=H^{*} \gg 0$ and $V$ is a symmetric $H$-compact operator, i.e., $V H^{-1}$ is compact in $\mathfrak{H}$;
iii)  $G$ is symmetric and $\mathcal{D}(G) \supset \mathcal{D}\left(H^{1 / 2}\right)$;
iv)  $T(\lambda)>0$ for all $\lambda \in \mathbb{R}$ with $|\lambda|>r_{0}$ provided $r_0$ is large enough.

{\bf Definition 2.2} An elliptic pencil $T(\lambda)$ is called {\it strongly elliptic} if there exist a number $\varepsilon>0$ and a symmetric $H$-compact operator $V'$ such that
\begin{equation}\label{13::2.1}
T(\lambda) \geq \varepsilon\left(\lambda^{2} I+H\right)-V^{\prime} \quad \forall \lambda \in \mathbb{R}.
\end{equation}

{\bf Definition 2.3} An elliptic pencil $L(\lambda)$ is called {\it regular elliptic} if
\begin{equation}\label{13::2.2}
\left\|H T^{-1}(\lambda)\right\|+|\lambda|^{2}\left\|T^{-1}(\lambda)\right\| \leq \text { const } \quad \forall \lambda \in \mathbb{R},|\lambda|>r_{0}.
\end{equation}

Let $\mathfrak{H_\theta}$ be the scale of the Hilbert spaces generated by the operator $H^{1 / 2}$, i.e., $\mathfrak{H}_{\theta}=\mathcal{D}\left(H^{\theta / 2}\right)$ and$ \|x\|_{\theta}=\left\|H^{\theta / 2} x\right\|$. Recall that the abstract Sobolev space $W^{s}(a, b ; \mathfrak{H})$ consists of functions $f(t)$ defined on $(a, b) \subset \mathbb{R}$, taking values in $\mathfrak{H}$, and having a finite norm
$$
\|f\|_{s}^{2}=\int_{a}^{b}\left(\left\|f^{(s)}(t)\right\|^{2}+\left\|H^{s / 2} f(t)\right\|^{2}\right) d t.
$$
(see details in \cite[Ch. 1]{13::LM}).

{\bf Definition 2.4} An elliptic pencil $T(\lambda)$ is called {\it strongly regular} if for all functions $v(t) \in W^{2}[0, \infty ; \mathfrak{H}]$ subject to the condition $v(0)=0$, the following estimate holds
\begin{equation}\label{13::2.3}
\left\|T_{\chi}\left(-i \frac{d}{d t} v(t)\right)\right\|_{\mathrm{L}_{2}} \geq \varepsilon\|v\|_{2}, \quad \varepsilon>0
\end{equation}
where $T_{\varkappa}(\lambda)=\lambda^{2} F+\varkappa \lambda G+H$ and $\varepsilon$ does not depend on $v$ and $\varkappa \in[0,1]$.

It is known \cite[A \S3]{13::Shi}  that estimate \eqref{13::2.3} implies \eqref{13::2.2}, i.e., a strongly regular elliptic pencil is regular elliptic but not vice versa. We remark that for usual elliptic operators estimate \eqref{13::2.1} is equivalent to the Garding inequality (see \cite{13::Sh2}), estimate \eqref{13::2.2} is known as the Agmon-Nirenberg or the Agranovich--Vishik estimate for regular elliptic problems with parameter on smooth bounded domains (in this context $ \mathfrak{H}=L_{2}(\Omega)$, where $\Omega$ is a smooth bounded domain in $\mathbb{R}^{n}$), and estimate \eqref{13::2.3}is known as the Bermstein-Ladyzhenskaya inequality
(see \cite[|\S3, 6]{13::Shi}). More details on the motivation of the above definitions can be found in the papers \cite{13::Shi}, \cite{13::Sh2}.

The verification of estimate \eqref{13::2.3} is not trivial even for concrete pencils. We shall use the following result.

{\bf Proposition 2.5} {\it Let $T(\lambda)$ be elliptic. Suppose that $ F=F_{0}+F_{1}, G=G_{0}+G_{1}$, $H=H_{0}+H_{1}$, where $F_1, G_1, H_1$ are symmetric operators such that $F_{1}, G_{1} H^{-1 / 2}, H_{1} H^{-1}$ are compact, and $H_{0} \gg 0$. If the estimate
\begin{equation}\label{13::2.4}
\left\|F_{0}^{-\frac{1}{2}} G y\right\| \leq(2-\varepsilon)\left\|H_{0}^{1 / 2} y\right\|, \quad y \in \mathcal{D}\left(H_{0}^{1 / 2}\right)
\end{equation}
holds with some $\varepsilon > 0$, then $T(\lambda)$ is strongly regular.}

{\bf Proof:} It can be found in \cite[\S9]{13::Shi}.	$\square$

Let the operator $H$ have discrete spectrum. It follows from the theorem on holomorphic operator functions (see \cite[Ch. 1]{13::GK2}and \cite[\S1.4]{13::Shi} for the version of this theorem for pencils with unbounded coefficients) that the spectrum of an elliptic pencil $T(\lambda)$ in this case is also discrete.

It is known \cite{13::Ke} that the principal part of the Laurent expansion of the resolvent $T^{-1}(\lambda)$ admits a representation
$$
\sum_{k=1}^{N} \sum_{s=0}^{p_{k}} \frac{\left(\cdot, z_{k}^{p_{k}-s}\right) x_{k}^{s}}{(\lambda-\mu)^{p_{k}+1-s}}
$$
where
\begin{equation}\label{13::2.5}
x_{k}^{0}, \ldots, x_{k}^{p_{k}}, \quad k=1, \ldots, N
\end{equation}
is a canonical system of eigen and associated vectors of the pencil  $T(\lambda)$, and
$$
z_{k}^{0}, \ldots, z_{k}^{p_{k}}, \quad k=1, \ldots, N
$$
is the adjoint canonical system of the eigen and associated vectors of $T(\lambda)$ corresponding to the eigenvalue $\bar{\mu}$. The following result is essential in the sequel.

{\bf Proposition 2.6} {\it Canonical system \eqref{13::2.5} corresponding to a real eigenvalue $\mu$, can be chosen so that
$$
x_{k}^{s}=\varepsilon_{k} z_{k}^{s}, \quad k=1, \ldots, N, \quad \forall s=0, \ldots, p_{k},
$$

where $\varepsilon_ k = \pm1$ and $\varepsilon_{k}=\operatorname{sign}\left(L^{\prime}(\mu) x_{k}^{0}, x_{k}^{p_{k}}\right)$.}

{\bf Proof:} See \cite[Lemma 1.2]{13::KS}. $\square$

Other definitions of the sign characteristics are given in \cite{13::GLR}. Actually, the sign characteristics are important only for Jordan chains of odd length. Further, we assume for convenience that the {\it sign characteristics of Jordan chains of even length equal zero}.

Canonical systems that possess the properties formulated in Proposition 2.6 are called {\it normal}. Normal canonical systems and sign characteristics are important in the analysis of the factorization problem. We shall give here short historical comments concerning the problem of the factorization of operator polynomials with respect to the real axis. Krein and Langer \cite{13::KL} studied pencils of the form
$$
L(\lambda)=I+\lambda B+\lambda^{2} C,
$$
where $B$ and $C$ are bounded self-adjoint operators and $C > 0$. They proved that $L(\lambda)$ possesses a right divisor of the form $(\lambda Z-I)$, whose spectrum is located in the closed upper half-plane. The real spectrum of this divisor was investigated by Kostyuchenko and Orazov \cite{13::KO}. The factorization of higher order operator polynomials was out carried by Langer \cite{13::L}; the detailed analysis of the real spectrum of divisors of polynomials with Hermitian matrix coefficients was done by Gohberg, Lancaster and Rodman \cite{13::GLR}. We should mention that the problem of factorization of non-negative operator pencils (and operator functions) on the real axis has its own history (see the book of Rosenblum and Rovnjak \cite{13::RR}). In the paper \cite{13::KS} the authors proposed a new analytic approach to the factorization of quadratic pencils, investigated the properties of a linear operator $Z$ participating in the factorization and proved the first factorizatio
 n theorem for pencils with unbounded coefficients. Further developments of the theory was carried out by Shkalikov \cite{13::Shi}-\cite{13::Sh3}.

Now, let us define the half of the eigen and associated vectors of an operator pencil $T(\lambda)$. Let canonical system \eqref{13::2.5} be normal. Its half consists of the vectors
\begin{equation}\label{13::2.6}
x_{k}^{0}, x_{k}^{1}, \ldots, x_{k}^{l_{k}}, \quad k=1, \ldots, N
\end{equation}
where $l_{k}=\left(p_{k}-1+\varepsilon_{k}\right) / 2$ and $\varepsilon_ k$ are the sign characteristics (we assumes $\varepsilon_ k=0$ for Jordan chains of even length). We imply that in the case $l_k = -1$ the corresponding set in \eqref{13::2.6} is empty. The set of all canonical systems of $T(\lambda)$ corresponding to the eigenvalues from the open upper half-plane and of the halves of canonical systems corresponding to the real eigenvalues, is called the{\it half of eigen and associated vectors of $T(\lambda)$}. We point out a particular important case (connected with the scattering problem): if the lengths of Jordan chains corresponding to the real eigenvalues do not exceed 2, then the half contains the canonical systems corresponding to the eigenvalues from the upper half-plane and only the eigenvectors $x_k$ corresponding to $\lambda_{k} \in \mathbb{R}$ subject to the condition $\left(T^{\prime}\left(\lambda_{k}\right) x_{k}, x_{k}\right) \geq 0$.

To formulate the basic results we shall introduce the class of operator pencils whose resolvents are meromorphic functions of finite order having polynomial growth on some rays in the complex plane.

{\bf Definition 2.7} We say $T(\lambda)$ belongs to the class $K$ if

i)  the eigenvalues of the operator $H$ satisfy the estimate
$$
\lambda_{j}(H) \geq c j^{p}
$$
with some constants $c$ and $p$.

ii)  either $p \geqslant 2$ or $p < 2$ and there exist rays $\mathrm{z} \gamma_{j}=\left\{\lambda | \arg \lambda=\theta_{j}\right\}$,  $j = 1, \ldots, N, \theta_{j+1}>\theta_{j}$, in the upper half-plane such that
$$
\max \left(\theta_{1}, \theta_{j+1}-\theta_{j}, \pi-\theta_{N}\right)<2 \pi / p
$$
and
\begin{equation}\label{13::2.7}
\left\|H^{1 / 2} T^{-1}(\lambda) H^{1 / 2}\right\| \leq c|\lambda|^{m} \quad \text { for } \quad \lambda \in \gamma_{j}
\end{equation}
with some constants $c$, $m$, provided $|\lambda|$ is large enough.

It is proved in \cite{13::Sh2} that the inequality
$$
|(T(\lambda) x, x)| \geq \varepsilon\left[r^{2}(x, x)+(H x, x)\right], \quad \forall \lambda \in \gamma,|\lambda|=r>r_{0}
$$
is sufficient for the validity of estimate \eqref{13::2.7} with $m =0$ on the ray $\gamma$. The last inequality is easier to verify for concrete problems. In particular, \eqref{13::2.7} holds on the real line for strongly elliptic pencils. If $T(\lambda)$ is elliptic and the condition $\lambda_{j}(H) \geq c j^{p}$ holds, then the resolvent $T^{-1}(\lambda)$ is a meromorphic function of order $\leqslant p/2$ (see \cite[\S2]{13::Shi}). Applying the Phragmen--Lindelof theorem we find: if an elliptic pencil $T(\lambda)$ is strongly regular, belongs to the class $K$ and $F(\lambda)=\left(T^{-1}(\lambda) f, f\right)$ is an entire function for some $f \in \mathfrak{H}$, then $F(\lambda) \equiv 0$. Hence conditions i) and ii) in Definition 2.7 are needed to prove the completeness theorems for eigenvectors (see details in \cite{13::Sh2}).

Now let us specify the understanding of solutions of the equation
\begin{equation}\label{13::2.8}
T\left(i \frac{d}{d t}\right) u(t)=-F u^{\prime \prime}(t)+i G u^{\prime}(t)+(H+V) u(t)=0.
\end{equation}

A	function $u(t) \in W^{1}(a, b ; \mathfrak{H})$ is said to be a generalized solution	of equation \eqref{13::2.8} if for all functions $v(t) \in W^{1}(a, b ; \mathfrak{H})$ subject to the conditions	$v(a)	 =v(b) =0$, the equality
$$
\left(F u^{\prime}, v^{\prime}\right)+i\left(u^{\prime}, G v\right)+\left(\left(I+V^{\prime}\right) H^{1 / 2} u, H^{1 / 2} v\right)=0
$$
holds, where $V^{\prime}=H^{-1 / 2} V H^{-1 / 2}$ and the scalar product is taken in $\mathbf{L}_{2}(a, b ; \mathfrak{H})$. Details clarifying this definition see in \cite{13::Sh2}.

A function $u(t) \in W^{2}(a, b ; \mathfrak{H})$ is called a classical solution of equation \eqref{13::2.8} on $(a, b)$ if \eqref{13::2.8} holds as equality of functions in $\mathbf{L}_{2}(a, b ; \mathfrak{H})$.

We say that a classical (generalized) solution of equation \eqref{13::2.8} satisfies the radiation condition at $+\infty$ if
\begin{equation}\label{13::2.9}
u(t)=\sum_{\varepsilon_{k} \geq 0} c_{k} e^{i \lambda_{k} t} x_{k}+u_{0}(t)
\end{equation}
where $u_0(t)$ is a classical (generalized) solution on $\mathbb{R}^+$ satisfying the condition $\left\|u_{0}(t)\right\|_{1} \rightarrow 0$  as $ t \rightarrow \infty$.	 Here	the	first term in \eqref{13::2.9} is a finite sum of elementary solutions corresponding to the real eigenvalues of nonnegative type, and for simplicity we have assumed that the lengths of Jordan chains corresponding to the real eigenvalues do not exceed 2. In the general case the sum has to contain all elementary solutions corresponding to the halves of Jordan chains \eqref{13::2.6}.

Let us formulate the basic results on elliptic pencils.

{\bf Theorem 2.8} {\it The half of eigen and associated vectors of a self-adjoint elliptic pencil $T(\lambda)$ is minimal (i.e., there exists a biorthogonal system) in the spaces $\mathfrak{H}_\theta$, $0 \leqslant \theta \leqslant 1$. It is complete in the same spaces if $T(\lambda)$ is either strongly or regular elliptic and belongs to the class $K$. If $T(\lambda)$ is strongly regular and belongs to the class $K$, then the half is a complete system in for $0 <\leqslant \theta \leqslant 3/2$.}

{\bf Theorem 2.9} {\it Let $T(\lambda)$ be either strongly elliptic or regular elliptic pencil and belong to the class $K$. Then
\begin{equation}\label{13::2.10}
T(\lambda)=\left(\lambda-Z_{1}\right) F(\lambda-Z),
\end{equation}
where the operator $Z$ possesses the properties:
\begin{enumerate}

\item $\mathcal{D}(Z)=\mathfrak{H}_{1}$ and $Z-\lambda=K H^{1 / 2}$, where $K$ is bounded and boundedly invertible in $\mathfrak{H}$, provided $\lambda \notin \sigma(Z)$;

\item the spectrum of $Z$ lies in the closed upper half-plane and the system of its eigen and associated vectors coincides with the half of those of $T(\lambda)$ ;

\item $ iZ$ generates a holomorphic semigroup in the spaces $\mathfrak{H}_{\theta}, 0 \leq \theta \leq 1$.
\end{enumerate}

If in addition $T(\lambda)$ is strongly regular, then property 3) remains valid in the spaces  $\mathfrak{H}_{\theta}$ for $0 \leq \theta \leq 3 / 2$.}

{\bf Theorem 2.10 } {\it Let $T(\lambda)$ be either strongly or regular elliptic self-adjoint pencil belonging to the class $K$. Then for any vector $f \in \mathfrak{H}_{\theta}, 0 \leq \theta \leq 1$, there exists a unique function $u(t)$ which is a generalized solution of equation \eqref{13::2.8} on $(\varepsilon, \infty)$ for any $\varepsilon > 0$ and satisfies the radiation condition \eqref{13::2.9} and the initial condition
\begin{equation} \label{13::2.11}
\lim _{t \rightarrow 0+}\|u(t)-f\|_{\theta}=0.
\end{equation}
If $T(\lambda)$is strongly regular, then the same is true for $ 0 \leq \theta \leq 3/2$. Moreover, $u(t)$ is a classical solution for $t > 0$ and is represented by the formula
\begin{equation}\label{13::2.12}
u(t)=\frac{1}{2 \pi i} \int_{\gamma} e^{i t \lambda}(Z-\lambda)^{-1} d \lambda,
\end{equation}
where a contour $\gamma$ contains the spectrum of the operator $Z$ and lies asymptotically in the upper half-plane.}

{\bf Theorem 2.11} {\it Let $T(\lambda)$ be a strongly regular self-adjoint pencil. Then there exists a unique classical solution of equation \eqref{13::2.8} on the semi-axis satisfying radiation condition \eqref{13::2.9} and initial condition \eqref{13::2.11} for $\theta = 3/2$.}

The most important fact of the last theorem is that existence and uniqueness of solutions of the half-range Cauchy problem on the semiaxis $\mathbb{R}^+$ is true for operator pencils not necessarily belonging to the class $K$.

The proofs of Theorems 2.9-2.11 can be found in \cite{13::Shi}, \cite{13::Sh2} (see also \cite{13::KS}, where the first results of this kind were obtained).

\subsection{ Existence and Uniqueness of the Solution of the Plane Scattering Problem}

It suffices to define a solution of the scattering problem in the semi-strip
$$
\Omega_{0}=\{x, y | 0 \leq x \leq 2 \pi, \quad a(x) \leq y<\infty\}
$$
The corresponding solution in the whole half-plane $y > a(x$) is restored by quasi-periodic conditions \eqref{13::1.3}.

The substitution of
\begin{equation}\label{13::3.1}
\xi=x, \quad \eta=y-a(x)
\end{equation}
maps the semi-strip $\Omega_0$ onto the standard semi-strip $\Omega'_0$ (see Figure 2). Taking into account that
$$
\xi_{x}=1, \quad \xi_{y}=0, \quad \xi_{x x}=0, \quad \xi_{y y}=0;
$$
$$
\eta_{x}=-a^{\prime}(x), \quad \eta_{y}=1, \quad \eta_{x x}=-a^{\prime \prime}(x), \quad \eta_{y y}=0,
$$
and
$$
u_{x x}=u_{\xi \xi} \xi_{x}^{2}+2 u_{\xi} \eta \xi_{x} \eta_{x}+u_{\eta \eta} \eta_{x}^{2}+u_{\xi} \xi_{x x}+u_{\eta} \eta_{x x},
$$
$$
u_{y y}=u_{\xi \xi} \xi_{y}^{2}+2 u_{\xi} \eta \xi_{y} \eta_{y}+u_{\eta \eta} \eta_{y}^{2}+u_{\xi} \xi_{y y}+u_{\eta} \eta_{y y},
$$
we find that the Helmholtz equation is transformed into the following one
\begin{equation}\label{13::3.2}
u_{\xi \xi}-2 u_{\xi \eta} a^{\prime}(\xi)+u_{\eta \eta}\left(a^{\prime}(\xi)^{2}+1\right)-u_{\eta} a^{\prime \prime}(\xi)+k^{2} u=0
\end{equation}
The form of this equation is more intricate but its coefficients do not depend on $\eta$ and the advantage is that $\xi$ and $\eta$ belong to the domain $\Omega'_0$, where the separation of

%Figure 2

variables can be realized. Looking for solutions of the form $u(\xi, \eta)=e^{i \lambda \eta} f(\xi)$ and taking into account quasi-periodic conditions \eqref{13::1.3} we come to the spectral problem
\begin{equation}\label{13::3.3}
-f^{\prime \prime}-k^{2} f+\lambda i\left(2 a^{\prime}(\xi) f^{\prime}+a^{\prime \prime}(\xi) f\right)+\lambda^{2}\left(a^{\prime}(\xi)^{2}+1\right) f=0
\end{equation}
with the boundary conditions
\begin{equation}\label{13::3.4}
f(0)=e^{-i v} f(2 \pi), \quad f^{\prime}(0)=e^{-i v} f^{\prime}(2 \pi).
\end{equation}

Since functions \eqref{13::1.5} represent a complete set of elementary quasi-periodic solutions of the Helmholtz equation, the functions
\begin{equation}\label{13::3.5}
\begin{aligned} f_{n}^{\pm}(x) &=u_{n}^{\pm}(x, a(x))=e^{\pm i \lambda_{n} a(x)} e^{i \mu_{n} x} \\ \mu_{n} &=\frac{v}{2 \pi}+n, \quad \lambda_{n}=\sqrt{k^{2}-\mu_{n}^{2}} \end{aligned}
\end{equation}
form a complete set of all eigenfunctions of the spectral problem \eqref{13::3.3} and \eqref{13::3.4}, which correspond to the eigenvalues $\lambda_n$. The set of eigenfunctions of the problem \eqref{13::3.3}, \eqref{13::3.4} can also be found by a straightforward calculation. Substituting $f(\xi)=z(\xi) e^{i \lambda a(\xi)}$ in \eqref{13::3.3} and \eqref{13::3.4} we find
$$
-z^{\prime \prime}+\left(\lambda^{2}-k^{2}\right) z=0,
$$
$$
z(0)=e^{-i v} z(2 \pi), \quad z^{\prime}(0)=e^{-i v} z^{\prime}(2 \pi).
$$
Since $e^{i \mu_{n} x}$ form a complete set of eigenfunctions of this problem, we obtain that
set \eqref{13::3.5} possesses the same property with respect to problem \eqref{13::3.3}, \eqref{13::3.4}.	

It is easily seen that in the non-resonant case all eigenvalues $\lambda_{n}=\sqrt{k^{2}-\mu_{n}^{2}}$ are simple provided $ \nu / \pi \notin \mathbb{Z}$. The location of $\lambda_n$ for values $\nu$ close to $0$ is shown in Figure 3.

If $ \nu / \pi \in \mathbb{Z}$ but $ \nu / 2\pi \notin \mathbb{Z}$, then two eigenfunctions
$$
e^{i a(x) \sqrt{k^{2}-(n+1 / 2)^{2}}} \sin (n+1 / 2) x, \quad e^{i a(x) \sqrt{k^{2}-(n+1 / 2)^{2}}} \cos (n+1 / 2) x
$$

%Figure 3

correspond to the eigenvalues $\lambda_{n}=\sqrt{k^{2}-(n+1 / 2)^{2}}$. Finally, if $ \nu / 2\pi \in \mathbb{Z}$  then the pair of eigenfunctions
$$
e^{i a(x) \sqrt{k^{2}-n^{2}}} \cos n x, \quad e^{i a(x) \sqrt{k^{2}-n^{2}}} \sin n x,
$$
correspond to all eigenvalues $\lambda_{n}=\sqrt{k^{2}-n^{2}} \neq \pm k$. The extremal real eigenvalues $\pm \lambda_{0}=\pm k$ are simple; the corresponding eigenfunctions are $e^{\pm i k a(x)}$.

In the resonant case zero is the eigenvalue of pencil \eqref{13::3.3}, \eqref{13::3.4} of algebraic multiplicity 2 or 4. If $ \nu / \pi \notin \mathbb{Z}$, then the only eigenfunction $f_{n_{0}}(x)=e^{i\left(n_{0}+\nu / 2 \pi\right) x}$ corresponds to this eigenvalue (here no is defined by the equality $k = n_0 + \nu/2\pi$), and there is an associated function that coincides with $f_{n_{0}}(x)$ (we omit here elementary calculations). If $ \nu / \pi \in \mathbb{Z}$, then two eigenfunctions $e^{\pm i\left(n_{0}+\nu/ 2 \pi\right) x}$ correspond to 0, and there are associate functions coinciding with the previous ones.

According to the general definition (see Section 2) the functions $\left\{f_{n}^{+}(x)\right\}_{n=-\infty}^{\infty}$ defined in \eqref{13::3.5} form the half of the root functions of pencil \eqref{13::3.3}, \eqref{13::3.4}. The same is true in the resonant case, since the lengths of Jordan chains do not exceed 2. It is worth mentioning that Rayleigh \cite{13::Rl} calculated these waves in the case of a vertical incident wave ($\nu = 0$) assuming that $a(x)$ is an even function with respect to $x = 0$ and $x = \pi$. In this case a solution $u(x, y)$ of the scattering problem the same property $u^{\prime}(0, y)=u^{\prime}(\pi, y)=0$.

If conditions \eqref{13::3.4} are replaced by $f^{\prime}(0)=f^{\prime}(\pi)=0$, then the functions
\begin{equation}\label{13::3.6}
f_{n}(x)=e^{i a(x) \sqrt{k^{2}-n^{2}}} \cos n x, \quad n=0,1, \ldots
\end{equation}
form the half of the root functions of the corresponding pencil. System \eqref{13::3.6} is called the Rayleigh system. Suppose that the Rayleigh system is minimal, say
in the space $\mathbf{L}_{2}(0, \pi)$. In this case a solution of the scattering problem can be represented by a formal series
$$
u(x, y)=\sum_{n=0}^{\infty}\left(v_{\varphi}, f_{n}^{*}\right) f_{n}(x) e^{i \sqrt{k^{2}-n^{2}}}(y-a(x))
$$
where $v_{\varphi}=e^{i k a(x)}$ and $\left\{f_{n}^{*}\right\}$ is a biorthogonal system with respect to $\left\{f_{n}^{+}\right\}$. However, we are not aware of papers where the minimality or the completeness of the Rayleigh system (or the generalized Rayleigh system defined by \eqref{13::3.5}) is proved. Moreover, the minimality and the completeness do not guarantee the convergence of the series to the solution $u(x,y)$. Hence a rigorous justification of the Fourier method for the Rayleigh problem seems to be a hard task (see Theorem 3.4 below).

We intend to apply the results of Section 2 to solve the scattering problem. Let us represent pencil  \eqref{13::3.3},  \eqref{13::3.4} in the abstract form. For $s = 1$ ($s = 2$) denote by $W_{U}^{S}[0,2 \pi]$ the subspace of the Sobolev space $W_{2}^{s}[0,2 \pi]$ consisting of functions satisfying the first boundary condition \eqref{13::3.4} (both conditions \eqref{13::3.4}). The intermediate spaces
$$
W_{U}^{\theta}[0,2 \pi]=\left[W_{U}^{2}[0,2 \pi], L_{2}[0,2 \pi]\right]_{\theta}, \quad 0 \leq \theta \leq 2
$$
are defined by interpolation (see \cite[Ch.1]{13::LM}).

In the space $\mathbf{L}_{2}[0,2 \pi]$, let us define the operators
$$
\begin{aligned} H f &=-f^{\prime \prime}+f, & \mathcal{D}(H) &=W_{U}^{2}[0,2 \pi] \\ G f &=i\left(2 a^{\prime}(x) f^{\prime}+a^{\prime \prime}(x) f\right), & \mathcal{D}(G)&=W_{U}^{1}[0,2 \pi] \\ F f &=\left(a^{\prime}(x)^{2}+1\right) f, &  \mathcal{D}(F) &=\mathbf{L}_{2}[0,2 \pi]. \end{aligned}
$$
Further it is assumed that $a(x) \in W_{2}^{2}[0,2 \pi]$. Now, problem \eqref{13::3.3}, \eqref{13::3.4} is represented in the form
$$
T(\lambda) f=0, \quad T(\lambda)=\lambda^{2} F+\lambda G+H-V, \quad V=\left(k^{2}+1\right) I
$$
{\bf Proposition 3.1} {\it The pencil $T (\lambda)$ is self-adjoint and strongly elliptic.}

{\bf Proof:} It is obvious that $H=H^{*} \gg 0$. Integrating by parts we find
\begin{equation}\label{13::3.7}
i\left(2 a^{\prime}(x) f^{\prime}+a^{\prime \prime}(x) f, f\right)=i\left(a^{\prime}(x) f^{\prime}, f\right)-i\left(f, a^{\prime}(x) f^{\prime}\right).
\end{equation}
Hence the quadratic form $(Gf, f)$ is real and $G$ is symmetric. Since
$$
(H f, f)=\left(f^{\prime}, f^{\prime}\right)+(f, f)=\left\|H^{1 / 2} f\right\|^{2}
$$
we have $\mathcal{D}\left(H^{1 / 2}\right)=W_{U}^{1}[0,2 \pi] \subset \mathcal{D}(G)$. Finally, let us prove estimate \eqref{13::2.1}. We shall use the inequality
$$
(f, f) \geq M^{-2}\left(a^{\prime} f, a^{\prime} f\right), \quad \text { where } \quad M=\max _{0 \leq x \leq 2 \pi}\left|a^{\prime}(x)\right|
$$

Bearing in mind \eqref{13::3.7}, we obtain for $\lambda \in \mathbb{R}$

\begin{multline}
(T(\lambda) f, f) \geq\left(f^{\prime}, f^{\prime}\right)-k^{2}(f, f)-2\left|\lambda \operatorname{Im}\left(f^{\prime}, a^{\prime} f\right)\right| \\
+\lambda^{2}\left(a^{\prime} f, a^{\prime} f\right)+\lambda^{2}(f, f) \geq\left\|f^{\prime}\right\|^{2}-2|\lambda|\left\|f^{\prime}\right\|\left\|a^{\prime} f\right\| \\
+\left(1+\frac{1}{2} M^{-2}\right)\left\|a^{\prime} f\right\|^{2} \lambda^{2}+\frac{1}{2} M^{-2} \lambda^{2}\|f\|^{2}-k^{2}\|f\|^{2} \\
\geq\left[1-\left(1+\frac{1}{2} M^{-2}\right)^{-1}\right]\left\|f^{\prime}\right\|^{2}+\frac{1}{2} M^{-2} \lambda^{2}\|f\|^{2}-k^{2}\|f\|^{2} \\
>\left(2 M^{2}+1\right)^{-1}\left[(H f, f)+\lambda^{2}(f, f)\right]-\left(k^{2}+1\right)(f, f).
\end{multline}

This proves the proposition. $\square$

{\bf Proposition 3.2}{\it The pencil $T (\lambda)$ is regular elliptic and, moreover, strongly regular elliptic.}

{\bf Proof:} Estimate \eqref{13::2.2} can be obtained by a straightforward calculation of the resolvent kernel of the integral operator  $T^{-1}(\lambda)$. First, one has to prove by standard means that there is a pair of solutions of equation \eqref{13::3.3} having the asymptotics
$$
f^{\pm}(\lambda, x)=e^{(a(x) \pm i x) \lambda}\left(1+O\left(\lambda^{-1}\right)\right), \lambda \rightarrow \infty
$$
if $\lambda$, is located in one of the quadrants that are formed by the real and the imaginary axes. Then these solutions have to be substituted in the well-known formulas for the Green function (see \cite[Ch.1]{13::Na}). A detailed proof of estimates of the type  \eqref{13::2.2} for ordinary differential pencils of arbitrary order can be found in the work of Pliev \cite{13::P}.

To prove the strong regularity we recall Proposition 2.5. Since the operator
$a^{\prime \prime}(x) H^{-1 / 2}$ is compact in the space $\mathfrak{H}=\mathbf{L}_{2}$ (provided $a \in W_{2}^{2}$), it suffices to obtain the estimate
\begin{equation}\label{13::3.8}
\left\|F^{-1 / 2} a^{\prime}(x) y^{\prime}\right\|^{2} \leq(1-\varepsilon)^{2}(H y, y), \quad \varepsilon>0
\end{equation}
for functions $y \in \mathcal{D}(H)$. We have
$$
\left|F^{-1 / 2} a^{\prime}(x)\right|=\left|\left(1+a^{\prime}(x)^{2}\right)^{-1 / 2} a^{\prime}(x)\right|<1-\varepsilon
$$
for some $\varepsilon > 0$. Therefore,
$$
\left\|F^{-1 / 2} a^{\prime}(x) y^{\prime}\right\|^{2} \leq(1-\varepsilon)^{2}\left\|y^{\prime}\right\|^{2}=(1-\varepsilon)^{2}(H y, y).
$$
The proposition is proved. $\square$

{\bf Proposition 3.3} {\it The sign characteristics of the eigenfunctions of $T (\lambda)$ are defined by the relations
$$
\varepsilon_{n}^{\pm}=\operatorname{sign}\left(T^{\prime}\left(\pm \lambda_{n}\right) f_{n}^{\pm}, f_{n}^{\pm}\right)=\operatorname{sign}\left(\pm \lambda_{n}\right), \quad \lambda_{n} \in \mathbb{R}
$$
Hence the functions$\left\{f_{n}^{+}(x)\right\}_{n=-\infty}^{\infty}$ defined in \eqref{13::3.5} form the half of the eigen and associated functions of $T (\lambda)$.}

{\bf Proof:} For $\lambda_{n}>0$ we find
$$
\begin{array}{l}{\left(T^{\prime}\left(\lambda_{n}\right) f_{n}^{+}, f_{n}^{+}\right)=\left(G f_{n}^{+}, f_{n}^{+}\right)+2 \lambda_{n}\left(F f_{n}^{+}, f_{n}^{+}\right)} \\ {=-\mu_{n}\left(a^{\prime} f_{n}^{+}, f_{n}^{+}\right)+i\left(a^{\prime \prime} f_{n}^{+}, f_{n}^{+}\right)+2 \lambda_{n}\left(f_{n}^{+}, f_{n}^{+}\right)} \\ {=\mu_{n}\left(a^{\prime}, 1\right)+i\left(a^{\prime \prime}, 1\right)+4 \pi \lambda_{n}=4 \pi \lambda_{n}}\end{array}
$$
as the functions $a$ and $a'$ are periodic.$\square$

{\bf Theorem 3.4} {\it The generalized Rayleigh system $\left\{f_{n}^{+}(x)\right\}_{n=-\infty}^{\infty}$defined in \eqref{13::3.5}  is minimal and complete in the spaces $W_{U}^{\theta}[0,2 \pi]$ if $0 \leq \theta \leq 3 / 2$. For any function $g \in W_{U}^{\theta}$ the Fourier series
\begin{equation}\label{13::3.9}
u(x, \eta)=\sum_{n=-\infty}^{\infty}\left(g, f_{n}^{*}\right) f_{n}^{+}(x) e^{i \lambda_{n} \eta},
\end{equation}
converges for $\eta>\eta_{0}$ in the norm of $W_{U}^{\theta}$ provided $\eta_0$ is large enough (here $\left\{f_{n}^{*}\right\}$  the biorthogonal system in $\mathbf{L}_2$ with respect to $\left\{f_{n}\right\}$). Moreover, the function $u(x, \eta)$) admits a holomorphic continuation in a sector $|\arg \eta|<\varepsilon$ for sufficiently small $\varepsilon$, and there exists
$$
\mathrm{s}-\lim _{\eta \rightarrow+0} u(x, \eta)=g
$$
(the limits is understood in the norm of $W_{2}^{\theta}$)}

{\bf Proof:} The completeness and the minimality is the consequence of Theorem 2.8. The convergence of series \eqref{13::3.9} follows from representation \eqref{13::2.12} if there is a sequence of semicircles $|\lambda|=r_{k} \rightarrow \infty$	in	the	upper	 half-plane such that
\begin{equation}\label{13::3.10}
\left\|(Z-\lambda)^{-1}\right\| \leq e^{\eta_{0}|\operatorname{Im} \lambda|}=e^{\eta_{0}|\sin \varphi|_{\kappa_{k}}}, \quad|\lambda|=r_{k}
\end{equation}
Let us prove \eqref{13::3.10}. Without loss of generality suppose that $Z$ is invertible (equivalently, $T(0)$ is invertible). It follows from Theorem 2.9 that
$$
Z=K H^{1 / 2}, \quad Z_{1}=\left(H^{1 / 2}-\left(k^{2}+1\right) H^{-1 / 2}\right) K^{-1} F^{-1},
$$
where $K$ and $K^{-1}$ are bounded. Hence
\begin{equation}\label{13::3.11}
(Z-\lambda)^{-1}=T^{-1}(\lambda)\left(Z_{1}-\lambda\right) F.
\end{equation}
Recall that the Green function of the integral operator $T^{-1}(\lambda)$ (see Proposition 3.2) is a meromorphic function of order 1 and of finite type. By virtue of the Titchmarsh theorem for any $C$ exceeding the type there is a sequence	$r_{k} \rightarrow \infty$ such
that $ \left\|T^{-1}(\lambda)\right\| \leq \exp (C|\lambda|) \text { for }|\lambda|=r_{k}$. Hence estimate \eqref{13::3.10} outside a double sector $\Lambda_{\varphi_{0}}$ containing the real axis follows from  \eqref{13::3.11} . Since $T(\lambda)$ is strongly elliptic, estimate \eqref{13::3.10} inside a small double sector follows from \cite[Theorem 1.7]{13::Sh2}. Hence, \eqref{13::3.10} is proved and series \eqref{13::3.9} converges for $\eta > C$. The last assertion of the theorem follows from the fact that the operator $Z$ is a generator of a holomorphic semigroup in the spaces $W_{U}^{\theta}$.	 $\square$

We remark that series \eqref{13::3.9} does not converge for all $\eta > 0$ and arbitrary functions $g$, i.e., the system in question does not form a basis for the Abel summability method of order 1. Let us clarify our claim for Rayleigh system \eqref{13::3.6} assuming in addition that $a(x)$ is holomorphic.

{\bf Proposition 3.5}{\it Let $\lambda_{n}=\sqrt{k^{2}-n^{2}}$, $f_n(x)$ be defined by \eqref{13::3.6} and $a(x)$ be holomorphic on $\mathbb{R}$. If the series
\begin{equation}\label{13::3.12}
u(x, \eta)=\sum_{n=0}^{\infty}\left(g, f_{n}^{*}\right) f_{n}(x) e^{i \lambda_{n} \eta}
\end{equation}
converges in $\mathbf{L}_{2}$ for all $\eta > 0$, then $g(x)$ is holomorphic at all points $x \in(0, \pi)$ except the points where $a(x)$ attains the maximum.}

{\bf Proof:} Let $M=\max a(x)$. For any $\varepsilon > 0$ we have $\left\|f_{n}\right\| \geq e^{(M-\varepsilon) n}$ for all sufficiently large $n$. Assuming that series \eqref{13::3.12} converges in $\mathbf{L}_{2}$  for $\eta=\varepsilon$ we get the estimate |($\left|\left(g, f_{n}^{*}\right)\right|<e^{(2 \varepsilon-M) n}$ (under our assumption the norms of functions in series \eqref{13::3.12} tend to 0 as $n \rightarrow \infty$). If $a (\xi) < M$ and $\varepsilon$ is small enough, then there is a neighborhood $U$ of $\xi$ such that series \eqref{13::3.12} converges uniformly for all $x \in U$ and $0 \leq \eta \leq \varepsilon$. Since the terms of the series are holomorphic functions of $x$, the sum is also holomorphic at $\xi$. According to Theorem 3.4 this sum coincides with $g(x)$. Thus $g(x)$ is holomorphic at $\xi$.	$\square$

{\bf Remark 3.6} It follows from the proof of Proposition 3.5 that $g(x)$ admits a holomorphic continuation in the domain
$$
\Lambda=\{z | \operatorname{Re} a(z)-M<0\}
$$
if the series \eqref{13::3.12} is summable by the Abel method of order 1 (i.e., converges for all $\eta > 0$). We do not know if the converse assertion is also true.

Let us formulate the basic result of this section.

{\bf Theorem 3.7} {\it There is the only solution $u(x,y)$ of scattering problem \eqref{13::1.2}- \eqref{13::1.4},  \eqref{13::1.6}. If $C_{R}=\left\{x, y | x^{2}+y^{2}<R^{2}\right\}$  then $u(x, y) \in W_{2, l o c}^{2}\left(\Omega \cap C_{R}\right)$ for any $R > 0$. For large $y$ the solution $u(x,y)$ is represented by the Fourier series with respect to the generalized Rayleigh system $\left\{f_{n}^{+}(x)\right\}$.}

{\bf Proof: }It suffices to put $\eta=y-a(x)$ and recall Theorems 2.11 and 3.4.	$\square$

\subsection{Scattering by Two-periodic Surfaces in M3}

Let a smooth function $z = a(x, y)$ be $2\pi$-periodic with respect to both variables $x,y$. This function defines the surface (the grating) $S$ in $\mathbb{R}^{3}$. Let
$$
v(x, y, z)=e^{-i k(x \cos \alpha+y \cos \beta+z \cos \gamma)}
$$
be the wave incident onto this surface with directing vector $\bar{\varphi}=(\cos \alpha, \cos \beta, \cos \gamma)$, $|\bar{\varphi}|=1$.

%Figure 4

The wave $v$ satisfies the Helmholtz equation
\begin{equation}\label{13::4.1}
\Delta u+k^{2} u=0, \quad u=u(x, y, z)
\end{equation}
and quasi-periodic conditions
\begin{equation}\label{13::4.2}
\begin{array}{ll}{u(0, y, z)=e^{-i \nu} u(2 \pi, y, z),} & {u_{x}^{\prime}(0, y, z)=e^{-i \nu} u_{x}^{\prime}(2 \pi, y, z)}, \\ {u(x, 0, z)=e^{-i \delta} u(x, 2 \pi, z),} & {u_{y}^{\prime}(x, 0, z)=e^{-i \delta} u_{y}^{\prime}(x, 2 \pi, z)}\end{array},
\end{equation}

where $\nu=2 \pi k \cos \alpha, \delta=2 \pi k \cos \beta$. Scattered waves do the same. It is assumed that the wave v reaches all the points of the surface (there are no shadows). The problem is to find a quasi-periodic solution of \eqref{13::4.1} that is represented (in some sense) as a superposition of scattered waves and satisfies the "full reflection" condition
\begin{equation}\label{13::4.3}
u(x, y, a(x, y))=v(x, y, a(x, y)).
\end{equation}

To define the scattered waves let us find all quasi-periodic solutions of equation \eqref{13::4.1} in the cylinder with the base $K_{2 \pi}=[0,2 \pi] \times[0,2 \pi]$. Separating the variables $x$ and $y$, we find that the elementary solutions have the representation
\begin{equation}\label{13::4.4}
u_{n j}^{\pm}(x, y, z)=e^{i \mu_{n} x} e^{i \rho_{j} y} e^{\pm i \lambda_{n j} z},
\end{equation}
where
\begin{equation}\label{13::4.5}
\begin{aligned} \mu_{n} &=\frac{\nu}{2 \pi}+n, \quad \rho_{j}=\frac{\delta}{2 \pi}+j \\ \lambda_{n j} &=\sqrt{k^{2}-\mu_{n}^{2}-\rho_{j}^{2}}, \quad n, j \in \mathbb{Z} \end{aligned}.
\end{equation}
The branch of the square root is chosen so that either $\lambda_{n j} \geq 0$ or $\operatorname{Im} \lambda_{n j}>0$. The set of scattered waves consists of outgoing propagating waves and decaying waves. It coincides with the system $\left\{u_{n j}^{+}\right\}$. We can not guarantee that the reflected solution is a finite or infinite superposition of the scattered waves. Therefore, we are looking for solutions satisfying the radiation condition at $z \rightarrow \infty$, namely
\begin{equation}\label{13::4.6}
u(x, y, z)=\sum_{\mu_{n}^{2}+\rho_{j}^{2} \leq k^{2}} c_{n j} e^{i\left(\mu_{n} x+\rho_{j} y+\lambda_{n j} z\right)}+o(1).
\end{equation}
Here $o(1) \rightarrow 0$ as $z \rightarrow \infty$ and $c_{nj}$ are unknown constants to be determined.

The substitution of
$$
\xi=x, \quad \eta=y, \quad \zeta=z-a(x, y)
$$
maps the half-cylinder
$$
\Omega=\{x, y, z | 0 \leq x \leq 2 \pi, 0 \leq y \leq 2 \pi, a(x, y) \leq z<\infty\}
$$
onto a usual half-cylinder whose base is the square $K_{2 \pi}$. This substitution transforms equation \eqref{13::4.1} to the form
$$
\begin{array}{l}{u_{\xi \xi}+u_{\eta \eta}+u_{\zeta \zeta}-2 u_{\xi} a_{\xi}^{\prime}-2 u_{\eta \zeta} a_{\eta}^{\prime}} \\ {\quad+u_{\zeta \zeta}\left(a_{\xi}^{\prime 2}+a_{\eta}^{\prime 2}\right)-u_{\zeta} a_{\xi}^{\prime \prime}-u_{\zeta} a_{\eta}^{\prime \prime}+k^{2} u=0}\end{array}
$$
Separating the variable $\zeta$ by putting $u=f(\xi, \eta) e^{i \lambda \zeta}$ and taking into account boundary conditions \eqref{13::4.2} , we find
\begin{equation}\label{13::4.7}
T(\lambda) f=\left(F \lambda^{2}+\lambda G+H-\left(k^{2}+1\right) I\right) f=0,
\end{equation}
where the operators $H$, $G$, $F$  are defined as follows
$$
\begin{aligned} H f &=-f_{\xi \xi}^{\prime \prime}-f_{\eta \eta}^{\prime \prime}+f, &  \mathcal{D}(H)&=W_{U}^{2}\left[K_{2 \pi}\right] \\ G f &=i\left(2 a_{\xi}^{\prime} f_{\xi}^{\prime}+2 a_{\eta}^{\prime} f_{\eta}^{\prime}+\left(a_{\xi \xi}^{\prime \prime}+a_{\eta \eta}^{\prime \prime}\right) f\right), &  \mathcal{D}(G)&=W_{U}^{1}\left[K_{2 \pi}\right] \\ F f &=\left(1+a_{\xi}^{\prime 2}+a_{\eta}^{\prime 2}\right) f, & \mathcal{D}(F)&=\mathbf{L}_{2}\left(K_{2 \pi}\right) \end{aligned}
$$
Here $W_{U}^{s}\left[K_{2 \pi}\right]$ is the subspace of the Sobolev space $W_{2}^{s}\left[K_{2 \pi}\right]$ consisting of functions subject to the quasi-periodic boundary conditions. It is easily seen that the
eigenfunctions of pencil \eqref{13::4.7} coincide with the traces of elementary solutions \eqref{13::4.4} on the surface $z = a(x,y)$, i.e.,
\begin{equation}\label{13::4.8}
f_{n j}^{\pm}(x, y)=e^{\pm i \lambda_{n j} a(x, y)} e^{i \mu_{n} x} e^{i \rho_{j} y}.
\end{equation}
This can be checked independently by a straightforward calculation if one puts in \eqref{13::4.7}
$$
f(\xi, \eta)=v(\xi, \eta) e^{i \lambda a(\xi, \eta)}.
$$
Then the function $v$ satisfies the equation
$$
\Delta v+\left(k^{2}-\lambda^{2}\right) v=0
$$
and quasi-periodic boundary conditions. This holds for functions $v_{n j}=e^{i \mu_{n} x} e^{i \rho_{j} y}$, where $\mu_n$ and $\rho_j$ are defined by \eqref{13::4.5}. We remark that the multiplicities of the
eigenvalues $\lambda_{n j}=\sqrt{k^{2}-\mu_{n}^{2}-\rho_{j}^{2}}$ may grow as $k \to \infty$. For example, if $\nu =
\delta = 0$ and $k^2 =50$, then zero is the eigenvalue of the geometric multiplicity 16 ($\lambda_{nj} = 0$ for $n = \pm 1, \pm 4, \pm 6, \pm 7$ and $j = \pm 7, \pm 6, \pm 4, \pm1$, respectively) and of algebraic multiplicity 32 (all Jordan chains have length 2 and associated functions coincide with eigenfunctions).

Let us prove that $T (\lambda)$ is an elliptic pencil. The properties
$$
H=H^{*} \gg 0, \quad 0 \ll F \ll \infty, \quad G \subset G^{*},
$$
$$
\mathcal{D}(G) \subset W_{U}^{1}\left(K_{2 \pi}\right)=\mathcal{D}\left(H^{1 / 2}\right),
$$
are trivial to check. Further,
$$
\begin{array}{l}{(T(\lambda) f, f)=\left\|f_{\xi}^{\prime}\right\|^{2}+\left\|f_{\eta}^{\prime}\right\|^{2}-2\left(f_{\xi}^{\prime}, a_{\xi}^{\prime} f\right)-2 \lambda \operatorname{Im}\left(f_{\eta}^{\prime}, a_{\eta}^{\prime} f\right)} \\ {\quad+\lambda^{2}\left(\left(1+a_{\xi}^{\prime 2}+a_{\eta}^{\prime 2}\right) f, f\right)-k^{2}(f, f)>0},\end{array}
$$
provided $\lambda ^2 > k^2$, $\lambda \in \mathbb{R}$. Since the embedding $I: W_{U}^{2}\left(K_{2 \pi}\right) \rightarrow L_{2}\left(K_{2 \pi}\right)$ is compact, the last estimate implies that $T(\lambda)$ is strongly elliptic. It is important in the sequel to find explicitly a double angle where estimate\eqref{13::2.1} holds.

{\bf Proposition 4.1} {\it Let
$$
M_{1}=\max _{x, y \in K_{2 \pi}}\left|a_{x}^{\prime}(x, y)\right|, \quad M_{2}=\max _{x, y \in K_{2 \pi}}\left|a_{y}^{\prime}(x, y)\right|,
$$
$$
\varphi=\operatorname{arctg} \frac{1}{\sqrt{1+M_{1}^{2}+M_{2}^{2}}}.
$$
If $|\arg \pm \lambda|=\theta<\varphi$, then for $|\lambda|=r>r_{0}$ the following estimate holds
\begin{equation}\label{13::4.9}
\operatorname{Re}(T(\lambda) f, f) \geq \varepsilon\left(r^{2}(f, f)+(H f, f)\right)
\end{equation}}

{\bf Proof:} Let $\lambda=r e^{i \theta}$ and let $\alpha, \beta$ be positive numbers such that $\alpha+\beta=1$. We have
\begin{equation}\label{13::4.10}
\operatorname{Re}\left(\left\|f_{\xi}^{\prime}\right\|^{2}-2 r e^{i \theta} \operatorname{Im}\left(f_{\xi}^{\prime}, a_{\xi}^{\prime} f\right)+r^{2} e^{2 i \theta}\left(\alpha+{a_{\xi}^{\prime}}^{2}\right)(f, f)\right)\geq \varepsilon\left(\left\|f_{\xi}^{\prime}\right\|^{2}+r^{2}(f, f)\right)
\end{equation}
for some $\varepsilon=\varepsilon(\theta)$, provided
\begin{equation}\label{13::4.11}
\cos ^{2} \theta-\cos 2 \theta\left(1+\frac{\alpha}{M_{1}^{2}}\right)<0.
\end{equation}
Similarly,
\begin{equation}\label{13::4.12}
\begin{array}{l}{\operatorname{Re}\left[\left\|f_{\eta}^{\prime}\right\|^{2}-2 r e^{i \theta} \operatorname{Im}\left(f_{\eta}^{\prime}, a_{\eta}^{\prime} f\right)+r^{2} e^{2 i \theta}\left(\beta+a_{\xi}^{\prime 2}\right)(f, f)\right]} \\ {\quad \geq \varepsilon\left(\left\|f_{\eta}^{\prime}\right\|^{2}+r^{2}(f, f)\right)}\end{array}
\end{equation}
provided
\begin{equation}\label{13::4.13}
\cos ^{2} \theta-\cos 2 \theta\left(1+\frac{\beta}{M_{2}^{2}}\right)<0.
\end{equation}
The inequalities \eqref{13::4.11} and \eqref{13::4.13} are equivalent to the following ones
$$
\operatorname{tg}^{2} \theta \leq \frac{\alpha}{M_{1}^{2}+\alpha}, \quad \operatorname{tg}^{2} \theta<\frac{\beta}{M_{2}^{2}+\beta}.
$$
For $\alpha=\frac{M_{1}^{2}}{M_{1}^{2}+M_{2}^{2}}, \beta=\frac{M_{2}^{2}}{M_{1}^{2}+M_{2}^{2}}$ last inequalities are equivalent to the condition $|\theta|<\varphi$ . Summing inequalities\eqref{13::4.10} and \eqref{13::4.12} with the chosen $\alpha$ and $\beta$ we obtain estimate \eqref{13::4.9}. $\square$

Analyzing the proof of Proposition 4.1 one can understand that the bound for $\varphi$ is precise, i.e., estimate \eqref{13::4.9} does not hold generally inside the angle $\varphi<\arg \lambda<\pi - \varphi$. Seemingly, $T^{-1}(\lambda)$ has an exponential growth inside this angle. The eigenvalue asymptotics of the Laplace operator on a bounded domain is known, hence, in our problem we have $\lambda_{j}(H)=c j^{p}$ with $p = 1$. Therefore, we can guarantee that the pencil $T (\lambda)$ belong to the class $K$ and we can claim (by virtue of Theorem 2.8) the completeness of the traces $\left\{f_{n j}^{+}(x, y)\right\}$ of the scattered waves only in the case $\varphi<\pi / 4$. The problem whether the system  $\left\{f_{n j}^{+}\right\}$ is complete in the case $\varphi \geq \pi / 2$ is open. Nevertheless, the following basic result is true.

{\bf Theorem 4.2} {\it There exists the only solution $u(x,y,z)$ of scattering problem
\eqref{13::4.1} -\eqref{13::4.3}, \eqref{13::4.6}.}

{\bf Proof:} Elliptic pencil \eqref{13::4.7} is strongly regular. One can prove this fact repeating the arguments of Proposition 2.5. Putting $z=\zeta+a(x, y)$ and recalling Theorem 2.11 we obtain the assertion of the theorem.	 $\square$

\renewcommand{\refname}{{\large\rm Bibliography for Section~13}}

\newpage
\section{On the stability of a top with a cavity filled with a viscous fluid}
\subsection{Introduction}

We consider small oscillations of a rotating top with a cavity entirely filled with an incompressible viscous fluid. There are no restrictions on the distribution of mass in the body and on the form of the cavity. In the nonperturbed state, the velocity field of the fluid is that of a rigid body rotating together with the top about the vertical axis that coincides with one of the principal axes of inertia of the system. The shell has a fixed point. The system is moving in a gravitational field. The center of gravity lies on the principal axis of the nonperturbed rotation.
The main goal of this work is to obtain a stability criterion for the system and to investigate spectral properties of the evolution operator corresponding to the linear equations. This problem has a long history. In the first place, the investigations of Sobolev \cite{14::Sob}, Rumyantsev \cite{14::Rum2}, and Chernous'ko \cite{14::Cher3}, including sufficient conditions for stability and instability of the top, must be noted (see also the monographs \cite{14::Rum4},\cite{14::Kop5}). Rumyantsev \cite{14::Rum2},\cite{14::Rum4} showed that the top is stable if both
\begin{equation}\label{eq:1}
n_{1}=a_{0}-a_{1}-k / \omega^{2} \quad \text{and} \quad n_{2}=a_{0}-a_{2}-k / \omega^{2}
\end{equation}
are positive. Here $\omega$ is the angular velocity of the top. The other constants are defined in \S1. In particular, if the top has a symmetry axis of order greater than one, then $a_1= a_2$ and the positiveness of the number $n_1 =n_2$ is sufficient for the stability. In the absence of gravity, this sufficient condition means that the axis of the nonperturbed rotation is the axis of the greatest central moment of inertia. Rumyantsev obtained this condition with the help of the Lyapunov second method, following Sobolev, who obtained the same condition for the case of an ideal fluid (i.e., zero viscosity). He showed that the evolution operator corresponding to the linear problem is self-adjoint in a Hilbert space with indefinite metric. The metric is definite if $n_1 =n_2 > 0$. On the other hand, Chernous'ko \cite{14::Cher3} obtained conditions of instability in the problem. He found asymptotic expansions for solutions of the linear equations of motion of the top with respect 
 to the powers of the Reynolds number and showed that the top is not stable if the sods of the nonperturbed rotation is the axis of the minimal or middle moment of inertia. Later, Smirnova \cite{14::Smir21} obtained this result without certain symmetry assumptions of Chernous'ko. She also showed that the rotation about the axis of greatest moment of inertia is the unique stable rotation if the viscosity of the fluid is sufficiently small and the cavity is toroidal \cite{14::Smir6}. Chernous'ko \cite{14::Cher7} obtained this result for the top with a spherical cavity. In \cite{14::Rum4}, Rumyantsev obtained some implicit conditions of instability.

In this paper, we show that the Unear equation for the problem can be represented in the form with an evolution operator that is dissipative in a space with indefinite metric (a Pontryagin space). Hence, we can use nontrivial results of the theory of such operators and obtain necessary and sufficient conditions of stability. Moreover, we obtain the exact value of the instability index of the problem. The assumptions on the symmetry of the top, the absence of gravity, and on the large or small viscosity turn out to be superfluous. Furthermore, unlike the previous works, we.investigate the stability of the infinite-dimensional system as a whole but not its finite-dimensional ``rigid'' part. One of the main results of the paper is that the instability index of the problem does not depend on the nonzero viscosity and on the form of the
cavity. It is equal to the number of negative eigenvalues of the matrix $\left(\begin{array}{cc}{n_{1}} & {0} \\ {0} & {n_{2}}\end{array}\right)$. This fails to be true if the viscosity is zero (see \cite{14::Sob},\cite{14::Yur8}).

The methodology of our approach is close to that of the work of Sobolev \cite{14::Sob} However, in the derivation of the equations of motion of the system, Sobolev used the coordinate system rotating with respect to the inertial frame at the same angular velocity as the non-perturbed top. These equations turn out to be complicated for analysis. They admit a representation in an operator form only in the case of zero viscosity (see \S5). Following [4], we write out the equations of motion of the top in a coordinate system attached to the rigid shell and obtain linear equations that admit a convenient operator form (cf. \cite{14::Yur8},\cite{14::Yur9}).
The work is divided into five sections. In \S1 we more exactly define the setting of the problem and write out the equations of motion of the system. In \S2 the linear equations are represented in the form of an operator equation in a Hilbert space. We show that the evolution operator $T$ of the system is a maximal dissipative operator in the Pontryagin space. In \S3 we find the number of eigenvalues of the operator $T$ in the open lower half-plane and prove that this number coincides with the instability index. In \S4 the basis property of the eigenfunctions of the operator T is investigated and the description of the spectrum for large viscosity is given. In \S5 we consider a symmetric top. The family of invariant subspaces of the evolution operator $T$ is given. The subspaces that contain unstable motions of the top are singled out. The correspondence between the operator $T$ and Sobolev's operator is given.

%\subsection{Equations of Motion of the Top with the Fluid}
\
subsection*{Notations. Spaces and operators of fluid mechanics.}
Suppose that a rigid body occupies a bounded region $\Omega_1$ in $ \mathbb{R}^3$ and contains a cavity $\Omega$ inside, that is, $\Omega\subset\Omega_1$. The cavity is completely filled with fluid. Further, assume that $\Omega$ is a domain in $\mathbb{R}^3$. The rigid body itself occupies the domain $\Omega_1\setminus \Omega$ and is called the shell. Speaking of a body with frozen fluid, we imply the following: the region ft) as a whole is considered as a rigid body; the density of this body is equal to that of the original one in $\Omega_1\setminus \Omega$ and is equal to the density of the fluid in $\Omega$. The body is rotating about a fixed origin of coordinates. Let $e_o$, $e_1$, $e_2$ be the unit vectors of the principal axes of inertia of the body with frozen fluid. Consider the orthonormal frame $Oe_0e_1e_2$ rigidly attached to the body and a fixed orthonormal frame $Of_0f_1f_2$. In the non-perturbed state, $e_0$ coincides with $f_0$. The motion of the system ``(body+fluid)'
 ' is described by three variables $(z, w, v(x))$, namely, the coordinates $z$ of $f_0$, the angular velocity w of the shell, and the velocity field $v(x)$ of the fluid. Here $x = (x_0, x_1,x_2)\in\Omega$ are the coordinates of a point. The four variables $z$, $w$, $v(x)$, and $x$ axe vectors in the frame $Oe_0e_1e_2$.

The finiteness of kinetic energy implies that the values of the velocity field $v$ belong to the subspace $J_0(\Omega)\subset L^3_2(\Omega)$. The subspace $J_0(\Omega)$ is the closure in $L^3_2$ of the set of smooth solenoidal vectors compactly supported in $\Omega$ (i.e., $\div w = 0$). We have the Weyl decomposition
$$
L_{2}^{3}(\Omega)=J_{0}(\Omega) \oplus G(\Omega)
$$
Here $G(\Omega)$ consists of the functions $\operatorname{grad} p(x)$, where $p(x)$ is a scalar locally square integrable function whose first generalized derivatives belong to $L_2(\Omega)$.

The following operators are well known in fluid mechanics (see \cite{14::Lad10} and \cite{14::Kop5}). The vector
$$
B^{*} v=\int_{\Omega}[x, v(x)] d x
$$
is called the gyrostatic moment of the fluid. Here $[\cdot,\cdot]$ denotes the vector product in $\mathbb{R}^3$. Obviously, $B^{*}: J_{0}(\Omega) \rightarrow \mathbb{R}^{3}$ is a bounded operator and the adjoint operator is defined by $B w=P_{0}([w, x])$, where $P_0$ denotes the orthogonal projection of $L_{2}^{3}(\Omega)$ onto $J_0(\Omega)$. It is easily shown that $\operatorname{rot} Bw \equiv 2w$. It is known [4] that if the cavity is simply connected, then the operator $B: \mathbb{R}^{3} \rightarrow J_{0}(\Omega)$ is uniquely determined by this identity.

Let $A$ be the Laplace operator acting in the subspace $J_0(\Omega)$. The domain of this operator consists of smooth functions $v \in J_0(\Omega)$ compactly supported in $\Omega$. Since the operator $R = P_0\Delta$ is nonnegative, we can define its Friedrichs extension. As before, denote it by $R$. This operator is called the Stokes operator. It is difficult to describe the domain of $R$ if the boundary is not smooth, but it is easy to find the domain of its quadratic form (the domain of $R^{1/2}$). Namely, it coincides with $J_{0}(\Omega) \cap \stackrel{\circ}{H}_{1}(\Omega)$, where $\stackrel{\circ}{H}_{1}(\Omega)$ is
the closure of the set of smooth functions compactly supported in $\Omega$ with respect to the metric in the Sobolev space $H_1(\Omega)$. It is known \cite{14::Lad10} that the Stokes operator is uniformly positive and has a discrete spectrum.

In the sequel, we work with the complex space $\mathbb{C}^3$ instead of the real space $\mathbb{R}^3$. Consider the operators $A$ and $A_0$ on $\mathbb{C}^3$ defined by the following quadratic forms:
$$
(A w, w)=\int_{\Omega_{1}}\|[w, x]\|^{2} d \rho(x), \quad\left(A_{0} w, w\right)=\int_{\Omega_{1} \backslash \Omega}\|[w, x]\|^{2} d \rho(x),
$$
where $\rho(x)$ is the distribution of mass in the shell $\Omega_1\setminus \Omega$ and the distribution of mass of the fluid in $\Omega$. The corresponding matrices are the inertia tensors of the body with frozen fluid and the shell, respectively. According to the definition of the frame $Oe_0e_1e_2$, we have $A = \operatorname{diag}\{a_0, a_1, a_2\}$, where the $a_j$ are the moments of inertia of the body with frozen fluid with respect to the principal axes of inertia $Ox_j$. Further, we assume that $A_0 > 0$, that is, the shell is not weightless.

Consider the operator $H$ in $\mathbb{C}^3$ defined by
\begin{equation}\label{eq:2}
i H x=\left[x, e_{0}\right] \quad \text {or} \quad H=\left(\begin{array}{ccc}{0} & {0} & {0} \\ {0} & {0} & {-i} \\ {0} & {i} & {0}\end{array}\right)
\end{equation}
The operator
$$
G = -2 P_0H
$$
acting in the space $J_0(\Omega)$ is called the gyroscopic operator. Obviously, $G$ is a bounded self-adjoint operator. By $I$ denote the identity operator in $\mathbb{C}^3$ and in $J_0(\Omega)$.

By $v$ we denote the viscosity of the fluid. Throughout the following, unless otherwise specified, we assume that $v\ne 0$. By $k$ we denote the gravity force divided by the cosine of the angle between $f_0$ and $e_0$. We have $k = glm$, where $l$ is the distance from the fixed point to the center of gravity of the system, $m$ is the mass of the system, and $g$ is the gravitational acceleration. We can assume, without loss of generality, that the density of the fluid is equal to 1.

\subsection*{Equations of motion of the top with fluid.}

The evolution of the system is described by the equations (e.g., see \cite{14::Rum4})
\begin{equation}\label{eq:3}
\begin{aligned} \dot{z} &=[z, w] \\ A \dot{w}+B^{*} \dot{v} &=\left[A w+B^{*} v, w\right]+k\left[z, e_{0}\right] \\ \nabla p+\nu \Delta v &=\dot{v}+(v \cdot \nabla) v+2[w, v]+[w,[w, x]]+[\dot{w}, x], \quad \operatorname{div} v=0 \end{aligned}
\end{equation}
The first equation is the kinematic relation, and $z$ is called the Poisson vector. The second equation describes the evolution of the kinetic moment of the system with respect to the point $O$. The third equation is the Navier--Stokes equation of motion of the fluid in the coordinate system attached to the rigid body. Applying the projector $P_0$ to this equation and using the relation
$$
P_{0} \nabla p=0, \quad P_{0}[w,[w, x]]=0, \quad P_{0}(u \cdot \nabla) u=P_{0}([v, \operatorname{rot} v]),
$$
we obtain the equation
$$
B \dot{w}+\dot{v}=P_{0}([v,(2 w+\operatorname{rot} v)])+\nu P_{0} \Delta v.
$$
Let us linearize the first two equations of the motion of the system and the obtained third equation at the following stationary solution of system \eqref{eq:3}:
$$
z_{0}=e_{0}, \quad w_{0}=\omega e_{0}, \quad v_{0}=0,
$$
where $\omega$ is a constant (the angular velocity of the non-perturbed rotation).

Let us substitute the shifts $z_0 + z$, $w_0 + w$, $v_0 + v$ for $z$, $w$, $v$ into \eqref{eq:1} and retain only the linear terms. Taking into account \eqref{eq:2}, we obtain the linearized equations of motion
\begin{equation}\label{eq:4}
\begin{aligned} \dot{z} &=i \omega\left(H z-\omega^{-1} H w\right) \\ A \dot{w}+B^{*} \dot{v} &=i \omega\left(k \omega^{-1} H z-H\left(a_{0} I-A\right) w+H B^{*} v\right) \\ B \dot{w}+\dot{v} &=i \omega\left(-2 P_{0} H+i \nu \omega^{-1} P_{0} \Delta\right) v \end{aligned}.
\end{equation}
It is more convenient to investigate an operator form of these equations.

\subsection{Analysis of the Operator Equation Corresponding to the Linear Equations of Evolution of the System}

The operator form of the linear equations. Consider the Hilbert space
$$
\mathfrak{H}=\mathbb{C}^{3} \times \mathbb{C}^{3} \times J_{0}(\Omega).
$$
The elements of this space are columns $u = (z, w, v)^t$. Here $t$ denotes the transposition. The first two entries of a column lie in $\mathbb{C}^3$ and the last one belongs to $_0(\Omega)$. To each operator in $\mathfrak H$, a $3\times 3$ matrix is naturally assigned. The entries of this matrix are operators acting in the spaces $\mathbb{C}^3$ and $J_0(\Omega)$ and between them. Obviously, system \eqref{eq:4} admits the operator form
\begin{equation}\label{eq:5}
\mathbf{W} \dot{\mathbf{u}}=i \omega \mathbf{M} \mathbf{u}, \quad \mathbf{u} \in \mathfrak{H},
\end{equation}
where $W$ and $M$ are the operator matrices
$$
\mathbf{w}=\left(\begin{array}{ccc}{I} & {0} & {0} \\ {0} & {A} & {B^{*}} \\ {0} & {B} & {I}\end{array}\right), \quad \mathbf{M}=\left(\begin{array}{ccc}{H} & {-\omega^{-1} H} & {0} \\ {k_{0}-1} & {H\left(A-a_{0} I\right)} & {H B^{*}} \\ {0} & {0} & {D}\end{array}\right)
$$
Here the operator $D$ acting in $J_0(\Omega)$ is defined by
$$
D = G + \frac{\nu}{\omega}R,
$$
where $G$ is the gyroscopic operator and $R$ is the Stokes operator. All the operators occurring in the preceding are bounded with the exception of the Stokes operator $R$. Let
$$
\mathcal{D}(\mathbf{M})=\mathbb{C}^{3} \times \mathbb{C}^{3} \times \mathcal{D}(R)
$$
be the domain of $M$. Since the operator $M$ does not possess any special properties, it is difficult to investigate Eq. \eqref{eq:5}. This operator is neither symmetric nor dissipative. Our idea is to ``guess'' an operator $S$ such that Eq.\eqref{eq:5} becomes more symmetric after applying the operator S to it. Namely, $\mathbf{SM}$ becomes dissipative and $\mathbf{SW}$ remains self-adjoint. Let

$$
\mathbf{S}=\left(\begin{array}{ccc}{\omega^{2}(A+N)} & {-\omega I} & {0} \\ {-\omega A} & {I} & {0} \\ {-\omega B} & {0} & {I}\end{array}\right), \quad \text { where } N=\left(\begin{array}{ccc}{1} & {0} & {0} \\ {0} & {n_{1}} & {0} \\ {0} & {0} & {n_{2}}\end{array}\right)
$$
with $n_1$ and $n_2$ defined in \eqref{eq:1}.

\begin{proposition}
$$
\begin{array}{c}{\mathbf{J}:=\mathbf{S W}=\left(\begin{array}{ccc}{\omega^{2}(A+N)} & {-\omega A} & {-\omega B^{*}} \\ {-\omega A} & {A} & {B^{*}}, \\ {-\omega B} & {B} & {I}\end{array}\right)} \\ {\mathbf{L}:=\mathbf{S M}=\left(\begin{array}{ccc}{\left(\omega^{2} a_{0}-2 k\right) H} & {-\omega H A+k \omega^{-1} H} & {-\omega H B^{*}} \\ {-\omega A H+k \omega^{-1} H} & {A H+H A-a_{0} H} & {H B^{*}} \\ {-\omega B H} & {B H} & {D}\end{array}\right)}\end{array}
$$
\end{proposition}
\begin{proof}
The proof is by straightforward calculation using the formula $N H=\left(a_{0} I-A-k \omega^{-2} I\right) H$.
\end{proof}

\begin{proposition} If $N$ is nonsingular, then the operator $S$ acting in $\mathfrak H$ is a bounded, boundedly invertible operator.
\end{proposition}
\begin{proof}
The proposition follows from the relation
$$
\left(\begin{array}{cc}{\omega^{2}(A+N)} & {-\omega I} \\ {-\omega A} & {I}\end{array}\right)=\left(\begin{array}{cc}{I} & {-\omega I} \\ {0} & {I}\end{array}\right)\left(\begin{array}{cc}{\omega^{2} N} & {0} \\ {0} & {I}\end{array}\right)\left(\begin{array}{cc}{I} & {0} \\ {-\omega A} & {I}\end{array}\right).
$$
\end{proof}

\begin{proposition}
The operator $\mathbf{W}$ acting in $\mathfrak H$ is uniformly positive.
\end{proposition}
\begin{proof} Since $\mathbf{W}$ is a finite-dimensional perturbation of the identity operator in $\mathfrak H$, by virtue of the Fredholm alternative it suffices to show that $\mathbf{W} > 0$. It follows from the definitions of $A$, $A_0$, and $B$ that
$$
(\mathbf{W} \mathbf{u}, \mathbf{u})=(z, z)+\left(A_{0} w, w\right)+\int_{\Omega}([w, x]+v(x),[w, x]+v(x)) d x.
$$
Hence $\mathbf{W}> \mathbf{C}$. Suppose that $\mathbf{Wu} = 0$; then $z=0$ and $w = 0$ (we recall that $A_0 > 0$). Thus $v(x)\equiv 0$.	
\end{proof}
\begin{proposition}
Let $N$ be nonsingular. Then $J$ is a bounded invertible self-adjoint operator acting in $\mathfrak H$. Its negative spectrum consists of finitely many eigenvalues. The number of negative eigenvalues of the operator $J$ coincides with that of the matrix $N$.
\end{proposition}

\begin{proof}
Consider the following operators in $\mathfrak H$:
$$
\mathbf{J}_{0}=\left(\begin{array}{ccc}{\omega^{2} N} & {0} & {0} \\ {0} & {A-B^{*} B} & {0} \\ {0} & {0} & {I}\end{array}\right), \quad \mathbf{S}_{0}=\left(\begin{array}{ccc}{I} & {0} & {0} \\ {-\omega I} & {I} & {0} \\ {-\omega B} & {B} & {I}\end{array}\right)
$$
By virtue of Proposition 3, $W$ is a uniformly positive operator. Hence, so is the operator
$$
\left(\begin{array}{cc}{A} & {B^{*}} \\ {B} & {I}\end{array}\right)=\left(\begin{array}{cc}{I} & {0} \\ {B} & {I}\end{array}\right)\left(\begin{array}{cc}{A-B^{*} B} & {0} \\ {0} & {I}\end{array}\right)\left(\begin{array}{cc}{I} & {B^{*}} \\ {0} & {I}\end{array}\right)
$$
Therefore $A -B^*B>>0$ and $\pi_-(J_0)=\pi_-(N)$, where $\pi_-$ denotes the number of negative eigenvalues
of the operators. Further, it can be easily checked that $J =S^*_0J_0S_0$. Hence, $J$ is invertible. Moreover, $\pi_-(J_0)=\pi_-(J)=\pi_-(N)$. This concludes the proof.
\end{proof}

We recall that an operator $C$ is called dissipative if $\Im(Cx,x) > 0$ for all $x\in\mathcal{D}(C)$. If the open lower half-plane $\mathbb{C}^-$ (or at least one point $\mu\in \mathbb{C}^-$) belongs to the resolvent set of the operator $C$, then $C$ is called maximal dissipative. Further, we use the notion of a dissipative operator in a Pontryagin space. Suppose that $J$ is a bounded invertible operator and $\pi_-=\varkappa <\infty$. Then the Hilbert space $\mathfrak{H}$ with the indefinite scalar product $(Jx,x)$ is called a Pontryagin space and is denoted by $\Pi_\varkappa=\{J,\mathfrak{H}\}$. An operator $C$ is called dissipative in if $\Im(JCx, x) > 0$ for all $x\in\mathcal{D}(C)$. If there is a point $\mu\in \mathbb{C}^-$ (but not obligatorily the whole half-plane) that belongs to the resolvent set of $C$, then $C$ is called maximal dissipative in $\Pi_\varkappa$. It is known (see \cite[ Chap. 2, Theorem 2.10]{14::Az11}) that $C$ is maximal dissipative in $\Pi_\varkappa$ if
  and only if $JC$ is maximal dissipative in $\mathfrak{H}$.

The main goal of the previous transformations is to prepare the proof of the following result.
\begin{theorem}
Suppose that $N$ is nonsingular; then Eq.\eqref{eq:5} of the evolution of the system is equivalent to the equation
\begin{equation}\label{eq:6}
\dot{\mathbf{u}}=i \omega \mathbf{T} \mathbf{u}
\end{equation}
where the operator $T = J^{-1}L$ with the domain $\mathbb{C}^3\times \mathbb{C}^3 \times \mathcal{D}(R)$ is maximal dissipative in the Pontryagin space $\Pi_\varkappa=\{J,\mathfrak{H}\}$, where $\varkappa=\pi_-(N)$.
\end{theorem}
\begin{proof}
It follows from the representation
$$
\mathbf{L}=\mathbf{V}+\mathbf{D}, \quad \mathbf{D}=\operatorname{diag}\{I, I, D\}, \quad D=G+i \nu \omega^{-1} R
$$
where $\mathbf{V}$ is a self-adjoint, finite-dimensional operator, $G = G*$ is the gyroscopic operator, and $R = R*>>0$ is the Stokes operator, that $\Im(JTx, x) = \Im(Lx, x) > 0$ for all $x\in\mathcal{D}(T)$.

Obviously, $D$ is maximal dissipative in $\mathfrak{H}$. Since $V$ is bounded and self-adjoint,	$L$  	is	 also	maximal dissipative. Hence, since $T$ has no nontrivial $J$-dissipative extensions, $T$ is maximal $J$-dissipative in $\Pi_\varkappa$. The result also follows from the discreteness of the spectrum of $T$, which is proved inthe	sequel.	
\end{proof}

\begin{rem} Observe the following useful fact. If
$$
\mathbf{S}_{1}=\left(\begin{array}{ccc}{\left(\omega^{2} N\right)^{-1}} & {A(\omega N)^{-1}} & {(\omega N)^{-1} B^{*}} \\ {(\omega N)^{-1}} & {(A+N) N^{-1}} & {N^{-1} B^{*}} \\ {0} & {0} & {I}\end{array}\right)
$$
then
$$
\mathbf{J}_{1}:=\mathbf{W S}_{1}=\left(\begin{array}{ccc}{\left(\omega^{2} N\right)^{-1}} & {A(\omega N)^{-1}} & {(\omega N)^{-1} B^{*}} \\ {A(\omega N)^{-1}} & {A(A+N) N^{-1}} & {(A+N) N^{-1} B^{*}} \\ {B(\omega N)^{-1}} & {B(A+N) N^{-1}} & {B N^{-1} B^{*}+I}\end{array}\right)
$$
$$
\mathbf{L}_{1}:=\mathbf{M S}_{1}=\left(\begin{array}{ccc}{0} & {-\omega^{-1} H} & {0} \\ {-\omega^{-1} H} & {-a_{0} H} & {0} \\ {0} & {0} & {D}\end{array}\right).
$$
Moreover, it follows from the non-singularity of the matrix
$\left(\begin{array}{cc}{\omega^{-2}} & {A \omega^{-1}} \\ {\omega^{-1}} & {A+N}\end{array}\right)$
that $\mathbf{S}_1$ is invertible. Hence, after the substitution $\mathbf{u} = \mathbf{S}_1f$, Eq.\eqref{eq:5} becomes
$$
\mathbf{J}_{1} \dot{\mathbf{f}}=i \omega \mathbf{L}_{1} \mathbf{f}, \quad \text { or } \quad \dot{\mathbf{r}}=i \omega \mathbf{T}_{1} \mathbf{f}, \quad \mathbf{T}_{1}=\mathbf{J}_{1}^{-1} \mathbf{L}_{1}
$$
where $\mathbf{J}_1$ is self-adjoint and $\mathbf{L}_1$ is maximal dissipative in $\mathfrak{H}$. In addition, $\mathbf{J}_0$, $\mathbf{J}$, and $\mathbf{J}_1$ are congruent. Namely,
\begin{equation}\label{eq:7}
\mathbf{J}_{1}=\mathbf{S}_{1}^{*} \mathbf{J S}_{1}, \quad \text { where } \mathbf{J}=\mathbf{S}_{0}^{*} \mathbf{J}_{0} \mathbf{S}_{0}.
\end{equation}
Thus, $\mathbf{J}_1$ generates a Pontryagin metric and $\pi_{-}\left(\mathbf{J}_{1}\right)=\pi_{-}(\mathbf{J})=\pi_{-}\left(\mathbf{J}_{0}\right)=\pi_{-}(N)$. The substitution makes the operator $M$ more convenient for the investigation, since $\mathbf{L}_1$ is the direct sum of two operators acting in $\mathbb{C}^3\times \mathbb{C}^3$ and $\mathbf{J}_0(\Omega)$, respectively. This fact is used in the sequel. Also, note that the operators L and Li are congruent: $\mathbf{L}_1=\mathbf{S}^*_1\mathbf{LS}_1$.
\end{rem}

\subsection{A Stability Criterion and the Instability Index}

The absence of eigenvalues of the operator $\mathbf{T}$ in the open lower half-plane is a necessary condition for the stability of Eq. \eqref{eq:6}. The following theorem describes the spectrum of $\mathbf{T}$ in the closed lower half-plane.
\begin{theorem}
The spectrum of $\mathbf{T}$ is discrete;$\lambda = 0$ is the only real point of the spectrum. The subspace $\Ker\mathbf{T}$ is $\mathbf{J}$-positive and coincides with the linear span of $\mathbf{x}_0=(e_0,0,0)^t$ and $\mathbf{x}_1=(0,e_0,0)^t$ There are no associated vectors corresponding to the zero eigenvalue. There are exactly $\varkappa=\pi_-(N)$ eigenvalues of $\mathbf{T}$ in the open lower half-plane. In particular, if $\varkappa= 0$, then $\mathbb{C}^-$ is a subset of the resolvent set of $\mathbf{T}$.
\end{theorem}
\begin{proof}
Let $\mathbf{D} = diag(I,I, D)$. Obviously, it follows from the compactness of $^{-1}$ in $J_0(\Omega)$ that $\mathbf{D}^{-1}$ is a compact operator in $\mathfrak{H}$. We have
$$
\mathbf{T}-\lambda \mathbf{I}=\mathbf{J}^{-1}\left[\mathbf{I}+(\mathbf{V}-\lambda \mathbf{J}) \mathbf{D}^{-1}\right] \mathbf{D}, \quad \mathbf{V}=\mathbf{L}-\mathbf{D},
$$
where $\mathbf{V}$ is finite-dimensional. Since $\mathbf{T}$ is maximal dissipative, its resolvent set is not empty. Hence, it follows from the representation written out above and the theorem on a holomorphic operator function (see [12, Chap. 1]) that the spectrum of $\mathbf{T}$ is discrete.

Let us show that there are no nonzero real eigenvalues of $\mathbf{T}$. Since the spectrum of $\mathbf{T}$ is discrete, it suffices to show that $\Ker(\mu\mathbf{J} - \mathbf{L}) = 0$ whenever $0\ne\mu\in\mathbb{R}$. Suppose that
$$
(\mu \mathbf{J}-\mathbf{L}) \mathbf{u}=0, \quad \mathbf{u}=(z, w, v)^{t} \in \mathbf{f}, \quad 0 \neq \mu \in \mathbb{R},
$$

Then $0 =\Im((\mu\mathbf{J} - \mathbf{L})u,u) = (Rv, v)$; hence $v = 0$. Using the non-singularity of $S$ and the equality $\mathbf{S}(\mu\mathbf{W} - \mathbf{M}) =\mu\mathbf{J} - \mathbf{L}$, we get $(\mu\mathbf{W} \mathbf{M})\{z, w, 0\}$. From the definitions of $\mathbf{W}$ and $\mathbf{M}$, we successively obtain
$$
Bw = 0,\quad  w = 0, \quad Hz = 0, \text{ and } (\mu I - H)z = 0.
$$
Therefore, $u = 0$. Hence $\mathbf{T}$ has no real nonzero eigenvalues.
Let us show that zero is an eigenvalue and calculate the corresponding eigenvector. Suppose that $\mathbf{T}u = 0$. Using $\mathbf{T}=\mathbf{W}^{-1}\mathbf{M}$, we get $\mathbf{M}u = 0$. Using the matrix representation of $\mathbf{M}$, we prove that the null space of $\mathbf{M}$ is the linear span of $\mathbf{x}_0 = (e_0, 0, 0)^t$ and $\mathbf{x}_1 = (0,e_0, 0)^t$. Let $\mathbf{x} =\alpha\mathbf{x}_0 + \beta\mathbf{x}_1$. Hence, if $\mathbf{x}\ne0$, then
$$
(\mathbf{J}x, x) = \omega^{-2}|\alpha|^2 + a_0(a_0 + 1)|\beta|^22 + 2a_0\omega^{-1}\Re (\alpha\overline{\beta}) > 0,
$$
i. e., the null space of $\mathbf{T}$ is $\mathbf{J}$-positive. Thus, there are no associated vectors corresponding to $\mathbf{x} \in \Ker\mathbf{T}$ (see \cite[Chap. 2]{14::Az11}).

Let us use the fundamental theorem on the existence of invariant subspaces of maximal dissipative operators in Pontryagin spaces to prove a statement on the number of eigenvalues of $\mathbf{T}$ in the lower halfplane. Essentially, this theorem is due to Pontryagin \cite{14::Pon13} (although it is stated there only for self-adjoint operators). In the following form (and even more generally) it was obtained by Krein and Langer \cite{14::Kr14} and Azizov \cite{14::Az15}.

\textbf{Theorem on invariant subspaces.} \textit{Suppose that $\mathbf{T}$ is a maximal dissipative operator in a Pontryagin space $\Pi_\varkappa = \{J, \mathfrak{H}\}$. Then there exists a $\mathbf{T}$-invariant $\mathbf{J}$-non-positive $\varkappa$-dimensional subspace $\mathfrak{H}^-$ such that the spectrum of the restriction $\left.\mathbf{T}\right|_{\mathfrak{H}^-}$ lies in the closed lower half-plane and in the open half-plane $\mathbb{C}^-$ it coincides with the spectrum of $\mathbf{T}$.}

We proved that the whole real line except zero belongs to the resolvent set of $\mathbf{T}$. Zero is an eigenvalue of this operator, but the corresponding subspace is $\mathbf{J}$-positive. Hence, the intersection of this subspace with the $\mathbf{J}$-non-positive invariant subspace$\mathfrak{H}^-$ is zero. Therefore, the spectrum of the restriction $\left.\mathbf{T}\right|_{\mathfrak{H}^-}$ lies in the open lower half-plane $\mathbb{C}^-$. But according to the previous theorem, in $\mathbb{C}^-$ the spectrum of the restriction coincides with the spectrum of $\mathbf{T}$. Thus, the number of eigenvalues of $\mathbf{T}$ in $\mathbb{C}^-$ is equal to the number $\varkappa=\dim\mathfrak{H}^- =\pi_-(N)$.
\end{proof}

The dimension of the quotient space of the solution space for Eq. \eqref{eq:6} by the linear space of bounded solutions is called the instability index (i.e., the number of linearly independent unbounded solutions modulo bounded ones). Let $\nu(\mathbf{T})$ be the instability index of Eq. \eqref{eq:6}. Obviously,$\nu(\mathbf{T})$ is greater than or equal to the number of eigenvalues of $\mathbf{T}$ in $\mathbb{C}^-$. Hence, using Theorem 2, we have $\nu(\mathbf{T})\ge \pi_-(N)$.
In the general situation, the instability index can be greater than the number of eigenvalues	of the	 operator in the half-plane $\mathbb{C}^-$ (e.g., see \cite{14::Mil16}). However, in our case equality takes place.

\begin{theorem}
Suppose that $n_1n_2>0$; then the instability index $\nu(\mathbf{T})$ of Eq. \eqref{eq:6} is equal	to	 $\pi_-(N)$.	In particular, the problem is stable if and only if the numbers $n_1$ and $n_2$ are positive.
\end{theorem}
\begin{proof}
If $N >0$, then $\mathbf{J}$ is uniformly positive. Hence, the metric $(\mathbf{J}x,x)$ in $\mathfrak{H}$ is equivalent to the original one. The operator $T$ is maximal dissipative with respect to the metric $(\mathbf{J}x,x)$. Hence, it generates a contraction semigroup (see [17, Chap. 9]). Therefore, every solution of Eq. \eqref{eq:6} (understood as an equality in $\mathfrak{H}$) satisfies the condition $\|\mathbf{J}^{1/2}u(t)\|\le \|\mathbf{J}^{1/2}u(0)\|$ whenever $t > 0$. Hence, the problem is stable.

Suppose that $\pi_-(N) > 0$. By virtue of the inequality $\nu(\mathbf{T}) \ge \pi_-(N)$, the problem is unstable. Let us calculate the index. Let $\mathfrak{H}_-$ be the one- or two-dimensional linear span of the eigenvectors corresponding to the eigenvalues from $\mathbb{C}^-$. Let us show that $\mathfrak{H}_-$ is a $\mathbf{J}$-negative subspace.

Suppose that $\mathbf{x}\in \mathfrak{H}_-$ and $\mathbf{u}(t)$ is an elementary solution of Eq. \eqref{eq:6} such that $\mathbf{u}(0) = \mathbf{x}$ (if $\pi_-(n) = 1$, then	$\mathbf{u}(t)= e^{i\omega \lambda_0t}\mathbf{x}$, where	$\lambda_0\in \mathbb{C}^-$). Using (6), we get
\begin{multline*}
\frac{d}{d t}(\mathbf{J} \mathbf{u}, \mathbf{u})=(\mathbf{J} \dot{\mathbf{u}}, \mathbf{u})+(\mathbf{J} \mathbf{u}, \dot{\mathbf{u}})=(i \omega \mathbf{J} \mathbf{T} \mathbf{u}, \mathbf{u})+(\mathbf{u}, i \omega \mathbf{J} \mathbf{T} \mathbf{u})={}\\ {}=-2 \omega \operatorname{Im}(\mathbf{J} \mathbf{T} \mathbf{u}, \mathbf{u})=-2 \nu(R v, v)
\end{multline*}

It is	clear	that $\mathbf{u}(t)\to 0$ as	$t\to-\infty$. Integrating the last equality, we	get
$$
(\mathbf{J}\mathbf{x},\mathbf{x}) = - 2\nu\int_{-\infty}^0 (Rv(t) ,v(t))\, dt.
$$
Since $R > 0$, we have $(\mathbf{J}x,x) \le 0$. Suppose that $(\mathbf{J}x,x) = 0$; then $v(t)\equiv 0$. Using the matrix representation of Eq. \eqref{eq:5}, we easily get $\dot{w}(t)\equiv 0$ and $\dot{z}(t)\equiv 0$. Hence, the vector $\mathbf{x} = (z, w, 0)^t$ belongs to the null space of $\mathbf{T}$, i.e., corresponds to the zero eigenvalue. This implies $\mathbf{x} = 0$.

Let us consider the operator $\mathbf{J}^{-1}\mathbf{T}^*\mathbf{J} =\mathbf{ J}^{-1}\mathbf{L}^*$. It follows from the matrix representation of $\mathbf{L}$ that $\mathcal{D}(\mathbf{J}^{-1}\mathbf{T}^*\mathbf{J}) = \mathcal{D}(\mathbf{T})$. Obviously, the operator $\mathbf{J}^{-1}\mathbf{T}^*\mathbf{ J}$ is maximal $\mathbf{J}$-dissipative. By $\mathfrak{H}^*_-$ we denote the invariant subspace of $\mathbf{J}^{-1}\mathbf{T}^*\mathbf{ J}$, that is, the linear span of the eigenvectors corresponding to the eigenvalues from the upper half-plane. Arguing as above, we see that $\mathfrak{H}^*_-$ is $\mathbf{J}$-negative. Since $\mathfrak{H}^*_-$ is finite-dimensional, it is uniformly $\mathbf{J}-negative$. By virtue of the Pontryagin theorem ([13, Theorem 1]), the $\mathbf{J}$-orthogonal complement $\mathfrak{H}_+$ of $\mathfrak{H}^*_-$ is uniformly $\mathbf{J}$-positive. Obviously, the space $\mathfrak{H}_+$ is $\mathbf{T}$-invariant and the spectrum of the restriction $\mathbf{T}_+
  =\left. \mathbf{T}\right|_{\mathfrak{H}_+}$ lies in the closed upper half-plane. Hence, $\mathfrak{H}$ is the direct sum of two $\mathbf{T}$-invariant uniformly $\mathbf{J}$-definite subspaces $\mathfrak{H}_-$ and $\mathfrak{H}_+$. Hence, it is clear that Eq. \eqref{eq:6} in $\mathfrak{H}_+$ is stable. Thus, the instability index $\nu(\mathbf{T})$ is equal to $\dim \mathfrak{H}_- =\pi_(N)$.	
\end{proof}

In addition to Theorem 3, note that the instability of Eq. \eqref{eq:6} implies the instability of the top, i.e., the rigid part of the system. Namely, the following result is valid.

\begin{theorem}\label{thm:4}
Suppose that $\mathbf{u}(t) = (z(t), w(t), v(t))^t$ is an unbounded solution of Eq. \eqref{eq:6}; then $z(t)$ is also unbounded.
\end{theorem}
\begin{proof}
It follows from the proof of Theorem 3 that $\mathfrak{H}$ is the direct sum of two $\mathbf{T}$-invariant uniformly $\mathbf{J}$-definite subspaces: $\mathfrak{H}= \mathfrak{H}_+\dot{+}\mathfrak{H}_+$. In particular, any unbounded solution $\mathbf{u}(t)$ of Eq. \eqref{eq:6} has the form
$$
\mathbf{u}(t) =\mathbf{u}_+(t)+\mathbf{u}_-(t),\quad \text{where}\quad  (\mathbf{J}\mathbf{u}_-(t),\mathbf{ u}_-(t))\to-\infty \quad \text{as}\quad  t\to+\infty.
$$
Let
$$
\mathbf{u}_-(t) = (z(t), w(t), v(t)^t,\qquad 	 \mathbf{u}^0_-(t):=\mathbf{S}_0\mathbf{u}_-(t)=	 (z(t),w_0(t),v_0(t))^t.
$$
Hence, using \eqref{eq:7}, we get
\begin{multline*}
(\mathbf{J}\mathbf{u}_(t),\mathbf{u}_(t)) = (\mathbf{J}_0\mathbf{u}^0_-(t),\mathbf{u}^0_-(t)) ={}\\
{}=\omega^2 (Nz(t),z(t)) + ((A- B^* B) w_0(t), w_0(t)) + (v_0(t), v_0(t)).
\end{multline*}
Since the last two summands on the right side are positive and the sum tends to $-\infty$ as $t\to+\infty$, we see that $z(t)$ is an unbounded function.
\end{proof}

\begin{rem}
The upper bound for the number of eigenvalues of $\mathbf{T}_1 =\mathbf{J}^{-1}_1\mathbf{L}_1$ with negative imaginary part was obtained in \cite{14::Yur9}. It can be shown that the statements of Theorems 2 and 3 concerning the number of eigenvalues in $\mathbb{C}^-$ and the instability index are consequences of the results of \cite{14::Sh18}, \cite{14::Sh19}. But we prefer to give an independent proof here.
\end{rem}

\subsection{Further Properties of the Evolution Operator $\mathbf{T}$ of the System. The Spectrum at Large Viscosity}

\subsection*{The numerical range of $\mathbf{T}$ and properties of its eigenfunctions.} The range of the quadratic form $(Fx, x)$ for $x\in\mathcal{D}(F)$ and $\|x\| = 1$ is called the numerical range of an operator $F$. It is clear that the spectrum of an operator is a subset of its numerical range. Let us recall the concept of a basis for the Abel summability method of order a, which is due to Lidskii [20]. Suppose that $F$ has a discrete spectrum and all of its eigenvalues except for finitely many lie in the sector $|\arg\lambda -\pi/2|<\theta < \pi$. To be concise, suppose that the eigenvalues are simple. We say that the system $\{y_k\}$ of eigenfunctions of $F$ corresponding to the eigenvalues $\{\lambda_k\}$ is a basis for the Abel summability method of order $\alpha$, $\alpha<2\theta/\pi$, if for any $f\in\mathfrak{H}$ the series
$$
f(t)=\sum e^{(-i\lambda_k)^\alpha t}(f,z_k)y_k
$$
strongly converges in $\mathfrak{H}$ for all $t > 0$ (it is permitted to put groups of terms in brackets independently of $f$ and $t$) and $f(t)$ strongly tends to $f$ as $t\to 0$. Here $\{z_k\}$ is the set of eigenfunctions of the adjoint operator. It is biorthogonal to $\{y_k\}$, and $(-i\lambda_k)^\alpha$ is the main branch of $\lambda^\alpha$. Obviously, every basis for the Abel summability method is a complete system in $\mathfrak{H}$.

As before, by $\mathbf{T}_\pm$ we denote the restriction of the generator $\mathbf{T}$ of Eq. (6) to (see Theorem 2).

\begin{theorem}
The numerical range of $\mathbf{T}_+$ lies in the half-strip $|\Re\lambda| < c_0$, $\Im\lambda > -c_1$ for some constants $c_0$, $c_1$> 0. The eigenvalues $\lambda_k = \lambda_k(\mathbf{T})$ satisfy the condition
\begin{equation}\label{eq:8}
\lambda_{k}=i\left(\frac{3 \pi^{2} k}{\operatorname{mes}(\Omega)}\right)^{2 / 3} \frac{\nu}{\omega}(1+o(1)), \quad k=1,2,3, \ldots
\end{equation}
and lie in a half-strip of the same form. The spectrum is symmetric about the imaginary cuds. The set of all eigenvectors of $\mathbf{T}$ is a basis for the Abel summability method	of order a	whenever	$\alpha > 1/2$.	 In particular, if $\alpha= 1$ and $t > 0$, then any solution of Eq. (6) is represented by a	convergent series in	the eigenfunctions of $\mathbf{T}$. The operator $i\mathbf{T}$ generates a holomorphic semigroup in $\mathfrak{H}$.
\end{theorem}
\begin{proof}
Suppose that $\mathbf{Q}_+$ is the orthogonal projection onto $\mathfrak{H}_+$	 and $\mathbf{J}_+ =\mathbf{	 Q}_+\mathbf{J}\mathbf{Q}_+$.	Since $\mathfrak{H}_+$	is
uniformly $\mathbf{J}$-positive, it follows that $\mathbf{J}_+$ is uniformly positive in $\mathfrak{H}_+$.	If $\mathbf{x}\in\mathfrak{H}_+$,	then
\begin{equation}\label{eq:9}
(\mathbf{T}\mathbf{x}, \mathbf{x}) = (\mathbf{J}_+\mathbf{T}\mathbf{x}, \mathbf{J}^{-1}_+\mathbf{x}) =
(\mathbf{JTx}, \mathbf{J}^{-1}_+\mathbf{x}) = (\mathbf{Lx}, \mathbf{J}^{-1}_+\mathbf{x}) = (\mathbf{Q}_+\mathbf{LQ}_+\mathbf{x}, \mathbf{J}^{-1}_+\mathbf{x}).
\end{equation}
Note that if $C$ is uniformly positive, then the real form $(Fx, x$) is bounded (semi-bounded) if and only if so is the form $(Fx, Cx)$. It follows from the matrix representation of $\mathbf{L}$ that the real part of $(\mathbf{Lx}, \mathbf{x})$ is bounded and the imaginary part is semibounded. Hence, so is form (9). Therefore, the numerical range of $\mathbf{T}_+$ lies in a half-strip. The following estimate for the resolvent holds outside the numerical range $\mathbf{W}(\mathbf{T}_+)$:
\begin{equation}\label{eq:10}
\|(\lambda-\mathbf{T}_+)^{-1}\| \le  l/d(\lambda, W(\mathbf{T}_+)),
\end{equation}
where $d$ is the distance from $\lambda$ to $W(\mathbf{T}_+$. In particular, if $\lambda$ asymptotically lies outside a small sector containing the imaginary axis, then the resolvent is of maximal decay ($\le c|\lambda|^{-1}$). It follows from the well- known results of semigroup theory (see [17, Chap. 9]) that $i\mathbf{T}_+$ is a generator of a holomorphic semigroup in $\mathfrak{H}_+$. Since $\mathbf{T} =\mathbf{ T}_+ + \mathbf{T}_-$, so is $i\mathbf{T}$. Here $\mathbf{T}_-$ is finite dimensional.

Let $\mathbf{R} = \operatorname{diag}(I,I, R)$, where $R$ is the Stokes operator. We have
$$
\mathbf{J} =\mathbf{I} + \mathbf{K},\quad \mathbf{L}= (i\nu\omega^{-1}\mathbf{I}+ \mathbf{V})\mathbf{R},\quad
\mathbf{V} = (\mathbf{L}-i\nu\omega^{-1}\mathbf{R})\mathbf{R}^{-1},
$$
where $\mathbf{K}$ is finite-dimensional and $\mathbf{V}$ is compact. Hence, $\mathbf{T} =i\nu\omega^{-1}(\mathbf{I} + \mathbf{V}_1)\mathbf{R}$, where $\mathbf{V}_1$ is a compact operator.

The eigenvalues of $\mathbf{R}$ and $R$ coincide, except for one, $\lambda= 1$. The asymptotics for the eigenvalues of $R$ is known \cite{14::Kop5}. It is given in (8) up to the coefficient $\nu/\omega$. The Keldysh--Gokhberg--Krein theorem (see \cite[Chap. 5]{14::Go12}) implies a similar asymptotics for the compact perturbation $\mathbf{T}$ of $\mathbf{R}$. Since the matrices $A$, $N$, $iH$ are real, it follows that the spectrum of the pencil $\lambda \mathbf{J} - \mathbf{L}$ or of $\mathbf{T}$ is symmetric about the imaginary axis.

Now it suffices to prove that the set of eigenvectors of the operator $\mathbf{T}_+$ acting in $\mathfrak{H}_+$ is a basis for the Abel summability method. We proved that the estimate (10) for the resolvent holds outside a half-strip. Obviously, the $s$-numbers of the operators $\mathbf{T}^{-1}$ and $\mathbf{R}^{-1}$ satisfy the condition $s_k(\mathbf{T}) \sim k^p$, $p = 2/3$ (see \cite[Chap. 2]{14::Go}). Thus, it readily follows from the result of \cite{14::Smir21} that the system is a basis for the Abel summability method of order $a > p^{-1} = 1/2$.	
\end{proof}

The spectrum of $\mathbf{T}$ for large viscosity. It is natural that the spectrum of the evolution operator approaches the spectrum of the small oscillations of the body with frozen fluid as $\nu\to\infty$.

The equation of motion for the small oscillations of the body with frozen fluid is the projection of Eq. (5) onto the subspace $\mathbb{C}^3\times \mathbb{C}^3 \times {0}$ of $\mathfrak{H}$. Hence, the spectrum of the corresponding problem coincides with that of the matrix pencil
$$
V(\lambda)=\lambda\left(\begin{array}{cc}{I} & {0} \\ {0} & {A}\end{array}\right)-\left(\begin{array}{cc}{H} & {\omega^{-1} H} \\ {k \omega^{-1} H} & {H\left(A-a_{0} I\right)}\end{array}\right)=: \lambda W^{(s)}-M^{(s)}
$$
where $W^{(s)}$ and $M^{(s)}$ are the projections of $\mathbf{W}$ and $\mathbf{M}$ onto the ``rigid'' part
$\mathbb{C}^3\times \mathbb{C}^3$ of $\mathfrak{H}$. Obviously, zero is a double eigenvalue of the pencil $V(\lambda)$. Consider the operators $\mathbf{S}_1$, $\mathbf{J}_1$, $\mathbf{L}_1$, defined in Remark 1. Note that the eigenvalues of the pencils $\mathbf{V}(\lambda) =\lambda\mathbf{W}-\mathbf{M}$ and $\mathbf{V}_1(\lambda) = \lambda\mathbf{J}_1-\mathbf{L}_1$ coincide. The eigenvalues of their projections $\mathbf{V}(\lambda)$ and
\begin{multline*}
V_{1}(\lambda)=V(\lambda) S_{1}^{(s)}={}\\ {}=\lambda\left(\begin{array}{cc}{\left(\omega^{2} N\right)^{-1}} & {A(\omega N)^{-1}} \\ {A(\omega N)^{-1}} & {A(A+N) N^{-1}}\end{array}\right)-\left(\begin{array}{cc}{0} & {-\omega^{-1} H} \\ {-\omega^{-1} H} & {-a_{0} H}\end{array}\right)=: \lambda J_{1}^{(s)}-L_{1}^{(s)}
\end{multline*}
onto $\mathbb{C}^3\times \mathbb{C}^3$ also coincide.

\begin{theorem}\label{thm:6}
Suppose that $n_1n_2\ne0$ and $\nu\to\infty$. Then zero is a double eigenvalue of $\mathbf{T}$; four eigenvalues approach the nonzero eigenvalues of the pencil $V(\lambda)$, which corresponds to the rigid part of the system. The degree of approximation is estimated by $O(\nu^{-1/2})$ if the four eigenvalues are simple and by $O(\nu^{-1/4})$ if they are double. The spectrum of $V(\lambda)$ and $V_1(\lambda)$ is symmetric about the real and imaginary axes. All other eigenvalues approach infinity and satisfy

\begin{equation}\label{eq:11}
\lambda_k=\nu\omega^{-1}\lambda_k^{(f)}(l + o(l))\quad as \nu\to\infty,
\end{equation}
where the $\lambda_k^{(f)}$ are the nonzero eigenvalues of the pencil $\lambda\left(\begin{array}{cc}{A} & {B^{*}} \\ {B} & {I}\end{array}\right)-i\left(\begin{array}{cc}{0} & {0} \\ {0} & {R}\end{array}\right) \text { in } \mathbb{C}^{3} \times J_{0}(\Omega)$.
\end{theorem}

\begin{proof}
Since $V_1(\lambda)$ is self-adjoint, it follows that the spectrum of $V_1(\lambda)$ and $V(\lambda)$ is symmetric about the real axis. Since the matrices $A$, $N$, $iH$ are real, it follows that the spectrum of $V_1 (\lambda)$ is symmetric about the imaginary axis.

Let $\mu=\omega/\nu$ and $\rho=\mu^{1/2}$. Multiplying the pencil $\lambda\mathbf{J}_1 - \mathbf{L}_1$ by the operator $\operatorname{diag}(I,I, \mu^{1/2}R^{-1/2})$ on the left and on the right, we see that the eigenvalues of $\mathbf{T}$ coincide with those of the pencil
\begin{equation}\label{eq:12}
\mathbf{S}_{\rho}(\lambda)=\lambda\left(\begin{array}{cc}{J_{1}^{(s)}} & {\rho C_{1}^{*}} \\ {\rho C_{1}} & {\rho^{2} C}\end{array}\right)-\left(\begin{array}{cc}{L_{1}^{(s)}} & {0} \\ {0} & {\rho^{2} G+i I}\end{array}\right)
\end{equation}
where $C$ and $C_1$ are some unbounded operators. Using the Rellich--Kato theorem (see \cite[Chap. 7]{14::Kato17}), we see that if $\rho$ is small ($\nu$ is large), then the eigenvalues of the pencil $\mathbf{S}\rho(\lambda)$ lie in neighborhoods of the eigenvalues of $\mathbf{S}_0(\lambda)$, i.e., in neighborhoods of zero, infinity, and four nonzero eigenvalues of $\mathbf{V}_1(A)$.

In the general case, the nonzero eigenvalues are simple and their perturbations depend analytically on $\rho$. Hence, the degree of approximation is proportional to $\rho= \mu^{1/2}$. If the non-perturbed eigenvalues are double,	then they	are	represented by the Puiseux series in powers	of	$\rho^{1/2}$.	Therefore,	the	degree of approximation is	 estimated by $O(\mu^{1/4})$. It follows from the symmetry	of	the spectrum that	there	 are no triple or quadruple eigenvalues.

Note that
$$
\rho(\lambda\mathbf{W} -\mathbf{M}) =\rho\lambda\mathbf{W} - i\operatorname{diag}{0, 0, R} + \mathbf{K}(\rho),
$$
where $\mathbf{K}(\rho) = O(\rho)$. Using the Rellich-Kato theorem again, we get (11).
\end{proof}

The real eigenvalues of $\mathbf{V}(\lambda)$ are of particular interest. How do the eigenvalues of $\mathbf{T}$ approach them as $\nu\infty$? The following theorem gives the answer.

\begin{theorem}\label{thm:7}
Suppose that $\lambda_0$ is a simple real eigenvalue of the pencil $\mathbf{V}(\lambda)$ and $h_0= (z_0,w_0)^t$ is the corresponding eigenvector. Then the eigenvalue $\lambda$ of $\mathbf{T}$ close to $\lambda_0$ satisfies
\begin{equation}\label{eq:13}
\lambda=\lambda_{0}+c i \mu+O\left(\mu^{3 / 2}\right), \quad c=\frac{\omega \lambda_{0}^{2}\left(B^{*} R^{-1} B w_{0}, w_{0}\right)}{\left(J^{(s)} h_{0}, h_{0}\right)}
\end{equation}

Here $J^{(s)}$ is the restriction of $\mathbf{J}$ to $\mathbb{C}^3\times \mathbb{C}^3$. Hence, $\lambda$ is in the upper or lower half-plane when $(J^{(s)}h_0, h_0)$ is positive or negative, respectively.
\end{theorem}
\begin{proof}
It follows from the representation (12) for $\mathbf{S}_\rho(\lambda)$ and the Rellich-Kato theorem that the eigenvalue $\lambda$ and the corresponding eigenvector $\mathbf{f}$depend analytically on $\rho$. Namely,
$$
\lambda=\lambda_0 +c_1\rho +c_2\rho^2 +\ldots,\quad  \mathbf{f} =\mathbf{ f}_0 + \rho\mathbf{f}_1 +\rho^2\mathbf{f}_2 +\ldots
$$
Let $\mathbf{f}_j = (g_j,r_j)^t$, $j = 1,2,3$, where the $g_j$ are the projections of $\mathbf{f}_j$, onto $\mathbb{C}^3 \times\mathbb{C}^3$. Let us write out lower-order terms in the equation $\mathbf{S}_\rho(\lambda)\mathbf{f} = 0$ with respect to the powers of $\rho$. Writing out the first two terms, we get
$$
\begin{array}{c}{\lambda_{0} J_{1}^{(s)} g_{0}-L_{1}^{(s)} g_{0}=0, \quad r_{0}=0} \\ {\lambda_{0} C_{1} g_{0}-i r_{1}=0, \quad c_{1} J_{1}^{(s)} g_{0}+\lambda_{0} J_{1}^{(s)} g_{1}-L_{1}^{(s)} g_{1}-C_{1} . r_{0}=0}\end{array}
$$
Hence $c_1(J^{(s)}_1g_0,g_0) = 0$. It is known that any eigenvector of the pencil $V_1(\lambda)$ corresponding to a simple real eigenvalue is $J^{(s)_1}$-definite. Hence $c_1= 0$. Further, the coefficient of $\rho^2$ is equal to
$$
c_{2} J_{1}^{(s)} g_{0}+\lambda_{0} C_{1}^{*} r_{1}+\lambda_{0} J_{1}^{(s)} g_{2}-L_{1}^{(s)} g_{2}=0.
$$
Hence,
\begin{equation}\label{eq:14}
c_2(J^{(s)}_1g_0,g_0)-i\lambda^2_0(C^*_1C_1g_0,g_0)=0.
\end{equation}
Obviously, the eigenvector $h_0= (z_0,w_0)^t$	of $V(\lambda)$ is equal to	 $S^{(s)}_1g_0$. From the representation of	$C_1$ in (12), using the equalities $\mathbf{WS}_1=\mathbf{J}_1$ and	 $W^{(s)}S^{(s)}_1=J^{(s)}_1$, we	 get
$$
C_1g_0=R^{-1/2}W^{(s)}h_0 = R^{-1/2}Bw_0.
$$
Now it follows from the formula $J^{(s)}_1=(S^{(s)}_1)^*J^{(s)}S^{(s)}_1$(cf. (7)) and Eq. (14) that $c_2 = ic$ in (13).
\end{proof}

\begin{rem}
The sign of $J^{(s)}_1g_0,g_0 = (J^h_0, h_0)$ determines the type of eigenvalues $\lambda_0$ of the self- adjoint pencil $V_1(\lambda)$. It follows from the	symmetry that the	eigenvalues	$\lambda_0$ and $-\lambda_0$ are of the	same type. If there are exactly two real eigenvalues, then they are	of positive	 type, because otherwise	 there
would be three eigenvalues of $\mathbf{T}$ in the lower half-plane for large $\nu$, which is impossible. Suppose that there are four real eigenvalues, i.e., $n_1n_2 > 0$; then these eigenvalues are of positive type if $n_1$ and $n_2$ are positive. Otherwise, one symmetric pair has the negative type. In particular, the conclusion of Theorem 2 concerning the number of eigenvalues of $\mathbf{T}$ in the lower half-plane readily follows from Theorem 7 and from the fact that there are no nonzero real eigenvalues.
\end{rem}

\begin{rem}
The operator $B^*R^{-1}B$ is evaluated in \cite{14::Cher3} for some forms of the cavities. Formula (13) is obtained in \cite{14::Cher3} and \cite{14::Smir21} for $k = 0$, i.e., when there is no gravity.
\end{rem}

We have supposed that $n_1n_2\ne 0$. However, it is possible to investigate the behavior of the eigenvalues of $\mathbf{T}$ without this assumption. Also, it is interesting to observe the transition of the eigenvalues from the upper half-plane to the lower half-plane as $\omega$ passes through a critical value.

\begin{proposition}\label{prop:5}
Take $\omega$	such that $n_1= 0$ and $n_2\ne0$. Then	zero	is	a triple semisimple	eigenvalue	of
the operator $\mathbf{T}= \mathbf{T}(w)$. If e is sufficiently small, then $\mathbf{T}(\omega + \varepsilon)$	 has	a	simple	eigenvalue	 $\lambda(\varepsilon)$	satisfying
$$
\lambda(\varepsilon)=\frac{2 k \varepsilon}{\omega^{2}\left(D^{-1} B e_{1}, B e_{1}\right)}+O\left(\varepsilon^{2}\right) \quad \text { as } \varepsilon \rightarrow \pm 0.
$$
Here $e_1$ is the unit vector defined in \S1. The number $(D^{-1}Be_1, Be_1)$ is pure imaginary.
\end{proposition}
\begin{proof}
Arguing as in the proof of Theorem 2, we see that zero is a triple semi-simple eigenvalue. It follows from (7) that the eigenvalues of $\mathbf{T}$ and those of the pencil $\lambda\mathbf{J}_0-\mathbf{L}_0$ coincide. Here $\mathbf{L}_0 = (\mathbf{S}^*_0)^{-1}\mathbf{L}\mathbf{S}^{-1}_0$. It is easy to write out the expansions of $\mathbf{J}_0$ and $\mathbf{L}_0$ in powers of $\varepsilon$. Using the first two terms of the expansions, it is easy to show that $(GB_1, Be_1)=0$. Thus, $\lambda(\varepsilon)$ passes through zero along the imaginary axis (here we use the symmetry). Arguing as in Theorem 7, we can conclude the proof. Therefore, we omit the details.
\end{proof}

\subsection{The Symmetric Top}
\subsection*{Invariant subspaces.} If symmetry occurs, then it is possible to single out a family of $\mathbf{T}$-invariant subspaces and to test the stability in them. Let $Ox_0$ be the axis of a $p$th-order symmetry. In other words, the top is invariant with respect to the rotation by the angle of $2\pi/p$ about the axis $Ox_0$. Further, assume that $p> 2$. If the top is a solid of revolution, then we set $p =\infty$. Let us introduce the operators
$$
U_{s}=\left(\begin{array}{ccc}{1} & {0} & {0} \\ {0} & {\cos 2 \pi s / p} & {\sin 2 \pi s / p} \\ {0} & {-\sin 2 \pi s / p} & {\cos 2 \pi s / p}\end{array}\right), \quad \mathbf{U}_{s}=\left(\begin{array}{ccc}{U_{s}} & {0} & {0} \\ {0} & {U_{s}} & {0} \\ {0} & {0} & {V_{s}}\end{array}\right)
$$
where $V_s$, is defined by $V_sv = U_s(v(U_{-s}x))$. If $p =\infty$, then we substitute $2\pi$ for $p$ in these equalities.

\begin{proposition}\label{Prop:6}
If $Ox_0$ is the axis of a $p$th-order symmetry for the top, then
\begin{equation}\label{eq:15}
AU_s= U_sA,\quad  HU_s = U_sH,\quad  B^*V_s= V_sB^*,\quad 	V_sD = DV_s
\end{equation}
for all $s\in\mathbb{Z}$ (if $p =\infty$, then $s\in \mathbb{R}$). The operator $\mathbf{T}$ commutes with $U_s$ . The operators
$$
\mathbf{Q}_{j}=\frac{1}{p} \sum_{l=0}^{p-1} e^{2 \pi i l j / p} \mathbf{U}_{l} \quad\left(\mathbf{Q}_{j}=\frac{1}{2 \pi} \int_{0}^{2 \pi} e^{i j x} \mathbf{U}_{x} d x \quad \text { if } p=\infty\right)
$$
are orthoprojections in $\mathfrak{H}$, and $\sum_{j=0}^{p-1}\mathbf{Q}_j =\mathbf{I}$ (if $p =\infty$, then $\sum_{j=-\infty}^\infty\mathbf{Q}_j=\mathbf{I}$). The subspaces $\mathfrak{H}_j = Q_j(\mathfrak{H})$ are $\mathbf{T}$-invariant.
\end{proposition}
\begin{proof}
Equations (15) follow from the definition (see [8, Lemma 1]). The other conclusions are direct consequences of these equations. Note that if $p =\infty$, then we use the notion of the Pettis integral in the definition of $\mathbf{Q}_j$.
\end{proof}

It is readily seen that $\mathbf{Q}_j=\mathbf{Q}_{j-p}$. In the sequel, it is convenient to consider $\mathbf{Q}_{p-1}$ instead of $\mathbf{Q}_1$.

\begin{theorem}\label{thm:8}
If there is a $p$th-order symmetry ($2 < p <\infty$), then Eq. (6) in $\mathfrak{H}$ is stable whenever $j\ne\pm 1$. If $n_1(=n_2) < 0$, then the instability index of the restriction of the equation to $\mathfrak{H}_\pm1$ is equal to 1.
\end{theorem}
\begin{proof}
To be definite, suppose that $p < \infty$. If $p =\infty$, the modifications are obvious. The projections $\mathbf{Q}_j$ are equal to $\operatorname{diag} (R_j, R_j, Q_j)$, where
$$
R_{-1}=\frac{1}{2}\left(\begin{array}{ccc}{0} & {0} & {0} \\ {0} & {1} & {i} \\ {0} & {-i} & {1}\end{array}\right), \quad R_{0}=\left(\begin{array}{ccc}{1} & {0} & {0} \\ {0} & {0} & {0} \\ {0} & {0} & {0}\end{array}\right)
$$
\centerline{$R_1 = \overline{R_{-1}}$,\quad  $R_j = 0$ whenever $j\ne 0,\pm 1$.}

Here $Q_j$ is the orthoprojection of $J_0(\Omega)$ onto the subspace of vector fields $v(x)$ such that after rotation by an angle of $2\pi/p$ about the axis $Ox_0$ they coincide with $e^{-2\pi ij/p}v(x)$. Hence, if $j\ne0,\pm 1$, then Eq. (6) in $\mathfrak{H}_j = \{0\}\times\{0\}\times Q_j(J_0(\Omega))$ is equivalent to $\dot{v} = i\omega Dv$. Since $D$ is dissipative, it follows that the equation is stable.

The subspace $\mathfrak{H}_0$ is equal to $\mathbb{C}e_0\times \mathbb{C}e_0\times Q_0(J_0(\Omega)$. It follows from the relation $R_{0} B^{*}=B^{*} Q_{0}$ that $\mathfrak{H}_0$ is $\mathbf{L}$-invariant. We have $\mathbf{J}=\mathbf{S}_{0}^{*} \mathbf{J}_{0} \mathbf{S}_{0}$. Note that
$$
\mathrm{S}_{0}\left(\mathfrak{H}_{0}\right) \subset \mathbb{C} e_{0} \times \mathbb{C} e_{0} \times J_{0}(\Omega)=: \mathfrak{H}_{0}^{\prime}
$$
and $\mathbf{J}_0$ is positive in $\mathfrak{H}'_0$. Hence, $\mathbf{J}$ is positive in $\mathfrak{H}_0$. Therefore, $\mathbf{T} = \mathbf{J}^{-1}\mathbf{L}$ is maximal dissipative with respect to the inner product $(\mathbf{Jx}, \mathbf{x})$ in $\mathfrak{H}_0$, which is equivalent to the original one. Thus, Eq. (6) is stable in $\mathfrak{H}_0$.

Since $\mathbf{T\overline{u}} = -\mathbf{\overline{Tu}}$ (the bar denotes complex conjugation) and $\mathfrak{H}_{-1}$ is the complex conjugate of $\mathfrak{H}_{-1}$, it follows that the instability index of (6) in $\mathfrak{H}_1$ is equal to that in $\mathfrak{H}_{-1}$. If $n_1< 0$, then the instability index in $\mathfrak{H}_1$ is neither zero nor greater than one (otherwise, we would have a contradiction with Theorem 3). Hence, it is equal to 1.
\end{proof}

\subsection*{Connection with Sobolev's equation.} In [1], Sobolev investigated the case of a symmetric top and ideal fluid ($\nu=0$). He introduced the operator $\mathbf{F}$ in $\mathbb{C} \times \mathbb{C} \times J_{0}(\Omega)$ defined by
$$
\mathbf{F}(z, w, v)=\left(z^{1}, w^{1}, v^{1}\right).
$$
Here the numbers $z^1$, $w^1$ and the function $v^1 = (v^1_0,v^1_1,v^1_2)\in  \mathbf{J}_0(\Omega)$ satisfy the equations
$$
\begin{aligned} v_{0}^{1}=i \frac{\partial p}{\partial x_{0}}, & v_{2}^{1}=-2 \omega i v_{1}+i \frac{\partial p}{\partial x_{2}}+2 \omega w \frac{\partial \bar{\chi}}{\partial x_{1}} \\ v_{1}^{1}=2 \omega i v_{2}+i \frac{\partial p}{\partial x_{1}}-2 \omega w \frac{\partial \bar{\chi}}{\partial x_{2}}, & z^{1}=\omega w, \quad \operatorname{div} v=\operatorname{div} v^{1}=0 \\ v_{n}^{1} | \partial \Omega=0, & A_{1} w^{1}=A_{2} w+A_{3} z+\int_{\Omega}\left(v_{2} \frac{\partial \chi}{\partial x_{1}}-v_{1} \frac{\partial \chi}{\partial x_{2}}\right) d \Omega \end{aligned}
$$

The function $\chi$ is defined by
$$
\Delta \chi=0,\left.\quad \frac{\partial \chi}{\partial n}\right|_{\partial \Omega}=x_{0}\left(\cos n x_{2}+i \cos n x_{1}\right)-\left(x_{2}+i x_{1}\right) \cos n x_{0}.
$$
It can be readily shown that $\chi$ satisfies the equality
\begin{equation}\label{eq:16}
P_{o}\left(\frac{\partial \chi}{\partial x_{2}},-\frac{\partial \chi}{\partial x_{1}}, 0\right)^{t}=(B H-G B) e, \quad e=\frac{1}{\sqrt{2}}(0,1, i)^{t}
\end{equation}
and $\mathbf{F}$ is the restriction of $\mathbf{T}_0 =\mathbf{S}_0\mathbf{T}\mathbf{S}_0^{-1}$ to $\mathfrak{H}$ (see [8, 9]). Hence, if $\nu = 0$, then we can consider $\mathbf{T}_0$ instead of Sobolev's operator $\mathbf{F}$. We do not know the definition of Sobolev's operator if the fluid is viscous. The obstacle is that we cannot substitute $D$ for $G$ in (16) because the vector Be does not belong to the domain of the Stokes operator $R$. Indeed, suppose that $Be\subset\mathcal{D}(R)$; then $Be = [e, x] + \nabla p$. Hence,
$$
R B e=\Delta([e, x])+\nabla \Delta p+\nabla q=\nabla(\Delta p+q).
$$
Here $p$ and $q$ are smooth functions. Thus, $P_0RBe = RBe = 0$. This contradicts the condition $>>0$.

The content of this section is base on the paper by  A.~G.~ Kostyuchenko, 
A.~ A.~Shkalikov and  M.~ Yu.~~Yurkin, 

\medskip
\renewcommand{\refname}{{\large\rm Bibliography for Section~14}}

%%%%%%%%%%%%%%%%%%%%%%%%%%%%%%%%%%%%%%%%%%%%%%%%%%%%%%%%%%%%%%%%%%%%%%%%%%%%%%%%%%%%%%%%%%%%%%%%%%%%%%%%
%%%%%%%%%%%%%%%%%%%%%%%%%%%%%%%%%%%%%%%%%%%%%%%%%%%%%%%%%%%%%%%%%%%%%%%%%%%%%%%%%%%%%%%%%%%%%%%%%%%%%%%%

\end{document}